\documentclass[12pt]{article}
\usepackage{latexsym}
\usepackage{epsfig}
\usepackage{amsthm,amsfonts,amssymb,amsmath}
\usepackage{stmaryrd}
\usepackage[utf8]{inputenc}
\date{}
\title{Minimax and adaptive tests for detecting abrupt and possibly transitory changes in a Poisson process}

\author{Magalie Fromont\\
\small{{\em magalie.fromont@univ-rennes2.fr}}\\
Fabrice Grela\\
\small{{\em fabrice.grela@univ-rennes2.fr}}\\
\small{ and } \\
Ronan Le Gu\'evel\\
\small{{\em ronan.leguevel@univ-rennes2.fr}}\\
\small{{\em Univ. Rennes, Equipe de Statistique Irmar, UMR CNRS 6625}}\\
\small{{\em Place du Recteur Henri Le Moal, CS 24307, 35043 RENNES Cedex, France}}}

\RequirePackage{amsthm,amsmath,amsfonts,amssymb}
\RequirePackage[numbers]{natbib}
\RequirePackage[colorlinks,citecolor=blue,urlcolor=blue]{hyperref}
\RequirePackage{graphicx}
\usepackage{mathrsfs} 

\usepackage[utf8]{inputenc}
\usepackage[T1]{fontenc}
\usepackage[english]{babel}
\usepackage{dsfont}

\theoremstyle{plain}
\newtheorem{lemma}{Lemma}

\newtheorem{proposition}[lemma]{Proposition}
\newtheorem{corollary}[lemma]{Corollary}

\theoremstyle{remark}

\newcommand{\R}{\mathbb{R}} 
\newcommand{\N}{\mathbb{N}} 

\newcommand{\calS}{\mathcal{S}}

\newcommand{\calD}{\mathcal{D}}

\newcommand{\calT}{\mathcal{T}}

\newcommand{\bbL}{\mathbb{L}}

\renewcommand{\a}{\alpha}

\newcommand{\pa}[1]{\left( \left. #1 \right. \right)} 
\newcommand{\cro}[1]{\left[ \left. #1 \right. \right]} 
\newcommand{\set}[1]{\left\{ \left. #1 \right. \right\}}

\newcommand{\abs}[1]{\left\lvert #1 \right\rvert}

\renewcommand{\P}{\mathbb{P}}

\newcommand{\1}[1]{\ensuremath{\mathds{1}_{\left\{ #1 \right\}}}}	
\newcommand{\un}[1]{\ensuremath{\mathds{1}_{ #1}}}	

\newcommand{\SR}{{\rm{SR}}}

\newcommand{\SRb}{{\rm{SR}_\beta}}
\newcommand{\mSRab}{{\rm{mSR}_{\alpha,\beta}}}

\newcommand{\hzero}{\pa{H_0}}
\newcommand{\hone}{\pa{H_1}}

\newcommand{\bbul}{\boldsymbol{\cdot}}

\begin{document}
 \maketitle

\begin{abstract}
Motivated by applications in cybersecurity and epidemiology, we consider the problem of detecting an abrupt change in the intensity of a Poisson process, characterised by a jump (non transitory change) or a bump (transitory change) from constant.
We propose a complete study from the nonasymptotic minimax testing point of view, when the constant baseline intensity is known or unknown. The question of minimax adaptation with respect to each parameter (height, location, length) of the change is tackled, leading to a comprehensive overview of the various minimax separation rate regimes. We exhibit three such regimes and identify the factors of the two phase transitions, by giving the cost of adaptation to each parameter. For each alternative hypothesis, depending on the knowledge or not of each change parameter, we propose minimax or minimax adaptive tests based on linear statistics, close to CUSUM statistics, or quadratic statistics more adapted to the $\bbL_2$-distance considered in our minimax criteria and typically more powerful in practice, as our simulation study shows. When the change location or length is unknown, our adaptive tests are constructed from a scan aggregation principle combined with Bonferroni or min-$p$ level correction, and a conditioning trick when the baseline intensity is unknown.
\end{abstract}

%
%

\section{Introduction}

As explained in the introduction of the book of Daley and Vere-Jones \cite{DaleyVJ}, historically the theory of point processes seems to emerge with the study of the first life tables and renewal processes, and of counting problems in the research of Poisson \cite{Poisson}. Since recently,  point processes are largely deployed in the epidemiology, genetics, neuroscience and communications engineering literature. At the origin of this work, we were actually interested in some applications in public health and healthcare surveillance, where a point process on a bounded interval may represent occurrences of a medical event in a particular context, and in cyber security, where it may represent a packet or session arrival process in internet traffic or occurrences of certain cyber attacks or intrusions. In these contexts, being able to define conditions for abnormal behaviours to be detectable and to detect such anomalies as efficiently as possible is of particular importance. 

\paragraph*{Change detection in a Poisson process model} Despite rather widespread debates regarding the real nature of the point process that can model observations in the above applications, the Poisson process model is the most frequently encountered in the dedicated articles, probably due to its convenient theoretical properties as well as its ability to fit the data. 
An abrupt change in the intensity of the Poisson process may reveal a significant health phenomenon when the process models epidemiological data (see \cite{Sonesson02} for a review), malicious activity or intrusion attempt when it models packet or session arrival processes in internet traffic (see \cite{Polunchenko2012}, \cite{Cao2003}, \cite{Karagiannis2004}, \cite{Vishwanath2009} or \cite{Soltani2008}), or a change of attack pattern when it models the occurrences of cyber attacks against a cyber system (see \cite{Daras2014} and \cite{Holm2013}). Another cyber security problem, considered in \cite{Soltani2017}, \cite{Soltani2020} and \cite{Wang2018cyber}, concerns communication over Poisson packet channels. In such a channel, an authorised transmitter sends packets to an authorised receiver according to a Poisson process, and a covert transmitter wishes to communicate some informations to a covert receiver on the same channel without being detected by a watchful adversary. Different models of covert transmissions have been studied by authors, treating the cases where the covert transmitter is restricted to packet insertion  or where he or she can only alter the packet timing by slowing down the incoming process to a lower rate to convey the information. The question of detectability of such covert transmissions, translated as a question of detectability of a bump in the Poisson process intensity, is clearly related to the testing minimax point of view adopted here and described below. Considering a Poisson process observed on a bounded fixed interval, we are thus interested in the problem of detecting an abrupt change in its distribution, characterised by a jump or a bump in its intensity.

This problem comes within the much more general framework of statistical change-point analysis. In view of the long history, going back to the 1940-1950's with the seminal works of Wald \cite{Wald1945}, Girshick and Rubin \cite{Girshick1952}, Page \cite{Page1954}, Fisher \cite{Fisher1958}, and the extensive literature on change-point analysis, we can not pretend to present a comprehensive state of the art. Detailed overviews will be found in the monographs of Basseville and Nikirov \cite{Basseville1993}, Carlstein et al. \cite{Carlstein1994}, Csörgö and Horváth \cite{Csorgo1997}, Brodsky and Darkhovsky \cite{Brodsky2013, Brodsky2013bis},  Tartakovsky et al. \cite{Tartakovsky2014}, and a structured and annotated bibliography in the paper by Lee \cite{Lee2010}.

Statistical change-point problems can essentially be classified into two main classes, depending on whether they are formulated as on-line or off-line change-point problems. 

On-line change-point analysis, also referred to as sequential analysis or disorder problems, generally deals with time sequences of random variables or stochastic processes, and aims at constructing a stopping time as close as possible to an unknown time of disorder or change in the distribution. For presentations of the most common performance measures and optimisation criteria used to this end, see for instance \cite{Lai2001}, \cite{Moustakides2008} or \cite{Polunchenko2012}, and references therein.

Off-line change-point analysis, also referred to as a posteriori change-point analysis, in fact raises two distinct questions : the one of detecting a given number of change-points or estimating the change-points number, and the one of estimating some or all the parameters of such change-points (jump locations and/or heights), once detected. 

Though most of these questions can be, as explained in \cite{Niu2016}, formulated or interpreted as single or multiple hypotheses testing problems, since they are usually all treated together, rather few attention seems to be paid to the testing performances themselves: detection rates results are not always explicitly stated and well formalised in the literature.

\paragraph*{A nonasymptotic minimax testing point of view}

Our work, which focuses on the question of detecting a jump or a bump in the intensity of a Poisson process, precisely aims at proposing a nonasymptotic minimax testing set-up and a guided progressive approach to construct minimax and minimax adaptive  detection procedures. It can thus also be viewed as a necessary preliminary step towards a further rigorous minimax study of multiple testing procedures designed for change-point localisation as in \cite{Niu2016}.

Let us consider a (possibly inhomogeneous) Poisson process $N=(N_t)_{t\in[0,1]}$ observed on the interval $[0,1]$, with intensity $\lambda$ with respect to some measure $\Lambda$ on $[0,1]$, and whose distribution is denoted by $P_\lambda$. As in \cite{FLRB2011} and \cite{FLRB2013}, we assume that the measure $\Lambda$ satisfies $d\Lambda(t)=Ldt$, where $L$ is a positive number. Note that when $L$ is an integer, this assumption amounts to considering the Poisson process $N$ as $L$ pooled i.i.d. Poisson processes with the same intensity $\lambda$, with respect to $dt$: $L$ can therefore be seen as a growing number when comparisons with asymptotic existing results in other frequentist models are needed. 

Depending on the intended application, and the level of knowledge on the baseline intensity of the process $N$ it induces, the questions of detecting a jump or a bump
in $\lambda$ are here formulated as problems of testing the null hypothesis $\hzero$ $"\lambda \in \calS_0"$ versus the alternative $\hone$ $"\lambda\in \calS_1"$, where $\calS_0$ is either the set of a single known constant intensity, or the set of all constant intensities on $[0,1]$, and $\calS_1$ is 
a set of alternative intensities defined as positive piecewise constant functions, with one jump or one bump. As mentioned above, the point of view that we adopt here for our theoretical study is nonasymptotic, based on minimax criteria in accordance with the Neyman-Pearson principle. So, given a first kind error level $\alpha$ in $(0,1)$, any of our (nonrandomised) tests $\phi$, with values in $\{0,1\}$ and rejecting $\hzero$ when $\phi(N)=1$, is primarily required to be of level $\alpha$, that is to satisfy 
$$\sup_{\lambda\in\calS_0} P_\lambda\pa{\phi(N)=1}\leq \alpha\enspace.$$
Then, given a second kind error level $\beta$ in $(0,1)$, any of our level $\alpha$ tests $\phi_\alpha$ is secondarily required to achieve, over the considered set of alternatives $\calS_1$, the $(\alpha,\beta)$-minimax separation rate defined as follows.

Considering the usual metric $d_2$ of $\bbL_2([0,1])$, and a level $\alpha$ test $\phi_\a$ of $\hzero$ versus $\hone$, the $\beta$-uniform separation rate of $\phi_\a$ over $\calS_1$ is defined by
\begin{equation}\label{defSR}
\SRb\pa{\phi_\a,\calS_1}=\inf\set{r>0,\  \sup_{\lambda\in \calS_1,\ d_2\pa{\lambda,\calS_0} \geq r} P_\lambda\pa{\phi_\alpha(N)=0} \leq \beta}\enspace.
\end{equation}
The corresponding $(\alpha,\beta)$-minimax separation rate over $\calS_1$ is defined by
\begin{equation}\label{defmSR}
\mSRab\pa{\calS_1}=\inf_{\phi_\a,\  \sup_{\lambda\in\calS_0} P_\lambda\pa{\phi_\alpha(N)=1}\leq \alpha}\SR_\beta\pa{\phi_\a,\calS_1}\enspace,
\end{equation}
where the infimum is taken over all possible nonrandomised level $\alpha$ tests.

A level $\alpha$ test $\phi_\alpha$ is said to be $\beta$-minimax over $\calS_1$ if $\SRb\pa{\phi_\a,\calS_1}$ achieves $\mSRab\pa{\calS_1}$, possibly up to a multiplicative constant depending on $\alpha$ and $\beta$. 

These definitions due to Baraud \cite{Baraud2002} translate, in a nonasymptotic framework, the (asymptotic) minimax testing criteria that originate in Ingster's work \cite{Ingster1982, Ingster1984, Ingster1993}, and that have now several variants in the literature, among them the asymptotic minimax testing with exact separation constants criteria introduced in \cite{Lepski2000}.

For each choice of $\calS_0$, several sets of alternatives $\calS_1$ are investigated, according to whether the jump or bump parameters are known or not. Following the terminology adopted since Spokoiny's paper \cite{Spokoiny1996}, a complete minimax adaptivity study of the problem is therefore conducted: when one of the alternative parameters is unknown at least, the corresponding minimax tests are said to be minimax adaptive with respect to this unknown parameter.

After determining lower bounds for the $(\alpha,\beta)$-minimax separation rates over all these alternative sets, we construct nonasymptotic minimax and minimax adaptive detection tests. To the best of our knowledge, no such minimax results in the present Poisson process model have already been established. 

\paragraph*{Change-point detection procedures in Poisson processes models}

References dealing with change-point detection in a Poisson process are actually mainly dedicated to the construction of optimal on-line detection rules (see e.g. \cite{Peskir2002}, \cite{Herberts2004}, \cite{Brown2006}, and  \cite{Bayraktar2006} for Bayesian approaches; \cite{DKW1953a}, \cite{DKW1953b}, \cite{Mei2011} or \cite{ElKaroui2015, Elkaroui2017} and references therein for non-Bayesian approaches), or asymptotic off-line detection tests.  On the one hand, a few off-line procedures are derived from the Bayesian perspective, such as the ones in \cite{Akman1986b}, \cite{Raftery1986a} and \cite{Raftery1994} dealing with the single change-point case, \cite{Green1995}, \cite{Young2001} or \cite{Shen2012} dealing with the multiple change-points case. On the other hand, non-Bayesian off-line procedures are numerous, due to the variety of Poisson processes convenient properties. Since the earliest procedures of Neyman and Pearson \cite{Neyman1928}, Sukhatme \cite{Sukhatme1936}, Maguire, Pearson and Wynn \cite{Maguire1952}, many contributions have been made considering the exponential distribution of the homogeneous Poisson process inter-arrivals, like 
in \cite{Matthews1985}, \cite{Worsley1986},  \cite{Siegmund1988}, or more recently \cite{Antoch2007}. 

Recalling that the Poisson process $N$ is homogeneous if and only if for every positive integer $n$, given $N_1=n$, the points of the process are independent and uniformly distributed on $[0,1]$, any test of uniformity on $[0,1]$ in a density model can be directly applied conditionally to $N_1$, or used as a source of inspiration to obtain a Poisson process adapted test of homogeneity. Closer to the tests we propose in the present work, many other existing tests for the single change-point problem are thus based on or inspired from the historical likelihood ratio, Cramér von-Mises or Kolmogorov-Smirnov statistics, with various weighting or other transforming strategies, as the ones of Rubin \cite{Rubin1961}, Lewis \cite{Lewis1965}, or Kendall and Kendall \cite{Kendall1980}.
Deshayes and Picard \cite{Deshayes1984, Deshayes1985} study the optimality of weighted Kolmogorov-Smirnov and likelihood ratio tests in the non local asymptotic sense of Bahadur \cite{Bahadur1967} and Brown \cite{Brown1971}, and their equivalence in the local asymptotic sense of Le Cam \cite{Lecam1970}. 
Asymptotic properties of point and interval change-point estimators deduced from these tests can be found in \cite{Akman1986a}, \cite{Loader1990}, and \cite{Galeano2007} where Galeano also integrates these tests in a binary segmentation algorithm to
further address multiple change-points detection. More recently, Dachian, Kutoyants and Yang \cite{Dachian2016b} (see also Yang's \cite{Yang2014} PhD thesis, and \cite{Dachian2015} and \cite{Yang2020} where tests derived from the Bayesian perspective are also proposed) and Farinetto \cite{Farinetto2017} consider a single change-point detection problem in the more general framework of inhomogeneous Poisson processes. 

\paragraph*{Testing procedures for close purposes in Poisson processes models}

On related topics, it may be worth first mentioning the foundational paper by Davies \cite{Davies1977}, whose goodness-of-fit test is also discussed in Section 13.1 of \cite{DaleyVJ}. Several procedures for testing goodness-of-fit or homogeneity of a Poisson process versus the alternative hypothesis that it has an increasing intensity have then been introduced and explored through experimental comparative studies in a series of papers by Bain et al. \cite{Bain1985}, Engelhardt et al. \cite{Engelhardt1990}, Cohen and Sackrowitz \cite{Cohen1993}, \cite{Ho1993}, \cite{Ho1995}.  Although these procedures are not initially designed to handle the change-point detection problem, they can nevertheless be applied to this end. Among them, the so-called Laplace and $Z$ tests introduced by Cox \cite{Cox1955} and Crown \cite{Crow1974}, whose extensions have been proposed in \cite{Pena1998a}, \cite{Agustin1999}, and \cite{Bhattacharjee2004}, stand out when they are used to detect a positive jump.
Fazli and Kutoyants \cite{Fazli2005},  Fazli \cite{Fazli2007}, and more recently Dachian, Kutoyants and Yang \cite{Dachian2016a} consider the goodness-of-fit testing problem where the null hypothesis corresponds to a given inhomogeneous Poisson process, and the alternatives correspond to single or one-sided parametric Poisson processes families.
The problem of testing that a point process is a given homogeneous Poisson process versus it belongs to a stationary self-exciting or stress-release point processes family is treated in \cite{Dachian2006} and \cite{Dachian2009}.

\paragraph*{Related minimax studies}


Focusing now on the minimax point of view, one can cite Ingster and Kutoyants \cite{Ingster2007} and Fromont et al.'s studies of goodness-of-fit or homogeneity tests, where the alternative hypotheses \cite{FLRB2011}, corresponding to Poisson processes with nonparametric intensities in Sobolev and Besov spaces with known and unknown smoothness parameters respectively, are however not suited for change-points detection problems. 

To find minimax tests devoted to change-points detection problems in the existing literature, it is actually needed to switch to other statistical models. 

Using the conditioning trick explained above, which enables to treat change-points detection problems in the Poisson model as particular change-points detection problems in the classical density model, Rivera and Walther \cite{Rivera2013} propose two positive bump detection tests based on scan or average aggregation of likelihood ratio statistics. Though their optimality results are not directly transposable to the minimax set-up that we consider here due to deconditioning difficulties, they nevertheless give preview of possible approaches towards more Poisson processes-specific minimax tests. 
In the classical density model, Dümbgen and Walther \cite{Dumbgen2008} had already tackled the problem of detecting local increases and decreases of the density or the failure rate. The introduced procedures, based on aggregation of local order statistics and spacings, were proved to satisfy asymptotic minimax adaptation properties.

Of course, the most complete bibliography on jump or bump detection from the minimax testing point of view lies in the basic Gaussian framework, where the observation is modelled by a Gaussian vector  $Y=(Y_1,\ldots,Y_n)$ with variance $\sigma^2I_n$. Arias-Castro et al. \cite{Arias-Castro2005} first studied the minimax separation rate for the problem of detecting a bump, that is a change in mean from zero over an interval, when considering the $\ell_2$ metric on the mean vectors, related to the $\bbL_2$-distance between the corresponding Gaussian distributions and also to the \emph{signal to noise ratio} or \emph{signal energy} often mentioned in regression models analysis.
When the height and the length of the change are unknown, they exhibited a minimax separation rate of order $\sqrt{\log n}$  with an exact constant equal to $\sqrt{2}$. In other words, they proved that no test can reliably detect $Y$ such that $\mathbb{E}\cro{Y_i}=\delta \1{i \in [\tau,\tau+\ell[}$ (with $\tau\in \{1,\ldots,n\}$, $\ell\in \{1,\ldots,n+1-\tau\}$)  unless the condition $|\delta|\sqrt{\ell} \geq \sqrt{2(1+\eta)\log n }$ with $\eta>0$ is satisfied, and they introduced minimax adaptive tests based on a scan aggregation of the Neyman-Pearson test statistics $\ell^{-1/2}\sum_{i=\tau}^{\tau+\ell-1} Y_i$ designed to detect non-zero mean, either over all the possible intervals $[\tau,\tau+\ell[$ or over intervals of dyadic type $[k2^j, (k+1)2^{j}[$. Chan and Walther \cite{Chan2013} constructed three other tests, based on the same Neyman-Pearson test statistics, but combined according to different aggregation schemes. All these tests were proved to be consistent as soon as  the refined condition $|\delta| \sqrt{\ell} \geq \sqrt{2\log (n/\ell)}+b_n$ with $b_n \to +\infty$ holds, which slightly improves  Arias-Castro et al.'s lower bound at least when $\ell/n:=\ell_n/n$ is allowed to tend to $0$ with $n$ tending to $+\infty$. A nonasymptotic counterpart of this improved lower bound has been very recently provided by Verzelen et al. \cite{Verzelen2021}. But the procedures introduced in this work go beyond the scope of the present minimax testing study as they further address the twin problems of detecting and localising multiple change-points. In the case where the change height is known, equal to $1$, Brunel \cite{Brunel2014} constructed a test based on a scanning of the shifted test statistic $\sum_{i=\tau}^{\tau+\ell-1} Y_i-\ell/2$, which is consistent as soon as $\ell/\log n\to+\infty$. Still considering the $\ell_2$ metric on the mean vectors, but considering, among piecewise monotone signals estimation problems, the special problem of detecting a jump from an unknown constant mean, Gao et al. \cite{Gao2020} obtained a lower bound of order $\sqrt{\log \log n}$. More precisely, they proved that no test can reliably detect $Y$ such that $\mathbb{E}\cro{Y_i}=\delta \1{i \in [\tau,n]}$ (with $\tau\in \{2,\ldots,n\}$)  unless $|\delta|\sqrt{(\tau-1)(n+1-\tau)/n} \!\geq \!c\sqrt{\log \log (16 n) }$.  Verzelen et al. \cite{Verzelen2021} provide a nonasymptotic lower bound equal to $\sqrt{2(1-c)(1-n^{-1/2})\log \log n}$ for $c$ in $(0,2/3)$ and $n$ large enough. As for a corresponding upper bound, Gao et al. \cite{Gao2020} refer to the asymptotic test of Csörgö and Horváth \cite{Csorgo1997}, based on the scan statistic $\max_{1\leq \tau \leq n} \sqrt{{n}/\pa{\tau(n-\tau)}} |\sum_{i=1}^\tau Y_i-({k}/{n})\sum_{i=1}^n Y_i|$, which is proved to be optimal from asymptotic inequalities in the spirit of the Iterated Logarithm Law. Notice that this scan statistic is closely related to the well-known CUSUM statistics, which have a long history in the single change-point analysis literature from Hinkley's \cite{Hinkley1970} work, as well as in multiple change-points analysis references, where they are at the core of binary segmentation approaches. Verzelen et al. \cite{Verzelen2021} introduce a test based on a max penalized CUSUM statistic, with location-dependent penalties, whose separation rate is of the optimal order  $\sqrt{2\log \log n}$ (thus proving, combined with their lower bound, that the exact constant is $\sqrt{2}$ as in the bump detection case), with possible refinement when restricting to particular change locations. 

In more complex Gaussian models, with sparse high dimensional, heterogeneous or dependence properties, it is worth mentioning at least the work of Enikeeva and Harchaoui \cite{EnikeevaHarchaoui2019}, Enikeeva et al. \cite{EnikeevaMunk2018}, Liu et al. \cite{Liu2021} and Enikeeva et al. \cite{EnikeevaMunk2020}, addressing bump detection problems from asymptotic minimax points of view, that are quite close to the one we adopt here.

Notice that we do not tackle the problem of detecting multiple change-points with more than two change-points, nor the problem of localising change-points, that  we  consider as out of the scope of the present paper and a basis for future work.

\paragraph*{Our contribution} The present work address the question of detecting a jump or a bump in the intensity of a Poisson process from the nonasymptotic minimax point of view described above. At this end, we will determine the minimax separation rates that correspond to :
\begin{description}
\item  - the detection of a change from a known or an unknown constant baseline intensity,
\item  - the detection of a non transitory change formalised as a jump in the intensity, or a transitory change formalised as a bump in the intensity, 
\item - the detection of a change with known or unknown height, length and location.
\end{description}
We will thus provide a comprehensive overview of the various minimax separation rate regimes, with a special focus on the phase transitions, and their determining factors. Considering each parameter as known or unknown, one by one, indeed enables us to precisely identify what causes such phase transitions, and the precise cost of minimax adaptation to each of these parameters. Among the main results of this study, we find a phase transition from a $\sqrt{\log \log L/L}$ order minimax separation rate for jump detection to a $\sqrt{\log L/L}$ order minimax separation rate for bump detection, when the jump or bump height, the jump or bump location and the bump length are together unknown: this phase transition is similar to the Gaussian one. But we also exhibit minimax separation rates that are not even known in the basic Gaussian model, up to our knowledge. 

For the bump detection problem, we indeed prove that the minimax separation rate is of order $\sqrt{\log L/L}$ when both location and length of the bump are unknown, whether the height is known or not, of order $\sqrt{\log \log L/L}$ as in the jump detection problem when the only location  of the bump is known (with height and length unknown), of order $\sqrt{1/L}$ in the other cases. 

For the jump detection problem, the results could be more easily anticipated: we prove that the minimax separation rate is of order $\sqrt{\log \log L/L}$ when both height and location of the jump are unknown, as in the Gaussian model, and of order $\sqrt{1/L}$ in the other cases. 

Such minimax separation rates are as usual obtained in two steps.  Lower bounds are first deduced from classical Bayesian arguments, originating from Le Cam's theory, and clearly outlined by Ingster \cite{Ingster1982, Ingster1984, Ingster1993} and Baraud \cite{Baraud2002} in an asymptotic and a nonasymptotic framework respectively. Combined with these Bayesian arguments, the Poisson processes properties, and mainly Girsanov's Lemma, are key points of the proofs.  Then, matching upper bounds are derived from the construction of minimax or minimax adaptive tests. The tests that we propose are based on either linear statistics adapted from the Neyman-Pearson test in the case where all the bump or jump parameters are known, and close to the CUSUM like statistics used in the Gaussian framework, or more novel quadratic statistics that we felt better suitable for the estimation of the distance $d_2$, considered here, between $\lambda$ and $\calS_0$. Our simulation study 
actually come to support the use of such quadratic statistics, as the corresponding tests are mostly more powerful than the tests based on linear statistics, especially when the bump or jump height is negative, that is especially for depression detection. Minimax adaptation when some change parameters are unknown is obtained from scan aggregation approaches, that all differ depending on which parameters are unknown.  The critical values involved in the scan aggregation approaches are also differently adjusted, with an additional crucial conditioning trick already used in \cite{FLRB2011} when the baseline intensity is unknown, to lead to a nonasymptotic level $\alpha$ and nonasymptotic minimax optimality. Upper bounding these critical values often was the main and most difficult point of the proofs. We had to use a wide variety of exponential and concentration inequalities, from historical 
ones due to \cite{pyke1959}  to very recent ones due to Le Guével \cite{LeGuevel2021} which are specific to suprema of counting processes and their related square martingales when dealing with detection of a change from a known baseline intensity, plus exponential inequalities for suprema or oscillations of empirical processes and $U$-statistics due to Mason, Shorack and Wellner (see \cite{ShorackWellner}) and Houdré and Reynaud-Bouret \cite{HoudreRB}, or obtained from Bernstein and Bennett's inequalities as stated in \cite{BercuDelyonRio}, refined through combination with chaining techniques.

\paragraph*{Organisation of the paper} 
Section \ref{Sec:knownbaseline} of the paper is devoted to the problem of detecting a change from a known baseline intensity, while Section \ref{Sec:unknownbaseline} deals with the problem of detecting a change from an unknown intensity. For each problem, all the possible sets of alternatives according to whether each parameter of the change (height, location and length) is known or not, including the special case where the change is non transitory (jump detection), are handled.  And for each of the resulting ten sets of alternatives, lower bounds for minimax separation rates are provided, as a preliminary basis for corresponding upper bounds (when appropriate, that is when at least the height or the length is unknown). As explained above, these upper bounds are obtained by constructing minimax or minimax adaptive tests, which are mainly based on aggregation of either linear or quadratic statistics, coupled with adjusted critical values. A simulation study is presented in Section~\ref{SimulationStudy}, whose aim is to compare linear and quadratic type tests, and also to compare them with standard tests used to detect nonhomogeneity of Poisson processes in practice. Proofs of the core results are postponed to Section \ref{Sec:Proofs}, and proofs of technical results 
mostly based on exponential inequalities and devoted to quantiles and critical values upper bounds are postponed to Section \ref{Sec:FTresults}, which also contains fundamental and general results for lower bounds.

\paragraph*{Notation} 
Concerning the Poisson Process $N=(N_t)_{t\in[0,1]}$, we use the notation $N(\tau_1,\tau_2]$ for $N_{\tau_2}-N_{\tau_1}$ for every $\tau_1,\tau_2$ in $[0,1]$. As usual, $dN$ stands for the point measure associated with $N$, and  $E_\lambda$ and $\mathrm{Var}_\lambda$ respectively stand for the expectation and the variance under $P_{\lambda}$, that is when $N$ has $\lambda$ as intensity with respect to $d\Lambda(t)=Ldt$. The distance $d_2$ has been introduced above, and associated with this distance, we consider the usual norm of $\bbL_2([0,1])$ denoted by $\|.\|_2$. 
For all $x$ and $y$ in $\R$, $x \vee y$ (resp. $x\wedge y$) denotes the maximum (resp. minimum) between $x$ and $y$, and the sign function $\mathrm{sgn}$ is defined by $\mathrm{sgn}(x)= \mathds{1}_{x>0}- \mathds{1}_{x<0}$.

All along the article, we will introduce some positive constants denoted by $C(\alpha,\beta,\ldots)$ and $L_0(\alpha,\beta,\ldots)$, meaning that they may depend on $(\alpha,\beta,\ldots)$. Though they are denoted in the same way, they may vary from one line to another. When they appear in the main results about lower and upper bounds, we do not intend to precisely evaluate them. However, some possible, probably pessimistic, explicit expressions for them are proposed in the proofs.

\section{Detecting an abrupt, possibly transitory, change in a known baseline intensity}\label{Sec:knownbaseline}

As a first step of work, and because this also addresses particular applications, we are here interested in the problem of detecting an abrupt change in the intensity of the Poisson process $N$, when its baseline is assumed to be known, equal to a positive constant function $\lambda_0$ on $[0,1]$. For the sake of simplicity, the constant function $\lambda_0$ and its value on $[0,1]$ are often confused in the following. The null hypothesis of the present section can therefore be expressed as $\hzero \ "\lambda\in \calS_0[\lambda_0]=\{\lambda_0\}"$,
while the alternative hypothesis varies according to the height, length and location of the intensity jump or bump knowledge.


In order to further cover the full range of alternatives in a unified notation, we introduce for $\delta^*$ in $ (-\lambda_0,+\infty)\setminus\{0\}$, $\tau^*$ in $ (0,1)$, $\ell^*$ in $(0,1-\tau^*]$ the set $\calS_{\delta^*,\tau^*,\ell^*}[\lambda_0]$ of intensities with a change of height $\delta^*$, location $\tau^*$ and length $\ell^*$ from $\lambda_0$:
\begin{equation}\label{allknown}
\bold{[Alt.1]}\quad \calS_{\delta^*,\tau^*,\ell^*}[\lambda_0]= \lbrace \lambda:[0,1]\to (0,+\infty),\ \forall t\in [0,1]\ \lambda(t)= \lambda_{0} + \delta^* \mathds{1}_{(\tau^*,\tau^*+\ell^*]}(t) \rbrace\enspace.\end{equation}
Testing $\hzero$ versus $\hone\  "\lambda\in \calS_{\delta^*,\tau^*,\ell^*}[\lambda_0]"$ falls within the scope of the Neyman-Pearson fundamental lemma and an Uniformly Most Powerful (UMP) test exists, thus achieving the minimax separation rate over $\calS_{\delta^*,\tau^*,\ell^*}[\lambda_0]$. Details are provided below.

Then, when the question of adaptivity with respect to unknown parameters is tackled, the unknown parameters are replaced by single, double or a triple dots in the notation $\calS_{\delta^*,\tau^*,\ell^*}[\lambda_0]$.

Notice that for any alternative intensity $\lambda= \lambda_{0} + \delta \mathds{1}_{(\tau,\tau+\ell]}$ with $\delta$ in $ (-\lambda_0,+\infty)\setminus\{0\}$, $\tau$ in $ (0,1)$, and $\ell$ in $(0,1-\tau]$, $d_2(\lambda,\calS_0[\lambda_0])=|\delta| \sqrt{\ell}$.
Hence, as soon as $\lambda$ has a known change height $\delta=\delta^*$ and a known change length $\ell=\ell^*$, the distance $d_2(\lambda,\calS_0[\lambda_0])$ is fixed, equal to $|\delta^*| \sqrt{\ell^*}$. The $\beta$-uniform separation rate of any level $\alpha$ test over $\calS_{\delta^*,\tau^*,\ell^*}[\lambda_0]$ or $\calS_{\delta^*,\bbul\bbul,\ell^*}[\lambda_0]$ as defined by \eqref{defSR} is therefore either $0$ or $+\infty$ (with the usual convention $\inf \emptyset =+\infty$), as well as the minimax separation rate. In these only two cases, studying our tests from the minimax point of view would have no sense. Nevertheless, once having ensured that their first kind error rate is at most $\alpha$, in order to follow the same line as the minimax results obtained in the other cases, we establish conditions expressed as a sufficient minimal distance $d_2(\lambda,\calS_0[\lambda_0])$, guaranteeing that their second kind error rate is at most equal to some prescribed level $\beta$. 

\subsection{Uniformly most powerful detection of a  possibly transitory change with known location and length} \label{UMPknownSec}

Let us now give more details about the above problem of testing the simple null hypothesis $\hzero \ "\lambda\in \calS_0[\lambda_0]=\{\lambda_0\}"$ versus the simple alternative hypothesis $\hone\  "\lambda\in \calS_{\delta^*,\tau^*,\ell^*}[\lambda_0]"$ with $\calS_{\delta^*,\tau^*,\ell^*}[\lambda_0]$ defined by \eqref{allknown}. Notice that for any $\lambda$ in $\calS_{\delta^*,\tau^*,\ell^*}[\lambda_0]$, then 
\[d_2(\lambda,\calS_0[\lambda_0])=|\delta^*| \sqrt{\ell^*}\enspace.\]

Given $\alpha$ in $(0,1)$, Neyman-Pearson tests of $\hzero$ versus $\hone\  "\lambda\in \calS_{\delta^*,\tau^*,\ell^*}[\lambda_0]"$ of size $\alpha$  can be constructed. To this end, we recall Girsanov's lemma (see \cite{Bremaud} for a proof).

\begin{lemma}[Girsanov] \label{lemmegirsanov} Let $N=(N_{t})_{t \in [0,1]}$ be an inhomogeneous Poisson process with jump locations $(X_j)_{j \geq 1}$, with bounded intensity $\lambda$ with respect to some measure $\Lambda$ on $[0,1]$,  and with distribution denoted by $P_\lambda$ under the probability $\P$. Assume that $\lambda_{0}$ is a bounded nonnegative function such that for every $j \geq 1,$ $\lambda_{0}(X_j)>0$ $\P$-almost surely. Then 
  \[ \frac{dP_{\lambda}}{dP_{\lambda_{0}}}(N) = \exp \left[ \int_{0}^{1} \ln \left( \frac{\lambda(t)}{\lambda_{0}(t)} \right) dN_t  - \int_{0}^{1} (\lambda(t)-\lambda_{0}(t))  d\Lambda_t \right]\enspace.\]
\end{lemma}

From this fundamental lemma, we deduce the likelihood ratio, for $\lambda$ in $\calS_{\delta^*,\tau^*,\ell^*}[\lambda_0]$,
\begin{equation}\label{LR}
\frac{dP_{\lambda}}{dP_{\lambda_0}}(N)= \exp \left[ \ln \pa{1+ \frac{\delta^{*}}{\lambda_{0}} }N(\tau^{*},\tau^* + \ell^{*}]- \delta^{*}\ell^{*} L\right]\enspace,
\end{equation}
which leads to the following size $\alpha$ Neyman-Pearson tests:
\begin{equation}\label{NPtest}
\left\{\begin{array}{llll}
\phi_{1,\alpha}^{-}(N)&= \mathds{1}_{N(\tau^{*}, \tau^{*}+ \ell^{*}] < p_{\lambda_{0}\ell^* L}(\alpha)}&+ \gamma^-(\alpha) \mathds{1}_{N(\tau^{*},\tau^{*}+ \ell^{*}] = p_{\lambda_{0} \ell^* L}(\alpha)}&\textrm{if }\delta^{*}<0\enspace,\\
\phi_{1,\alpha}^{+}(N)&= \mathds{1}_{N(\tau^{*},\tau^{*} + \ell^{*}] > p_{\lambda_{0} \ell^* L}(1-\alpha)}&+ \gamma^+(1-\alpha) \mathds{1}_{N(\tau^{*}, \tau^{*} + \ell^{*}] = p_{\lambda_{0} \ell^* L}(1-\alpha)}&\textrm{if  }\delta^{*}>0\enspace,
\end{array}\right.
\end{equation}
where $p_{\xi}(u)$ denotes the $u$-quantile of the Poisson distribution with parameter $\xi$, and 
\begin{equation}\label{defgamma}
\gamma^-(u)= \frac{u -  P_{\lambda_0}(N(\tau^*, \tau^* + \ell^*] <p_{\lambda_{0}\ell^* L}(u))}{P_{\lambda_0}(N(\tau^{*}, \tau^{*} + \ell^{*}] =p_{\lambda_{0}\ell^* L}(u))},\quad \gamma^+(u)=1-\gamma^-(u)\enspace.
\end{equation}

\begin{proposition}[Second kind error rates control for $\bold{[Alt.1]}$]\label{bNP1}
Let $\alpha$ and $\beta$ be fixed levels in $(0,1)$, $L\geq 1$, $\lambda_{0}>0$, $\delta^*$ in $(-\lambda_0,+\infty)\setminus\{0\}$, $\tau^{*}$ in $(0,1)$ and $\ell^*$ in $(0,1- \tau^{*}]$. 

\smallskip

$(i)$ If $\delta^*>0$, the test $\phi_{1,\alpha}^{+}$ of $\hzero$ versus $\hone \ "\lambda\in \calS_{\delta^*,\tau^*,\ell^*}[\lambda_0]"$ is 
a UMP test of size $\alpha$. Moreover $P_{\lambda}(\phi_{1,\alpha}^{+}(N)=0) \leq \beta$ as soon as $\lambda$ belongs to $\calS_{\delta^*, \tau^*, \ell^*}[\lambda_0]$ with
\begin{equation}\label{cond_alt1} 
d_2\pa{\lambda,\calS_{0}[\lambda_0]} \geq \pa{\sqrt{(\lambda_{0}+ \delta^*)/\beta} + \sqrt{\lambda_{0}/\alpha}} /\sqrt{L}\enspace.
\end{equation}

$(ii)$ If $-\lambda_0<\delta^*<0$, the test $\phi_{1,\alpha}^{-}$ of $\hzero$ versus $\hone \ "\lambda\in \calS_{\delta^*,\tau^*,\ell^*}[\lambda_0]"$ is a UMP test of size $\alpha$. Moreover, $P_{\lambda}(\phi_{1,\alpha}^{-}(N)=0) \leq \beta$ as soon as $\lambda$ belongs to $\calS_{\delta^*, \tau^*, \ell^*}[\lambda_0]$ with \eqref{cond_alt1}.

\end{proposition}
\emph{Comments.} Notice first that the same result holds with $P_{\lambda}(\phi_{1,\alpha}^{+}(N)=0)$ replaced by the second kind error rate $E_{\lambda}[1-\phi_{1,\alpha}^{+}(N)]$. Then, as explained above, studying the present tests from the minimax point is not really relevant. One can however notice that the uniform separation rate of a UMP test necessarily provides the minimax separation rate over any set of alternatives. Since for $\lambda$ in $\calS_{\delta^*, \tau^*, \ell^*}[\lambda_0]$, $d_2\pa{\lambda,\calS_{0}[\lambda_0]}=|\delta^*|\sqrt{\ell^*}$, the above proposition implies that if $L \geq \big(\sqrt{(\lambda_{0}+ \delta^*)/\beta} + \sqrt{\lambda_{0}/\alpha}\big)^2/ \pa{{\delta^*}^2 \ell^{*}}$,  then $P_{\lambda}(\phi_{1,\alpha}^{+}(N)=0) \leq \beta$ when $\delta^*>0$ and $P_{\lambda}(\phi_{1,\alpha}^{-}(N)=0) \leq \beta$ when $-\lambda_0<\delta^*<0$.
Therefore, in this case, the $\beta$-uniform separation rate of $\phi_{1,\alpha}^{+}$ over $\calS_{\delta^*, \tau^*, \ell^*}[\lambda_0]$ with $\delta^*>0$ is  equal to $0$, as well as the  $\beta$-uniform separation rate of $\phi_{1,\alpha}^{-}$ over $\calS_{\delta^*, \tau^*, \ell^*}[\lambda_0]$ with $-\lambda_0<\delta^*<0$, and consequently, the $(\alpha,\beta)$-minimax separation rate $\mSRab\pa{\calS_{\delta^*, \tau^*, \ell^*}[\lambda_0]}$.

\smallskip

Let us now consider the question of adaptation with respect to the change height only.

To this end, we introduce, for $\tau^*$ in $(0,1)$ and $\ell^*$ in $(0,1-\tau^*]$ the set
\begin{multline} \label{latauleknown}
\bold{[Alt.2]}\quad\calS_{\bbul,\tau^*,\ell^*}[\lambda_0]=\{ \lambda:[0,1]\to (0,+\infty),\ \exists \delta \in (- \lambda_{0}, + \infty)\setminus\lbrace 0 \rbrace,\\
\forall t\in [0,1]\quad \lambda(t) =\lambda_{0} + \delta \mathds{1}_{(\tau^*,\tau^*+\ell^*]}(t)\}\enspace,
\end{multline}
and we consider the problem of testing $\hzero$ versus $\hone\  "\lambda\in \calS_{\bbul,\tau^*,\ell^*}[\lambda_0]"$.

\smallskip

The following result gives a nonasymptotic lower bound for the $(\alpha,\beta)$-minimax separation rate over the set of alternatives $\calS_{\bbul,\tau^*,\ell^*}[\lambda_0]$ of the parametric order $L^{1/2}$, which is obtained from a now classical Bayesian approach that originates in Le Cam's theory and Ingster's work \cite{Ingster1993} in an asymptotic perspective, and that has been next adapted to the nonasymptotic perspective by Baraud \cite{Baraud2002}. For the sake of clarity and completeness, the main points of this approach are recalled in Section \ref{Keys}, and the complete proof can be found in Section \ref{Sec:Proofs}.

 \begin{proposition}[Minimax lower bound for $\bold{[Alt.2]}$] \label{LBalt2}
Let $\alpha, \beta$  in $(0,1)$ such that $\alpha+\beta<1$, $\lambda_{0}>0$, $\tau^{*}$ in $(0,1)$ and $\ell^* $ in $(0,1- \tau^{*}]$. For all $L \geq 1$, the following lower bound holds:
\[\mSRab\pa{\calS_{\bbul, \tau^*, \ell^*}[\lambda_0]} \geq \sqrt{{\lambda_{0} \log C_{\alpha, \beta}}/{L}},\textrm{ with }C_{\alpha, \beta}=1+4(1- \alpha -\beta)^2\enspace.\]
\end{proposition}

In order to prove that this lower bound is sharp with respect to $L$, we introduce two tests, whose test statistics both derive from the estimation of a certain distance between $\lambda$ in $\calS_{\bbul, \tau^*, \ell^*}[\lambda_0]$ and $\lambda_0$: the $\bbL_1$-distance is used for the first test while the $\bbL_2$-distance is used for the second test.
The first test is therefore based on a linear statistic of the Poisson process while the second one is based on a more complex quadratic statistic, which may be more adapted to our chosen performance evaluation criterion based on the $\bbL_2$-distance $d_2$. Let thus
\begin{multline} \label{test1}
\phi_{2,\alpha}^{(1)}(N)=\1{N(\tau^{*}, \tau^{*} + \ell^*] > p_{\lambda_0 \ell^* L}(1-\alpha_1)}+\gamma^+(\alpha_1)\1{N(\tau^{*}, \tau^{*} + \ell^*] = p_{\lambda_0 \ell^* L}(1-\alpha_1)}\\
+\1{N(\tau^{*}, \tau^{*} + \ell^{*}] < p_{\lambda_0 \ell^* L}(\alpha_2)}+
\gamma^-(\alpha_2)\1{N(\tau^{*}, \tau^{*} + \ell^{*}] = p_{\lambda_0 \ell^* L}(\alpha_2)}\enspace,
\end{multline}
where $\alpha_1$ and $\alpha_2 $ in $ (0,1)$ are determined by
\begin{equation} \label{bUMPU1}
\alpha_1 + \alpha_2 = \alpha \quad \textrm{and} \quad  
E_{\lambda_0} [N(\tau^{*}, \tau^{*}+ \ell^{*}] \phi_{2,\alpha}^{(1)}(N)] = \alpha E_{\lambda_0} [N(\tau^{*}, \tau^{*}+ \ell^{*}]]\enspace.\end{equation}
Remark that $E_{\lambda_0} [N(\tau^{*}, \tau^{*}+ \ell^{*}] \phi_{2,\alpha}^{(1)}(N)]$ and $E_{\lambda_0} [N(\tau^{*}, \tau^{*}+ \ell^{*}]]$ can be explicitly computed as functions of the parameters $\lambda_0,\alpha, \tau^*, \ell^*$, and that $\phi_{2,\alpha}^{(1)}(N)$ is the maximum of two UMP tests when considering the alternatives with $\delta>0$ and $\delta<0$ separately. 

\smallskip

Since our testing problem amounts to a problem of testing $"\delta=0"$ versus $"\delta \neq 0"$ in the exponential model ${dP_{\lambda}}/{dP_{\lambda_0}}(N)= \exp \left[ \ln \pa{1+{\delta}/{\lambda_{0}} }N(\tau^{*},\tau^* + \ell^{*}]- \delta\ell^{*} L\right]$, applying the result of Chapter 4.2 in \cite{Lehmann2006} allows to see that $ \phi_{2,\alpha}^{(1)}(N)$ is an Uniformly Most Powerful Unbiased (UMPU) test of size $\alpha$.

\smallskip

Let us now deal with the estimation of the $\bbL_2$-distance $d_2\pa{\lambda,\lambda_0}$ when $\lambda\in\calS_{\bbul, \tau^*, \ell^*}[\lambda_0]$. Let for $\tau_1, \tau_2$ such that $0 \leq \tau_1 < \tau_2 \leq 1$, $ \varphi_{(\tau_1,\tau_2]}=  \mathds{1}_{(\tau_1, \tau_2 ]}/\sqrt{\tau_2 - \tau_1}$, $V_{\tau_1,\tau_2}= \mathrm{Vect}(\varphi_{(\tau_1,\tau_2]})$, and $\Pi_{V_{\tau_1,\tau_2}}$ be the orthogonal projection onto $V_{\tau_1,\tau_2}$ in $\bbL_2([0,1])$. An unbiased estimator of $\Vert \Pi_{V_{\tau_1,\tau_2}} (\lambda - \lambda_{0}) \Vert_2^2$ is given by the quadratic statistic
\begin{equation} \label{estimateurSB}
 T_{\tau_1, \tau_2}(N) = \frac{1}{L^2 (\tau_2 - \tau_1)} \left(  N \left( \left.  \tau_1, \tau_2  \right] \right.^2 -N \left( \left.  \tau_1 , \tau_2  \right] \right.  \right) - \frac{2 \lambda_{0}}{L} N \left( \left.  \tau_1, \tau_2  \right] \right. + \lambda_{0}^2 (\tau_2 - \tau_1)\enspace.
 \end{equation}

We therefore consider the particular statistic $T_{\tau^*, \tau^*+\ell^*}(N)$ which is an unbiased estimator of $d_2^2(\lambda,\lambda_0)$ when $\lambda$ belongs to $\calS_{\bbul, \tau^*, \ell^*}[\lambda_0]$, leading to the test
\begin{equation} \label{test1N2}
\phi_{2,\alpha}^{(2)}(N) = \mathds{1}_{T_{\tau^*, \tau^*+\ell^*}(N) > t_{\lambda_0,\tau^*, \tau^* + \ell^*}(1-\alpha)}\enspace,
\end{equation}
where $t_{\lambda_0,\tau_1,\tau_2}(u)$ denotes the $u$-quantile of the distribution of $T_{\tau_1,\tau_2}(N)$ under $\hzero$.

\begin{proposition}[Minimax upper bound for $\bold{[Alt.2]}$] \label{UBalt2} Let $L\geq 1$, $\alpha,\beta$ in $(0,1)$, $\lambda_{0}>0$, $\tau^*$ in $(0,1)$ and $\ell^*$ in $(0,1- \tau^{*}]$. Let $\phi_{2,\alpha}^{(1/2)}$ be one of the tests $\phi_{2,\alpha}^{(1)}$ and $\phi_{2,\alpha}^{(2)}$ of $\hzero$ versus $\hone\ "\lambda\in\calS_{\bbul,\tau^*,\ell^*}[\lambda_0]"$ defined by \eqref{test1}-\eqref{bUMPU1} and \eqref{test1N2}. The test $\phi_{2,\alpha}^{(1/2)}$ is of level $\alpha$, that is $ P_{\lambda_0}(\phi_{2,\alpha}^{(1/2)}(N)=1) \leq \alpha$ (the randomised test $\phi_{2,\alpha}^{(1)}$ is even of size $\alpha$, that is  $E_{\lambda_0}[\phi_{2,\alpha}^{(1)}(N)]=\alpha$). Moreover, there exists $C(\alpha, \beta, \lambda_{0}, \tau^*, \ell^*)>0$ such that
$$ \SRb(\phi_{2,\alpha}^{(1/2)},\calS_{\bbul,\tau^*, \ell^*}[\lambda_0]) \leq {C(\alpha, \beta, \lambda_{0},\ell^*)}/{\sqrt{L}}\enspace,$$
which entails in particular $\mSRab\pa{\calS_{\bbul, \tau^*, \ell^*}[\lambda_0]} \leq C(\alpha, \beta, \lambda_{0}, \ell^*)/\sqrt{L}$.
\end{proposition}

\emph{Comments.} Both tests $\phi_{2,\alpha}^{(1)}$ and $\phi_{2,\alpha}^{(2)}$ are therefore $\beta$-minimax (up to a possible multiplicative constant) over the set of alternatives $\calS_{\bbul,\tau^*, \ell^*}[\lambda_0]$, where the height of the change is unknown, with an optimal uniform separation rate of the expected parametric order $1/\sqrt{L}$.

\smallskip

Notice that the present study involves the particular non transitory change or jump detection problem, with a known change location, taking $\ell^*=1-\tau^*$. The study of the general non transitory change or jump detection problem (of unknown location) is conducted in Section~\ref{Sec:generalsingle} as a particular case of change detection problem, with unknown location and length.

\subsection{Minimax detection of a transitory change with known length}\label{Sec:knownlength}

The present subsection is dedicated to the  problem of testing $\hzero \ "\lambda\in \calS_0[\lambda_0]=\{\lambda_0\}"$ versus alternatives where the length of the change from the baseline intensity is known, with adaptation with respect to the change location, and with or without adaptation with respect to the height of the change. We therefore introduce, for $\ell^*$ in $(0,1)$ and $\delta^*$ in $(- \lambda_{0}, + \infty)\setminus\{0\}$, the two following sets:
\begin{multline}\label{laledeknown}
\bold{[Alt.3]}\ \calS_{\delta^*,\bbul\bbul,\ell^*}[\lambda_0]= \big\{\lambda:[0,1]\to (0,+\infty),\ \exists \tau \in (0,1-\ell^*),\\
\forall t\in [0,1]\quad   \lambda(t) = \lambda_{0} + \delta^* \mathds{1}_{(\tau,\tau+\ell^*]}(t) \big\}\enspace,
\end{multline}
\vspace{-0.8cm}
\begin{multline}\label{laleknown}
\bold{[Alt.4]}\ \calS_{\bbul,\bbul\bbul,\ell^*}[\lambda_0]=
 \big\{ \lambda\:[0,1]\to (0,+\infty),\ \exists \delta \in (- \lambda_{0}, + \infty)\setminus\lbrace 0\rbrace ,\exists \tau \in (0,1-\ell^*),\\
\forall t\in [0,1]\quad    \lambda(t) = \lambda_{0} + \delta \mathds{1}_{(\tau,\tau+\ell^*]}(t) \big\}\enspace.
\end{multline}

 As seen in the above subsection, the knowledge of the change height $\delta^*$ is not necessary to construct an UMP test of $\hzero$ versus $\hone\  "\lambda\in \calS_{\delta^*,\tau^*,\ell^*}[\lambda_0]"$ as the test statistic, which is the exhaustive statistic in the considered exponential model, does not depend on the value of $\delta^*$. This enables to directly extend it to an UMPU test of $\hzero$ versus $\hone\  "\lambda\in \calS_{\bbul,\tau^*,\ell^*}[\lambda_0]"$ based on the same exhaustive statistic $N(\tau^*,\tau^*+\ell^*]$. 

The only significant question is hence the one of adaptation to the change location $\tau^*$.

A natural approach to handle this question is to take the same linear and quadratic statistics as the ones used for testing  $\hzero$ versus $\hone\  "\lambda\in \calS_{\delta^*,\tau^*,\ell^*}[\lambda_0]"$ or $\hone\  "\lambda\in \calS_{\bbul,\tau^*,\ell^*}[\lambda_0]"$, but making $\tau^*$ varying in the whole set of possible change locations, or an appropriate restricted set of possible change locations. This approach, known as statistics scanning in the signal and image processing literature or statistics aggregation in the minimax testing literature, has close connections with multiple tests that were investigated in \cite{FLRB2016} (see Section~\ref{discussionlevels}), and that will be exploited in a further work dedicated to the change localisation problem.

We therefore first introduce the following linear statistic based aggregated tests:
\begin{equation} \label{testN1alt3}
\left\{\begin{array}{ll}
\phi_{3,\alpha}^{(1)-}(N)=&\1{\min_{\tau\in [0,1-\ell^*]}N(\tau, \tau + \ell^*] < p_{\lambda_0,\ell^*}^-(\alpha)}\enspace,\\
\phi_{3,\alpha}^{(1)+}(N)=&\1{\max_{\tau\in [0,1-\ell^*]}N(\tau, \tau + \ell^*] > p_{\lambda_0,\ell^*}^+(1-\alpha)}\enspace,
\end{array}\right.
\end{equation}
where $p_{\lambda_0,\ell^*}^-(u)$ and $p_{\lambda_0,\ell^*}^+(u)$ respectively denote the $u$-quantiles of the distributions of $\min_{\tau\in [0,1-\ell^*]}N(\tau, \tau + \ell^*]$ and $\max_{\tau\in [0,1-\ell^*]}N(\tau, \tau + \ell^*]$ under $\hzero$.

From these unilateral tests, we construct the bilateral test
\begin{equation} \label{testN1alt4}
\phi_{4,\alpha}^{(1)}(N)=\phi_{3,\alpha/2}^{(1)-}(N)\vee \phi_{3,\alpha/2}^{(1)+}(N)\enspace,
\end{equation}
intended to address the change height adaptation issue.

Finally, considering $M = \lceil 2/\ell^* \rceil$ and $u_\alpha= {\alpha}/{\lceil (1-\ell^*) M \rceil}$, the test statistic $T_{{k}/{M},{k}/{M}+\ell^*}(N)$ defined by \eqref{estimateurSB} and its $u$-quantile  $t_{\lambda_0,{k}/{M}, {k}/{M}+\ell^*}(u)$ under $\hzero$, we introduce the quadratic statistic based aggregated test
\begin{equation} \label{testN2alt3-4}
\phi_{3/4,\alpha}^{(2)}(N) = \1{\max_{k\in \lbrace 0,\ldots, \lceil (1-\ell^*) M \rceil -1 \rbrace} \left( T_{\frac{k}{M}, \frac{k}{M}+\ell^*}(N) - t_{\lambda_0,\frac{k}{M}, \frac{k}{M}+\ell^*} \pa{1 - u_\alpha} \right)>0}\enspace.
\end{equation}
As the set $\calS_{\delta^*,\bbul\bbul,\ell^*}[\lambda_0]$ defined in \eqref{laledeknown} is composed of alternatives with known change height $\delta^*$ and length $\ell^*$, the distance between any of its elements and $\calS_{0}[\lambda_0]=\{\lambda_0\}$ is fixed, equal to $|\delta^*|\sqrt{\ell^*}$. Therefore, it is not discussed from the minimax point of view. We only provide in the following proposition sufficient conditions for the tests $\phi_{3,\alpha}^{(1)+}$, $\phi_{3,\alpha}^{(1)-}$ and $\phi_{3/4,\alpha}^{(2)}$ to have a second kind error rate controlled by a prescribed level $\beta$ under $P_\lambda$ when $\lambda\in \calS_{\delta^*,\bbul\bbul,\ell^*}[\lambda_0]$. 
The proofs of these results, which are postponed to Section \ref{Sec:Proofs}, mainly rely on sharp bounds for the quantiles $p_{\lambda_0,\ell^*}^-(\alpha)$, $p_{\lambda_0,\ell^*}^+(1-\alpha)$ and $ t_{\lambda_0,{k}/{M},{k}/{M}+\ell^*} \pa{1 - u_\alpha} $, that are deduced from two very recent exponential inequalities for the supremum and the oscillation modulus of the square martingale associated with a counting process due to  Le Guével \cite{LeGuevel2021}. Recall that the technical proofs of such quantiles bounds are detailed in Section \ref{Sec:FTresults}.

\begin{proposition}[Second kind error rate control for $\bold{[Alt.3]}$] \label{UBalt3}
Let $L\geq 1$, $\alpha$ and $\beta$ in $(0,1)$, $\lambda_{0}>0$, $\delta^*$ in $(-\lambda_0,+\infty)\setminus\{0\}$  and $\ell^*$ in $(0,1)$. 
Considering the problem of testing $\hzero$ v.s. $\hone \ "\lambda\in \calS_{\delta^*, \bbul\bbul,\ell^*}[\lambda_0]"$, let $\phi_{3,\alpha}^{(1/2)}$ be one of the tests $\phi_{3,\alpha}^{(1)+}$ or $\phi_{3/4,\alpha}^{(2)}$ if $\delta^*>0$, and one of the tests $\phi_{3,\alpha}^{(1)-}$ or $\phi_{3/4,\alpha}^{(2)}$ if $\delta^*<0$ (see \eqref{testN1alt3} and \eqref{testN2alt3-4}  for definitions of the tests). The test $\phi_{3,\alpha}^{(1/2)}$ is of level $\alpha$, that is $P_{\lambda_0}(\phi_{3,\alpha}^{(1/2)}(N)=1)\leq \alpha$. Moreover, there exists $C(\alpha,\beta,\lambda_0,\delta^*,\ell^*)>0$ such that $P_{\lambda}\pa{\phi_{3,\alpha}^{(1/2)}(N)=0} \leq \beta$ as soon as $\lambda$ belongs to $\calS_{\delta^*, \bbul\bbul, \ell^*}[\lambda_0]$ with
$$d_2\pa{\lambda,\calS_{0}[\lambda_0]} \geq {C(\alpha,\beta,\lambda_0,\delta^*,\ell^*)}/{\sqrt{L}}\enspace.$$
\end{proposition}

\emph{Comments.} Remarking that for $\lambda$ in $\calS_{\delta^*, \bbul\bbul, \ell^*}[\lambda_0]$, $d_2\pa{\lambda,\calS_{0}[\lambda_0]}=|\delta^*|\sqrt{\ell^*}$, Proposition \ref{UBalt3} leads to exhibit a sufficient minimal value $L_0(\alpha,\beta,\lambda_{0},\delta^*,\ell^{*})$ for $L$ so that the second kind error rates of the above tests are controlled by $\beta$. Anecdotally, it furthermore shows that if $L\geq L_0(\alpha,\beta,\lambda_{0},\delta^*,\ell^{*})$,  the $\beta$-uniform separation rate of the above tests over $\calS_{\delta^*,\bbul \bbul, \ell^*}[\lambda_0]$ is  equal to $0$, as well as the $(\alpha,\beta)$-minimax separation rate $\mSRab\pa{\calS_{\delta^*, \bbul \bbul, \ell^*}[\lambda_0]}$.

\smallskip

Turning now to the change height adaptation issue, the lower bound for $\mSRab\pa{\calS_{\bbul,\tau^*, \ell^*}[\lambda_0]} $ given in Proposition \ref{LBalt2} directly leads, using the monotonicity property of the minimax separation rate recalled in Lemma \ref{mSR}, to the following lower bound for $\mSRab\pa{\calS_{\bbul,\bbul\bbul, \ell^*}[\lambda_0]}$.
\begin{corollary}[Minimax lower bound for $\bold{[Alt.4]}$] \label{LBalt4} Let $\alpha, \beta$ in $(0,1)$  such that $\alpha+\beta<1$, $\lambda_0>0$ and $\ell^* $ in $(0,1)$. For $L \geq 1$, 
\[\mSRab\pa{\calS_{\bbul,\bbul\bbul, \ell^*}[\lambda_0]} \geq  \sqrt{\lambda_{0} \log C_{\alpha, \beta}/L},\textrm{ with }C_{\alpha, \beta}=1+4(1- \alpha -\beta)^2\enspace.\]
\end{corollary}

\begin{proposition}[Minimax upper bounds for $\bold{[Alt.4]}$] \label{UBalt4} 
Let $L\geq 1$, $\alpha$ and $\beta$ in $(0,1)$, $\lambda_{0}>0$, and $\ell^*$ in $(0,1)$. Let $\phi_{4,\alpha}^{(1/2)}$ be one of the tests $\phi_{4,\alpha}^{(1)}$ and $\phi_{3/4,\alpha}^{(2)}$ of $\hzero$ versus $\hone\ "\lambda\in\calS_{\bbul,\bbul\bbul,\ell^*}[\lambda_0]"$ defined respectively by \eqref{testN1alt4} and \eqref{testN2alt3-4}. Then $\phi_{4,\alpha}^{(1/2)}$ is of level $\alpha$, that is $P_{\lambda_0}(\phi_{4,\alpha}^{(1/2)}(N)=1)\leq \alpha$.
Moreover, there exists $C(\alpha, \beta, \lambda_{0},\ell^*)>0$ such that
$$ \SRb\big(\phi_{4,\alpha}^{(1/2)},\calS_{\bbul,\bbul\bbul, \ell^*}[\lambda_0]\big) \leq {C(\alpha, \beta, \lambda_{0},\ell^*)}/{\sqrt{L}}\enspace,$$
which entails in particular $\mSRab\big(\calS_{\bbul, \bbul\bbul, \ell^*}[\lambda_0]\big) \leq C(\alpha, \beta, \lambda_{0}, \ell^*)/\sqrt{L}$.
\end{proposition}

\emph{Comments.} The proof of Proposition  \ref{UBalt4} mainly relies as the proof of Proposition \ref{UBalt3} on the quantile control of Lemma~\ref{QuantilesT} deduced from Theorem 6 in \cite{LeGuevel2021}. This result with Corollary~\ref{LBalt4} means that the tests $\phi_{4,\alpha}^{(1)}$ and $\phi_{3/4,\alpha}^{(2)}$ of $\hzero$ versus $\hone\ "\lambda\in\calS_{\bbul,\bbul\bbul,\ell^*}[\lambda_0]"$ are minimax. Moreover and importantly, regarding the results obtained for $\bold{[Alt.2]}$, Proposition \ref{UBalt4} with Corollary \ref{LBalt4} also means that minimax adaptation with respect to the change location can be achieved with a minimax separation rate of the parametric order, hence without any additional price to pay (possibly except multiplicative constants), as soon as the only change length is known. This may contrast with the common idea (maybe spread by results in the jump detection problem where adaptation to the change location is equivalent to adaptation to the change length) that adaptation to the change location is the main cause of an unavoidable  logarithmic cost. Here, by considering all the cases separately and step by step, we aim at precisely exhibiting the various regimes of minimax separation rates: this allows us in particular to specify - where relevant - the price to pay for adaptation to the different alternative parameters.

\subsection{Minimax detection of a transitory change with known location} \label{Sec:knownlocation}

In this subsection, we consider the problem of testing the null hypothesis $\hzero \ "\lambda\in \calS_0[\lambda_0]=\{\lambda_0\}"$ versus alternative hypotheses where the location of the change from the baseline intensity is known, with adaptation with respect to the change length, and with or without adaptation with respect to the height of the change. Contrary to the study of Section \ref{Sec:knownlength}, while adaptation to the change length only can be done without any incidence on the minimax separation rate order, adaptation to the change height in addition to the change length has a non-negligible impact. We therefore examine these two questions in two separate subsections.

\subsubsection{Known change height}

Let us first investigate the problem of testing $\hzero$ versus
$\hone\ "\lambda\in\calS_{\delta^*,\tau^*,\bbul\bbul\bbul}[\lambda_0]"$,
where the set $\calS_{\delta^*,\tau^*,\bbul\bbul\bbul}[\lambda_0]$ is defined for $\delta^*$ in $(- \lambda_{0}, + \infty)\setminus\{0\}$ and $\tau^*$ in $(0,1)$ by
\begin{multline}\label{alt5}
\bold{[Alt.5]}\ \calS_{\delta^*,\tau^*,\bbul\bbul\bbul}[\lambda_0]= \big\{\lambda:[0,1]\to (0,+\infty),\ \exists \ell \in (0,1-\tau^*),\\
\forall t\in [0,1]\quad  \lambda(t) = \lambda_{0} + \delta^* \mathds{1}_{(\tau^*,\tau^*+\ell]}(t) \big\}\enspace.
\end{multline}

 As explained above, we will see that the minimax separation rate over this alternative set remains unchanged, of the parametric order $L^{-1/2}$. A lower bound is easily obtained from the key arguments given in Section \ref{Keys}. Therefore the major point here is the construction of a minimax adaptive test, which has to take the knowledge of the change height $\delta^*$ into account. In order to determine the most relevant way to integrate this knowledge, we have used an exact expression for the probability distribution function as well as an exponential inequality for the supremum of Poisson processes with shift, both due to Pyke \cite[Equation~(6) and Theorem 3]{pyke1959}. This has led to a new procedure which is rather atypical regarding the other tests of this paper, and which can be related to Brunel's \cite{Brunel2014} scan test in the Gaussian set-up.

\begin{proposition}[Minimax lower bound for $\bold{[Alt.5]}$] \label{LBalt5}
Let $\alpha,\beta$ in $(0,1)$ such that $\alpha+\beta<1$, $\lambda_0>0$, $\delta^*$ in $(- \lambda_{0}, + \infty)\setminus\{0\}$ and $\tau^*$ in $(0,1)$. For all $L \geq \lambda_{0}  \log C_{\alpha, \beta} /(\delta^{*}{}^2 (1-\tau^{*}{}))$, 
\[\mSRab\pa{\calS_{\delta^*,\tau^*,\bbul\bbul\bbul}[\lambda_0]} \geq  \sqrt{\lambda_{0} \log C_{\alpha, \beta}/L},\textrm{ with }C_{\alpha, \beta}=1+4(1- \alpha -\beta)^2\enspace.\]
\end{proposition}

Let us now introduce the aggregated test
\begin{equation} \label{testalt5}
\phi_{5,\alpha}(N) = \1{\sup_{\ell \in (0,1- \tau^{*})}S_{\delta^*,\tau^*,\tau^*+\ell}(N) >s_{\lambda_0,\delta^*,\tau^*,L}^+(1-\alpha)}\enspace,
\end{equation}
where $S_{\delta^*,\tau_1,\tau_2}(N)$ is the statistic defined for $0\leq \tau_1<\tau_2\leq 1$ by
\begin{equation}\label{stat_alt5}
S_{\delta^*,\tau_1,\tau_2}(N)=  \mathrm{sgn}(\delta^{*}{}) \Big(N(\tau_1,\tau_2] - \lambda_{0} L(\tau_2-\tau_1)\Big) - \vert \delta^{*}{} \vert L(\tau_2-\tau_1)/2 \enspace,
\end{equation}
and $s_{\lambda_0,\delta^*,\tau^*,L}^+(u)$ is the $u$-quantile of $\sup_{\ell \in (0,1- \tau^{*})}S_{\delta^*,\tau^*,\tau^*+\ell}(N)$ under $(H_0)$.

Lemma \ref{QuantilessupShifted} provides a control of the quantile $s_{\lambda_0,\delta^*,\tau^*,L}^+(1-\alpha)$, which is deduced from Pyke's results  \cite{pyke1959}, and which is the main argument to prove that $\phi_{5,\alpha}$ has an uniform separation rate of parametric order $L^{-1/2}$ and thus show that the lower bound of Proposition \ref{LBalt5} is sharp.

\begin{proposition}[Minimax upper bound for $\bold{[Alt.5]}$] \label{UBalt5}
Let $L\geq 1$, $\alpha$ and  $\beta$ in $(0,1)$, $\lambda_{0}>0$, $\delta^*$ in $(- \lambda_{0}, + \infty)\setminus\{0\}$ and $\tau^*$ in $(0,1)$. Let $\phi_{5,\alpha}$ be the test of $\hzero$ versus $\hone\ "\lambda\in\calS_{\delta^*,\tau^*,\bbul\bbul\bbul}[\lambda_0]"$ defined by \eqref{testalt5}. Then $\phi_{5,\alpha}$ is of level $\alpha$, that is $P_{\lambda_0}\pa{\phi_{5,\alpha}(N)=1}\leq \alpha$.
Moreover, there exists a constant $C(\alpha, \beta, \lambda_{0},\delta^*)>0$ such that
$$ \SRb\pa{\phi_{5,\alpha},\calS_{\delta^*,\tau^*,\bbul\bbul\bbul}[\lambda_0]} \leq {C(\alpha, \beta, \lambda_{0},\delta^*)}/{\sqrt{L}}\enspace,$$
which entails in particular $\mSRab\pa{\calS_{\delta^*,\tau^*,\bbul\bbul\bbul}[\lambda_0]} \leq C(\alpha, \beta, \lambda_{0}, \delta^*)/\sqrt{L}$.
\end{proposition}

\subsubsection{Unknown change height} 

Now addressing the question of adaptation to the change height together with the change length, we consider for $\tau^*$ in $(0,1)$ the alternative set
\begin{multline}\label{alt6-pre}
\calS_{\bbul,\tau^*,\bbul\bbul\bbul}[\lambda_0]=
 \big\{ \lambda:[0,1]\to (0,+\infty),\ \exists \delta \in (- \lambda_{0}, + \infty)\setminus\lbrace 0\rbrace ,\exists \ell \in (0,1-\tau^*),\\
\forall t\in [0,1]\quad  \lambda(t) = \lambda_{0} + \delta \mathds{1}_{(\tau^*,\tau^*+\ell]}(t) \big\}\enspace.
\end{multline}
A first preliminary result in fact shows that this set of alternatives is too large to be relevantly studied from the minimax point of view: the minimax separation rate is infinite over it.
\begin{lemma} \label{LBalt6-pre}
Let $\alpha$ and  $\beta$ in $(0,1)$ such that $\alpha + \beta <1$, $\lambda_0>0$, and $\tau^*$ in $(0,1)$. For the problem of testing 
  $\hzero \ "\lambda\in \calS_0[\lambda_0]=\{\lambda_0\}"$ versus $\hone\ "\lambda\in\calS_{\bbul,\tau^*,\bbul\bbul\bbul}[\lambda_0]"$, with $\calS_{\bbul,\tau^*,\bbul\bbul\bbul}[\lambda_0]$ defined by \eqref{alt6-pre}, one has $\mSRab\pa{\calS_{\bbul,\tau^*,\bbul\bbul\bbul}[\lambda_0]}  = + \infty$.
\end{lemma}
This preliminary result leads us to consider, for $R >\lambda_0$, the restricted set of alternatives
\begin{multline}\label{alt6}
\bold{[Alt.6]}\  \calS_{\bbul,\tau^*,\bbul\bbul\bbul}[\lambda_0,R]=
 \big\{ \lambda:[0,1]\to (0,R],\ \exists \delta \in (- \lambda_{0},R- \lambda_{0}]\setminus\lbrace 0\rbrace ,\\
 \exists \ell \in (0,1-\tau^*),\ \forall t\in[0,1]\quad  \lambda(t) = \lambda_{0} + \delta \mathds{1}_{(\tau^*,\tau^*+\ell]}(t) \big\}\enspace.
\end{multline}
For the problem of testing $\hzero \ "\lambda\in \calS_0[\lambda_0]=\{\lambda_0\}"$ versus $\hone\ "\lambda\in\calS_{\bbul,\tau^*,\bbul\bbul\bbul}[\lambda_0,R]"$,
we then obtain the following lower bound.

\begin{proposition}[Minimax lower bound for $\bold{[Alt.6]}$] \label{LBalt6}
Let $\alpha,\beta$ in $(0,1)$ with $\alpha + \beta <1/2$, $\lambda_0\!>\!0$, $R\!>\!\lambda_0$, $\tau^*$ in $(0,1)$. There exists $L_0(\alpha,\beta,\lambda_0,R)\!>\!0$  such that for $L\!\geq\! L_0(\alpha,\beta,\lambda_0,R)$,
 \[\mSRab\pa{\calS_{\bbul,\tau^*,\bbul\bbul\bbul}[\lambda_0,R]} \geq  \sqrt{{\lambda_{0}\log\log L}/{L}}\enspace.\]
\end{proposition}

Let us now assume that $L\geq 3$. In order to prove that the above lower bound is of sharp order (with respect to $L$), we construct two aggregated tests: a first one  based on 
a linear statistic and a second one based on quadratic statistic as in Section \ref{Sec:knownlength}.

We thus consider the discrete subset of $(0,1-\tau^*)$ of the dyadic form
\[\left\{ \ell_{\tau^*,k}=\pa{1 - \tau^{*}{}}{2^{-k}}; ~ k \in \lbrace 1,\ldots, \lfloor \log_{2} L \rfloor \rbrace \right\}\enspace,\]
and the corrected level $u_\alpha=\alpha/\lfloor \log_{2}(L) \rfloor$, which allow to define the two following tests:
\begin{equation} \label{testN1alt6}
\phi_{6,\alpha}^{(1)}(N) = \1{\max_{k\in \lbrace 1,\ldots, \lfloor \log_{2} L \rfloor\rbrace} \left( \left|S_{\tau^*,\tau^*+\ell_{\tau^*,k}}(N)\right| - s_{\lambda_0,\tau^*,\tau^*+\ell_{\tau^*,k}} \pa{1 - u_\alpha} \right)>0}\enspace,
\end{equation}
where $S_{\tau_1,\tau_2}(N)$ is the linear statistic defined for $0\leq \tau_1<\tau_2\leq 1$ by
\begin{equation}\label{stat_alt6_N1}
S_{\tau_1,\tau_2}(N)= N(\tau_1, \tau_2] - \lambda_{0} (\tau_2-\tau_1) L \enspace,
\end{equation}
and $s_{\lambda_0,\tau_1,\tau_2}(u)$ stands for the $u$-quantile of $\big| S_{\tau_1,\tau_2}(N)\big|$ under the null hypothesis $\hzero$, and
\begin{equation} \label{testN2alt6}
\phi_{6,\alpha}^{(2)}(N) = \1{\max_{k\in \lbrace 1,\ldots, \lfloor \log_{2} L \rfloor \rbrace} \left(  T_{\tau^*, \tau^*+\ell_{\tau^*,k} }(N)  - t_{\lambda_0,\tau^*,\tau^*+\ell_{\tau^*,k}} \pa{1 - u_\alpha} \right)>0}\enspace,
\end{equation}
where $T_{\tau_1,\tau_2}(N)$ is the quadratic statistic \eqref{estimateurSB}, and $t_{\lambda_0,\tau_1,\tau_2}(u)$ its $u$-quantile under $\hzero$.

\begin{proposition}[Minimax upper bound for $\bold{[Alt.6]}$] \label{UBalt6}
Let $\alpha$ and  $\beta$ in $(0,1)$, $\lambda_{0}>0$, $R>\lambda_0$ and $\tau^*$ in $(0,1)$. Let $\phi_{6,\alpha}^{(1/2)}$ be one of the tests $\phi_{6,\alpha}^{(1)}$ and  $\phi_{6,\alpha}^{(2)}$ of $\hzero$ versus $\hone\ "\lambda\in\calS_{\bbul,\tau^*,\bbul\bbul\bbul}[\lambda_0,R]"$ respectively defined by \eqref{testN1alt6} and \eqref{testN2alt6}. Then $\phi_{6,\alpha}^{(1/2)}$ is of level $\alpha$, that is $P_{\lambda_0}\pa{\phi_{6,\alpha}^{(1/2)}(N)=1}\leq \alpha$.
Moreover, there exists $C(\alpha, \beta, \lambda_{0},R)>0$ such that
$$ \SRb\pa{\phi_{6,\alpha}^{(1/2)},\calS_{\bbul,\tau^*,\bbul\bbul\bbul}[\lambda_0,R]} \leq C(\alpha, \beta, \lambda_{0},R)\sqrt{{\log \log L}/{L}}\enspace,$$
which entails in particular $\mSRab\pa{\calS_{\bbul,\tau^*,\bbul\bbul\bbul}[\lambda_0,R]} \leq C(\alpha, \beta, \lambda_{0},R)\sqrt{{\log \log L}/{L}}$.
\end{proposition}

\emph{Comments.} The proofs of these upper bounds are mainly based on a sharp control of the quantile $s_{\lambda_0,\tau_1,\tau_2}(u)$  derived from an exponential inequality for the supremum of the martingale associated with a counting process due to Le Guével \cite{LeGuevel2021} (see details in Section~\ref{Sec:FTresults}), and of the quantile  $t_{\lambda_0,\tau_2,\tau_2}(u)$ already used in the proof of Proposition \ref{UBalt3}. This result, combined with its corresponding lower bound, brings out a first phase transition in the minimax separation rates orders, from the parametric rate order $1/\sqrt{L}$ to $\sqrt{{\log \log L}/{L}}$. This means that adaptation with respect to both height and length of the bump has a $\sqrt{\log \log L}$ cost, while adaptation to only one of these parameters does not cause any additional price, nor adaptation to both height and location as noticed above. Though a comparable phase transition has already been observed in Gaussian models when dealing with the jump detection problem (where adaptation with respect to the location is equivalent to adaptation with respect to the length), up to our knowledge, such results did not appear yet in the bump detection literature.

\subsection{Minimax detection of a possibly transitory change with unknown location and length} \label{Sec:unknownlocationlength}

In this subsection, we address the final problem of testing the null hypothesis $\hzero \ "\lambda\in \calS_0[\lambda_0]=\{\lambda_0\}"$ versus alternative hypotheses where both location and length of the change from the baseline intensity are unknown, distinguishing the case where the change is transitory from the particular case where it is not transitory.

\smallskip

Still adopting the minimax point of view, we will see that when considering the transitory change detection problem, adaptation to both change location and length has a minimax separation rate cost of order $\sqrt{\log L}$, and this whether the change height is known or not. 

This highly contrasts with the study of the particular non transitory change or jump detection problem, which makes two different regimes of minimax separation rates appear, with a maximal cost of order $\sqrt{\log \log L}$ for change height adaptation. 

Let us underline that the non transitory change or jump detection problem can be viewed as perfectly symmetrical to the above transitory change with known location detection problem  (see Section \ref{Sec:knownlocation}). In the first problem, one can consider that the length of the change is unknown but the endpoint of the change is known,  while in the second problem the length of the change is unknown but the starting point of the change is known. The study of the first non transitory change detection problem will therefore use very similar arguments as the study of the second transitory change with known location detection problem, finally leading to the same minimax separation rates. This is why we conduct it first here. 

\subsubsection{Non transitory change}\label{Sec:generalsingle}

In order to investigate the problem of detecting a non transitory change or jump with unknown location, but known height, we introduce for $\delta^*$ in $(- \lambda_{0}, + \infty)\setminus\{0\}$ the alternative set
\begin{equation}\label{alt7}
\bold{[Alt.7]}\ \calS_{\delta^*,\bbul\bbul,1-\bbul\bbul}[\lambda_0]= \lbrace \lambda:[0,1]\to (0+\infty),\ \exists \tau\in (0,1),~   \lambda(t) = \lambda_{0} + \delta^*\mathds{1}_{(\tau,1]}(t) \rbrace\enspace.
\end{equation}

This allows us to formalise the considered detection problem as a problem of testing the null hypothesis $\hzero \ "\lambda\in \calS_0[\lambda_0]=\{\lambda_0\}"$ versus
the alternative $\hone\ "\lambda\in\calS_{\delta^*,\bbul\bbul,1-\bbul\bbul}[\lambda_0]"$, with the corresponding minimax lower bound stated below. 

\begin{proposition}[Minimax lower bound for $\bold{[Alt.7]}$] \label{LBalt7}
Let $\alpha$ and  $\beta$ in $(0,1)$  such that $\alpha+\beta<1$, $\lambda_0>0$ and $\delta^*$ in $(- \lambda_{0}, + \infty)\setminus\{0\}$. For all $L \geq \lambda_{0}  \log C_{\alpha, \beta} /\delta^{*}{}^2$, 
\[\mSRab\pa{\calS_{\delta^*,\bbul\bbul,1-\bbul\bbul}[\lambda_0]} \geq  \sqrt{\lambda_{0} \log C_{\alpha, \beta}/L},\textrm{ with }C_{\alpha, \beta}=1+4(1- \alpha -\beta)^2\enspace.\]
\end{proposition}

Following the study and the notation of Section \ref{Sec:knownlocation}, we define the test
\begin{equation} \label{testalt7}
\phi_{7,\alpha}(N) = \1{\sup_{\tau \in (0,1)} S_{\delta^*,\tau,1}(N) >s_{\lambda_0,\delta^*,L}^+(1-\alpha)}\enspace,
\end{equation}
where $S_{\delta^*,\tau_1,\tau_2}(N)$ is the statistic defined  by \eqref{stat_alt5} and $s_{\lambda_0,\delta^*,L}^+(u)$ stands for the $u$-quantile of $\sup_{\tau \in (0,1)}S_{\delta^*,\tau,1}(N)$ under $\hzero$.

\begin{proposition}[Minimax upper bound for $\bold{[Alt.7]}$] \label{UBalt7}
Let $L\geq 1$, $\alpha$ and  $\beta$ in $(0,1)$, $\lambda_{0}>0$ and $\delta^*$ in $(- \lambda_{0}, + \infty)\setminus\{0\}$. Let $\phi_{7,\alpha}$ be the test of $\hzero$ versus $\hone\ "\lambda\in\calS_{\delta^*,\bbul\bbul,1-\bbul\bbul}[\lambda_0]"$ defined by \eqref{testalt7}. Then $\phi_{7,\alpha}$ is of level $\alpha$, that is $P_{\lambda_0}\pa{\phi_{7,\alpha}(N)=1}\leq \alpha$. Moreover, there exists a constant $C(\alpha, \beta, \lambda_{0},\delta^*)>0$ such that
$$ \SRb\pa{\phi_{7,\alpha},\calS_{\delta^*,\bbul\bbul,1-\bbul\bbul}[\lambda_0]} \leq {C(\alpha, \beta, \lambda_{0},\delta^*)}/{\sqrt{L}}\enspace,$$
which entails in particular $\mSRab\pa{\calS_{\delta^*,\bbul\bbul,1-\bbul\bbul}[\lambda_0]} \leq C(\alpha, \beta, \lambda_{0}, \delta^*)/\sqrt{L}$.
\end{proposition}

Let us now tackle the question of adaptation with respect to the change height and therefore introduce to this end a preliminary alternative set
\begin{multline}\label{alt8-pre}
\calS_{\bbul,\bbul\bbul,1-\bbul\bbul}[\lambda_0]= \lbrace \lambda\!: \exists \delta \in (- \lambda_{0}, + \infty)\setminus\{0\}, \exists \tau\in (0,1),~   \lambda(t) = \lambda_{0} + \delta \mathds{1}_{(\tau,1]}(t) \rbrace\enspace.
\end{multline}

As in Section \ref{Sec:knownlocation}, we underline that the minimax separation rate over this set is infinite.
\begin{lemma} \label{LBalt8-pre}
Let $\alpha,\beta$ in $(0,1)$ such that $\alpha + \beta <1$. For the problem of testing
  $\hzero \ "\lambda\in \calS_0[\lambda_0]=\{\lambda_0\}"$ versus $\hone\ "\lambda\in\calS_{\bbul,\bbul\bbul,1-\bbul\bbul}[\lambda_0]"$, with $\calS_{\bbul,\bbul\bbul,1-\bbul\bbul}[\lambda_0]$ defined by \eqref{alt8-pre}, one has 
$\mSRab\pa{\calS_{\bbul,\bbul\bbul,1-\bbul\bbul}[\lambda_0]}  = + \infty$. 
\end{lemma}
We thus consider for $R >\lambda_0$ the more suitable set of alternatives bounded by $R$, defined by
\begin{multline}\label{alt8}
\bold{[Alt.8]}\  \calS_{\bbul,\bbul\bbul,1-\bbul\bbul}[\lambda_0,R]=
 \big\{ \lambda:[0,1]\to(0,R],\ \exists \delta \in (- \lambda_{0},R- \lambda_{0}]\setminus\lbrace 0\rbrace ,\\
 \exists \tau \in (0,1),\ \forall t\in [0,1]\quad  \lambda(t) = \lambda_{0} + \delta \mathds{1}_{(\tau,1]}(t) \big\}\enspace.
\end{multline}
Considering the problem of testing $\hzero \ "\lambda\in \calS_0[\lambda_0]=\{\lambda_0\}"$ versus $\hone\ "\lambda\in\calS_{\bbul,\bbul\bbul,1-\bbul\bbul}[\lambda_0,R]"$, 
we then obtain the following lower bound.

\begin{proposition}[Minimax lower bound for $\bold{[Alt.8]}$] \label{LBalt8}
Let $\alpha, \beta$ in $(0,1)$  with $\alpha + \beta <1/2$, $\lambda_0>0$ and $R>\lambda_0$. There exists $L_0(\alpha,\beta,\lambda_0,R)>0$  such that for $L\geq L_0(\alpha,\beta,\lambda_0,R)$,
 \[\mSRab\pa{\calS_{\bbul,\bbul\bbul,1-\bbul\bbul}[\lambda_0,R]} \geq  \sqrt{{\lambda_{0}\log\log L}/{L}}\enspace.\]
\end{proposition}

Again, following the study and the notation of Section \ref{Sec:knownlocation}, we assume now that $L\geq 3$, we consider the discrete subset of $(0,1)$ of the dyadic form
\[\left\{\tau_k=1 -2^{-k}; ~ k \in \lbrace 1,\ldots,\lfloor \log_{2} L \rfloor\rbrace \right\}\enspace,\]
and  we set $u_\alpha=\alpha/\lfloor \log_{2}(L) \rfloor$, which allows to define the two following tests:
\begin{equation} \label{testN1alt8}
\phi_{8,\alpha}^{(1)}(N) = \1{\max_{k\in \lbrace 1,\ldots, \lfloor \log_{2} L \rfloor \rbrace} \left( |S_{\tau_k,1}(N)| - s_{\lambda_0,\tau_k,1} \pa{1 - u_\alpha} \right)>0}\enspace,
\end{equation}
where $S_{\tau_1,\tau_2}(N)$ is the linear statistic defined by \eqref{stat_alt6_N1}, $s_{\lambda_0,\tau_1,\tau_2}(u)$ the $u$-quantile of $| S_{\tau_1,\tau_2}(N)|$ under $\hzero$, and
\begin{equation} \label{testN2alt8}
\phi_{8,\alpha}^{(2)}(N) = \1{\max_{k\in \lbrace 1,\ldots, \lfloor \log_{2} L \rfloor \rbrace} \left(  T_{\tau_k,1 }(N)  - t_{\lambda_0,\tau_k,1} \pa{1 - u_\alpha} \right)>0}\enspace,
\end{equation}
where $T_{\tau_1,\tau_2}(N)$ is the quadratic statistic \eqref{estimateurSB} and $t_{\lambda_0,\tau_1,\tau_2}(u)$ its $u$-quantile  under $\hzero$.

\begin{proposition}[Minimax upper bound for $\bold{[Alt.8]}$] \label{UBalt8}
Let $\alpha, \beta$ in $(0,1)$ with $\alpha+\beta<1$, $\lambda_{0}>0$ and $R>\lambda_0$. Let $\phi_{8,\alpha}^{(1/2)}$ be one of the tests $\phi_{8,\alpha}^{(1)}$ and  $\phi_{8,\alpha}^{(2)}$ of $\hzero$ versus $\hone\ "\lambda\in\calS_{\bbul,\bbul\bbul, 1-\bbul\bbul}[\lambda_0,R]"$ respectively defined by \eqref{testN1alt8} and \eqref{testN2alt8}. Then $\phi_{8,\alpha}^{(1/2)}$ is of level $\alpha$, that is $P_{\lambda_0}\pa{\phi_{8,\alpha}^{(1/2)}(N)=1}\leq \alpha$.
Moreover, there exists a constant $C(\alpha, \beta, \lambda_{0},R)>0$ such that
$$ \SRb\pa{\phi_{8,\alpha}^{(1/2)},\calS_{\bbul,\bbul\bbul,1-\bbul\bbul}[\lambda_0,R]} \leq C(\alpha, \beta, \lambda_{0},R)\sqrt{{\log \log L}/{L}}\enspace,$$
which entails in particular $\mSRab\pa{\calS_{\bbul,\bbul\bbul,1-\bbul\bbul}[\lambda_0,R]} \leq C(\alpha, \beta, \lambda_{0},R)\sqrt{{\log \log L}/{L}}$.
\end{proposition}

\subsubsection{Transitory change} \label{Sec:transitory}

In this section, we address the transitory change detection problem, focusing here on the question of adaptation to unknown location and length.

As explained above, we will see that minimax adaptation to these both parameters has the most important cost in the present study, whose order is as large as $\sqrt{\log L}$, so that adaptation to the height will have no additional cost.

\smallskip

Let us first give lower bounds for the minimax separation rates, focusing on the case where the change height is known since the general case where all three parameters, location, length and height of the change are unknown then follows easily.

Hence, we introduce for $\delta^*$ in $(- \lambda_{0}, + \infty)\setminus\{0\}$ the alternative set
\begin{multline}\label{alt9}
\bold{[Alt.9]}\ \calS_{\delta^*,\bbul\bbul,\bbul\bbul\bbul}[\lambda_0]= \big\{ \lambda:[0,1]\to(0,+\infty),\ \exists \tau\in (0,1),\exists \ell \in (0,1-\tau),\\
\forall t\in [0,1]\quad \lambda(t) = \lambda_{0} + \delta^*\mathds{1}_{(\tau,\tau+\ell]}(t) \big\}\enspace.
\end{multline}

\begin{proposition}[Minimax lower bound for $\bold{[Alt.9]}$] \label{LBalt9}
Let $\alpha, \beta$ in $(0,1)$, $\lambda_0>0$ and $\delta^*$ in $(- \lambda_{0}, + \infty)\setminus\{0\}$. There exists $L_0(\alpha,\beta,\lambda_0,\delta^*)>0$ such that for all $L \geq L_0(\alpha,\beta,\lambda_0,\delta^*)$, 
\[\mSRab\pa{\calS_{\delta^*,\bbul\bbul,\bbul\bbul\bbul}[\lambda_0]} \geq  \sqrt{\lambda_{0} \log L/(2L)}\enspace.\]
\end{proposition}

Now considering the very general alternative set 
\vspace{-0.3cm}
\begin{multline}\label{alt10-pre}
\calS_{\bbul,\bbul\bbul,\bbul\bbul}[\lambda_0]= \big\{ \lambda:[0,1]\to(0,+\infty),\  \exists \delta \in (- \lambda_{0}, + \infty)\setminus\{0\}, \exists \tau\in (0,1),\\
\exists \ell \in (0,1-\tau),\ \forall t\in[0,1]\quad   \lambda(t) = \lambda_{0} + \delta \mathds{1}_{(\tau,\tau+\ell]}(t) \big\}\enspace,
\end{multline}
since it contains $\calS_{\bbul,\tau^*,\bbul\bbul\bbul}[\lambda_0]$ defined by \eqref{alt6-pre} for any $\tau^*$ in $(0,1)$, Lemma \ref{LBalt6-pre} straightforwardly leads to an infinite minimax separation rate lower bound.
\begin{corollary} \label{LBalt10-pre}
Let $\alpha,\beta$ in $(0,1)$ such that $\alpha + \beta <1$. For the problem of testing
  $\hzero \ "\lambda\in \calS_0[\lambda_0]=\{\lambda_0\}"$ versus $\hone\ "\lambda\in\calS_{\bbul,\bbul\bbul,\bbul\bbul\bbul}[\lambda_0]"$, with $\calS_{\bbul,\bbul\bbul,\bbul\bbul\bbul}[\lambda_0]$ defined by \eqref{alt10-pre}, one has $\mSRab\pa{\calS_{\bbul,\bbul\bbul,\bbul\bbul\bbul}[\lambda_0]}  = + \infty$.
\end{corollary}
We therefore restrict the alternative set to the one defined for $R \geq \lambda_0$ by
\begin{multline}\label{alt10}
\bold{[Alt.10]}\  \calS_{\bbul,\bbul\bbul,\bbul\bbul\bbul}[\lambda_0,R]=
 \big\{ \lambda:[0,1]\to(0,R],\ \exists \delta \in (- \lambda_{0},R- \lambda_{0}]\setminus\lbrace 0\rbrace ,\exists \tau \in (0,1),\\
 \exists \ell\in (0,1-\tau),\ \forall t\in[0,1]\quad \lambda(t) = \lambda_{0} + \delta \mathds{1}_{(\tau,\tau+\ell]}(t) \big\}\enspace,
\end{multline}
 and deal with the problem of testing $\hzero \ "\lambda\in \calS_0[\lambda_0]=\{\lambda_0\}"$ versus $\hone\ "\lambda\in\calS_{\bbul,\bbul\bbul,\bbul\bbul\bbul}[\lambda_0,R]"$. This alternative set $\calS_{\bbul,\bbul\bbul,\bbul\bbul\bbul}[\lambda_0,R]$ includes $\calS_{R-\lambda_0,\bbul\bbul,\bbul\bbul\bbul}[\lambda_0]$   defined by \eqref{alt9} when $R>\lambda_0$, and $\calS_{-\lambda_0/2,\bbul\bbul,\bbul\bbul\bbul}[\lambda_0]$ when $\lambda_0=R$. Therefore, Proposition \ref{LBalt9} has the following direct corollary, whose proof as well as the proof of Corollary \ref{LBalt10-pre} is omitted for simplicity.
  
\begin{corollary}[Minimax lower bound for $\bold{[Alt10]}$] \label{LBalt10}
Let $\alpha,\beta$ in $(0,1)$ with $\alpha+\beta<1$, $\lambda_0>0$ and $R\geq \lambda_0$. There exists $L_0(\alpha,\beta,\lambda_0,R)>0$ such that for all $L \geq L_0(\alpha,\beta,\lambda_0,R)$, 
\[\mSRab\pa{\calS_{\bbul,\bbul\bbul,\bbul\bbul\bbul}[\lambda_0,R]} \geq  \sqrt{\lambda_{0} \log L/(2L)}\enspace.\]
\end{corollary}

In order to prove that the above lower bounds are sharp, we secondly construct two novel minimax adaptive tests according to a scanning aggregation principle again.
 
As expected, the test named $\phi_{9/10,\alpha}^{(1)}$ is based on the linear statistic $S_{\tau_1,\tau_2}(N)$ defined by \eqref{stat_alt6_N1} and the $u$-quantiles of $| S_{\tau_1,\tau_2}(N)|$ under $\hzero$ denoted by $s_{\lambda_0,\tau_1,\tau_2}(u)$, while the test named $\phi_{9/10,\alpha}^{(2)}$ is based on the quadratic statistic $T_{\tau_1,\tau_2}(N)$ defined by \eqref{estimateurSB} and its $u$-quantiles  under $\hzero$ denoted by $t_{\lambda_0,\tau_1,\tau_2}(u)$. Since the lower bound shows an additional cost for adaptation to change location and length of order $\sqrt{\log L}$ instead of at most $\sqrt{\log \log L}$ when dealing with adaptation to only one of these parameters, we do not necessarily need to consider a dyadic set of aggregated tests. More precisely, setting  $u_\alpha=2\alpha/( \lceil L \rceil( \lceil L \rceil+1))$, we define
\begin{equation} \label{testN1alt9-10}
\phi_{9/10,\alpha}^{(1)}(N) = \1{\max_{k\in \lbrace 0,\ldots, \lceil L \rceil -1\rbrace, k'\in \lbrace 1,\ldots, \lceil L \rceil-k\rbrace} \left( \big|S_{\frac{k}{\lceil L\rceil},\frac{k+k'}{\lceil L\rceil}}(N)\big| - s_{\lambda_0,\frac{k}{\lceil L\rceil},\frac{k+k'}{\lceil L\rceil}}\pa{1 - u_\alpha} \right)>0}\enspace,
\end{equation}
and
\begin{equation} \label{testN2alt9-10}
\phi_{9/10,\alpha}^{(2)}(N) = \1{\max_{k\in \lbrace 0,\ldots, \lceil L \rceil -1\rbrace, k'\in \lbrace 1,\ldots, \lceil L \rceil-k\rbrace} \left(  T_{\frac{k}{\lceil L\rceil},\frac{k+k'}{\lceil L\rceil}}(N)  - t_{\lambda_0,\frac{k}{\lceil L\rceil},\frac{k+k'}{\lceil L\rceil}} \pa{1 - u_\alpha} \right)>0}\enspace.
\end{equation}
\begin{proposition}[Minimax upper bound for $\bold{[Alt.10]}$] \label{UBalt10}
Let $\alpha,\beta$  in $(0,1)$, $\lambda_{0}>0$ and $R\geq \lambda_0$. Let $\phi_{9/10,\alpha}^{(1/2)}$ be one of the tests $\phi_{9/10,\alpha}^{(1)}$ and  $\phi_{9/10,\alpha}^{(2)}$ of $\hzero$ versus $\hone\ "\lambda\in\calS_{\bbul,\bbul\bbul,\bbul\bbul\bbul}[\lambda_0,R]"$ respectively defined by \eqref{testN1alt9-10} and \eqref{testN2alt9-10}. Then $\phi_{9/10,\alpha}^{(1/2)}$ is of level $\alpha$, that is $P_{\lambda_0}\pa{\phi_{9/10,\alpha}^{(1/2)}(N)=1}\leq \alpha$.
Moreover, there exists a constant $C(\alpha, \beta, \lambda_{0},R)>0$ such that
$$ \SRb\pa{\phi_{9/10,\alpha}^{(1/2)},\calS_{\bbul,\bbul\bbul,\bbul\bbul\bbul}[\lambda_0,R]} \leq C(\alpha, \beta, \lambda_{0},R)\sqrt{{\log L}/{L}}\enspace,$$
which entails in particular $\mSRab\pa{\calS_{\bbul,\bbul\bbul,\bbul\bbul\bbul}[\lambda_0,R]} \leq C(\alpha, \beta, \lambda_{0},R)\sqrt{{\log L}/{L}}$.
\end{proposition}

As in particular $\calS_{\delta^*,\bbul\bbul,\bbul\bbul\bbul}[\lambda_0]$ is included in $\calS_{\bbul,\bbul\bbul,\bbul\bbul\bbul}[\lambda_0,\lambda_0+\delta^*]$ for any $\delta^*>0$ and $\calS_{\bbul,\bbul\bbul,\bbul\bbul\bbul}[\lambda_0,\lambda_0]$ for any $\delta^*$ in $(- \lambda_{0},0)$, Proposition \ref{UBalt10} has the following immediate corollary, which closes the study of possibly transitory change in a known baseline intensity detection.

\begin{corollary}[Minimax upper bound for $\bold{[Alt.9]}$] \label{UBalt9}
Let $\alpha,\beta$ in $(0,1)$, $\lambda_{0}>0$ and $\delta^*$ in $(- \lambda_{0}, + \infty)\setminus\{0\}$. Let $\phi_{9/10,\alpha}^{(1/2)}$ be one of the tests $\phi_{9/10,\alpha}^{(1)}$ and  $\phi_{9/10,\alpha}^{(2)}$ of $\hzero$ versus $\hone\ "\lambda\in\calS_{\delta^*,\bbul\bbul,\bbul\bbul\bbul}[\lambda_0]"$ respectively defined by \eqref{testN1alt9-10} and \eqref{testN2alt9-10}. Then $\phi_{9/10,\alpha}^{(1/2)}$ is of level $\alpha$, that is $P_{\lambda_0}\pa{\phi_{9/10,\alpha}^{(1/2)}(N)=1}\leq \alpha$.
Moreover, there exists $C(\alpha, \beta, \lambda_{0},\delta^*)>0$ such that
$$ \SRb\pa{\phi_{9/10,\alpha}^{(1/2)},\calS_{\delta^*,\bbul\bbul,\bbul\bbul\bbul}[\lambda_0]} \leq C(\alpha, \beta, \lambda_{0},\delta^*)\sqrt{{\log L}/{L}}\enspace,$$
which entails in particular $\mSRab\pa{\calS_{\delta^*,\bbul\bbul,\bbul\bbul\bbul}[\lambda_0]} \leq C(\alpha, \beta, \lambda_{0},\delta^*)\sqrt{{\log L}/{L}}$.
\end{corollary}
\emph{Comment.} The upper bounds in Proposition \ref{UBalt10} and Corollary \ref{UBalt9}, combined with their corresponding lower bounds, bring out a second phase transition in the minimax separation rate orders, from the rate order $\sqrt{{\log \log L}/{L}}$ when considering adaptation with respect to both bump height and length to $\sqrt{{\log L}/{L}}$, obtained when dealing with adaptation to at least bump location and length (with no additional cost when adapting to the bump height). As comparable minimax separation rates were already known in Gaussian models with \cite{Arias-Castro2005} and \cite{Brunel2014}, these results were more expected that some of the above ones.

\subsection{Choice of individual levels for aggregated tests and link with multiple tests}\label{discussionlevels} Except the tests $\phi_{1,\alpha}^{-}$,  $\phi_{1,\alpha}^{+}$, $\phi_{2,\alpha}^{(1)}$ and $\phi_{2,\alpha}^{(2)}$ that are classical single tests derived from the fundamental Neyman-Pearson Lemma, all the tests introduced in the above study are based on aggregation principles. Among them, we can essentially distinguish two kinds of such aggregated tests.

The first aggregated test type is of the form 
\begin{equation}\label{aggregated1}
\phi_{agg1,\alpha}(N)=\1{\sup_{\theta\in\Theta} S_{\theta}(N)>s_{\lambda_0}^{+}(1-\alpha)}\enspace,
\end{equation}
where:
\begin{itemize}
\item $\theta$ is one possible parameter or couple of parameters among  the location $\tau$ or length $\ell$ of the bump/jump in the alternative  intensity, and $\Theta$ is a subset of possible values for $\theta$,
\item  $S_\theta$ is a statistic designed to test $(H_0)$ "$\lambda=\lambda_0$" versus $(H_1)$ "$\lambda$ has a jump or a bump with parameter or parameters $\theta$", such that $\sup_{\theta\in\Theta} S_{\theta}(N)$ has a computable, exactly or by a Monte Carlo method, $(1-\alpha)$-quantile $s_{\lambda_0}^+(1-\alpha)$ under $(H_0)$.
\end{itemize}
The tests $\phi_{3,\alpha}^{(1)-}$, $\phi_{3,\alpha}^{(1)+}$, $\phi_{5,\alpha}$ and $\phi_{7,\alpha}$ can all be written in this way.

\smallskip

The second aggregated test type is of the form 
\begin{equation}\label{aggregated2}
\phi_{agg2,\alpha}(N)=\1{\sup_{\theta\in\Theta} \pa{S_{\theta}(N)-s_{\lambda_0,\theta}(1-u_\alpha)}>0}\enspace,
\end{equation}
where:
\begin{itemize}
\item $\theta$ is one possible parameter or couple of parameters among the location $\tau$ or length $\ell$ of the bump/jump in the alternative  intensity, and $\Theta$ is a finite subset of possible values for $\theta$,
\item  $S_\theta$ is a statistic designed to test $(H_0)$ "$\lambda=\lambda_0$" versus $(H_1)$ "$\lambda$ has a jump or a bump with parameter or parameters $\theta$", with computable $(1-u)$-quantiles $s_{\lambda_0,\theta}(1-u)$ under $(H_0)$,
\item $u_\alpha$ is an individual, adjusted and smaller than $\alpha$, level of test.
\end{itemize}
The tests $\phi_{4,\alpha}^{(1)}$, $\phi_{3/4,\alpha}^{(2)}$, $\phi_{6,\alpha}^{(1)}$, $\phi_{6,\alpha}^{(2)}$, $\phi_{8,\alpha}^{(1)}$, $\phi_{8,\alpha}^{(2)}$, $\phi_{9/10,\alpha}^{(1)}$, $\phi_{9/10,\alpha}^{(2)}$  can all be written in this way.

\smallskip

Notice that both $\phi_{agg1,\alpha}$ and $\phi_{agg2,\alpha}$ aggregated test types can be expressed as 
$$ \1{\sup_{\theta\in\Theta}(S_{\theta}(N)-c_{\lambda_0,\theta,\alpha})>0}=\1{\exists \theta \in \Theta,\ S_{\theta}(N)>c_{\lambda_0,\theta,\alpha}}\enspace,$$ where $c_{\lambda_0,\theta,\alpha}$ is a critical value such that $c_{\lambda_0,\theta,\alpha}=s_{\lambda_0}^{+}(1-\alpha)$ does not vary with $\theta$ in $\phi_{agg1,\alpha}$ and $c_{\lambda_0,\theta,\alpha}=s_{\lambda_0,\theta}(1-u_\alpha)$ varies with $\theta$ in  $\phi_{agg2,\alpha}$. This therefore means that these tests reject $(H_0)$ when, scanning all the parameters or couples of parameters $\theta$ in $\Theta$, at least one single test in the collection $\big\{\un{S_{\theta}(N)>c_{\lambda_0,\theta,\alpha}},\ \theta\in\Theta\big\}$ rejects $(H_0)$, which explains the name of scan aggregation principle.
All the single tests $\un{S_{\theta}(N)>c_{\lambda_0,\theta,\alpha}}$ in the considered collection are of level $\alpha$, but they can be in fact, individually, very conservative, otherwise their aggregation would not preserve the level $\alpha$ property \emph{in fine}. In the particular case of $\phi_{agg2,\alpha}$, the single tests are of individual level $u_\alpha$, taken here equal to $u_\alpha=\alpha/|\Theta|$. 

A better choice for $u_\alpha$, leading to a less conservative aggregated test, was first proposed in another context by Baraud et al. \cite{BHL2005}. In our context, this choice corresponds to
\begin{equation}\label{ualpha_minp}
u_\alpha'=\sup \set{u\in (0,1),\ P_{\lambda_0}\pa{\sup_{\theta\in\Theta} \pa{S_{\theta}(N)-s_{\lambda_0,\theta}(1-u)}>0}\leq \alpha}\enspace.
\end{equation}
Since $u_\alpha\leq u_\alpha'$, by definition, $s_{\lambda_0,\theta}(1-u_\alpha')\leq s_{\lambda_0,\theta}(1-u_\alpha)$. Any upper bound for $s_{\lambda_0,\theta}(1-u_\alpha)$, such as those used in the proofs of the minimax separation rates upper bounds and deduced from the quantiles bounds of Section~\ref{TechnicalRes}, therefore remains valid for $s_{\lambda_0,\theta}(1-u_\alpha')$. As a consequence, all the above tests of type $\phi_{agg2,\alpha}$ but with $u_\alpha'$ instead of $u_\alpha$, that we can denote by $\phi_{agg2,\alpha}'$, satisfy the same minimax properties as $\phi_{agg2,\alpha}$.

The fact that such adjusted aggregated tests $\phi_{agg2,\alpha}'$ are more powerful than $\phi_{agg2,\alpha}$ is not discernable in minimax results whereas it is clearly noticeable in practice, is a known shortcoming of the present nonasymptotic minimax point of view, where exact constants (making lower and upper bound match, up to a possible negligible term) are not expected, which is not solved yet up to our knowledge in any testing framework. Our simulation study presented in Section \ref{SimulationStudy} focuses on the performances of adjusted aggregated tests of the form $\phi_{agg2,\alpha}'$.

\smallskip

Let us now turn to the links that can be highlighted between such minimax adaptive, aggregated tests $\phi_{agg2,\alpha}$ or adjusted aggregated tests $\phi_{agg2,\alpha}'$ and multiple tests. The parallel between such aggregated tests and multiple tests has been established in \cite{FLRB2016}, as the foundation of a minimax theory for multiple tests. Notice that when each single test $\un{S_{\theta}(N)>c_{\lambda_0,\theta,\alpha}}$ in the collection $\big\{\un{S_{\theta}(N)>c_{\lambda_0,\theta,\alpha}},\ \theta\in\Theta\big\}$ can be interpreted as a test of a null hypothesis $(H_{0,\theta})\ \lambda\in \calS_{0,\theta}$ versus $(H_{1,\theta})\ \lambda\not \in \calS_{0,\theta}$,  with $\calS_0\subset \cap_{\theta\in\Theta} \calS_{0,\theta}$, it appears that our first choice of individual level $u_\alpha=\alpha/|\Theta|$ can be related to a Bonferroni type multiple test of the set of hypotheses  $\{(H_{0,\theta}),\  \theta\in \Theta\}$, while our second choice $u_\alpha'$ defined by \eqref{ualpha_minp}
can be related to a min-$p$ type multiple test of the same set of hypotheses.

Notice or recall that $S_{\tau_1,\tau_2}(N)$ and $T_{\tau_1,\tau_2}(N)$ defined by \eqref{stat_alt6_N1} and  \eqref{estimateurSB}  are unbiased estimators of $L(\tau_2-\tau_1) \Pi_{V_{\tau_1,\tau_2}}(\lambda-\lambda_0)$ and $\|\Pi_{V_{\tau_1,\tau_2}}(\lambda-\lambda_0)\|_2^2$ respectively (see Section~\ref{UMPknownSec}). Therefore, for instance, the single tests  $\un{|S_{\tau_k,1}(N)|>s_{\lambda_0,\tau_k,1}(1-u)}$ and  $\un{T_{\tau_k,1}(N) >s_{\lambda_0,\tau_k,1}(1-u)}$ involved in $\phi_{8,\alpha}^{(1)}(N)$ and $\phi_{8,\alpha}^{(2)}(N)$ can be viewed as single tests of  $(H_{0,\tau_k})\ \Pi_{V_{\tau_k,1}}(\lambda-\lambda_0)=0$ v.s. $(H_{1,\tau_k})\ \Pi_{V_{\tau_k,1}}(\lambda-\lambda_0)\neq 0$. In the same way, the single tests involved in $\phi_{9/10,\alpha}^{(1)}(N)$ and $\phi_{9/10,\alpha}^{(2)}(N)$ can be viewed as tests of $\Pi_{V_{k/\lceil L\rceil,(k+k')/{\lceil L\rceil}}}(\lambda-\lambda_0)=0$ versus $\Pi_{V_{k/\lceil L\rceil,(k+k')/{\lceil L\rceil}}}(\lambda-\lambda_0)\neq 0$. Our aggregated tests of the form $\phi_{agg2,\alpha}$ are thus clearly aggregated tests constructed from Bonferroni multiple tests, while the corresponding adjusted aggregated tests of the form $\phi_{agg2,\alpha}'$ are constructed from min-$p$ multiple tests of such collections of hypotheses (see \cite{FLRB2016} for some detailed study).

\subsection{Summary and discussion}\label{Discussion_known}
We present below a summary of the results stated above in a tabular form. Recall (\emph{c.f.} \eqref{estimateurSB}, \eqref{stat_alt5} and \eqref{stat_alt6_N1}) that for $0\leq \tau_1<\tau_2\leq 1$, 
$$ T_{\tau_1, \tau_2}(N) = \frac{1}{L^2 (\tau_2 - \tau_1)} \left(  N \left( \left.  \tau_1, \tau_2  \right] \right.^2 -N \left( \left.  \tau_1 , \tau_2  \right] \right.  \right) - \frac{2 \lambda_{0}}{L} N \left( \left.  \tau_1, \tau_2  \right] \right. + \lambda_{0}^2 (\tau_2 - \tau_1)\enspace,$$
$S_{\delta^*,\tau_1,\tau_2}(N)=  \mathrm{sgn}(\delta^{*}{}) (N(\tau_1,\tau_2] - \lambda_{0} L(\tau_2-\tau_1)) - \vert \delta^{*}{} \vert L(\tau_2-\tau_1)/2$, $S_{\tau_1,\tau_2}(N)= N(\tau_1, \tau_2] - \lambda_{0} (\tau_2-\tau_1) L$, and that
 $t_{\lambda_0,\tau_1,\tau_2}(u)$ and $s_{\lambda_0,\tau_1,\tau_2}(u)$ stand for the $u$-quantiles of $T_{\tau_1, \tau_2}(N)$ and  $| S_{\tau_1,\tau_2}(N)|$ under $\hzero$ respectively.  

\bigskip

\begin{tabular}{l c l}
\hline
\multicolumn{3}{l}{{\bf Transitory change or bump detection}}\\
\hline
Alternative set & $\mSRab$ & Test statistics\\
\hline
$\calS_{\delta^*,\tau^*,\ell^*}[\lambda_0]$ & - & $N(\tau^{*}, \tau^{*}+ \ell^{*}] $\\
\hline
$\calS_{\bbul,\tau^*,\ell^*}[\lambda_0]$ & $L^{-1/2}$ & $N(\tau^{*}, \tau^{*}+ \ell^{*}] $\\
&&$T_{\tau^*, \tau^*+\ell^*}(N)$\\
\hline
$ \calS_{\delta^*,\bbul\bbul,\ell^*}[\lambda_0]$ & - & $\max_{\tau\in [0,1-\ell^*]}N(\tau, \tau + \ell^*]$, $\min_{\tau\in [0,1-\ell^*]}N(\tau, \tau + \ell^*]$\\
&& $\max_{k\in \lbrace 0,\ldots, \lceil (1-\ell^*) M \rceil -1 \rbrace} \left( T_{\frac{k}{M}, \frac{k}{M}+\ell^*}(N) - t_{\lambda_0,\frac{k}{M}, \frac{k}{M}+\ell^*} \pa{1 - u_\alpha} \right)$ $M = \lceil 2/\ell^* \rceil$\\
\hline
$\calS_{\bbul,\bbul\bbul,\ell^*}[\lambda_0]$ & $L^{-1/2}$ & $\max_{\tau\in [0,1-\ell^*]}N(\tau, \tau + \ell^*]$, $\min_{\tau\in [0,1-\ell^*]}N(\tau, \tau + \ell^*]$\\
&& $\max_{k\in \lbrace 0,\ldots, \lceil (1-\ell^*) M \rceil -1 \rbrace} \left( T_{\frac{k}{M}, \frac{k}{M}+\ell^*}(N) - t_{\lambda_0,\frac{k}{M}, \frac{k}{M}+\ell^*} \pa{1 - u_\alpha} \right)$ $M = \lceil 2/\ell^* \rceil$\\
\hline
$\calS_{\delta^*,\tau^*,\bbul\bbul\bbul}[\lambda_0]$ & $L^{-1/2}$ & $\sup_{\ell \in (0,1- \tau^{*})}S_{\delta^*,\tau^*,\tau^*+\ell}(N)$\\
\hline
$\calS_{\bbul,\tau^*,\bbul\bbul\bbul}[\lambda_0,R]$ & $\sqrt{\frac{\log\log L}{L}}$ & $\max_{k\in \lbrace 1,\ldots,\lfloor \log_{2} L \rfloor\rbrace} \Big( \big|S_{\tau^*,\tau^*+\pa{1 - \tau^*}{2^{-k}}}(N)\big| - s_{\lambda_0,\tau^*,\tau^*+\pa{1 - \tau^*}{2^{-k}}} \pa{1 - u_\alpha} \Big)$\\
&& $\max_{k\in \lbrace 1,\ldots, \lfloor \log_{2} L \rfloor \rbrace} \left(  T_{\tau^*, \tau^*+\pa{1 - \tau^*}{2^{-k}} }(N)  - t_{\lambda_0,\tau^*,\tau^*+\pa{1 - \tau^*}{2^{-k}}} \pa{1 - u_\alpha} \right)$\\
\hline
$\calS_{\delta^*,\bbul\bbul,\bbul\bbul\bbul}[\lambda_0]$ & $\sqrt{\frac{\log L}{L}}$ & $\max_{k\in \lbrace 0,\ldots, \lceil L \rceil -1\rbrace, k'\in \lbrace 1,\ldots, \lceil L \rceil-k\rbrace} \Big( \big|S_{\frac{k}{\lceil L\rceil},\frac{k+k'}{\lceil L\rceil}}(N)\big| - s_{\lambda_0,\frac{k}{\lceil L\rceil},\frac{k+k'}{\lceil L\rceil}}\pa{1 - u_\alpha} \Big)$\\
&&$\max_{k\in \lbrace 0,\ldots, \lceil L \rceil -1\rbrace, k'\in \lbrace 1,\ldots, \lceil L \rceil-k\rbrace} \left(  T_{\frac{k}{\lceil L\rceil},\frac{k+k'}{\lceil L\rceil}}(N)  - t_{\lambda_0,\frac{k}{\lceil L\rceil},\frac{k+k'}{\lceil L\rceil}} \pa{1 - u_\alpha} \right)$\\
\hline
$ \calS_{\bbul,\bbul\bbul,\bbul\bbul\bbul}[\lambda_0,R]$ &$\sqrt{\frac{\log L}{L}}$ & $\max_{k\in \lbrace 0,\ldots, \lceil L \rceil -1\rbrace, k'\in \lbrace 1,\ldots, \lceil L \rceil-k\rbrace} \Big( \big|S_{\frac{k}{\lceil L\rceil},\frac{k+k'}{\lceil L\rceil}}(N)\big| - s_{\lambda_0,\frac{k}{\lceil L\rceil},\frac{k+k'}{\lceil L\rceil}}\pa{1 - u_\alpha} \Big)$\\
&& $\max_{k\in \lbrace 0,\ldots, \lceil L \rceil -1\rbrace, k'\in \lbrace 1,\ldots, \lceil L \rceil-k\rbrace} \left(  T_{\frac{k}{\lceil L\rceil},\frac{k+k'}{\lceil L\rceil}}(N)  - t_{\lambda_0,\frac{k}{\lceil L\rceil},\frac{k+k'}{\lceil L\rceil}} \pa{1 - u_\alpha} \right)$\\
\hline
\end{tabular}

\begin{tabular}{l c l}
\hline
\multicolumn{3}{l}{{\bf Non transitory change or jump detection}}\\
\hline
Alternative set & $\mSRab$ & Test statistics\\
\hline
$\calS_{\delta^*,\tau^*,1-\tau^*}[\lambda_0]$ & - & $N(\tau^{*}, 1] $\\
\hline
$\calS_{\bbul,\tau^*,1-\tau^*}[\lambda_0]$ & $L^{-1/2}$ & $N(\tau^{*},1] $\\
&&$T_{\tau^*, 1}(N)$\\
\hline
$\calS_{\delta^*,\bbul\bbul,1-\bbul\bbul}[\lambda_0]$ &$L^{-1/2}$ & $\sup_{\tau \in (0,1)} S_{\delta^*,\tau,1}(N)$\\
\hline
$\calS_{\bbul,\bbul\bbul,1-\bbul\bbul}[\lambda_0,R]$ & $\sqrt{\frac{\log\log L}{L}}$ & $\max_{k\in \lbrace 1,\ldots, \lfloor \log_{2} L \rfloor \rbrace} \Big( \big|S_{1 -2^{-k},1}(N)\big| - s_{\lambda_0,1 -2^{-k},1} \pa{1 - u_\alpha} \Big)$\\
&& $\max_{k\in \lbrace 1,\ldots, \lfloor \log_{2} L \rfloor \rbrace} \left(  T_{1 -2^{-k},1 }(N)  - t_{\lambda_0,1 -2^{-k},1} \pa{1 - u_\alpha} \right)$\\
\hline
\end{tabular}

\medskip

The present overview notably enables to highlight two main phase transitions in minimax separation rates. A phase transition from the smallest parametric rate order $1/\sqrt{L}$ to the intermediate rate order $\sqrt{{\log\log L}/{L}}$, due to adaptation to both height and length of the bump when dealing with the bump detection problem (BDP), or both height and location of the jump (which is in fact equivalent to adaptation to the bump length here) when dealing with the jump detection problem (JDP). A similar phase transition was already known in the independent Gaussian model when dealing with the JDP as explained in the introduction (see \cite{Gao2020} and \cite{Verzelen2021}). But the tools used in this Gaussian model, mainly based on Law of Iterated Logarithm exponential inequalities could not be used here, which led us to circumvent the difficulty via new exponential inequalities of Le Guével \cite{LeGuevel2021} combined with a dyadic type scan aggregation approach. Two points which seem important to us here are: first, in the BDP, adaptation to both bump height and location can be conducted without any additional cost as soon as the bump length is known; second, when adaptation to the length in the BDP or the location in the JDP is considered, the knowledge of the bump or jump height suffices to cancel any price to pay for adaptation. Up to our knowledge, such results were not known, even in classical Gaussian models. Constructing minimax adaptive tests actually required a careful analysis of the shifted Poisson process. Then, a phase transition from the intermediate rate order $\sqrt{{\log\log L}/{L}}$ to the largest rate order $\sqrt{{\log L}/{L}}$, due to adaptation to both position and length of the bump when dealing with the BDP. Notice that this rate is so large that additional adaptation to the height has no supplementary cost. Notice also that similar minimax separation rates were already known in the independent Gaussian model: Arias-Castro et al. \cite{Arias-Castro2005} handled the case where the height is unknown, while Brunel \cite{Brunel2014} handled the case where the height is known, equal to $1$ (therefore positive) within the asymptotic perspective, with linear statistics in the spirit of well-known CUSUM statistics. From this angle, our study provides nonasymptotic and Poisson processes counterparts for the Gaussian tools used in \cite{Arias-Castro2005} and \cite{Brunel2014}. But we furthermore introduce, in the unknown height case, a novel scan aggregated quadratic statistic: if it leads to the same minimax adaptive testing properties as the scan aggregated linear one, our simulation study shows that the corresponding test is often more powerful, especially when the height is negative, that is in depression detection problems. 

\section{Detecting an abrupt, possibly transitory, change in an unknown intensity}\label{Sec:unknownbaseline}

We now turn to the problem of detecting an abrupt change in the intensity of the Poisson process $N$ when its constant baseline is not assumed to be known anymore. This detection problem probably more largely fits applications, especially with epidemiological data for which it is often more realistic not to assume that the baseline intensity of the underlying Poisson process is known. In the present section, we therefore consider the null hypothesis expressed as $\hzero \ "\lambda\in \calS_0^u[R]"$, where $\calS_0^u[R]$ is the set of all possible constant intensities upper bounded by a given $R>0$. As in the above section, we consider various alternative hypotheses, that are defined according to the persistent or transitory nature of the change, and its height, location and length knowledge. In order to further cover the full range of alternatives in a unified notation, we introduce for $\delta^*$ in $(-R,R)\setminus\{0\}$, $\tau^*$ in $ (0,1)$ and $\ell^*$ in $(0,1-\tau^*]$ the set $\calS_{\delta^*,\tau^*,\ell^*}^u[R]$ of intensities with a change of height $\delta^*$, location $\tau^*$ and length $\ell^*$ from an unknown $\lambda_0$ in $\calS_0^u[R]$, and still upper bounded by $R$,
\begin{multline}\label{def_alt1_u}
\bold{[Alt^u.1]}\quad \calS_{\delta^*,\tau^*,\ell^*}^u[R]= \Big\{ \lambda:[0,1]\to (0,R],~\exists \lambda_0 \in (-\delta^{*} \vee 0, (R-\delta^{*}) \wedge R],\\
~\forall t \in [0,1]\quad \lambda(t) = \lambda_{0}{} + \delta^{*}{} \mathds{1}_{(\tau^{*}{}, \tau^{*}{} + \ell^{*}{}]}(t)\Big\}\enspace.\end{multline}
Though it is not as immediate as in Section \ref{Sec:knownbaseline}, testing $\hzero$ versus $\hone\  "\lambda\in \calS_{\delta^*,\tau^*,\ell^*}^u[R]"$ also falls within the scope of Neyman-Pearson tests, and an Uniformly Most Powerful Unbiased (UMPU) test can be constructed by using a conditioning trick (see details below). As above, when the question of adaptivity w.r.t. some unknown parameters is tackled, the unknown parameters are replaced by single, double or triple dots in the notation $\calS_{\delta^*,\tau^*,\ell^*}^u[R]$.

Notice that for any intensity $\lambda$ such that $\lambda(t)= \lambda_{0} + \delta \mathds{1}_{(\tau,\tau+\ell]}(t)$ for  $\delta$ in $(-R,R)\setminus\{0\}$, $\tau$ in $ (0,1)$, $\ell$ in $(0,1-\tau]$ and $\lambda_0$ in $(-\delta\vee 0, (R-\delta) \wedge R]$,
\[d_2(\lambda,\calS_0^u[R])=|\delta| \sqrt{\ell(1-\ell)}\enspace.\]
Hence, as soon as an alternative intensity has known change height $\delta=\delta^*$ and length $\ell=\ell^*$, the distance $d_2(\lambda,\calS_0^u[R])$ is fixed, equal to $|\delta^*| \sqrt{\ell^*(1-\ell^*)}$. Hence, the $\beta$-uniform separation rate of any level $\alpha$ test over $\calS_{\delta^*,\tau^*,\ell^*}^u[R]$ or $\calS_{\delta^*,\bbul,\ell^*}^u[R]$  is either $0$ or $+\infty$, and so is the $(\alpha,\beta)$-minimax separation rate. In these cases, our tests are not studied from the minimax point of view. As in Section  \ref{Sec:knownbaseline}, we nevertheless  establish conditions, expressed as a sufficient minimal distance $d_2(\lambda,\calS_0^u[R])$, guaranteeing that their second kind error rate is controlled by $\beta$. 

\subsection{Uniformly most powerful detection of a  possibly transitory change with known location and length}

Let us first focus on the problem of testing $\hzero \ "\lambda\in \calS_0^u[R]"$ versus $\hone\  "\lambda\in \calS_{\delta^*,\tau^*,\ell^*}^u[R]"$ with $\calS_{\delta^*,\tau^*,\ell^*}^u[R]$ defined by \eqref{def_alt1_u} for $\delta^{*}{}$ in $(-R,R)\setminus\{0\}$, $\tau^{*}{}$ in $(0,1)$ and  $\ell^{*}{}$ in $(0, 1-\tau^{*}{}]$. Assume here that $\lambda$ belongs to
$\lbrace \lambda:[0,1]\to (0,R], ~\exists \delta \in (-R,R),   \exists \lambda_0 \in (-\delta \vee 0, (R-\delta) \wedge R],~ \lambda= \lambda_{0}{} + \delta \mathds{1}_{(\tau^{*}{}, \tau^{*}{} + \ell^{*}{}]} \rbrace\supset \calS_0^u[R]\cup \calS_{\delta^*,\tau^*,\ell^*}^u[R]$. In this model parametrised by $(\delta,\lambda_0)$ in $\set{(\delta,\lambda_0),\ \delta \in (-R,R),  \lambda_0 \in (-\delta \vee 0, (R-\delta) \wedge R]}$, the distribution $P_\lambda$ is dominated by $P_1$  (see Lemma \ref{lemmegirsanov}), with a Likelihood Ratio given by
\[\frac{dP_\lambda}{dP_1}(N)=e^{\pa{\log(1+\delta/\lambda_0) N(\tau^{*}{}, \tau^{*}{} + \ell^{*}{}] +\log(\lambda_0) N(0,1] -L (\lambda_0 +\delta \ell^* -1)}}\enspace.\]
Reparametrising the model by $\theta_1= \log\pa{1+ \delta/\lambda_0}$ and $\theta_2= \log(\lambda_0)$, this LR becomes
\begin{equation}\label{LRreparam}
\frac{dP_\lambda}{dP_1}(N)=e^{-L\left(e^{\theta_2}\left(1+(e^{\theta_1}-1) \ell^*\right) -1\right)} e^{\theta_1 N(\tau^{*}{}, \tau^{*}{} + \ell^{*}{}] +\theta_2 N(0,1]}\enspace.
\end{equation}
Our testing problem can then be viewed as a problem of testing $\hzero "\theta_1= 0"$ versus $\hone "\theta_1<0"$ or $\hone "\theta_1>0"$ (depending on the sign of $\delta^*$) in an exponential model with natural parameters $\theta=(\theta_1,\theta_2)$ and sufficient statistics $(N(\tau^{*}{}, \tau^{*}{} + \ell^{*}{}],N(0,1])$, and where $\theta_2$ can be interpreted as a nuisance parameter. From \eqref{LRreparam} and Lemma 2.7.2 of \cite{Lehmann2006} we can deduce that given $N_{1}=n$, the conditional distribution of $N(\tau^{*}{}, \tau^{*}{} + \ell^{*}{}]$ defines an exponential family with respect to some measure $\nu_n$, with natural parameter $\theta_1$, and is in particular free of $\theta_2$. In this conditional framework, one knows that there exists an UMP test of $\hzero$ versus $\hone$ of the Neyman-Pearson form. Recalling that given $N_{1}=n$, $N(\tau^{*}{}, \tau^{*}{} + \ell^{*}{}]$ has the same distribution as a binomial random variable $Y_{n,\ell^*}$ with parameters $(n,\ell^{*}{})$,  such conditional Neyman-Pearson tests lead us to consider the unilateral tests defined by
\begin{equation} \label{test_alt1_u}
\left\{\begin{array}{llll}
\phi_{1,\alpha}^{u,-}(N) &=& \mathds{1}_{N(\tau^{*}{} ,\tau^{*}{}+\ell^{*}{} ] < b_{N_{1},\ell^*}(\alpha)} &+ \gamma_{(N_{1},\ell^*)}^{-}(\alpha) \mathds{1}_{N(\tau^{*}{} ,\tau^{*}{} +\ell^{*}{}] = b_{N_{1},\ell^*}(\alpha)}\text{ if $\delta^{*}<0$} \\
\phi_{1,\alpha}^{u,+}(N)&=& \mathds{1}_{N(\tau^{*}{} ,\tau^{*}{} +\ell^{*}{} ] > b_{N_{1},\ell^*}(1-\alpha)} &+ \gamma_{(N_{1},\ell^*)}^{+}(1-\alpha) \mathds{1}_{N(\tau^{*}{}, \tau^{*}{}+\ell^{*}{}] = b_{N_{1},\ell^*}(1-\alpha)}\text{ if $\delta^{*}>0$}  \enspace,
\end{array}\right.
\end{equation}
where for all $n$ in $\mathbb{N}$, $b_{n,\ell^*}(u)$ denotes the $u$-quantile of the distribution of $Y_{n,\ell^*}$, and
\begin{equation} \label{bUPP2g}
\gamma_{(n,\ell^*)}^{-}(u)=\frac{u - \mathbb{P}(Y_{n,\ell^*} < b_{n,\ell^*}(u) )}{\mathbb{P}(Y_{n,\ell^*} = b_{n,\ell^*}(u)) },\quad
\gamma_{(n,\ell^*)}^{+}(u)= 1- \gamma_{(n,\ell^*)}^{-}(u) \enspace.
\end{equation}
From Theorem 4.4.1 in \cite{Lehmann2006} and the remark below its proof, we obtain the following result.

\begin{proposition}[Uniformly Most Powerful Unbiased tests] \label{bjumpbumpUMPtestg} ~\\
Let $L \geq 1$, $\alpha$ in $(0,1)$, $\delta^*$ in $(-R,R)\setminus\{0\}$,  $\tau^{*}{}$ in $(0,1)$ and $\ell^{*}{}$ in $(0, 1-\tau^{*}{}]$.
The tests $\phi_{1,\alpha}^{u,-}$ and $\phi_{1,\alpha}^{u,+}$ of $\hzero$ versus $\hone\  "\lambda\in \calS_{\delta^*,\tau^*,\ell^*}^u[R]"$,
defined by \eqref{test_alt1_u}, satisfy
\begin{equation}\label{Neymanstructure}
E_\lambda\cro{\phi_{1,\alpha}^{u,+}(N)\Big|N_1=n}=E_\lambda\cro{\phi_{1,\alpha}^{u,-}(N)\Big|N_1=n}=\alpha\quad \forall \lambda\in \calS_0^u[R]\enspace.
\end{equation}
Moreover, $\phi_{1,\alpha}^{u,-}$ and $\phi_{1,\alpha}^{u,+}$ are UMPU tests.
\end{proposition}

In order to follow the same line as the minimax results obtained when regarding other alternative hypotheses with unknown height and/or length change, we further study which minimal distance $d_2(\lambda,\calS_0^u[R])$ guarantees  a second kind error rate control. 

\begin{proposition}[Second kind error rates control for $\bold{[Alt^u.1]}$]\label{bNP1U}
Let $L\geq 1$, $\alpha$ in $(0,1)$, $\delta^*$ in $(-R,R)\!\setminus\!\{0\}$,  $\tau^{*}{}$ in $(0,1)$ and $\ell^{*}{}$ in $(0, 1-\tau^{*}{}]$, and let $\phi_{1,\alpha}^{u,-}$ and $\phi_{1,\alpha}^{u,+}$ be the tests of $\hzero$ versus $\hone\  "\lambda\in \calS_{\delta^*,\tau^*,\ell^*}^u[R]"$ defined by \eqref{test_alt1_u}.  There exists $C(\alpha,\beta,R,\ell^*)>0$ such that 

$(i)$ $P_{\lambda}(\phi_{1,\alpha}^{u,+}(N)=0) \leq \beta$ if $0<\delta^*<R$ and $\lambda$ belongs to $\calS_{\delta^*, \tau^*, \ell^*}^u[R]$ with
\begin{equation}\label{cond_alt1U} 
d_2(\lambda, \calS_0^u[R]) \geq C(\alpha,\beta,R,\ell^*)/{\sqrt{L}}\enspace,
\end{equation}
$(ii)$ $P_{\lambda}(\phi_{1,\alpha}^{u,-}(N)=0) \leq \beta$ if $-R<\delta^*<0$ and $\lambda$ belongs to $\calS_{\delta^*, \tau^*, \ell^*}^u[R]$ with \eqref{cond_alt1U}.
\end{proposition}
\emph{Comments.} Noticing that for any $\lambda$ in $\calS_{\delta^*, \tau^*, \ell^*}^u[R]$, $d_2\pa{\lambda,\calS_{0}^u[R]}=|\delta^*|\sqrt{\ell^*(1-\ell^*)}$, the above proposition in particular gives that if $L \geq C^{2}(\alpha,\beta,R,\ell^*)/{\delta^{*}{}^2 \ell^{*}{}(1- \ell^{*}{})}\enspace,$  then $P_{\lambda}(\phi_{1,\alpha}^{u,+}(N)=0) \leq \beta$ and $P_{\lambda}(\phi_{1,\alpha}^{u,-}(N)=0) \leq \beta$,  respectively when $0<\delta^*<R$ and $-R<\delta^*<0$.
Therefore, in this case, the $\beta$-uniform separation rates of $\phi_{1,\alpha}^{u,+}$  and $\phi_{1,\alpha}^{u,-}$ over $\calS_{\delta^*, \tau^*, \ell^*}^u[R]$, with $0<\delta^*<R$ and $-R<\delta^*<0$ respectively, are equal to $0$, and so are the corresponding $(\alpha,\beta)$-minimax separation rates $\mSRab(\calS_{\delta^*, \tau^*, \ell^*}^u[R])$.

\smallskip

In order to address the question of adaptation to the change height, we consider the problem of testing $\hzero$ v.s. $\hone\  "\lambda\in \calS^u_{\bbul,\tau^*,\ell^*}[R]"$, where for $R>0$, $\tau^*$ in $ (0,1)$ and $\ell^*$ in $(0,1-\tau^*]$,
\begin{multline}\label{def_alt2_u}
\bold{[Alt^u.2]}\quad \calS^u_{\bbul,\tau^*,\ell^*}[R]= \Big\{ \lambda:[0,1]\to (0,R],~\exists \lambda_0\in (0,R],\\
 ~\exists \delta \in (-\lambda_0,R-\lambda_0]\setminus \set{0},
~\forall t \in [0,1]\quad \lambda(t) = \lambda_{0}{} + \delta\mathds{1}_{(\tau^{*}{}, \tau^{*}{} + \ell^{*}{}]}(t)\Big\}\enspace.\end{multline}

Unsurprisingly, with the Bayesian arguments already used to prove Proposition \ref{LBalt2}, one obtains a lower bound for the minimax separation rate over $\calS^u_{\bbul,\tau^*,\ell^*}[R]$ of the parametric order $1/\sqrt{L}$.

 \begin{proposition}[Minimax lower bound for $\bold{[Alt^u.2]}$] \label{LB_alt2_u}
Let $\alpha$ and  $\beta$ in $(0,1)$, $R>0$, $\tau^{*}$ in $(0,1)$ and $\ell^* $ in $(0,1- \tau^{*}]$. For all $L \geq (2 \log C_{\alpha,\beta}/(R \ell^{*}{}))$,
\[\mSRab\pa{\calS^u_{\bbul, \tau^*, \ell^*}[R]} \geq  \sqrt{ {R (1-\ell^{*}{}) \log C_{\alpha,\beta}}/\pa{2L}},\textrm{ with }C_{\alpha, \beta}=1+4(1- \alpha -\beta)^2\enspace.\]
\end{proposition}

In order to prove that this lower bound is sharp, we construct two minimax adaptive tests. 

\smallskip

The first one is based on the linear statistic $N(\tau^{*}, \tau^{*} + \ell^*]$ and is very similar in spirit to the test $\phi_{2,\alpha}^{(1)}$ defined by \eqref{test1}, except that the associated critical values are based on the conditional distribution of $N(\tau^{*}, \tau^{*} + \ell^*]$ given $N_1=n$ under $(H_0)$ instead of the unconditional distribution (which is not free from the unknown constant baseline intensity under $(H_0)$). Let
\begin{multline} \label{testalt2_N1_u}
\phi_{2,\alpha}^{u(1)}(N)=\mathds{1}_{N(\tau^{*}, \tau^{*} + \ell^*] > b_{N_1,\ell^*}(1-\alpha_1)}+\gamma^+_{(N_1,\ell^*)}(1-\alpha_1)\mathds{1}_{N(\tau^{*}, \tau^{*} + \ell^*] = b_{N_1,\ell^*}(1-\alpha_1)}\\
+\mathds{1}_{N(\tau^{*}, \tau^{*} + \ell^{*}] < b_{N_1,\ell^*}(\alpha_2)}+
\gamma^-_{(N_1,\ell^*)}(\alpha_2)\mathds{1}_{N(\tau^{*}, \tau^{*} + \ell^{*}] = b_{N_1,\ell^*}(\alpha_2)}\enspace,
\end{multline}
where $\gamma^+_{(n,\ell^*)}$ and $\gamma^-_{(n,\ell^*)}$ are defined by \eqref{bUPP2g}, $b_{n,\ell^*}(u)$ denotes the $u$-quantile of the binomial distribution with parameters $(n,\ell^*)$ for all $n$ in $\N$, and $\alpha_1$ and $\alpha_2 $ in $ (0,1)$ are determined by
\begin{equation} \label{alphaUMPU_u}
\begin{cases}
\alpha_1 + \alpha_2 = \alpha\enspace, \\
E_\lambda\big[N(\tau^{*}, \tau^{*}+ \ell^{*}] \phi_{2,\alpha}^{u(1)}(N)\big|N_1=n\big]=\alpha E_\lambda\big[N(\tau^{*}, \tau^{*}+ \ell^{*}]\big|N_1=n\big]~ \forall \lambda\in \calS_0^u[R]\enspace.
\end{cases}
\end{equation}
Theorem 4.4.1 of \cite{Lehmann2006} again (but considering the bilateral test) shows that $ \phi_{2,\alpha}^{u(1)}$ is UMPU.

\smallskip

The second one is based on a quadratic statistic deduced from an estimation of the $\bbL_2$-distance between $\lambda$  in $\calS^u_{\bbul, \tau^*, \ell^*}[R]$ and $\calS_0^u[R]$. For $0  \!<  \!\tau_1  \!<  \!\tau_2 \!\leq \! 1$, we define $\psi_{0}= \mathds{1}_{[0,1]}$ and
$\psi_{\tau_1,\tau_2}= - \sqrt{\pa{\tau_2- \tau_1}\!/\!\pa{1-\tau_2+ \tau_1}} \pa{ \mathds{1}_{(0,\tau_1]}\!+\!\mathds{1}_{(\tau_2,1]} }  + \sqrt{\pa{1-\tau_2 + \tau_1}\!/\!\pa{\tau_2 - \tau_1}} \mathds{1}_{( \tau_1, \tau_2 ]}$.
Notice that $(\psi_{0}, \psi_{\tau_1, \tau_2})$ is an orthonormal family and set $W_{0}= \mathrm{Vect} \left(\psi_{0}\right)$ and $W_{\tau_1, \tau_2}= \mathrm{Vect} \left(\psi_{0}, \psi_{\tau_1, \tau_2}  \right)$. 
Denoting by $\Pi_{W_{0}}$ and $\Pi_{W_{\tau_1,\tau_2}}$ the orthogonal projections onto  $W_{0}$ and $W_{\tau_1,\tau_2}$ in $\bbL_2([0,1])$ respectively, the quadratic statistic
\begin{multline}\label{def_T'}
T'_{\tau_1, \tau_2}(N)= \frac{1}{L^2} \Big[ \frac{\tau_2 - \tau_1}{1- \tau_2 + \tau_1}  \left( \left( N(0,\tau_1] + N(\tau_2,1] \right)^2 - \left( N(0,\tau_1]  +N(\tau_2,1] \right)  \right) \\
+ \frac{1- \tau_2 + \tau_1}{\tau_2 - \tau_1}  \left( N(\tau_1,\tau_2]^2  - N(\tau_1,\tau_2]  \right) -2 N(\tau_1,\tau_2] \left( N(0, \tau_1]+ N(\tau_2,1] \right) \Big]\enspace,
\end{multline}
is an unbiased estimator of $\Vert \Pi_{W_{\tau_1, \tau_2}}(\lambda - \Pi_{W_{0}}(\lambda)) \Vert_2 ^2$. We therefore consider the particular statistic $T'_{\tau^*, \tau^*+\ell^*}(N)$ which is an unbiased estimator of the squared $\bbL_2$-distance between $\lambda$   in $\calS^u_{\bbul, \tau^*, \ell^*}[R]$ and the set of constant intensities, leading to the test defined by
\begin{equation} \label{testalt2_N2_u}
\phi_{2,\alpha}^{u(2)}(N) = \mathds{1}_{T'_{\tau^*, \tau^*+\ell^*}(N) > t'_{N_1,\tau^*, \tau^* + \ell^*}(1-\alpha)}\enspace,
\end{equation}
where $t'_{n,\tau_1,\tau_2}(u)$ is the $u$-quantile of the distribution of $T'_{\tau_1,\tau_2}(N)$ given $N_1=n$ under $\hzero$.

Since this conditional distribution under $(H_0)$ is the distribution of the renormalised $U$-statistic $L^{-2} \sum_{i \neq j=1}^{n} \psi_{\tau_1, \tau_2}(U_i) \psi_{\tau_1, \tau_2}(U_{j})$ based on a $n$-sample $(U_1,\ldots,U_n)$ of  i.i.d. uniform random variables, we use an exponential inequality for $U$-statistics of order $2$ due to Reynaud-Bouret and Houdré \cite{HoudreRB} to control the quantiles $t'_{n,\tau_1,\tau_2}(u)$ in theory (see Lemma~\ref{quantile_T'}), and Monte-Carlo methods to evaluate them in practice.

\begin{proposition}[Minimax upper bound for $\bold{[Alt^u.2]}$] \label{UBalt2_u} Let $L\geq 1$, $\alpha,\beta$ in $(0,1)$, $R>0$, $\tau^*$ in $(0,1)$ and $\ell^*$ in $(0,1- \tau^{*}]$. Let $\phi_{2,\alpha}^{u(1/2)}$ be one of the tests $\phi_{2,\alpha}^{u(1)}$ and $\phi_{2,\alpha}^{u(2)}$  of $\hzero$ versus $\hone\ "\lambda\in\calS^u_{\bbul,\tau^*,\ell^*}[R]"$  respectively defined by \eqref{testalt2_N1_u}-\eqref{alphaUMPU_u} and \eqref{testalt2_N2_u}. Then $\phi_{2,\alpha}^{u(1/2)}$ is of level $\alpha$, that is $ \sup_{\lambda\in\calS^u_0[R]}P_{\lambda}(\phi_{2,\alpha}^{u(1/2)}(N)=1) \leq \alpha$ ($\phi_{2,\alpha}^{u(1)}$ is even of size $\alpha$). Moreover, there exists $C(\alpha, \beta,R, \ell^*)>0$ such that
$$ \SRb(\phi_{2,\alpha}^{u(1/2)},\calS^u_{\bbul,\tau^*, \ell^*}[R]) \leq {C(\alpha, \beta,R,\ell^*)}/{\sqrt{L}}\enspace,$$
which entails in particular $\mSRab(\calS^u_{\bbul, \tau^*, \ell^*}[R]) \leq C(\alpha, \beta, R,\ell^*)/\sqrt{L}$.
\end{proposition}
\emph{Comments.} This result proves that both tests $\phi_{2,\alpha}^{u(1)}$ and $\phi_{2,\alpha}^{u(2)}$ are therefore minimax (up to a possible multiplicative constant) over the set of alternatives $\calS^u_{\bbul,\tau^*, \ell^*}[R]$, where the height of the change is unknown, with an optimal uniform separation rate of the parametric order $1/\sqrt{L}$, as expected regarding results for $\phi_{2,\alpha}^{(1)}$ and $\phi_{2,\alpha}^{(2)}$ in Section \ref{Sec:knownbaseline}.

\smallskip

Notice that this study involves the particular non transitory change or jump detection problem, with a known change location, taking $\ell^*=1-\tau^*$. Following the same layout as in Section \ref{Sec:knownbaseline}, we investigate the jump detection problem with unknown location in Section \ref{Sec:generalsingle_u}.

\subsection{Minimax detection of a transitory change with known length}\label{Sec:knownlength_u}

In this subsection, we deal with the  problem of testing the null hypothesis $\hzero \ "\lambda\in \calS_0^u[R]"$  versus alternatives where the length of the change from the unknown baseline intensity is known, with adaptation to the change location, and with or without adaptation to the change height.
We therefore introduce for $\ell^*$ in $(0,1)$ and $\delta^*$ in $(-R,R)\setminus\{0\}$ the two following sets:
  \begin{multline}\label{def_alt3_u}
\bold{[Alt^u.3]}\ \calS^u_{\delta^*,\bbul\bbul,\ell^*}[R]= \big\{\lambda:[0,1]\to (0,R],~ \exists \lambda_0 \in (-\delta^{*} \vee 0, (R-\delta^{*}) \wedge R],\\
\exists \tau \in (0,1-\ell^*),~\forall t\in[0,1]\quad \lambda(t) = \lambda_{0} + \delta^* \mathds{1}_{(\tau,\tau+\ell^*]}(t) \big\}\enspace,
\end{multline}
\vspace{-0.8cm}
\begin{multline}\label{def_alt4_u}
\bold{[Alt^u.4]}\ \calS^u_{\bbul,\bbul\bbul,\ell^*}[R]=
 \big\{ \lambda:[0,1]\to (0,R],~\exists \lambda_0\in (0,R],~\exists \delta \in (-\lambda_0,R-\lambda_0]\setminus \set{0},\\
 ~\exists \tau \in (0,1-\ell^*),~\forall t\in[0,1]\quad  \lambda(t) = \lambda_{0} + \delta \mathds{1}_{(\tau,\tau+\ell^*]}(t) \big\}\enspace.
\end{multline}
Adapting the ideas of Section \ref{Sec:knownlength}, we handle the question of adaptation to the change location $\tau^*$ by introducing aggregated tests based on the same linear and quadratic statistics as those used for testing  $\hzero$ versus $\hone\  "\lambda\in \calS^u_{\delta^*,\tau^*,\ell^*}[R]"$ above. We thus set on the one hand
\begin{eqnarray} 
\phi_{3,\alpha}^{u(1)-}(N)&=&\mathds{1}_{\min_{\tau\in [0,1-\ell^*\wedge(1/2)]}N(\tau, \tau + \ell^*\wedge(1/2)] < b_{N_1,\ell^*\wedge(1/2)}^-(\alpha)}\enspace,\label{testN1alt3_u_1}\\
\phi_{3,\alpha}^{u(1)+}(N)&=&\mathds{1}_{\max_{\tau\in [0,1-\ell^*\wedge(1/2)]}N(\tau, \tau + \ell^*\wedge(1/2)] > b_{N_1,\ell^*\wedge(1/2)}^+(1-\alpha)}\enspace,
\label{testN1alt3_u_2}\\
\phi_{4,\alpha}^{u(1)}(N)&=&\phi_{3,\alpha/2}^{u(1)-}(N)\vee \phi_{3,\alpha/2}^{u(1)+}(N)\enspace,\label{testN1alt4_u}
\end{eqnarray}
where $b_{n,\ell}^-(u)$ and $b_{n,\ell}^+(u)$ respectively denote the $u$-quantiles of the conditional distributions of $\min_{\tau\in [0,1-\ell]}N(\tau, \tau + \ell]$ and $\max_{\tau\in [0,1-\ell]}N(\tau, \tau + \ell]$ given $N_1=n$ under $\hzero$, for all $n$ in $\N$ and $\ell$ in $(0,1/2]$. Then, we introduce on the other hand the aggregated test
\begin{equation} \label{testN2alt3-4_u}
\phi_{3/4,\alpha}^{u(2)}(N) = \1{\max_{k\in \lbrace 0,\ldots, \lceil (1-\ell^*) M \rceil -1 \rbrace} \left( T'_{\frac{k}{M}, \frac{k}{M}+\ell^*}(N) - t'_{N_1,\frac{k}{M}, \frac{k}{M}+\ell^*} \pa{1 - u_\alpha} \right)>0}\enspace,
\end{equation}
where $M= \lceil 2/(\ell^*(1-\ell^*)) \rceil$, $u_\alpha= {\alpha}/{\lceil (1-\ell^*) M \rceil}$, $T'_{{k}/{M},{k}/{M}+\ell^*}$ is defined by \eqref{def_T'} and $t'_{n,{k}/{M},{k}/{M}+\ell^*} \pa{u}$ is the $u$-quantile of  $T'_{{k}/{M},{k}/{M}+\ell^*}(N)$ given $N_1=n$ under $\hzero$.

Since the set $\calS^u_{\delta^*,\bbul\bbul,\ell^*}[R]$ of \eqref{def_alt3_u} is composed of alternatives with known change height $\delta^*$ and length $\ell^*$, the distance between any of its elements and $\calS^u_{0}[R]$  is fixed, equal to $|\delta^*|\sqrt{\ell^*(1-\ell^*)}$. Hence, for this set, we only provide sufficient conditions for the tests $\phi_{3,\alpha}^{u(1)+}$, $\phi_{3,\alpha}^{u(1)-}$ and $\phi_{3/4,\alpha}^{u(2)}$ to have a second kind error rate controlled by a prescribed level $\beta$ when $\lambda\in \calS^u_{\delta^*,\bbul\bbul,\ell^*}[R]$. As in Proposition \ref{UBalt3}, the key points of the proofs of the following results for $\phi_{3,\alpha}^{u(1)-}$ and $\phi_{3,\alpha}^{u(1)+}$ are sharp lower or upper bounds for the involved quantiles $b_{n,\ell^*\wedge(1/2)}^-(\alpha)$ and $b_{n,\ell^*\wedge(1/2)}^+(\alpha)$, which are deduced from inequalities for oscillations of empirical processes found in \cite{ShorackWellner} (see Lemma \ref{bquantile_maxminNbis_u} for details). The result for $\phi_{3/4,\alpha}^{u(2)}$ relies on the control of the quantile $t'_{n,\tau_1,\tau_2}(u_\alpha)$  obtained in Lemma~\ref{quantile_T'} via the exponential inequality for $U$-statistics of order $2$ due to Reynaud-Bouret and Houdré \cite{HoudreRB}, as in Proposition \ref{UBalt2_u}.

\begin{proposition}[Second kind error rate control for $\bold{[Alt^u.3]}$] \label{UBalt3_u}
Let $L\geq 1$, $\alpha$ and $\beta$ in $(0,1)$, $\ell^*$ in $(0,1)$ and $\delta^*$ in $(-R,R)\setminus\{0\}$,
 and consider the problem of testing $\hzero$ v.s. $\hone \ "\lambda\in \calS^u_{\delta^*, \bbul\bbul,\ell^*}[R]"$. Let $\phi_{3,\alpha}^{u(1/2)}$ be one of the tests $\phi_{3,\alpha}^{u(1)+}$ or $\phi_{3/4,\alpha}^{u(2)}$ if $\delta^*>0$, and one of the tests $\phi_{3,\alpha}^{u(1)-}$ or $\phi_{3/4,\alpha}^{u(2)}$ if $\delta^*<0$ (see \eqref{testN1alt3_u_1},  \eqref{testN1alt3_u_2} and \eqref{testN2alt3-4_u}). The test $\phi_{3,\alpha}^{u(1/2)}$ is of level $\alpha$, that is $\sup_{\lambda_0\in \calS^u_0[R]}P_{\lambda_0}\big(\phi_{3,\alpha}^{u(1/2)}(N)=1\big)\leq \alpha$. Moreover, there exists $C(\alpha,\beta,R,\delta^*,\ell^*)>0$ such that $P_{\lambda}\big(\phi_{3,\alpha}^{u(1/2)}(N)=0\big) \leq \beta$ as soon as $\lambda$ belongs to $\calS^u_{\delta^*, \bbul\bbul, \ell^*}[R]$ with
$$d_2\pa{\lambda,\calS^u_0[R]}  \geq {C(\alpha,\beta,R,\delta^*,\ell^*)}/{\sqrt{L}}\enspace.$$
\end{proposition}

\emph{Comments.} Remarking that for $\lambda$ in $\calS^u_{\delta^*, \bbul\bbul, \ell^*}[\lambda_0]$, $d_2\pa{\lambda,\calS^u_{0}[R]}=|\delta^*|\sqrt{\ell^*(1-\ell^*)}$, Proposition \ref{UBalt3_u} provides a sufficient value $L_0(\alpha,\beta,R,\delta^*,\ell^{*})$ for $L$ so that the second kind error rates of the three tests is controlled by $\beta$. If $L\geq L_0(\alpha,\beta,R,\delta^*,\ell^{*})$,  their $\beta$-uniform separation rates over $\calS^u_{\delta^*,\bbul \bbul, \ell^*}[R]$ is equal to $0$, as well as the $(\alpha,\beta)$-minimax separation rate.

\smallskip

Now considering the alternative set $\calS^u_{\bbul,\bbul\bbul, \ell^*}[R] $, that is the change height adaptation issue, the following lower bound is directly deduced from the lower bound for $\mSRab\big(\calS^u_{\bbul,\tau^*, \ell^*}[R]\big)$ and the monotonicity property of the minimax separation rate recalled in Lemma \ref{mSR}.

\begin{corollary}[Minimax lower bound for $\bold{[Alt^u.4]}$]\label{LBalt4_u} 
Let $\alpha$ and  $\beta$ in $(0,1)$, $R>0$ and $\ell^* $ in $(0,1)$. For all $L \geq (2 \log C_{\alpha,\beta}/(R \ell^{*}{}))$,
\[\mSRab\pa{\calS^u_{\bbul,\bbul \bbul, \ell^*}[R]} \geq  \sqrt{{R (1-\ell^{*}{}) \log C_{\alpha,\beta}}/\pa{2L}},\textrm{ with }C_{\alpha, \beta}=1+4(1- \alpha -\beta)^2\enspace.\]
\end{corollary}

\begin{proposition}[Minimax upper bounds for $\bold{[Alt^u.4]}$] \label{UBalt4_u} Let $L\geq 1$, $\alpha,\beta$ in $(0,1)$, $R>0$ and $\ell^*$ in $(0,1)$. Let $\phi_{4,\alpha}^{u(1/2)}$ be one of the tests $\phi_{4,\alpha}^{u(1)}$ and $\phi_{3/4,\alpha}^{u(2)}$ of $\hzero$ versus $\hone\ "\lambda\in\calS^u_{\bbul,\bbul\bbul,\ell^*}[R]"$, defined by \eqref{testN1alt4_u} and \eqref{testN2alt3-4_u}. $\phi_{4,\alpha}^{u(1/2)}$ is of level $\alpha$, that is $\sup_{\lambda_0\in \calS^u_0[R]} P_{\lambda_0}(\phi_{4,\alpha}^{u(1/2)}(N)=1)\leq \alpha$, and there exists $C(\alpha, \beta, R,\ell^*)>0$ such that
$$ \SRb\big(\phi_{4,\alpha}^{u(1/2)},\calS^u_{\bbul,\bbul\bbul, \ell^*}[R]\big) \leq {C(\alpha, \beta,R,\ell^*)}/{\sqrt{L}}\enspace,$$
which entails in particular $\mSRab\big(\calS^u_{\bbul, \bbul\bbul, \ell^*}[R]\big) \leq C(\alpha, \beta,R, \ell^*)/\sqrt{L}$.
\end{proposition}

\emph{Comments.} Proposition  \ref{UBalt4_u} and Corollary \ref{LBalt4_u} mean that the tests $\phi_{4,\alpha}^{u(1)}$ and $\phi_{3/4,\alpha}^{u(2)}$ are minimax. Together with the ones obtained for $\bold{[Alt^u.2]}$, the two above results finally mean that, as when the baseline intensity is known, adaptation with respect to the change location can be achieved with a minimax separation rate of the parametric order, that is without any additional price to pay (possibly except multiplicative constants) as soon as the only change length is known.

\subsection{Minimax detection of a transitory change with known location} \label{Sec:knownlocation_u}

We consider the problem of testing the null hypothesis $\hzero \ "\lambda\in \calS^u_0[R]"$ versus alternative hypotheses where the location of the change from the baseline intensity is known, with adaptation to the change length, and with or without adaptation to the height. As in Section \ref{Sec:knownlocation}, we see that adaptation to the length can be done without any incidence on the minimax separation rate order, while adaptation to both height and length leads to a cost factor of order $\sqrt{\log\log L}$. 

\subsubsection{Known change height}

Let us first investigate the problem of testing $\hzero$ versus
$\hone\ "\lambda\in\calS^u_{\delta^*,\tau^*,\bbul\bbul\bbul}[R]"$,
where  for $R>0$, $\delta^*$ in $(-R,R)\setminus\{0\}$ and $\tau^*$ in $(0,1)$,
\begin{multline}\label{def_alt5_u}
\bold{[Alt^u.5]}\ \calS^u_{\delta^*,\tau^*,\bbul\bbul\bbul}[R]= \lbrace \lambda:[0,1]\to (0,R],~ \exists \lambda_0 \in (-\delta^{*} \vee 0, (R-\delta^{*}) \wedge R],\\
\exists \ell \in (0,1-\tau^*),~\forall t\in[0,1]\quad \lambda(t) = \lambda_{0} + \delta^* \mathds{1}_{(\tau^*,\tau^*+\ell]}(t) \rbrace\enspace.
\end{multline}

As in Section \ref{Sec:knownlocation}, the most intricate point here is the construction of a test achieving the minimax separation rate over $\calS^u_{\delta^*,\tau^*,\bbul\bbul\bbul}[R]$, which will be proved to be  of the parametric order $1/\sqrt{L}$, and therefore necessarily taking the knowledge of the change height $\delta^*$ into account. The test we propose is largely inspired from the aggregated test $\phi_{5,\alpha}$ defined by \eqref{testalt5}, where the test statistic is slightly adapted to compensate for the lack of the baseline intensity knowledge. Since the critical value can not be taken as a quantile of the test statistic, whose distribution under the null hypothesis is not free from the unknown baseline intensity anymore, we use the same conditioning trick as in the above subsections.

\begin{proposition}[Minimax lower bound for $\bold{[Alt^u.5]}$] \label{LBalt5_u}
Let $\alpha,\beta$ in $(0,1)$ with $\alpha + \beta <1$,  $R>0$, $\delta^*$ in $(-R,R)\setminus\{0\}$, $\tau^*$ in $(0,1)$. For $L > {((R- \delta^{*}{}) \wedge R) \log C_{\alpha,\beta}}/\pa{2 \delta^{*}{}^2 \tau^{*}{}(1-\tau^{*}{})} $, 
\[\mSRab\pa{\calS^u_{\delta^*,\tau^*,\bbul\bbul\bbul}[R]} \geq  \sqrt{{((R- \delta^{*}{}) \wedge R) \log C_{\alpha,\beta}}/\pa{2 L}  },\textrm{ with }C_{\alpha, \beta}=1+4(1- \alpha -\beta)^2\enspace.\]
\end{proposition}
Let us now introduce the test
\begin{equation} \label{testalt5_u}
\phi^u_{5,\alpha}(N) = \1{\sup_{\ell \in (0,1- \tau^{*})}S'_{\delta^*,\tau^*,\tau^*+\ell}(N) >s_{N_1,\delta^*,\tau^*,L}^{'+}(1-\alpha)}\enspace,
\end{equation}
where $S'_{\delta^*,\tau_1,\tau_2}(N)$ is the statistic defined for $0\leq \tau_1<\tau_2\leq 1$ by
\begin{equation}\label{stat_alt5_u}
S_{\delta^*,\tau_1,\tau_2}'(N)=  \mathrm{sgn}(\delta^{*}{}) \Big(N(\tau_1,\tau_2] - (\tau_2-\tau_1) N_1\Big) - \vert \delta^{*}{} \vert L(\tau_2-\tau_1)(1-\tau_2+\tau_1)/2 \enspace,
\end{equation}
and $s_{n,\delta^*,\tau^*,L}^{'+}(u)$ is the $u$-quantile of $\sup_{\ell \in (0,1- \tau^{*})}S'_{\delta^*,\tau^*,\tau^*+\ell}(N)$ given $N_1=n$ under $(H_0)$.

The main argument of the following upper bound is a control of the conditional quantile $s_{n,\delta^*,\tau^*,L}^{'+}(1-\alpha)$, provided in Lemma \ref{QuantilessupShifted_u}, and which is deduced from a refined Bernstein inequality based on some chaining techniques.

\begin{proposition}[Minimax upper bound for $\bold{[Alt^u.5]}$] \label{UBalt5_u}
Let $L\!\geq \!1$, $\alpha, \beta$  in $(0,1)$, $R\!>\!0$, $\delta^*$ in $(-R,R)\setminus\{0\}$ and $\tau^*$ in $(0,1)$. Let $\phi^u_{5,\alpha}$ be the test of $\hzero$ versus $\hone\ "\lambda\in\calS^u_{\delta^*,\tau^*,\bbul\bbul\bbul}[R]"$ defined by \eqref{testalt5_u}. $\phi^u_{5,\alpha}$ is of level $\alpha$, that is $\sup_{\lambda_0\in \calS_0^u[R]} P_{\lambda_0}\pa{\phi^u_{5,\alpha}(N)=1}\leq \alpha$.
Moreover, there exists a constant $C(\alpha, \beta,R,\delta^*)>0$ such that
$$ \SRb\pa{\phi^u_{5,\alpha},\calS^u_{\delta^*,\tau^*,\bbul\bbul\bbul}[R]} \leq {C(\alpha, \beta,R,\delta^*)}/{\sqrt{L}}\enspace,$$
which entails in particular $\mSRab\pa{\calS^u_{\delta^*,\tau^*,\bbul\bbul\bbul}[R]} \leq C(\alpha, \beta,R, \delta^*)/\sqrt{L}$.
\end{proposition}

\subsubsection{Unknown change height} 

Now addressing the question of adaptation to the change height and length together, we consider for $R>0$ and $\tau^*$ in $(0,1)$ the alternative set
\begin{multline}\label{alt6_u}
\bold{[Alt^u.6]}\  \calS^u_{\bbul,\tau^*,\bbul\bbul\bbul}[R]=\big\{ \lambda: [0,1]\to (0,R],~\exists \lambda_0\in (0,R],
 ~\exists \delta \in (-\lambda_0,R-\lambda_0]\setminus \set{0},\\~ \exists \ell \in (0,1-\tau^*),
~\forall t\in[0,1]\quad  \lambda(t) = \lambda_{0} + \delta \mathds{1}_{(\tau^*,\tau^*+\ell]}(t) \big\}\enspace.
\end{multline}
For the problem of testing $\hzero \ "\lambda\in \calS^u_0[R]"$ versus $\hone\ "\lambda\in\calS^u_{\bbul,\tau^*,\bbul\bbul\bbul}[R]"$, 
we obtain the following lower bound.

\begin{proposition}[Minimax lower bound for $\bold{[Alt^u.6]}$] \label{LBalt6_u}
Let $\alpha, \beta$ in $(0,1)$  with $\alpha + \beta <1/2$, $R>0$ and $\tau^*$ in $(0,1)$. There exists $L_0(\alpha,\beta,R,\tau^{*}{})>0$  such that for $L\geq L_0(\alpha,\beta,R,\tau^*),$
 \[\mSRab\pa{\calS^u_{\bbul,\tau^*,\bbul\bbul\bbul}[R]} \geq  \sqrt{{R \tau^{*}{}\log\log L}/\pa{2L}}\enspace.\]
\end{proposition}
Let us assume now that $L\geq 3$. In order to prove that the above lower bound is of sharp order (with respect to $L$), we construct two aggregated tests: a first one  based on 
a linear statistic and a second one based on a quadratic statistic as in Section \ref{Sec:knownlength_u}.

We thus consider the discrete subset of $(0,1-\tau^*)$ of the dyadic form
\[\left\{ \ell_{\tau^*,k}=\pa{1 - \tau^{*}{}}{2^{-k}}, ~ k \in \lbrace 1,\ldots,\lfloor \log_{2} L \rfloor \rbrace \right\}\enspace,\]
and $u_\alpha=\alpha/\lfloor \log_{2}(L) \rfloor$, which allows to define the two following tests:
\begin{equation} \label{testN1alt6_u}
\phi_{6,\alpha}^{u(1)}(N) = \1{\max_{k\in \lbrace 1,\ldots, \lfloor \log_{2} L \rfloor \rbrace} \left( \left|S'_{\tau^*,\tau^*+\ell_{\tau^*,k}}(N)\right| - s'_{N_1,\tau^*,\tau^*+\ell_{\tau^*,k}} \pa{1 - u_\alpha} \right)>0}\enspace,
\end{equation}
where $S'_{\tau_1,\tau_2}(N)$ is the linear statistic defined for $0\leq \tau_1<\tau_2\leq 1$ by
\begin{equation}\label{stat_alt6_N1_u}
S'_{\tau_1,\tau_2}(N)= N(\tau_1, \tau_2] - (\tau_2-\tau_1) N_1 \enspace,
\end{equation}
and $s'_{n,\tau_1,\tau_2}(u)$ stands for the $u$-quantile of  $\left| S'_{\tau_1,\tau_2}(N)\right|$ given $N_1=n$ under $\hzero$, whose sharp bound is obtained via Bennett's inequality (see Lemma \ref{QuantilesAbsS_u} for details), and

\begin{equation} \label{testN2alt6_u}
\phi_{6,\alpha}^{u(2)}(N) = \1{\max_{k\in \lbrace 1,\ldots, \lfloor \log_{2} L \rfloor\rbrace} \left(  T'_{\tau^*, \tau^*+\ell_{\tau^*,k} }(N)  - t'_{N_1,\tau^*,\tau^*+\ell_{\tau^*,k}} \pa{1 - u_\alpha} \right)>0}\enspace,
\end{equation}
where $T'_{\tau_1,\tau_2}(N)$ is the quadratic statistic defined in \eqref{def_T'} and $t'_{n,\tau_1,\tau_2}(u)$ still denotes the $u$-quantile of its conditional distribution  given $N_1=n$ under $\hzero$.

\begin{proposition}[Minimax upper bound for $\bold{[Alt^u.6]}$] \label{UBalt6_u}
Let $\alpha, \beta$ in $(0,1)$,  $R>0$, $\tau^*$ in $(0,1)$, and let $\phi_{6,\alpha}^{u(1/2)}$ be one of the tests $\phi_{6,\alpha}^{u(1)}$ and  $\phi_{6,\alpha}^{u(2)}$ of $\hzero$ versus $\hone\ "\lambda\in\calS^u_{\bbul,\tau^*,\bbul\bbul\bbul}[R]"$ respectively defined by \eqref{testN1alt6_u} and \eqref{testN2alt6_u}. Then $\phi_{6,\alpha}^{u(1/2)}$ is of level $\alpha$, that is $\sup_{\lambda_0\in \calS_0^u[R]} P_{\lambda_0}\pa{\phi_{6,\alpha}^{u(1/2)}(N)=1}\leq \alpha$.
Moreover, there exists $C(\alpha, \beta,R,\tau^*)>0$ such that
$$ \SRb\pa{\phi_{6,\alpha}^{u(1/2)},\calS^u_{\bbul,\tau^*,\bbul\bbul\bbul}[R]} \leq C(\alpha, \beta,R,\tau^*)\sqrt{{\log \log L}/{L}}\enspace,$$
which entails in particular $\mSRab\pa{\calS^u_{\bbul,\tau^*,\bbul\bbul\bbul}[R]} \leq C(\alpha, \beta,R,\tau^*)\sqrt{{\log \log L}/{L}}$.
\end{proposition}

\subsection{Minimax detection of a possibly transitory change with unknown location and length} \label{Sec:unknownlocationlength_u} Let us discuss as final stage the problem of testing the null hypothesis $\hzero \ "\lambda\in \calS_0^u[R]"$

 versus alternatives where both location and length of the change from the unknown baseline intensity are not known, distinguishing as in Section \ref{Sec:unknownlocationlength} the transitory change case from the non transitory change particular case.

From the minimax point of view, we will emphasize that regardless if the baseline intensity is known or not, adaptation to both location and length of the change has the same minimax separation rate cost of order $\sqrt{\log L}$ in the transitory change case, and of order $\sqrt{\log \log L}$ at most (possibly cancelled by the change height knowledge) in the non transitory change case. 

Since the non transitory change or jump detection problem, that we here study first, can be viewed as perfectly symmetrical to the transitory change with known location detection problem, our study uses tools and arguments that are very similar to the ones used in Section~\ref{Sec:knownlocation_u}.

\subsubsection{Non transitory change}\label{Sec:generalsingle_u}

In order to investigate the problem of detecting a non transitory change with unknown location, but known height, we introduce for $R>0$ and $\delta^{*}$ in $(-R,R) \setminus \lbrace 0 \rbrace$ the alternative set
\begin{multline}\label{alt7_u}
\bold{[Alt^u.7]}\ \calS^u_{\delta^*,\bbul\bbul,1-\bbul\bbul}[R]=  \big\{ \lambda :[0,1]\to (0,R],~\exists \lambda_{0} \in (- \delta^{*}{} \vee 0, (R- \delta^{*}{}) \wedge R ],\\
\exists \tau\in (0,1),~\forall t\in [0,1]\quad
\lambda(t) = \lambda_{0} + \delta^{*}{} \mathds{1}_{(\tau,1]}(t)\big\} \enspace.
\end{multline}
Considering the problem of testing the null hypothesis 
 $ \hzero \ " \lambda \in \calS^u_0[R]"$ versus the alternative hypothesis $ \hone \ "\lambda \in  \calS^u_{\delta^*,\bbul\bbul,1-\bbul\bbul}[R]",$ we obtain the following lower bound. 

 \begin{proposition}[Minimax lower bound for $\bold{[Alt^u.7]} $] \label{LBalt7_u}
 Let $\alpha,\beta$ in $(0,1)$ with $\alpha + \beta <1,$ $R>0$ and $\delta^{*}{}$ in $(-R,R) \setminus \lbrace 0 \rbrace$.
  For all $  L > 2((R- \delta^{*}{}) \wedge R) \log C_{\alpha,\beta}/\delta^{*}{}^2$,
 \begin{equation*}
 \mathrm{mSR}_{\alpha, \beta}( \calS^u_{\delta^*,\bbul\bbul,1-\bbul\bbul}[R]) \geq \sqrt{{((R- \delta^{*}{}) \wedge R) \log C_{\alpha,\beta}}/\pa{2L}  }, ~ \mathrm{with}~C_{\alpha,\beta} = 1 + 4(1-\alpha-\beta )^2 \enspace.
 \end{equation*}
 \end{proposition}
 Following the study and the notation of Section \ref{Sec:knownlocation_u}, we define the test
\begin{equation} \label{testalt7_u}
\phi^u_{7,\alpha}(N) = \1{\sup_{\tau \in (0,1)}S'_{\delta^*,\tau,1}(N) >s_{N_1,\delta^*,L}^{'+}(1-\alpha)}\enspace,
\end{equation}
where $S'_{\delta^*,\tau_1,\tau_2}(N)$ is the statistic defined for $0\leq \tau_1<\tau_2\leq 1$ by \eqref{stat_alt5_u}
and $s_{n,\delta^*,L}^{'+}(u)$ is the $u$-quantile of the conditional distribution of $\sup_{\tau \in (0,1)} S'_{\delta^*,\tau,1}(N)$ given $N_1=n$ under $\hzero$.

Notice that a control of this conditional quantile $s_{n,\delta^*,L}^{'+}(1-\alpha)$, provided in Lemma \ref{QuantilessupShifted_ubis}, and deduced from the same chaining trick combined with Bernstein's inequality as in the proof of Lemma \ref{QuantilessupShifted_u}, is the main argument of the following result.

\begin{proposition}[Minimax upper bound for $\bold{[Alt^u.7]} $] \label{UBalt7_u}
 Let $L \geq 1$, $\alpha$ and $\beta$ in $(0,1)$, $R>0$ and $\delta^{*}{}$ in $(-R,R) \setminus \lbrace 0 \rbrace$. Let $\phi^u_{7,\alpha}$ be the test of $\hzero$ versus $\hone \ " \lambda \in   \calS^u_{\delta^*,\bbul\bbul,1-\bbul\bbul}[R]"$ defined by \eqref{testalt7_u}. Then $\phi^u_{7,\alpha}$ is of level $\alpha$, that is $\sup_{\lambda_0 \in \calS_0^u[R]} P_{\lambda_0}( \phi^u_{7,\alpha}(N)=1) \leq \alpha$. Moreover, there exists a constant $C(\alpha, \beta,R,\delta^*)>0$ such that
$$ \SRb\pa{\phi^u_{7,\alpha},\calS^u_{\delta^*,\bbul\bbul,1-\bbul\bbul}[R]} \leq {C(\alpha, \beta,R,\delta^*)}/{\sqrt{L}}\enspace,$$
which entails in particular $\mSRab\pa{\calS^u_{\delta^*,\bbul\bbul,1-\bbul\bbul}[R]} \leq C(\alpha, \beta,R, \delta^*)/\sqrt{L}$.
\end{proposition}

To address the question of adaptation to the change height, we introduce the alternative set
\begin{multline}\label{alt8_u}
\bold{[Alt^u.8]}\  \calS^u_{\bbul,\bbul\bbul,1-\bbul\bbul}[R]= \big\{ \lambda :[0,1]\to (0,R],~\exists \lambda_{0} \in (0,R], ~\exists \delta \in  (-\lambda_{0}, R- \lambda_0] \setminus \lbrace 0 \rbrace ,\\
~ \exists \tau \in (0,1),~\forall t\in [0,1]\quad   \lambda(t) = \lambda_{0} + \delta \mathds{1}_{(\tau,1]}(t)\big\}\enspace,
\end{multline}
and we consider the problem of testing $ \hzero \ " \lambda \in \calS^u_0[R]"$ versus $\hone \ "\lambda \in \calS^u_{\bbul,\bbul\bbul,1-\bbul\bbul}[R]".$

As usual, we start with a lower bound for the corresponding minimax separation rate.

\begin{proposition}[Minimax lower bound for $\bold{[Alt^u.8]} $] \label{LBalt8_u}
Let $\alpha, \beta$ in $(0,1)$ with $\alpha + \beta <1/2$ and $R>0$. There exists $L_0(\alpha, \beta,R)>0$ such that for all $L \geq L_0(\alpha, \beta,R),$
\[ \mathrm{mSR}_{\alpha, \beta}(\calS^u_{\bbul,\bbul\bbul,1-\bbul\bbul}[R]) \geq \sqrt{{R\log  \log L}/\pa{4L}}\enspace.\]
\end{proposition}

Let us assume now that $L \geq 3$. In order to prove that the above lower bound is of sharp order (with respect to $L$), we consider the discrete subset of $(0,1)$ of the dyadic form
$$ \mathcal{D}_L=  \left\lbrace  2^{-k},~ k \in \lbrace 2,..., \lfloor \log_{2}(L)  \rfloor     \rbrace   \right\rbrace \cup  \left\lbrace  1- 2^{-k},~ k \in \lbrace 1,..., \lfloor \log_{2}(L)  \rfloor     \rbrace   \right\rbrace \enspace,$$
we set $u_\alpha = \alpha/(2\lfloor \log_{2}(L)  \rfloor   -1 )$ and we define the two following tests:
\begin{equation} \label{testN1alt8_u}
 \phi_{8,\alpha}^{u(1)}(N) = \1{\max_{\tau \in \mathcal{D}_L} \left( \left|S'_{\tau,1}(N)\right| - s'_{N_1,\tau,1} \pa{1 - u_\alpha} \right)>0}\enspace,
\end{equation}
where $S'_{\tau_1,\tau_2}(N)$ is the linear statistic defined by \eqref{stat_alt6_N1_u} and $s'_{n,\tau_1,\tau_2}(u)$ stands for the $u$-quantile of the conditional distribution of $| S'_{\tau_1,\tau_2}(N)|$ given $N_1=n$ under $\hzero$, and
\begin{equation} \label{testN2alt8_u}
\phi_{8,\alpha}^{u(2)}(N) = \1{\max_{ \tau  \in \mathcal{D}_L} \left(  T'_{\tau ,1 }(N)  - t'_{N_1,\tau ,1} \pa{1 - u_\alpha} \right)>0}\enspace,
\end{equation}
where $T'_{\tau_1,\tau_2}(N)$ is the quadratic statistic defined by \eqref{def_T'}, and $t'_{n,\tau_1,\tau_2}(u)$ denotes the $u$-quantile of its conditional distribution  given $N_1=n$ under $\hzero$.

\begin{proposition}[Minimax upper bound for $\bold{[Alt^{u}.8]}$] \label{UBalt8_u}
Let $\alpha,\beta$ in $(0,1)$, $R>0$, and let $\phi_{8,\alpha}^{u(1/2)}$ be one of the tests $\phi_{8,\alpha}^{u(1)}$ and  $\phi_{8,\alpha}^{u(2)}$ of $\hzero$ versus $\hone\ "\lambda\in\calS^u_{\bbul,\bbul\bbul, 1-\bbul\bbul}[R]"$ defined by \eqref{testN1alt8_u} and \eqref{testN2alt8_u}. Then $\phi_{8,\alpha}^{u(1/2)}$ is of level $\alpha$, that is $\sup_{\lambda_0 \in \mathcal{S}_0^u[R]}P_{\lambda_0}\pa{\phi_{8,\alpha}^{u(1/2)}(N)=1}\leq \alpha$.
Moreover, there exists a constant $C(\alpha, \beta,R)>0$ such that
$$ \SRb\pa{\phi_{8,\alpha}^{u(1/2)},\calS^u_{\bbul,\bbul\bbul,1-\bbul\bbul}[R]} \leq C(\alpha, \beta, R)\sqrt{{\log \log L}/{L}}\enspace,$$
which entails in particular $\mSRab\pa{\calS^u_{\bbul,\bbul\bbul,1-\bbul\bbul}[R]} \leq C(\alpha, \beta, R)\sqrt{{\log \log L}/{L}}$.
\end{proposition}

\subsubsection{Transitory change}

Let us investigate now the transitory change detection problem, focusing on adaptation to unknown location and length.
As in Section \ref{Sec:transitory}, we will prove that minimax adaptation to these two parameters together has a cost of order as large as $\sqrt{\log L}$, so that adaptation to the height will have no additional cost. We therefore treat the two corresponding alternative sets quasi-simultaneously. For $R>0$, $\delta^{*}$ in $(-R,R)\!\setminus\! \lbrace 0 \rbrace$, let
\begin{multline}\label{alt9_u}
\bold{[Alt^u.9]}\ \calS^u_{\delta^*,\bbul\bbul,\bbul\bbul\bbul}[R]=  \big\{ \lambda :[0,1]\to (0,R],~\exists \lambda_{0} \in (- \delta^{*}{} \vee 0, (R- \delta^{*}{}) \wedge R ],\\
\exists \tau\in (0,1),~\exists \ell  \in (0,1-\tau),~\forall t\in[0,1]\quad \lambda(t) = \lambda_{0} + \delta^{*}{} \mathds{1}_{(\tau, \tau + \ell ]}(t)\big\} \enspace,
\end{multline}
\vspace{-0.8cm}
\begin{multline}\label{alt10_u}
\bold{[Alt^u.10]}\  \calS^u_{\bbul,\bbul\bbul,\bbul\bbul\bbul}[R]= \lbrace \lambda :[0,1]\to (0,R],~ \exists \lambda_{0} \in (0, R ],~
 \exists \delta  \in  (-\lambda_{0}, R -  \lambda_0]\backslash \{0\},\\~ \exists \tau\in (0,1),~ \exists \ell  \in  (0,1-\tau),\forall t\in[0,1]\quad  \lambda(t) = \lambda_{0} + \delta \mathds{1}_{(\tau, \tau + \ell ]}(t)\rbrace \enspace.
\end{multline}

As usual, we begin by giving lower bounds for the minimax separation rates over these two alternative sets, noticing that the case where the change height is known can be straightforwardly extended (as a simple corollary then) to the general one, where all three parameters, location, length and height of the change are unknown.

\begin{proposition}[Minimax lower bound for $\bold{[Alt^u.9]}$] \label{LBalt9_u}
Let $\alpha,\beta$ in $(0,1)$ with $\alpha+\beta<1$,
 $R>0$, $\delta^*$ in $(- R, R)\!\setminus\!\{0\}$. There exists $L_0(\alpha,\beta,R,\delta^*)\!>\!0$ such that for $L \geq L_0(\alpha,\beta,R,\delta^*)$, 
\[\mSRab\pa{\calS^u_{\delta^*,\bbul\bbul,\bbul\bbul\bbul}[R]} \geq  \sqrt{{((R-\delta^{*}{} ) \wedge R) \log L}/\pa{4L}}\enspace.\]
\end{proposition}

Since $\calS^u_{\bbul,\bbul\bbul,\bbul\bbul\bbul}[R]$ includes $\calS^u_{R/2,\bbul\bbul,\bbul\bbul\bbul}[R]$, Proposition \ref{LBalt9_u} directly leads to the following corollary, whose proof is omitted for simplicity.
  
\begin{corollary}[Minimax lower bound for $\bold{[Alt^u.10]}$] \label{LBalt10_u}
Let $\alpha,\beta$ in $(0,1)$ with $\alpha+ \beta<1$ and $R>0$. There exists $L_0(\alpha,\beta,R)>0$ such that for all $L \geq L_0(\alpha,\beta,R)$, 
\[\mSRab\pa{\calS^u_{\bbul,\bbul\bbul,\bbul\bbul\bbul}[R]} \geq  \sqrt{R \log L/(8L)}\enspace.\]
\end{corollary}
In order to prove that the above lower bounds are sharp, we then construct two minimax adaptive tests, based on an aggregation principle. In order to customise the tests developed in Section \ref{Sec:transitory} to the lack of knowledge of the baseline intensity, we consider the linear statistic $S'_{\tau_1,\tau_2}(N)$ defined by \eqref{stat_alt6_N1_u} and the quadratic statistic $T'_{\tau_1,\tau_2}(N)$ defined by \eqref{def_T'}, combined with the conditional trick already used in the above studies
through the $u$-quantiles $s'_{n,\tau_1,\tau_2}(u)$ and $t'_{n,\tau_1,\tau_2}(u)$ of the conditional distributions of  $| S'_{\tau_1,\tau_2}(N)|$ and  $T'_{\tau_1,\tau_2}(N)$ given $N_1=n$ under $\hzero$ respectively. Introducing $M_L=\lceil L/\log L \rceil$,
\begin{eqnarray*}
\mathcal{K}_L^{(1)}&=&\set{(k,k'),\ k\in \lbrace 0,\ldots,  \lceil L \rceil   -1\rbrace, k'\in \lbrace 1,\ldots, \lceil L \rceil -k\rbrace}\enspace,\\
\mathcal{K}_L^{(2)}&=&\set{(k,k'),\ k\in \lbrace 0,\ldots,  M_L -1\rbrace, k'\in \lbrace 1,\ldots, M_L -k \rbrace}\setminus\set{(0,M_L)}\enspace,
\end{eqnarray*}
and the corrected levels $u_\alpha^{(1)}=2\alpha/( \lceil L \rceil( \lceil L \rceil+1))$ and $u_\alpha^{(2)}=2\alpha/( M_L (  M_L+1)-2)$, we can thus propose the two following tests:
\begin{equation} \label{testN1alt9-10_u}
\phi_{9/10,\alpha}^{u(1)}(N) = \1{\max_{(k,k')\in \mathcal{K}_L^{(1)}} \Big( \big|S'_{\frac{k}{ \lceil L \rceil},\frac{k+k'}{ \lceil L \rceil}}(N)\big| - s'_{N_1,\frac{k}{\lceil L \rceil},\frac{k+k'}{\lceil L \rceil}}\pa{1 - u_\alpha^{(1)}} \Big)>0}\enspace,
\end{equation}
\begin{equation} \label{testN2alt9-10_u}
\phi_{9/10,\alpha}^{u(2)}(N) = \1{\max_{(k,k')\in\mathcal{K}_L^{(2)}} \left(  T'_{\frac{k}{M_L},\frac{k+k'}{M_L}}(N)  - t'_{N_1,\frac{k}{M_L},\frac{k+k'}{M_L}} \pa{1 - u_\alpha^{(2)}} \right)>0}\enspace.
\end{equation}

\begin{proposition}[Minimax upper bound for $\bold{[Alt^{u}.9]}$ and $\bold{[Alt^{u}.10]}$] \label{UBalt10_u}
Let $\alpha,\beta$ in $(0,1)$, $R>0$ and $\delta^*$ in $(- R, R)\setminus\{0\}$. Let $\phi_{9/10,\alpha}^{u(1/2)}$ be one of the tests $\phi_{9/10,\alpha}^{u(1)}$ and $\phi_{9/10,\alpha}^{u(2)}$ defined by \eqref{testN1alt9-10_u} and \eqref{testN2alt9-10_u}. Then $\phi_{9/10,\alpha}^{u(1/2)}$ is of level $\alpha$ for the problems of testing $\hzero$ versus $\hone\ "\lambda\in\calS^u_{\delta^*,\bbul\bbul,\bbul\bbul\bbul}[R]$  or $\hone\ "\lambda\in\calS^u_{\bbul,\bbul\bbul, \bbul\bbul\bbul}[R]"$, that is $\sup_{\lambda_0 \in \mathcal{S}_0^u[R]}P_{\lambda_0}\pa{\phi_{9/10,\alpha}^{u(1/2)}(N)=1}\leq \alpha$.
Moreover, there exist $C(\alpha, \beta, R,\delta^*)>0$ and $C(\alpha, \beta,R)>0$ such that
\begin{eqnarray*}
\SRb\pa{\phi_{9/10,\alpha}^{u(1/2)},\calS^u_{\delta^*,\bbul\bbul,\bbul\bbul\bbul}[R]} &\leq& C(\alpha, \beta, R,\delta^*)\sqrt{{ \log L}{/L}}\enspace,\\
\SRb\pa{\phi_{9/10,\alpha}^{u(1/2)},\calS^u_{\bbul,\bbul\bbul,\bbul\bbul\bbul}[R]} &\leq &C(\alpha, \beta, R)\sqrt{{ \log L}/{L}}\enspace.
\end{eqnarray*}
These upper bounds entail both $\mSRab\pa{\calS^u_{\delta^*,\bbul\bbul,\bbul\bbul\bbul}[R]} \leq C(\alpha, \beta, R,\delta^*)\sqrt{{ \log L}/{L}}$ and $\mSRab\pa{\calS^u_{\bbul,\bbul\bbul,\bbul\bbul\bbul}[R]} \leq C(\alpha, \beta, R)\sqrt{{ \log L}/{L}}$.
\end{proposition}

\subsection{Adjustment  of individual levels for aggregated tests} 

As in Section \ref{discussionlevels}, we discuss here the possibility of adjusting the individual levels of the single tests involved in our aggregated tests to make them more powerful. Most of the tests introduced in the present section are based on aggregation principles coupled with conditional tricks. Among them, we can again distinguish aggregated tests of the form
\begin{equation}\label{aggregated1_u}
\phi_{agg1,\alpha}^u(N)=\1{\sup_{\theta\in\Theta} S_{\theta}(N)>s_{N_1}^{+}(1-\alpha)}\enspace,
\end{equation}
where $s_{n}^{+}(1-\alpha)$ is the $(1-\alpha)$-conditional quantile of $\sup_{\theta\in\Theta} S_{\theta}(N)$ given $N_1=n$ under $(H_0)$ (concerning $\phi_{3,\alpha}^{u(1)-}$, $\phi_{3,\alpha}^{u(1)+}$, $\phi_{5,\alpha}^u$ and $\phi_{7,\alpha}^u$),
from aggregated tests of the form
\begin{equation}\label{aggregated2_u}
\phi_{agg2,\alpha}^u(N)=\1{\sup_{\theta\in\Theta} \pa{S_{\theta}(N)-s_{N_1,\theta}(1-u_\alpha)}>0}\enspace,
\end{equation}
$s_{N_1,\theta}(1-u)$ being the $(1-u)$-conditional quantile of $S_{\theta}(N)$ under $(H_0)$ and  $u_\alpha=\alpha/|\Theta|$ (as $\phi_{4,\alpha}^{u(1)}$, $\phi_{3/4,\alpha}^{u(2)}$, $\phi_{6,\alpha}^{u(1)}$, $\phi_{6,\alpha}^{u(2)}$, $\phi_{8,\alpha}^{u(1)}$, $\phi_{8,\alpha}^{u(2)}$, $\phi_{9/10,\alpha}^{u(1)}$, $\phi_{9/10,\alpha}^{u(2)}$). Notice that for any $\lambda_0$  in $\mathcal{S}_{0}^{u}[R]$, when $N\sim P_{\lambda_0}$, the distribution of $S_{\theta}(N)$ given $N_1=n$ is free from $\lambda_0$.

\smallskip

We here want to point out that similarly to \eqref{ualpha_minp}, a better choice than $u_\alpha$ can be made for the levels of the single tests involved in the aggregated tests of the form $\phi_{agg2,\alpha}(N)$, namely
\begin{equation}\label{ualpha_minp_unknown}
u_{n,\alpha}'=\sup \set{u\in (0,1),\ \sup_{\lambda_0 \in \mathcal{S}_{0}^{u}[R]} P_{\lambda_0}\pa{\sup_{\theta\in\Theta} \pa{S_{\theta}(N)-s_{n,\theta}(1-u)}>0~\Big\vert N_1=n }\leq \alpha}\enspace.
\end{equation}

Since for all $n$ in $\N\setminus\{0\}$ $u_\alpha\leq u_{n,\alpha}'$, by definition, $s_{n,\theta}(1-u_{n,\alpha}')\leq s_{n,\theta}(1-u_\alpha)$, therefore all the above tests of type $\phi_{agg2,\alpha}^u$ but with $u_\alpha$ replaced by $u_{N_1,\alpha}'$, that we can denote by $\phi_{agg2,\alpha}'^{u}$, satisfy the same minimax properties as $\phi_{agg2,\alpha}^u$. Our simulation study presented in Section \ref{SimulationStudy} focuses on the practical performances of these adjusted aggregated tests $\phi_{agg2,\alpha}'^u$.

\subsection{Summary and discussion}\label{Discussion_unknown} As in Section \ref{Discussion_known}, we present a summary of the results stated above. Recall (\emph{c.f.} \eqref{def_T'}, \eqref{stat_alt5_u} and \eqref{stat_alt6_N1_u}) that for $0\leq \tau_1<\tau_2\leq 1$,
\vspace{-0.3cm}
 \begin{multline*}
T'_{\tau_1, \tau_2}(N)= \frac{1}{L^2} \Big[ \frac{\tau_2 - \tau_1}{1- \tau_2 + \tau_1}  \left( \left( N(0,\tau_1] + N(\tau_2,1] \right)^2 - \left( N(0,\tau_1]  +N(\tau_2,1] \right)  \right) \\
+ \frac{1- \tau_2 + \tau_1}{\tau_2 - \tau_1}  \left( N(\tau_1,\tau_2]^2  - N(\tau_1,\tau_2]  \right) -2 N(\tau_1,\tau_2] \left( N(0, \tau_1]+ N(\tau_2,1] \right) \Big]\enspace,
\end{multline*}
$S'_{\delta^*,\tau_1,\tau_2}(N)=  \mathrm{sgn}(\delta^{*}{}) \Big(N(\tau_1,\tau_2] - (\tau_2-\tau_1) N_1\Big) - \vert \delta^{*}{} \vert L(\tau_2-\tau_1)(1-\tau_2+\tau_1)/2$,\\
$S'_{\tau_1,\tau_2}(N)= N(\tau_1, \tau_2] - (\tau_2-\tau_1) N_1$, and that $t'_{n,\tau_1,\tau_2}(u)$ and $s'_{n,\tau_1,\tau_2}(u)$ stand for the $u$-quantiles of $T'_{\tau_1,\tau_2}(N)$ and $| S'_{\tau_1,\tau_2}(N)|$  given $N_1=n$ under $\hzero$ respectively.  

\medskip

\begin{tabular}{l c l}
\hline
\multicolumn{3}{l}{{\bf Transitory change or bump detection}}\\
\hline
Alternative set & $\mSRab$ & Test statistics\\
\hline
$\calS^u_{\delta^*,\tau^*,\ell^*}[R]$ & - & $N(\tau^{*}, \tau^{*}+ \ell^{*}] $\\
\hline
$\calS^u_{\bbul,\tau^*,\ell^*}[R]$ & $L^{-1/2}$ & $N(\tau^{*}, \tau^{*}+ \ell^{*}] $\\
&&$T'_{\tau^*, \tau^*+\ell^*}(N)$\\
\hline
$ \calS^u_{\delta^*,\bbul\bbul,\ell^*}[R]$ & - & $\max_{\tau\in [0,1-\ell^*\wedge(1/2)]}N(\tau, \tau + \ell^*\wedge(1/2)]$,\\
&& $\min_{\tau\in [0,1-\ell^*\wedge(1/2)]}N(\tau, \tau + \ell^*\wedge(1/2)]$\\
&& $\max_{k\in \lbrace 0,\ldots, \lceil (1-\ell^*) M \rceil -1 \rbrace} \left( T'_{\frac{k}{M}, \frac{k}{M}+\ell^*}(N) - t'_{N_1,\frac{k}{M}, \frac{k}{M}+\ell^*} \pa{1 - u_\alpha} \right)$\\
&& $M= \lceil 2/(\ell^*(1-\ell^*)) \rceil$\\
\hline
$\calS^u_{\bbul,\bbul\bbul,\ell^*}[R]$ & $L^{-1/2}$ & $\max_{\tau\in [0,1-\ell^*\wedge(1/2)]}N(\tau, \tau + \ell^*\wedge(1/2)]$,\\
&& $\min_{\tau\in [0,1-\ell^*\wedge(1/2)]}N(\tau, \tau + \ell^*\wedge(1/2)]$\\
&& $\max_{k\in \lbrace 0,\ldots, \lceil (1-\ell^*) M \rceil -1 \rbrace} \left( T'_{\frac{k}{M}, \frac{k}{M}+\ell^*}(N) - t'_{N_1,\frac{k}{M}, \frac{k}{M}+\ell^*} \pa{1 - u_\alpha} \right)$\\
\hline
$\calS^u_{\delta^*,\tau^*,\bbul\bbul\bbul}[R]$ & $L^{-1/2}$ & $\sup_{\ell \in (0,1- \tau^{*})}S'_{\delta^*,\tau^*,\tau^*+\ell}(N)$\\
\hline
$\calS^u_{\bbul,\tau^*,\bbul\bbul\bbul}[R]$ & $\sqrt{\frac{\log\log L}{L}}$ & $\max_{k\in \lbrace 1,\ldots,\lfloor \log_{2} L \rfloor\rbrace} \Big( \big|S'_{\tau^*,\tau^*+\frac{1 - \tau^*}{2^k}}(N)\big| - s'_{N_1,\tau^*,\tau^*+\frac{1 - \tau^*}{2^k}} \pa{1 - u_\alpha} \Big)$\\
&& $\max_{k\in \lbrace 1,\ldots, \lfloor \log_{2} L \rfloor \rbrace} \left(  T'_{\tau^*, \tau^*+\frac{1 - \tau^*}{2^k} }(N)  - t'_{N_1,\tau^*,\tau^*+\frac{1 - \tau^*}{2^k}} \pa{1 - u_\alpha} \right)$\\
\hline
$\calS^u_{\delta^*,\bbul\bbul,\bbul\bbul\bbul}[R]$ & $\sqrt{\frac{\log L}{L}}$ & $\max_{k\in \lbrace 0,\ldots, \lceil L \rceil -1\rbrace, k'\in \lbrace 1,\ldots, \lceil L \rceil-k\rbrace} \Big( \big|S_{\frac{k}{\lceil L\rceil},\frac{k+k'}{\lceil L\rceil}}(N)\big| - s_{\lambda_0,\frac{k}{\lceil L\rceil},\frac{k+k'}{\lceil L\rceil}}\pa{1 - u_\alpha} \Big)$\\
&&$\max_{k\in \lbrace 0,\ldots, \lceil L \rceil -1\rbrace, k'\in \lbrace 1,\ldots, \lceil L \rceil-k\rbrace} \left(  T_{\frac{k}{\lceil L\rceil},\frac{k+k'}{\lceil L\rceil}}(N)  - t_{\lambda_0,\frac{k}{\lceil L\rceil},\frac{k+k'}{\lceil L\rceil}} \pa{1 - u_\alpha} \right)$\\
\hline
$ \calS^u_{\bbul,\bbul\bbul,\bbul\bbul\bbul}[R]$ &$\sqrt{\frac{\log L}{L}}$ & $\max_{k\in \lbrace 0,\ldots,  \lceil L \rceil   -1\rbrace, k'\in \lbrace 1,\ldots, \lceil L \rceil -k\rbrace} \Big( \big|S'_{\frac{k}{ \lceil L \rceil},\frac{k+k'}{ \lceil L \rceil}}(N)\big| - s'_{N_1,\frac{k}{\lceil L \rceil},\frac{k+k'}{\lceil L \rceil}}\pa{1 - u_\alpha^{(1)}} \Big)$\\
&& $\max_{\substack{k\in \lbrace 0,\ldots,  M_L -1\rbrace, k'\in \lbrace 1,\ldots, M_L -k \rbrace\\ (k,k') \neq (0,M_L)}} \left(  T'_{\frac{k}{M_L},\frac{k+k'}{M_L}}(N)  - t'_{N_1,\frac{k}{M_L},\frac{k+k'}{M_L}} \pa{1 - u_\alpha^{(2)}} \right)$\\
&& $M_L=\lceil L/\log L \rceil  $\\
\hline
\end{tabular}

\begin{tabular}{l c l}
\hline
\multicolumn{3}{l}{{\bf Non transitory change or jump detection}}\\
\hline
Alternative set & $\mSRab$ & Test statistics\\
\hline
$\calS^u_{\delta^*,\tau^*,1-\tau^*}[R]$ & - & $N(\tau^{*}, 1] $\\
\hline
$\calS^u_{\bbul,\tau^*,1-\tau^*}[R]$ & $L^{-1/2}$ & $N(\tau^{*},1] $\\
&&$T'_{\tau^*, 1}(N)$\\
\hline
$\calS^u_{\delta^*,\bbul\bbul,1-\bbul\bbul}[R]$ &$L^{-1/2}$ & $\sup_{\tau \in (0,1)}S'_{\delta^*,\tau,1}(N)$\\
\hline
$\calS^u_{\bbul,\bbul\bbul,1-\bbul\bbul}[R]$ & $\sqrt{\frac{\log \log L}{L}}$  & $\max_{\tau \in \calD_L} \Big( \big|S'_{\tau,1}(N)\big| - s'_{N_1,\tau,1} \pa{1 - u_\alpha} \Big)$\\
&& $\max_{\tau  \in \calD_L} \left(  T'_{\tau ,1 }(N)  - t'_{N_1,\tau ,1} \pa{1 - u_\alpha} \right)$\\
&& $\calD_L=\left\lbrace  2^{-k},~ k \in \lbrace 2,\ldots, \lfloor \log_{2}(L)  \rfloor     \rbrace   \right\rbrace \cup  \left\lbrace  1- 2^{-k},~ k \in \lbrace 1,\ldots, \lfloor \log_{2}(L)  \rfloor     \rbrace   \right\rbrace$\\
\hline
\end{tabular}

\medskip

As compared with the above overview in Section \ref{Discussion_known}, this one enables to see that the minimax separation rates do not suffer from the lack of knowledge of the baseline distribution: they indeed remain of the same order as in the problem of detecting a change from a given intensity, with the same phase transitions. Of course, these results are obtained at the price of a more important complexity of the test statistics, whether they are of linear or quadratic nature. This, combined with the need to use conditional quantiles instead of direct quantiles as critical values, brought in more technical arguments in the proofs. It can be furthermore noticed that up to our knowledge, except in the work of Verzelen et al.  \cite{Verzelen2021} for the jump detection problem, this specific case of an unknown baseline distribution is in general not treated in the basic Gaussian model, where the only presence of a signal (that is a bump or jump from zero-mean) is tested.

\section{Simulation study}\label{SimulationStudy}

We study in this section the performance of our minimax adaptive tests from an experimental point of view, by giving estimations of their size and their power for various distributions of the observed Poisson process, characterised by a jump or a bump in its intensity. Motivated by some applications in epidemiology and in cybersecurity, we check the feasibility of our new change-points detection procedures in practice and compare them with existing procedures.

We focus here on the most general problems investigated in this article of detecting a change (a jump or a bump) in the intensity $\lambda$ when the change location and height are unknown. The baseline intensity of $\lambda$, denoted by $\lambda_0,$ is taken equal to $1$ on $[0,1]$ in all the sequel. Recall that this baseline intensity can be considered as a known parameter of the testing problem as in Section \ref{Sec:knownbaseline}  or as an unknown parameter as in Section \ref{Sec:unknownbaseline}.

For several piecewise constant intensities $\lambda$ with respect to the measure $d\Lambda(t)=Ldt,$  where we have chosen $L=100$, we take a level of test $\alpha=0,05.$ 

We compare the estimated powers of our procedures with more classical conditional tests previously studied in practice by other authors. For instance Cohen and Sackrowitz \cite{Cohen1993}, and Bain, Engelhardt and Wright \cite{Bain1985} considered six well-known tests in the context of detecting increasing intensities of a Poisson process. They showed that two of these six tests, namely the so-called Laplace and $Z$ tests (respectively studied first by Cox \cite{Cox1955} and Crow \cite{Crow1974}) are more efficient from a practical point of view.
The Laplace test, denoted by (\texttt{La}) is based on the statistic
\[ \calT^{(La)}_\alpha(N) = \sum_{i=1}^{N_1} X_i - q_{N_1}^{(La)}(1-\alpha) \enspace,\]
where $(X_1,\ldots,X_{N_1})$ are the points of the Poisson process $N$, and for all $n$ in $\mathbb{N}$, $q_{n}^{(La)}(1-\alpha)$ is the $(1-\alpha)$-quantile of the sum of $n$ independent random variables uniformly distributed on $[0,1]$. 
The $Z$ test, denoted by (\texttt{Z}), is based on the statistic
\[ \calT^{(Z)}_\alpha(N) = 2\sum_{i=1}^{N_1} \log(X_i) + q_{N_1}^{(Z)}(\alpha)\enspace,\]
where for every $n$ in $\mathbb{N}$, $q_{n}^{(Z)}(\alpha)$ is the $\alpha$-quantile of the chi-square distribution with $2n$ degrees of freedom. Note that these tests were especially designed to test homogeneity versus an increasing trend, with rejection of homogeneity when they take positive values. Therefore, we have had to adapt them to fit our context. More precisely, we decided to reject the null hypotheses $\hzero \ "\lambda\in \calS_0[\lambda_0]=\{\lambda_0\}"$ or $\hzero \ "\lambda\in \calS_0^u[R]"$ when $\calT^{(La)}_{\alpha/2}(N)>0$ or $\calT^{(La)}_{1-\alpha/2}(N)<0$ for the Laplace test (\texttt{La}), and when $\calT^{(Z)}_{\alpha/2}(N)>0$ or $\calT^{(Z)}_{1-\alpha/2}(N)<0$ for the $Z$ test (\texttt{Z}). Notice that a generalised version of the Laplace and the $Z$ tests have been studied by Peña \cite{Pena1998a} and Zenia, Agustin and Peña \cite{Agustin1999}. A simulation study has been performed for these generalised procedures, but we have found that they have very similar estimated powers to the more classical Laplace and $Z$ tests for the considered alternatives. The results are therefore omitted in this section.

\subsection{Detection of an abrupt change from a known baseline intensity}

We first consider the case where $\lambda_0$ is a known parameter, referring to the theoretical results of Section \ref{Sec:knownbaseline}. The minimax adaptive tests we introduced to detect a change with unknown parameters from such a known intensity are based on two kinds of statistics. The first statistic, of linear nature, is defined by $S_{\tau_1,\tau_2}(N)= N(\tau_1, \tau_2] - \lambda_{0} (\tau_2-\tau_1) L,$ while the second statistic, of quadratic nature, is defined by $$ T_{\tau_1, \tau_2}(N) = \frac{1}{L^2 (\tau_2 - \tau_1)} \left(  N \left( \left.  \tau_1, \tau_2  \right] \right.^2 -N \left( \left.  \tau_1 , \tau_2  \right] \right.  \right) - \frac{2 \lambda_{0}}{L} N \left( \left.  \tau_1, \tau_2  \right] \right. + \lambda_{0}^2 (\tau_2 - \tau_1)\enspace.$$

\subsubsection{Detection of a non transitory change or jump} \label{simuCP_Known}

Let us recall that our tests are based on an aggregation principle which involves a scanning of the above linear and quadratic statistics over a discrete subset of possible values for the change location on $(0,1)$. The subset introduced in Section \ref{Sec:generalsingle} is of the dyadic form
\[\Theta_d=\left\{1 -2^{-k}, ~ k \in \lbrace 1,\ldots,6 \rbrace \right\}\enspace.\]

\smallskip

Considering the alternative $\bold{[Alt.8]}$, the test statistic of our first procedure denoted by (\texttt{CP1}($\Theta_d$)) is thus
\[\calT_{\lambda_0,\alpha }^{(1)} \left(N\right)=\max_{\tau\in \Theta_d} \Big( \big|S_{\tau,1}(N)\big| - s_{\lambda_0,\tau,1} \pa{1 - u^{(1)}_\alpha} \Big) \enspace, \]
where $s_{\lambda_0, {\tau_1},{\tau_2}}(1-u)$ is the $(1-u)$-quantile of $|S_{{\tau_1},{\tau_2}}(N)|$ under $\hzero$ and $u^{(1)}_\alpha$ is defined as in \eqref{ualpha_minp} by
\begin{equation} \label{simuKnown_ualpha1}
u^{(1)}_\alpha = \sup \left\lbrace u \in (0,1),~ P_{\lambda_0} \left(\max_{\tau\in \Theta_d} \Big( \big|S_{\tau,1}(N)\big| - s_{\lambda_0,\tau,1} \pa{1 -u} \Big)  >0  \right) \leq \alpha   \right\rbrace \enspace,
\end{equation}
while the test statistic of our second procedure denoted by (\texttt{CP2}($\Theta_d$)) is 
\[ \calT_{\lambda_0,\alpha }^{(2)} \left( N\right) = \max_{\tau\in \Theta_d} \left(  T_{\tau,1 }(N)  - t_{\lambda_0,\tau,1} \pa{1 - u^{(2)}_\alpha} \right) \enspace, \] 
where $t_{\lambda_0, {\tau_1},{\tau_2}}(1-u)$ is the $(1-u)$-quantile of $T_{{\tau_1},{\tau_2}}(N) $ under $\hzero$ and 
\begin{equation} \label{simuKnown_ualpha2}
u^{(2)}_\alpha = \sup \left\lbrace u \in (0,1),~ P_{\lambda_0} \left(\max_{\tau\in \Theta_d} \left(  T_{\tau,1 }(N)  - t_{\lambda_0,\tau,1} \pa{1 - u} \right) >0  \right) \leq \alpha   \right\rbrace \enspace.
\end{equation}
The null hypothesis $\hzero \ "\lambda\in \calS_0[\lambda_0]=\{\lambda_0\}"$ is rejected when $\calT_{\lambda_0,\alpha }^{(1)}(N) >0$ for (\texttt{CP1}($\Theta_d$)), or when $\calT_{\lambda_0,\alpha }^{(2)} (N)>0$ for (\texttt{CP2}($\Theta_d$)).

\smallskip

Noticing that the test we use for the bump detection problem (see the following subsection) and the jump detection procedure studied in \cite{Verzelen2021} are based on regular grids of possible values of change location instead of the dyadic subset $\Theta_d$, we also consider the same tests but replacing $\Theta_d$ by a regular grid, with a cardinality close to the cardinality of $\Theta_d$, namely 
\[\Theta_r=\left\{\frac{k}{10},~ k \in \lbrace 1,\ldots,9 \rbrace \right\}\enspace.\]
The corresponding tests are then respectively denoted by (\texttt{CP1}($\Theta_r$)) and (\texttt{CP2}($\Theta_r$)).

\smallskip

For all $\tau$ in $\Theta_d$ or $\Theta_r$, we have estimated the quantities $u^{(1)}_\alpha$, $s_{\lambda_0, {\tau},1} (1- u^{(1)}_\alpha )$, $u^{(2)}_\alpha$ and $t_{\lambda_0, {\tau},1} (1-u^{(2)}_\alpha )$
by classical Monte Carlo methods based on the simulation of 200 000 independent samples of $\vert S_{{\tau},1}(N) \vert$ or $T_{{\tau},1}(N)$ under $\hzero.$ The approximations of $u^{(1)}_\alpha$  and $u^{(2)}_\alpha$ were obtained by dichotomy, such that the estimated probabilities occurring in \eqref{simuKnown_ualpha1} and \eqref{simuKnown_ualpha2} are less than $\alpha$, but as close to $\alpha$ as possible.

\subsubsection{Detection of a transitory change or bump}\label{simuTC_Known}

Let us consider the discrete set :
$$\Theta=\set{\pa{\frac{k}{100},\frac{k+k'}{100}},~ k \in \lbrace 0,\ldots,99 \rbrace,~k' \in \lbrace 1,\ldots,100-k \rbrace}\enspace.$$
Considering the alternative $\bold{[Alt.10]}$,  the test statistic of our first procedure denoted by (\texttt{TC1}) is
\[ \mathcal{T}_{\lambda_0,\alpha}^{(3)} (N) = \max_{(\tau,\tau')\in \Theta} \left( \abs{ S_{\tau,\tau'}(N) } - s_{\lambda_0,\tau,\tau'} \left( 1-u^{(3)}_\alpha \right) \right) \enspace,\]
where $u^{(3)}_\alpha$ is defined, again as in \eqref{ualpha_minp}, by 
\begin{equation} \label{simuKnown_ualpha3}
u^{(3)}_\alpha = \sup \left\lbrace u \in (0,1),~ P_{\lambda_0} \left(   \max_{(\tau,\tau')\in \Theta} \left( \abs{ S_{\tau,\tau'}(N) } - s_{\lambda_0,\tau,\tau'} \left( 1-u \right) \right) >0  \right) \leq \alpha   \right\rbrace \enspace,
\end{equation}
while the test statistic of our second procedure denoted by (\texttt{TC2}) is 
\[ \mathcal{T}_{\lambda_0,\alpha }^{(4)} \left( N\right) = \max_{(\tau,\tau')\in \Theta} \left(  T_{\tau,\tau'}(N)  - t_{\lambda_0, \tau,\tau'} \left( 1-u^{(4)}_\alpha \right) \right) \enspace, \] where $u^{(4)}_\alpha$ is defined by 
\begin{equation} \label{simuKnown_ualpha4}
u^{(4)}_\alpha = \sup \left\lbrace u \in (0,1),~ P_{\lambda_0} \left( \max_{(\tau,\tau')\in \Theta} \left(  T_{\tau,\tau'}(N)  - t_{\lambda_0, \tau,\tau'} \left( 1-u^{(4)}_\alpha \right) \right)>0  \right) \leq \alpha   \right\rbrace \enspace.
\end{equation}
The null hypothesis $\hzero \ "\lambda\in \calS_0[\lambda_0]=\{\lambda_0\}"$ is rejected when $\mathcal{T}_{\lambda_0,\alpha}^{(3)}(N)>0$ for (\texttt{TC1}), or when $\mathcal{T}_{\lambda_0,\alpha}^{(4)}(N)>0$ for (\texttt{TC2}). 

\smallskip

As above, for all $(\tau,\tau')$ in $ \Theta$, the quantities $u^{(3)}_\alpha,$ $s_{\lambda_0,\tau,\tau'}(1-u^{(3)}_\alpha), $ $u^{(4)}_\alpha$ and $t_{\lambda_0, \tau,\tau'}(1-u^{(4)}_\alpha) $ have been estimated  by Monte Carlo methods based on the simulation of 200 000 independent samples of $\vert S_{\tau,\tau'}(N) \vert$ or $T_{\tau,\tau'}(N) $  under $\hzero.$  The approximations of $u^{(3)}_\alpha$  and $u^{(4)}_\alpha$ have been obtained by dichotomy.

\subsubsection{Simulation results}

We compare the tests (\texttt{La}) and (\texttt{Z}) with (\texttt{CP1}($\Theta_d$)), (\texttt{CP2}($\Theta_d$)), (\texttt{CP1}($\Theta_r$)) and (\texttt{CP2}($\Theta_r$)) when addressing the jump detection problem (described in \ref{simuCP_Known}), and with (\texttt{TC1}) and (\texttt{TC2}) when addressing the bump detection problem (described in \ref{simuTC_Known}).

\paragraph*{Estimated sizes}
We first study the size of each test via simulation of 5 000 independent homogeneous Poisson processes of intensity $\lambda_0=1$ w.r.t. $\Lambda$ on $[0,1].$ The probabilities of first kind error of all the considered tests were simply  estimated by the number of rejections divided by 5 000. The results are given in Table \ref{size*}.

\begin{table}[h]
\caption{Estimated sizes}\label{size*}
\begin{center}
\begin{tabular}{cc}
\hline
\texttt{La}  & \texttt{Z} \\
0.051 &0.049\\
\hline
\texttt{CP1}($\Theta_d$) & \texttt{CP2}($\Theta_d$)\\
0.049 &0.049\\
\hline
\texttt{CP1}($\Theta_r$) & \texttt{CP2}($\Theta_r$)\\
0.048 &0.047\\
\hline
\texttt{TC1} & \texttt{TC2}\\  
 0.050 & 0.049\\
\hline
\end{tabular}
\end{center}
\end{table}

Notice that the estimated sizes of our tests always remain below the target $0.05$, as expected from the definitions of $u^{(1)}_\alpha$, $u^{(2)}_\alpha$, $u^{(3)}_\alpha$ and $u^{(4)}_\alpha$. It is in particular interesting to see that the Monte Carlo estimation, which is calibrated according to a balance precision/running time, does not affect here the first kind error rate control property.

\paragraph*{Estimated powers}
For both testing problems, we study the estimated power of each test under various alternatives. 

\smallskip

Let us start with the jump detection problem. We consider alternative intensities $\lambda_{\tau,\delta}$ defined for all $t$ in $[0,1]$ by 
\begin{equation} \label{simu_alt_cp}
\lambda_{\tau, \delta}(t) = 1 + \delta \mathds{1}_{(\tau,1]}(t) \enspace,
\end{equation}
 where $\delta \in \{-0.8 , -0.6  , -0.4  , -0.2  ,  0.2  , 0.4  ,  0.6  ,  0.8\},$ and $\tau=0.2$ (Table \ref{cp*0.2}), $\tau=0.5$ (Table~\ref{cp*0.5}), $\tau=0.8$ (Table \ref{cp*0.8}), $\tau=0.9$ (Table \ref{cp*0.9}) or $\tau=0.95$ (Table \ref{cp*0.95}).
%
 For each alternative, 1 000 independent inhomogeneous Poisson processes with intensity $\lambda_{\tau,\delta}$ w.r.t. $\Lambda$ on $[0,1]$ have been simulated. The power of the considered tests has then been simply estimated by the number of rejections divided by 1 000,  leading to the results gathered in Tables \ref{cp*0.2}-\ref{cp*0.95}.

 \begin{table}[h]
\caption{Estimated probability of detecting a jump with $\tau=0.2$}
\label{cp*0.2}
\begin{center}
\begin{tabular}{ccccccccc}
\hline
$\delta$ & -0.8  & -0.6 & -0.4 & -0.2 & 0.2 & 0.4 & 0.6 & 0.8 \\
\hline
\hline
\texttt{La} & 0.95&0.61 &0.27 &0.08 &  0.07&0.15 &0.28 &0.42\\
\texttt{Z} & 0.96 &0.69 &0.31 &0.08 &0.09 &0.18 &0.38 & 0.56 \\
\hline
\texttt{CP1}($\Theta_d$) &1 &1 &0.69 &0.14 &0.23 &0.65 &0.94 &0.99 \\
\texttt{CP2}($\Theta_d$) &1 &1 &0.75 &0.18 &0.20 &0.60 &0.92 &0.99 \\
\hline
\texttt{CP1}($\Theta_r$) &1 &1 &0.95 &0.31 &0.37 &0.87 &0.99 &1 \\
\texttt{CP2}($\Theta_r$) &1 &1 &0.96 &0.34 &0.35 &0.86 &0.99 &1\\
\end{tabular}
\end{center}
\end{table}

\begin{table}[h]
\caption{Estimated probability of detecting a jump with $\tau=0.5$}
\label{cp*0.5}
\begin{center}
\begin{tabular}{ccccccccc}
\hline
$\delta=$ & -0.8  & -0.6 & -0.4 & -0.2 & 0.2 & 0.4 & 0.6 & 0.8 \\
\hline
\hline
\texttt{La} &1 &0.87 &0.50 &0.15 &0.13 &0.35 &0.61 &0.84 \\
\texttt{Z} &0.92 &0.62 &0.31 &0.09 &0.11 &0.26 &0.46 &0.72\\
\hline
\texttt{CP1}($\Theta_d$) &1 &1 &0.68 &0.14 &0.22 &0.67 &0.94 &1\\
\texttt{CP2}($\Theta_d$) &1 &1 &0.73 &0.18 &0.18 &0.64 &0.92 &0.99\\
\hline
\texttt{CP1}($\Theta_r$) &1 &1 &0.76 &0.22 &0.28 &0.71 &0.96 &1\\
\texttt{CP2}($\Theta_r$) &1 &1 &0.77 &0.23 &0.25 &0.68 &0.94 &1\\
\hline
\end{tabular}
\end{center}
\end{table}

 \begin{table}[h]
\caption{Estimated probability of detecting a jump with $\tau=0.8$}
\label{cp*0.8}
\begin{center}
\begin{tabular}{ccccccccc}
\hline
$\delta=$ & -0.8  & -0.6 & -0.4 & -0.2 & 0.2 & 0.4 & 0.6 & 0.8 \\
\hline
\hline
\texttt{La} &0.70 &0.43 &0.19 &0.08 &0.08 &0.18 &0.36 &0.53\\
\texttt{Z} &0.28 &0.20 &0.09 &0.06&0.07 &0.11 &0.20 &0.30 \\
\hline
\texttt{CP1}($\Theta_d$) & 0.92 &0.55 &0.20 &0.07 &0.15 &0.33 &0.59 &0.78\\
\texttt{CP2}($\Theta_d$) &0.96 &0.63 &0.25 &0.08 &0.14 &0.30 &0.54 &0.74\\
\hline
\texttt{CP1}($\Theta_r$) & 0.99 &0.67 &0.27 &0.09 &0.16 &0.37 &0.63 &0.82\\
\texttt{CP2}($\Theta_r$) &1 &0.76 &0.33 &0.11 &0.14 &0.36 &0.61 &0.81\\
\hline
\end{tabular}
\end{center}
\end{table}

  \begin{table}[h]
\caption{Estimated probability of detecting a jump with $\tau=0.9$}
\label{cp*0.9}
\begin{center}
\begin{tabular}{ccccccccc}
\hline
$\delta=$ & -0.8  & -0.6 & -0.4 & -0.2 & 0.2 & 0.4 & 0.6 & 0.8 \\
\hline
\hline
\texttt{La} &0.22 &0.16 &0.08 &0.05 &0.08 &0.11 &0.17 &0.24\\
\texttt{Z} &0.10 &0.07 &0.05 &0.04 &0.07 &0.09 &0.10 &0.14 \\
\hline
\texttt{CP1}($\Theta_d$) & 0.49 &0.22 &0.08 &0.04 &0.10 &0.19 &0.36 &0.55\\
\texttt{CP2}($\Theta_d$) &0.63 &0.30 &0.11 &0.05 &0.09 &0.18 &0.35 &0.54\\
\hline
\texttt{CP1}($\Theta_r$) & 0.50 &0.19 &0.09 &0.06 &0.11 &0.17 &0.34 &0.52\\
\texttt{CP2}($\Theta_r$) &0.74 &0.30 &0.13 &0.07 &0.10 &0.16 &0.33 &0.51\\
\hline
\end{tabular}
\end{center}
\end{table}

  \begin{table}[h]
\caption{Estimated probability of detecting a jump with $\tau=0.95$}
\label{cp*0.95}
\begin{center}
\begin{tabular}{ccccccccc}
\hline
$\delta=$ & -0.8  & -0.6 & -0.4 & -0.2 & 0.2 & 0.4 & 0.6 & 0.8 \\
\hline
\hline
\texttt{La} &0.10 &0.07 &0.06 &0.04 &0.05 &0.07 &0.08 &0.10\\
\texttt{Z} &0.06 &0.04 &0.05 &0.05 &0.05 &0.05 &0.06 &0.06 \\
\hline
\texttt{CP1}($\Theta_d$) & 0.17 &0.06 &0.04 &0.04 &0.08 &0.12 &0.20 &0.31\\
\texttt{CP2}($\Theta_d$) &0.19 &0.08 &0.05 &0.04 &0.08 &0.11 &0.19 &0.30\\
\hline
\texttt{CP1}($\Theta_r$) & 0.08 &0.06 &0.04 &0.04 &0.07 &0.08 &0.12 &0.21\\
\texttt{CP2}($\Theta_r$) &0.12 &0.08 &0.05 &0.05 &0.07 &0.07 &0.11 &0.20\\
\hline
\end{tabular}
\end{center}
\end{table}

Let us now turn to the bump detection problem. We have considered alternative intensities $\lambda_{\tau,\ell,\delta}$ defined for all $t$ in $[0,1]$ by 
\begin{equation} \label{simu_alt_tc}
\lambda_{\tau,\ell, \delta}(t) = 1 + \delta \mathds{1}_{(\tau,\tau +\ell]}(t) \enspace,
\end{equation} 
where $\delta \in \{ -0.8 , -0.6 , -0.4 , -0.2 , 0.2 , 0.4 , 0.6 , 0.8 \},$ with $\tau=0.2$ and $\ell\in \{0.1,0.4,0.7\}$, and then with $\tau=0.5$ and $\ell=0.4$.

For each alternative, we have simulated 1 000 independent inhomogeneous Poisson processes with intensity $\lambda_{\tau,\ell,\delta}$ w.r.t. $\Lambda$ on $[0,1].$ The powers have been estimated for each test by the number of rejections divided by 1 000,  and the results are provided in Tables \ref{tc*0.2-0.1}-\ref{tc*0.5-0.4}.
 
 \begin{table}[h]
\caption{Estimated probability of detecting a bump with $\tau=0.2$ and $\ell=0.1$}
\label{tc*0.2-0.1}
\begin{center}
\begin{tabular}{ccccccccc}
\hline
$\delta=$ & -0.8  & -0.6 & -0.4 & -0.2 & 0.2 & 0.4 & 0.6 & 0.8 \\
\hline
\hline
\texttt{La} &0.11 &0.07 &0.06 &0.06 &0.05 &0.06 &0.07 &0.10\\
\texttt{Z} &0.08 &0.07 &0.06 &0.06 &0.06 &0.05 &0.04 &0.05\\
\hline
\texttt{TC1} &0.06 &0.04 &0.04 &0.04 &0.06 &0.11 &0.18 &0.30\\
\texttt{TC2} &0.10 &0.05 &0.04 &0.04 &0.07 &0.11 &0.18 &0.30 \\
\hline
\end{tabular}
\end{center}
\end{table}

 \begin{table}[h]
\caption{Estimated probability of detecting a bump with $\tau=0.2$ and $\ell=0.4$}
\label{tc*0.2-0.4}
\begin{center}
\begin{tabular}{ccccccccc}
\hline
$\delta=$ & -0.8  & -0.6 & -0.4 & -0.2 & 0.2 & 0.4 & 0.6 & 0.8 \\
\hline
\hline
\texttt{La} &0.30 &0.17 &0.11 &0.06 &0.06 &0.07 &0.10 &0.13\\
\texttt{Z} &0.10 &0.09 &0.07 &0.06 &0.05 &0.04 &0.04 &0.03\\
\hline
\texttt{TC1} &1 &0.71 &0.21 &0.04 &0.15 &0.46 &0.73 &0.94\\
\texttt{TC2} &1 &0.83 &0.30 &0.06 &0.15 &0.46 &0.74 &0.94\\
\hline
\end{tabular}
\end{center}
\end{table}

 \begin{table}[h]
\caption{Estimated probability of detecting a bump with $\tau=0.2$ and $\ell=0.7$}
\label{tc*0.2-0.7}
\begin{center}
\begin{tabular}{ccccccccc}
\hline
$\delta=$ & -0.8  & -0.6 & -0.4 & -0.2 & 0.2 & 0.4 & 0.6 & 0.8 \\
\hline
\hline
\texttt{La} &0.33 &0.19 &0.11 &0.08 &0.06 &0.07 &0.07 &0.10\\
\texttt{Z} &0.70 &0.39 &0.19 &0.10 &0.07 &0.13 &0.20 &0.32\\
\hline
\texttt{TC1} &1 &0.99 &0.54 &0.07 &0.24 &0.70 &0.94 &1\\
\texttt{TC2} &1 &1 &0.63 &0.09 &0.24 &0.71 &0.94 &1 \\
\hline
\end{tabular}
\end{center}
\end{table}

 \begin{table}[h]
\caption{Estimated probability of detecting a bump with $\tau=0.5$ and $\ell=0.4$}
\label{tc*0.5-0.4}
\begin{center}
\begin{tabular}{ccccccccc}
\hline
$\delta=$ & -0.8  & -0.6 & -0.4 & -0.2 & 0.2 & 0.4 & 0.6 & 0.8 \\
\hline
\hline
\texttt{La} &0.74 &0.48 &0.24 &0.09 &0.07 &0.15 &0.31 &0.47\\
\texttt{Z} &0.63 &0.37 &0.19 &0.07 &0.07 &0.14 &0.31 &0.44\\
\hline
\texttt{TC1} &1 &0.71 &0.16 &0.04 &0.14 &0.41 &0.75 &0.93\\
\texttt{TC2} &1 &0.85 &0.28 &0.06 &0.14 &0.43 &0.77 &0.94\\
\hline
\end{tabular}
\end{center}
\end{table}

\subparagraph*{Comments}
\begin{enumerate}
\item It first arises that our procedures have estimated powers significantly larger than the Laplace and the $Z$ tests for both testing problems corresponding to alternatives $\bold{[Alt.8]}$ and $\bold{[Alt.10]}$ in most cases. The lower performances of the Laplace and the $Z$ tests may be due to the fact that their construction does not take the knowledge of $\lambda_0$ into account. Moreover, among our testing procedures, it is to note that both procedures based on the quadratic statistics (\texttt{CP2}) and (\texttt{TC2}) are more worthwhile to use than the ones based on the linear statistics  (\texttt{CP1}) and (\texttt{TC1}). Indeed, in the case of negative change heights, the procedures (\texttt{CP2}) and (\texttt{TC2}) are distinctly mostly more powerful than (\texttt{CP1}) and (\texttt{TC1}), whereas the powers are very close when positive change heights occur. 
\item The comparison of the estimated powers of the different testing procedures confirm the intuition that detecting a bump is harder than detecting a jump. Moreover, the performances of each test are very different according to the sign of the change height. In both jump and bump detection problems, it is easier to detect a change with large negative height than a change with large positive height except for the cases where $\ell=0.1$ in the bump detection problem (the estimated powers are close to the size of the tests for small negative change height whatever the values of change location and length). Note that this capability of our tests to better detect a jump or a bump with negative height than with positive  height can be explained by the fact that the significant parameter to evaluate the detectability of a (transitory or not) change is not the change height itself but the ratio between the minimum and the maximum values of the intensity. In other words, this means that it is easier to detect an intensity increasing from $1$ to $2$ than an intensity increasing from $100$ to $101$ whereas in both cases, the jump height is equal to one. In the same way, it is hence easier to detect an intensity decreasing from $1$ to $0.8$ than an intensity increasing from $1$ to $1.2$.

\item By comparing the estimated powers of our jump detection procedures based on the dyadic and regular sets $\Theta_d$ and $\Theta_r$, one can take note that using the dyadic set is as expected more relevant when the jump is close to the observation interval ending point, that is the most difficult to detect.

\item Finally, for the bump detection problem, we have to notice that complementary experiments showed that the estimated powers of (\texttt{TC1}) and (\texttt{TC2}) are equivalent for a same value of change length whatever the values of the change location. For the procedures (\texttt{La}) and (\texttt{Z}), we noticed that this is also true for $\ell =0.1$, but  not for $\ell=0.4$ since we observe a larger power for  $\tau=0.5$ than for  $\tau=0.2$. 
\end{enumerate}

\subsection{Detection of an abrupt change from an unknown baseline intensity}

We now consider the case where $\lambda_0$ is an unknown parameter, referring to theoretical results in Section \ref{Sec:unknownbaseline}. The minimax adaptive tests we introduced to detect a change with unknown parameters from such an unknown intensity are still based on two kinds of statistics. The first statistic, of linear nature, is defined by $
S'_{\tau_1, \tau_2}(N) = N(\tau_1, \tau_2] - N(0,1] (\tau_2 - \tau_1)$, while the second statistic, of quadratic nature, is defined by 
\begin{align*}
T'_{\tau_1, \tau_2}(N)&= \frac{1}{L^2} \left[ \frac{\tau_2 - \tau_1}{1- \tau_2 + \tau_1}  \left( \left( N(0,\tau_1] + N(\tau_2,1] \right)^2 - \left( N(0,\tau_1]  +N(\tau_2,1] \right)  \right) \nonumber \right. \\ & \quad  \left. + \frac{1- \tau_2 + \tau_1}{\tau_2 - \tau_1}  \left( N(\tau_1,\tau_2]^2  - N(\tau_1,\tau_2]  \right) -2 N(\tau_1,\tau_2] \left( N(0, \tau_1]+ N(\tau_2,1] \right) \right].
\end{align*}

\subsubsection{Detection of a non transitory change or jump} \label{simuCP_Unknown}

As above, the aggregation approach we used to construct our new tests to detect a jump from an unknown baseline intensity consists in scanning these linear and quadratic statistics over a discrete subset of possible values for the change location on $(0,1)$. The subset introduced in Section \ref{Sec:generalsingle_u} is of the dyadic form
\[\Theta_d^u=\left\lbrace  2^{-k},~ k \in \lbrace 2,\ldots, 6     \rbrace   \right\rbrace \cup  \left\lbrace  1- 2^{-k},~ k \in \lbrace 1,\ldots, 6     \rbrace   \right\rbrace     \enspace.\]

\smallskip

Considering the alternative $\bold{[Alt^u.8]}$, the test statistic of our first procedure denoted by (\texttt{CP1}$^u(\Theta_d^u)$) is thus
\[ \mathcal{T}_{\bbul,\alpha}^{(1)} (N)  = \max_{{\tau} \in \Theta_d^u} \left( \left|S'_{\tau,1}(N)\right| - s'_{N_1,\tau,1} \pa{1 - u^{(1)}_{N_1,\alpha}} \right) \enspace,\]
where $s'_{n,\tau_1, \tau_2}(u)$ is the $u$-quantile of the conditional distribution of $\vert S'_{n,\tau_1, \tau_2}(N) \vert$ given $N_1=n$ under $\hzero$ and $u^{(1)}_{n,\alpha}$ is defined for all $n$ in $\mathbb{N}$ as in \eqref{ualpha_minp_unknown} by 
\begin{equation} \label{simuU_ualpha1}
u^{(1)}_{n,\alpha} = \sup \left\lbrace u \in (0,1),~ \sup_{\lambda_0 \in \mathcal{S}_{0}^{u}[R]} P_{\lambda_0} \left(
  \max_{{\tau} \in \Theta_d^u} 
  \left( \vert S'_{{\tau},1}(N) \vert - s'_{n,{\tau},1} (1-u)  \right) >0 ~\Big\vert N_1=n  \right) \leq \alpha   \right\rbrace \enspace.
\end{equation}
The test statistic of our second procedure denoted by (\texttt{CP2}$^u(\Theta_d^u)$) is
\[  \mathcal{T}_{\bbul,\alpha}^{(2)} (N)  = \max_{{\tau} \in \Theta_d^u}\left(T'_{{\tau},1}(N) - t'_{N_1,{\tau},1} \left( 1-u^{(2)}_{N_1,\alpha} \right) \right) \enspace,\]  where $t'_{n,\tau_1, \tau_2}(u)$ is the $u$-quantile of the conditional distribution of $T'_{\tau_1, \tau_2}(N) $ given $N_1=n$ under $\hzero$ and 
$u^{(2)}_{n,\alpha}$ is defined for all $n$ in $\mathbb{N}$ by 
\begin{equation} \label{simuU_ualpha2}
u^{(2)}_{n,\alpha} = \sup \left\lbrace u \in (0,1),~ \sup_{\lambda_0 \in \mathcal{S}_{0}^{u}[R]} P_{\lambda_0} \left(
  \max_{{\tau} \in \Theta_d^u}\left(T'_{{\tau},1}(N) - t'_{N_1,{\tau},1} \left( 1-u\right) \right) >0 ~\Big\vert N_1=n  \right) \leq \alpha   \right\rbrace \enspace.
\end{equation}

Then, the null hypothesis $\hzero \ "\lambda\in \calS_0^u[R]"$ is rejected when $ \mathcal{T}_{\bbul,\alpha}^{(1)} (N) >0$ for (\texttt{CP1}$^u(\Theta_d^u)$), and when $\mathcal{T}_{\bbul,\alpha}^{(2)} (N)  >0$ for (\texttt{CP2}$^u(\Theta_d^u)$). 

\smallskip

As in the known baseline case, we have also considered the same tests, but replacing the dyadic set $\Theta_d^u$ by the regular set
\[\Theta_r^u=\left\lbrace \frac{k}{10},~k\in\{1,\ldots,9\}  \right\rbrace     \enspace.\]
The corresponding testing procedures are then denoted  by (\texttt{CP1}$^u(\Theta_r^u)$) and  (\texttt{CP2}$^u(\Theta_r^u)$)

\smallskip

The quantities $u^{(1)}_{n,\alpha},$ $u^{(2)}_{n,\alpha},$ $s'_{n,{\tau},1} (1- u^{(1)}_{n,\alpha}  )$ and $t'_{n,{\tau},1} (1- u^{(2)}_{n,\alpha})$
have been estimated  by Monte Carlo methods based on the simulation of 200 000 samples of $n$ i.i.d. random variables with uniform distribution  on $[0,1]$. These samples were used to estimate the distribution of  $ |S'_{{\tau},1}(N)| $ and $ T'_{{\tau},1}(N) $  given $N_1=n$, and therefore to approximate the conditional probabilities occurring in \eqref{simuU_ualpha1} and \eqref{simuU_ualpha2}. The approximations of $u^{(1)}_{n,\alpha}$  and $u^{(2)}_{n,\alpha}$ were obtained by dichotomy.

\subsubsection{Detection of a transitory change or bump}\label{simuTC_Unknown}

Let us consider the discrete sets
$$\Theta_1=\set{\pa{\frac{k}{100},\frac{k+k'}{100}},~ k \in \lbrace 0,\ldots,99 \rbrace,~k' \in \lbrace 1,\ldots,100-k \rbrace}\enspace,$$
and
$$\Theta_2=\set{\pa{\frac{k}{22}, \frac{k+k'}{22}},~ k \in \lbrace 0,\ldots,21 \rbrace,~k' \in \lbrace 1,\ldots, 22-k \rbrace,~(k,k') \neq (0,22)\rbrace}\enspace.$$

Considering the alternative $\bold{[Alt^u.10]}$,  the test statistic of our first procedure denoted by (\texttt{TC1}$^u$) is
\[ \mathcal{T}_{\bbul,\alpha}^{(3)} (N) = \max_{(\tau,\tau')\in \Theta_1} \left(\abs{ S'_{\tau,\tau'}(N) } - s'_{N_1,\tau,\tau'}\left( 1-u^{(3)}_{N_1,\alpha} \right)    \right) \enspace,\]
where $u^{(3)}_{n,\alpha}$ is defined for all $n$ in $\mathbb{N}$ by
\begin{align} \label{simuU_ualpha3}
u^{(3)}_{n,\alpha} &= \sup \left\lbrace u \in (0,1),~ \sup_{\lambda_0 \in \mathcal{S}_{0}^{u}[R]} P_{\lambda_0} \left(
  \max_{(\tau,\tau')\in \Theta_1} \left(\abs{ S'_{\tau,\tau'}(N) } - s'_{N_1,\tau,\tau'}\left( 1-u \right)    \right) >0 ~\Big\vert N_1=n  \right)  \leq \alpha  \right\rbrace \enspace,
\end{align} 
while the test statistic of our second procedure denoted by (\texttt{TC2}$^u$) is
\[\mathcal{T}_{\bbul,\alpha}^{(4)} (N) =\max_{(\tau,\tau')\in \Theta_2} \left( T'_{\tau,\tau'}(N) - t'_{N_1,\tau,\tau'}\left( 1-u^{(4)}_{N_1,\alpha} \right) \right) \enspace,\]
where $u^{(4)}_{n,\alpha}$ is defined for all $n \in \mathbb{N}$ by 
\begin{align} \label{simuU_ualpha4}
u^{(4)}_{n,\alpha} = \sup \left\lbrace u \in (0,1), \sup_{\lambda_0 \in \mathcal{S}_{0}^{u}[R]} P_{\lambda_0} \left(
\max_{(\tau,\tau')\in \Theta_2} \left( T'_{\tau,\tau'}(N) - t'_{N_1,\tau,\tau'}\left( 1-u \right) \right)  >0 ~\Big\vert N_1=n  \right) \leq \alpha  \right\rbrace \enspace. 
\end{align}

The null hypothesis $\hzero \ "\lambda\in \calS_0^u[R]"$ is rejected when $\mathcal{T}_{\bbul,\alpha}^{(3)} (N) >0$ for (\texttt{TC1$^u$}), and when $\mathcal{T}_{\bbul,\alpha}^{(4)} (N) >0$ for (\texttt{TC2$^u$}).
The quantities $u^{(3)}_{n,\alpha}$, $s'_{n,\tau,\tau'} ( 1-u^{(3)}_{n,\alpha} )$, $u^{(4)}_{n,\alpha}$ and $t'_{n,\tau,\tau'} ( 1-u^{(4)}_{n,\alpha} )$ have been estimated by Monte Carlo methods based on the simulation of 200 000 independent samples of $\abs{ S'_{\tau,\tau'}(N) }$ and 
 of $ T'_{\tau,\tau'}(N) $ given $N_1=n$ under $\hzero$, obtained from the simulation of 200 000 samples of $n$ i.i.d. random variables uniformly distributed on $[0,1]$. These samples have been used to estimate the distribution of $\abs{ S'_{\tau,\tau'}(N) }$ and  $ T'_{\tau,\tau'}(N) $ given $N_1=n$, and to approximate the conditional probabilities occurring in \eqref{simuU_ualpha3} and \eqref{simuU_ualpha4} . The approximations of $u^{(3)}_{n,\alpha}$ and $u^{(4)}_{n,\alpha}$ have been obtained by dichotomy.
 
 \subsubsection{Simulation results}

We compare the tests (\texttt{La}) and (\texttt{Z}) with (\texttt{CP1}$^{u}(\Theta_d^u)$),  (\texttt{CP2}$^{u}(\Theta_d^u)$), (\texttt{CP1}$^{u}(\Theta_r^u)$), and (\texttt{CP2}$^{u}(\Theta_r^u)$) when addressing the jump detection problem (described in \ref{simuCP_Unknown}), and with  (\texttt{TC1$^u$}) and (\texttt{TC2$^u$})  when addressing the bump detection problem (described in \ref{simuTC_Unknown}).

\paragraph*{Estimated sizes}
We first study the size of each test by simulating 5 000 independent homogeneous Poisson processes of intensity $\lambda_0=1$ w.r.t. $\Lambda$ on $[0,1].$ The probabilities of first kind error of all the considered tests have been estimated by the number of rejections divided by 5 000. The results are given in Table \ref{size*}.

\begin{table}[h!]
\caption{Estimated sizes}
\label{size}
\begin{center}
\begin{tabular}{cc}
\hline
  \texttt{La}  & \texttt{Z}\\
 0.052 &0.049 \\
 \hline
 (\texttt{CP1}$^{u}(\Theta_d^u)$) &  (\texttt{CP2}$^{u}(\Theta_d^u)$) \\
 0.045 & 0.045\\
 \hline
 (\texttt{CP1}$^{u}(\Theta_r^u)$) &  (\texttt{CP2}$^{u}(\Theta_r^u)$) \\
0.049  &0.047 \\  
\hline
\end{tabular}
\end{center}
\end{table}

Notice again that the estimated sizes of our tests always remain below the target $0.05$, as expected from the definitions of $u^{(1)}_{n,\alpha}$, $u^{(2)}_{n,\alpha}$, $u^{(3)}_{n,\alpha}$ and $u^{(4)}_{n,\alpha}$: the Monte Carlo estimation procedure does not affect this first kind error rate control property.

\paragraph*{Estimated powers}
For both testing problems, we study the estimated power of each test under various alternatives. 

\smallskip

For the jump detection problem, we consider the same alternative intensities $\lambda_{\tau,\delta}$ as in the known baseline intensity case, defined for all $t$ in $[0,1]$ by 
\eqref{simu_alt_cp}, but with $\tau$ varying in $\{0.05, 0.1,0.2,0.5,0.8,0.9,0.95\}$.

 For each alternative, we have simulated 1 000  independent inhomogeneous Poisson processes with intensity $\lambda_{\tau,\delta}$ w.r.t. $\Lambda$ on $[0,1],$ and the powers have been estimated for each test by the number of rejections divided by 1 000. The results are gathered in Tables \ref{cp0.05}-\ref{cp0.95}.

  \begin{table}[h!]
\caption{Estimated probability of detecting a jump with $\tau=0.05$}
\label{cp0.05}
\begin{center}
\begin{tabular}{ccccccccc}
\hline
$\delta=$ & -0.8  & -0.6 & -0.4 & -0.2 & 0.2 & 0.4 & 0.6 & 0.8 \\
\hline
\hline
\texttt{La} &0.32 &0.11 &0.07 &0.05 &0.06  &0.06 &0.06 &0.08\\
\texttt{Z} & 0.58  &0.26 &0.13 &0.07 &0.06 &0.07 &0.10 &0.14  \\
\hline
\texttt{CP1}$^{u}(\Theta_d^u)$  &0.59 &0.30 &0.13 &0.08 &0.04 &0.04 &0.05 &0.05 \\
\texttt{CP2}$^{u}(\Theta_d^u)$  &0.59 &0.30 &0.13 &0.08 &0.05 &0.04 &0.05 &0.07\\
\hline
\texttt{CP1}$^{u}(\Theta_r^u)$ &0.39 &0.17 &0.08 &0.06 &0.04 &0.04 &0.05 &0.07 \\
\texttt{CP2}$^{u}(\Theta_r^u)$ &0.39 &0.17 &0.08 &0.06 &0.04 &0.04 &0.05 &0.08\\
\hline
\end{tabular}
\end{center}
\end{table}

 \begin{table}[h!]
\caption{Estimated probability of detecting a jump with $\tau=0.1$}
\label{cp0.1}
\begin{center}
\begin{tabular}{ccccccccc}
\hline
$\delta=$ & -0.8  & -0.6 & -0.4 & -0.2 & 0.2 & 0.4 & 0.6 & 0.8 \\
\hline
\hline
\texttt{La} &0.64 &0.28 &0.14 &0.06 &0.06  &0.09 &0.09 &0.15\\
\texttt{Z} & 0.81  &0.47 &0.21 &0.09 &0.06 &0.12 &0.17 &0.31  \\
\hline
\texttt{CP1}$^{u}(\Theta_d^u)$ &0.85 &0.47 &0.19 &0.07 &0.04 &0.06 &0.10 &0.16 \\
\texttt{CP2}$^{u}(\Theta_d^u)$  &0.85 &0.47 &0.18 &0.07 &0.04 &0.07 &0.11 &0.20\\
\hline
\texttt{CP1}$^{u}(\Theta_r^u)$ &0.87 &0.48 &0.19 &0.07 &0.06 &0.08 &0.13 &0.24 \\
\texttt{CP2}$^{u}(\Theta_r^u)$  &0.87 &0.47 &0.18 &0.07 &0.06 &0.08 &0.15 &0.27\\
\hline
\end{tabular}
\end{center}
\end{table}

 \begin{table}[h!]
\caption{Estimated probability of detecting a jump with $\tau=0.2$}
\label{cp0.2}
\begin{center}
\begin{tabular}{ccccccccc}
\hline
$\delta=$ & -0.8  & -0.6 & -0.4 & -0.2 & 0.2 & 0.4 & 0.6 & 0.8 \\
\hline
\hline
\texttt{La} &0.95 &0.64 &0.30 &0.10 &0.08  &0.16 &0.27 &0.42\\
\texttt{Z} & 0.95  &0.67 &0.32 &0.11 &0.09 &0.22 &0.39 &0.58  \\
\hline
\texttt{CP1}$^{u}(\Theta_d^u)$ &0.96 &0.64 &0.31 &0.09 &0.05 &0.11 &0.18 &0.38 \\
\texttt{CP2}$^{u}(\Theta_d^u)$ &0.96 &0.63 &0.29 &0.09 &0.06 &0.12 &0.22 &0.41\\
\hline
\texttt{CP1}$^{u}(\Theta_r^u)$ &0.98 &0.72 &0.35 &0.10 &0.08 &0.18 &0.31 &0.53 \\
\texttt{CP2}$^{u}(\Theta_r^u)$ &0.98 &0.70 &0.33 &0.10 &0.08 &0.19 &0.33 &0.55\\
\hline
\end{tabular}
\end{center}
\end{table}

 \begin{table}[h!]
\caption{Estimated probability of detecting a jump with $\tau=0.5$}
\label{cp0.5}
\begin{center}
\begin{tabular}{ccccccccc}
\hline
$\delta=$ & -0.8  & -0.6 & -0.4 & -0.2 & 0.2 & 0.4 & 0.6 & 0.8 \\
\hline
\hline
\texttt{La} &1 &0.90 &0.48 &0.15 &0.14 &0.37 &0.62 &0.83 \\
\texttt{Z} &0.92 &0.63 &0.31 &0.10 &0.11 &0.29 &0.48 &0.69\\
\hline
\texttt{CP1}$^{u}(\Theta_d^u)$&1 &0.86 &0.36 &0.10 &0.09 &0.29 &0.55 &0.77\\
\texttt{CP2}$^{u}(\Theta_d^u)$ &1 &0.86 &0.36 &0.10 &0.08 &0.30 &0.55 &0.78\\
\hline
\texttt{CP1}$^{u}(\Theta_r^u)$ &1 &0.89 &0.43 &0.13 &0.10 &0.34 &0.62 &0.82\\
\texttt{CP2}$^{u}(\Theta_r^u)$ &1 &0.90 &0.43 &0.13 &0.10 &0.33 &0.62 &0.82\\
\hline
\end{tabular}
\end{center}
\end{table}

 \begin{table}[h!]
\caption{Estimated probability of detecting a jump with $\tau=0.8$}
\label{cp0.8}
\begin{center}
\begin{tabular}{ccccccccc}
\hline
$\delta=$ & -0.8  & -0.6 & -0.4 & -0.2 & 0.2 & 0.4 & 0.6 & 0.8 \\
\hline
\hline
\texttt{La} &0.73 &0.42 &0.21 &0.08 &0.09 &0.20 &0.33 &0.54\\
\texttt{Z} &0.31 &0.17 &0.10 &0.06 &0.06 &0.13 &0.17 &0.31 \\
\hline
\texttt{CP1}$^{u}(\Theta_d^u)$& 0.78 &0.36 &0.14 &0.07 &0.08 &0.21 &0.34 &0.57 \\
\texttt{CP2}$^{u}(\Theta_d^u)$ &0.82 &0.41 &0.15 &0.08 &0.08 &0.20 &0.33 &0.56 \\
\hline
\texttt{CP1}$^{u}(\Theta_r^u)$ & 0.93 &0.53 &0.20 &0.08 &0.09 &0.27 &0.45 &0.68 \\
\texttt{CP2}$^{u}(\Theta_r^u)$ &0.95 &0.57 &0.23 &0.08 &0.09 &0.24 &0.43 &0.65 \\
\hline
\end{tabular}
\end{center}
\end{table}

 \begin{table}[h!]
 \caption{Estimated probability of detecting a jump with $\tau=0.9$}
\label{cp0.9}
\begin{center}
\begin{tabular}{ccccccccc}
\hline
$\delta=$ & -0.8  & -0.6 & -0.4 & -0.2 & 0.2 & 0.4 & 0.6 & 0.8 \\
\hline
\hline
\texttt{La} &0.25 &0.15 &0.10 &0.06 &0.07 &0.11 &0.16 &0.22\\
\texttt{Z} &0.10 &0.09 &0.06 &0.05 &0.06 &0.08 &0.09 &0.11 \\
\hline
\texttt{CP1}$^{u}(\Theta_d^u)$  & 0.31 &0.14 &0.09 &0.05 &0.07 &0.14 &0.24 &0.40 \\
\texttt{CP2}$^{u}(\Theta_d^u)$  &0.41 &0.19 &0.10 &0.05 &0.07 &0.14 &0.23 &0.37 \\
\hline
\texttt{CP1}$^{u}(\Theta_r^u)$  & 0.47 &0.20 &0.10 &0.05 &0.07 &0.15 &0.26 &0.42 \\
\texttt{CP2}$^{u}(\Theta_r^u)$  &0.60 &0.26 &0.11 &0.05 &0.06 &0.14 &0.25 &0.40 \\
\hline
\end{tabular}
\end{center}
\end{table}

 \begin{table}[h!]
\caption{Estimated probability of detecting a jump with $\tau=0.95$}
\label{cp0.95}
\begin{center}
\begin{tabular}{ccccccccc}
\hline
$\delta=$ & -0.8  & -0.6 & -0.4 & -0.2 & 0.2 & 0.4 & 0.6 & 0.8 \\
\hline
\hline
\texttt{La} &0.09 &0.06 &0.06 &0.05 &0.05 &0.07 &0.10 &0.10\\
\texttt{Z} &0.07 &0.05 &0.05 &0.04 &0.05 &0.05 &0.07 &0.07 \\
\hline
\texttt{CP1}$^{u}(\Theta_d^u)$ & 0.07 &0.05 &0.04 &0.05 &0.06 &0.09 &0.15 &0.25 \\
\texttt{CP2}$^{u}(\Theta_d^u)$ &0.12 &0.08 &0.05 &0.05 &0.06 &0.08 &0.14 &0.23 \\
\hline
\texttt{CP1}$^{u}(\Theta_r^u)$& 0.08 &0.05 &0.06 &0.06 &0.05 &0.05 &0.11 &0.14 \\
\texttt{CP2}$^{u}(\Theta_r^u)$ &0.10 &0.07 &0.06 &0.06 &0.05 &0.05 &0.10 &0.12 \\
\hline
\end{tabular}
\end{center}
\end{table}

Concerning the bump detection problem, we have considered the same alternative intensities $\lambda_{\tau,\ell,\delta}$ as in the known baseline intensity case. For each alternative, we have simulated 1 000 independent inhomogeneous Poisson processes with intensity $\lambda_{\tau,\ell,\delta}$ w.r.t. $\Lambda$ on $[0,1]$. The powers have been estimated for each test by the number of rejections divided by 1 000, giving the results presented in Tables  \ref{tc0.2-0.1}-\ref{tc0.5-0.4}.

 \begin{table}[h]
\caption{Estimated probability of detecting a bump with $\tau=0.2$ and $\ell=0.1$}
\label{tc0.2-0.1}
\begin{center}
\begin{tabular}{ccccccccc}
\hline
$\delta=$ & -0.8  & -0.6 & -0.4 & -0.2 & 0.2 & 0.4 & 0.6 & 0.8 \\
\hline
\hline
\texttt{La} &0.10  &0.08  &0.08 &0.06 &0.05 &0.07 &0.07 &0.11 \\
\texttt{Z} &0.08 &0.07 &0.06 &0.06 &0.05 &0.05 &0.04 &0.04 \\
\hline
\texttt{TC1$^u$} &0.17 &0.12 &0.08 &0.07 &0.07 &0.10 &0.16 &0.24 \\
\texttt{TC2$^u$} &0.20 &0.12 &0.09 &0.05 &0.05 &0.08 &0.14 &0.20 \\
\hline
\end{tabular}
\end{center}
\end{table}

\begin{table}[h]
\caption{Estimated probability of detecting a bump with $\tau=0.2$ and $\ell=0.2$}
\label{tc0.2-0.2}
\begin{center}
\begin{tabular}{ccccccccc}
\hline
$\delta=$ & -0.8  & -0.6 & -0.4 & -0.2 & 0.2 & 0.4 & 0.6 & 0.8 \\
\hline
\hline
\texttt{La} &0.25  &0.16  &0.09 &0.05 &0.05 &0.08 &0.12 &0.18 \\
\texttt{Z} &0.12 &0.08 &0.08 &0.05 &0.05 &0.05 &0.05 &0.05 \\
\hline
\texttt{TC1$^u$} &0.68 &0.29 &0.13 &0.09 &0.07 &0.14 &0.29 &0.47 \\
\texttt{TC2$^u$} &0.74 &0.34 &0.15 &0.07 &0.07 &0.15 &0.29 &0.46 \\
\hline
\end{tabular}
\end{center}
\end{table}

 \begin{table}[h]
\caption{Estimated probability of detecting a bump with $\tau=0.2$ and $\ell=0.4$}
\label{tc0.2-0.4}
\begin{center}
\begin{tabular}{ccccccccc}
\hline
$\delta=$ & -0.8  & -0.6 & -0.4 & -0.2 & 0.2 & 0.4 & 0.6 & 0.8 \\
\hline
\hline
\texttt{La} &0.27 &0.19 &0.12 &0.07 &0.04 &0.08 &0.09 &0.15 \\
\texttt{Z} &0.09 &0.08 &0.06 &0.05 &0.04 &0.04 &0.03 &0.03 \\
\hline
\texttt{TC1$^u$} &0.99 &0.69 &0.28 &0.11 &0.10 &0.22 &0.46 &0.67 \\
\texttt{TC2$^u$} &0.99 &0.71 &0.28 &0.10 &0.09 &0.20 &0.45 &0.67 \\
\hline
\end{tabular}
\end{center}
\end{table}

\begin{table}[h]
\caption{Estimated probability of detecting a bump with $\tau=0.2$ and $\ell=0.7$}
 \label{tc0.2-0.7}
 \begin{center}
\begin{tabular}{ccccccccc}
\hline
$\delta=$ & -0.8  & -0.6 & -0.4 & -0.2 & 0.2 & 0.4 & 0.6 & 0.8 \\
\hline
\hline
\texttt{La} &0.34 &0.19 &0.13 &0.08 &0.05 &0.06 &0.06 &0.08 \\
\texttt{Z} &0.66 &0.39 &0.20 &0.08 &0.08 &0.11 &0.19 &0.29 \\
\hline
\texttt{TC1$^u$} &0.98 &0.75 &0.32 &0.11 &0.09 &0.15 &0.26 &0.48 \\
\texttt{TC2$^u$} &0.98 &0.71 &0.30 &0.10 &0.08 &0.17 &0.31 &0.51  \\
\hline
\end{tabular}
\end{center}
\end{table}

\begin{table}[h]
\caption{Estimated probability of detecting a bump with $\tau=0.5$ and $\ell=0.4$} 
\label{tc0.5-0.4}
\begin{center}
\begin{tabular}{ccccccccc}
\hline
$\delta=$ & -0.8  & -0.6 & -0.4 & -0.2 & 0.2 & 0.4 & 0.6 & 0.8 \\
\hline
\hline
\texttt{La} &0.78 &0.47 &0.22 &0.09 &0.08 &0.17 &0.31 &0.51 \\
\texttt{Z} &0.63 &0.36 &0.17 &0.08 &0.07 &0.17 &0.31 &0.49 \\
\hline
\texttt{TC1$^u$} &0.98 &0.71 &0.27 &0.11 &0.09 &0.22 &0.44 &0.68 \\
\texttt{TC2$^u$} &0.99 &0.74 &0.29 &0.10 &0.08 &0.23 &0.45 &0.69 \\
\hline
\end{tabular}
\end{center}
\end{table}

\subparagraph*{Comments}
\begin{enumerate}
\item Considering the single change-point or jump detection problem, it first arises that among the (\texttt{La}) and (\texttt{Z}) procedures, neither is preferable to use: the Laplace  and $Z$ tests can have very low powers depending on when the change occurs. One can notice that their performances are significantly smaller than the ones of our procedures (\texttt{CP1}$^{u}$) and (\texttt{CP2}$^{u}$) when the jump occurs near to one, while the estimated powers remain comparable in the other cases.
Moreover, it is worthwhile to note again that the jump detection problem in a Poisson process is not a symmetric problem. Indeed, it is easier to detect large negative jumps occurring close to zero than close to one, and easier to detect large positive jumps occurring close to one than close to zero. 
\item Considering the transitory change or bump detection problem, our procedures have estimated powers significantly larger in all cases than the Laplace and $Z$ tests. Moreover, we have to mention that complementary experiments (omitted in this study) showed that the estimated powers of (\texttt{TC1$^u$}) and (\texttt{TC2$^u$}) are equivalent for a same value of change length whatever the change location. This assessment is not true for the procedures (\texttt{La}) and (\texttt{Z}) except for $\ell =0.1$, for which one observes better powers for $\ell =0.4$ and $\tau=0.5$ than for $\ell =0.4$ and $\tau=0.2$.
\item Among our testing procedures, the procedures (\texttt{CP2}$^{u}$) and (\texttt{TC2$^u$}), based on the quadratic statistics, are slightly more powerful than (\texttt{CP1}$^{u}$) and (\texttt{TC1$^u$}) based on the linear statistics for some negative jumps, and as expected, the aggregated tests based on dyadic sets are significantly more efficient that the ones based on regular sets when the change occurs near $0$ or $1$.
\item The comparison of the simulated powers of the different testing procedures confirm again the intuition that detecting a bump is harder than detecting a jump. The simulation study also highlights that it is substantially easier to detect a jump or a bump with negative change height than with positive change height. This phenomenon can still be explained by the fact that the significant parameter to evaluate the detectability of a (transitory or not) change is not the change height itself but the ratio between the minimum and the maximum values of the intensity.

\end{enumerate}

\section{Proofs of the main results}\label{Sec:Proofs}

\subsection*{Notation} As explained in the introduction, the main tools to prove our nonasymptotic minimax separation rates upper bounds are exponential inequalities. Many of these exponential inequalities  involve the function $g$ defined by:
\begin{equation}\label{defg}
g(x)=(1+x)\log(1+x)-x\quad  \forall x > 0\enspace,
\end{equation}
and its inverse function $g^{-1}$, which can be upper bounded as follows:
\begin{equation}\label{UBginv}
g^{-1}(x)\leq 2x/3 + \sqrt{2x} \quad \forall x>0\enspace.
\end{equation}

\subsection{Proof of Proposition \ref{bNP1}}
The first statement of Proposition \ref{bNP1} directly results from the Neyman-Pearson fundamental lemma and Girsanov's lemma recalled in Lemma \ref{lemmegirsanov}.

Assume that $\delta^* >0$ and notice that the assumption (\ref{cond_alt1}) leads to
 \begin{equation}\label{bNPpreuve3}
   \delta^{*}\ell^* L\geq \sqrt{\frac{(\lambda_{0}+ \delta^{*}) \ell^{*} L}{\beta}} +\sqrt{\frac{\lambda_{0} \ell^{*} L}{\alpha}}\enspace.
 \end{equation}
From \eqref{bNPpreuve3}, the quantile bound \eqref{controlquantilePb1} and the Bienayme-Chebyshev inequality, we obtain
 \begin{align*}
  P_{\lambda}(\phi_{1,\alpha}^{+}(N)=0) &\leq P_{\lambda}\left(N(\tau^{*},\tau^{*}+\ell^{*}]  \leq \sqrt{\frac{\lambda_{0}\ell^{*} L}{\alpha}} +\lambda_{0} \ell^{*} L\right)\\
  &\leq P_{\lambda} \left(  N(\tau^{*},\tau^{*} +\ell^{*}] \leq  (\lambda_{0}^*+ \delta^* ) \ell^{*}L  - \sqrt{\frac{(\lambda_{0}+ \delta^{*}) \ell^{*} L}{\beta}} \right)\\
  &\leq \beta \enspace.
 \end{align*}
 Assume now that $-\lambda_{0}^*<\delta^*<0$ and notice that the assumption \eqref{cond_alt1} leads to
 \begin{equation}\label{bControltemp2}
  \delta^{*}\ell^* L \leq-\sqrt{\frac{(\lambda_{0}+ \delta^{*})\ell^{*} L}{\beta}} -\sqrt{\frac{\lambda_{0}\ell^{*} L}{\alpha}}\enspace.
 \end{equation}
As above, using \eqref{bControltemp2},  \eqref{controlquantilePb1} and the Bienayme-Chebyshev inequality, we obtain
 \begin{align*}
  P_{\lambda}(\phi_{1,\alpha}^{-}(N)=0)  &\leq P_{\lambda}\left(N(\tau^{*},\tau^{*} +\ell^{*}] \geq -\sqrt{\frac{\lambda_{0}\ell^{*} L}{\alpha}} +\lambda_{0} \ell^{*} L \right)\\
  &\leq P_{\lambda} \left(  N(\tau^*,\tau^* +\ell^{*}] \geq  (\lambda_{0}+ \delta^{*} )\ell^{*} L  + \sqrt{\frac{(\lambda_{0}+ \delta^{*})\ell^{*}L}{\beta}} \right)\\
  &\leq \beta\enspace.
 \end{align*}

\subsection{Proof of Proposition \ref{LBalt2}}

Let $C_{\alpha, \beta}=1+4(1- \alpha -\beta)^2$. Let us introduce for $r>0$ the Poisson intensity $\lambda_r$ defined by 
\[\lambda_r(t) = \lambda_{0} + \frac{r}{\sqrt{\ell^*}}\un{(\tau^{*}, \tau^{*} +\ell^*]}(t) \textrm{ for all $t$ in [0,1]}\enspace.\]
Notice that $\lambda_r$ belongs to $\pa{\calS_{\bbul,\tau^{*},\ell^*}[\lambda_0]}_r= \lbrace \lambda \in \calS_{\bbul,\tau^{*},\ell^*}[\lambda_0],~d_2(\lambda, \calS_{0}[\lambda_0]) \geq r   \rbrace$, as defined by Lemma~\ref{mSR}. We get from 
 Lemma \ref{lemmegirsanov}  and Lemma \ref{momentPoisson} that
 \[E_{\lambda_0} \left[\left( \frac{d P_{\lambda_r}}{dP_{\lambda_0}} \right)^{2}(N)\right] = \exp \pa{\frac{r^2 L}{\lambda_{0}} }\enspace. \]
Choosing $r=(\lambda_{0} \log C_{\alpha, \beta}/L)^{1/2}$ then leads to $E_{\lambda_0} \left[\left(d P_{\lambda_r}/{dP_{\lambda_0}} \right)^{2}(N)\right]= C_{\alpha,\beta}$, and thanks to Lemma \ref{lemmebayesien}, we conclude that $\rho_\alpha\pa{\pa{\calS_{\bbul,\tau^{*},\ell^*}[\lambda_0]}_r}\geq \beta$ and $\mSRab\pa{\calS_{\bbul,\tau^{*},\ell^*}[\lambda_0]} \geq r.$

 \subsection{Proof of Proposition \ref{UBalt2}} 

The first statement of the proposition is straightforward.

For the test $\phi_{2,\alpha}^{(1)}$, let us consider first $\lambda=\lambda_{0} + \delta \mathds{1}_{(\tau^*,\tau^*+\ell^*]}$ in $ \calS_{\bbul, \tau^*, \ell^*}[\lambda_0]$ with $\delta>0$.
 From the quantile bound \eqref{controlquantilePb1}, one deduces that
\begin{align*}
 P_{\lambda}(\phi_{2,\alpha}^{(1)}=0)
 &\leq P_{\lambda} \pa{N(\tau^*,\tau^* + \ell^*] \leq  \sqrt{\frac{ \lambda_0\ell^* L}{\alpha_1}} +\lambda_0\ell^* L}\enspace,\\
 &=P_{\lambda}\pa{N(\tau^*,\tau^* + \ell^*]  - \pa{\lambda_0+ \delta}\ell^* L \leq  - \delta \ell^* L + \sqrt{\frac{ \lambda_0\ell^* L}{\alpha_1}}}\enspace.
\end{align*}
It remains to find a condition on $d_2\pa{\lambda,\calS_{0}[\lambda_0]}$ which will guarantee that
\begin{equation}\label{condition1}
- \delta \ell^* L + \sqrt{\frac{ \lambda_0\ell^* L}{\alpha_1}} \leq -\sqrt{\frac{(\lambda_0+ \delta)\ell^* L}{\beta}}\enspace,
\end{equation}
so that $P_{\lambda}(\phi_{2,\alpha}^{(1)}=0)\leq \beta$ thanks to the Bienayme-Chebyshev inequality.

Let us assume for instance that
\[d_2\pa{\lambda,\calS_{0}[\lambda_0]}\geq 2\sqrt{\frac{\lambda_0}{L}} \pa{\frac{1}{\sqrt{\beta}}+\frac{1}{\sqrt{\alpha_1}}}+\frac{1}{\beta \sqrt{\ell^*} L}\enspace.\]
Since $d_2\pa{\lambda,\calS_{0}[\lambda_0]}=\delta\sqrt{\ell^*}$,  this implies
\[\delta\sqrt{\ell^*}\geq 2\sqrt{\frac{\lambda_0}{L}} \pa{\frac{1}{\sqrt{\beta}}+\frac{1}{\sqrt{\alpha_1}}}+\frac{1}{\beta \sqrt{\ell^*} L}\enspace,\]
whereby
\[\delta\sqrt{\ell^*}-\pa{\frac{\delta\sqrt{\ell^*}}{2}+\frac{1}{2\beta \sqrt{\ell^*} L}}\geq \sqrt{\frac{\lambda_0}{L}} \pa{\frac{1}{\sqrt{\beta}}+\frac{1}{\sqrt{\alpha_1}}}\enspace.\]
Using the basic inequality $\sqrt{ab}\leq (a+b)/2$ then leads to
\[\delta\sqrt{\ell^*}-\sqrt{\frac{\delta}{\beta L}}\geq \sqrt{\frac{\lambda_0}{L}} \pa{\frac{1}{\sqrt{\beta}}+\frac{1}{\sqrt{\alpha_1}}}\enspace,\]
and \eqref{condition1} conveniently follows. 

Let us consider now $\lambda=\lambda_{0} + \delta \mathds{1}_{(\tau^*,\tau^*+\ell^*]}$ in $\calS_{\bbul, \tau^*, \ell^*}[\lambda_0]$ with $\delta$ in $(-\lambda_0,0)$.
 From the quantile bound \eqref{controlquantilePb1} again, one deduces that
\begin{align*}
 P_{\lambda}(\phi_{2,\alpha}^{(1)}=0)
 &\leq P_{\lambda} \pa{N(\tau^*, \tau^* +\ell^*] \geq  -\sqrt{\frac{\lambda_0\ell^* L}{\alpha_2}} +\lambda_0\ell^* L}\\
  &=P_{\lambda}\pa{N(\tau^*,\tau^* +\ell^*]  - \pa{\lambda_0+ \delta}\ell^* L\geq  - \delta \ell^* L - \sqrt{\frac{\lambda_0 \ell^* L}{\alpha_2}}} \enspace.
 \end{align*}
As above, it remains to find a condition on $d_2\pa{\lambda,\calS_{0}[\lambda_0]}$ which will guarantee that
\begin{equation}\label{condition2}
- \delta \ell^* L - \sqrt{\frac{\lambda_0\ell^* L}{\alpha_2}} \geq \sqrt{\frac{(\lambda_0+ \delta)\ell^* L}{\beta}}\enspace,
\end{equation}
so that $P_{\lambda}(\phi_{2,\alpha}^{(1)}=0)\leq \beta$.
Since $d_2\pa{\lambda,\calS_{0}[\lambda_0]}=-\delta\sqrt{\ell^*}$ and $\delta<0$, the following condition suffices
\[d_2\pa{\lambda,\calS_{0}[\lambda_0]}\geq \sqrt{\frac{\lambda_0}{L}} \pa{\frac{1}{\sqrt{\beta}}+\frac{1}{\sqrt{\alpha_2}}}\enspace.\]

Taking $C(\alpha,\beta,\lambda_0,\ell^*)=\textrm{max}\pa{2\sqrt{\lambda_0} \pa{\frac{1}{\sqrt{\beta}}+\frac{1}{\sqrt{\alpha_1}}}+\frac{1}{\beta \sqrt{\ell^*}}~,~\sqrt{\lambda_0} \pa{\frac{1}{\sqrt{\beta}}+\frac{1}{\sqrt{\alpha_2}}}}$ finally allows to conclude for the test $\phi_{2,\alpha}^{(1)}$.

\smallskip

As for the test $\phi_{2,\alpha}^{(2)}$, let us consider $\lambda=\lambda_{0} + \delta \mathds{1}_{(\tau^*,\tau^*+\ell^*]}$ in $ \calS_{\bbul, \tau^*, \ell^*}[\lambda_0]$.
Since  $T_{\tau^*, \tau^* + \ell^{*}}(N)$ is centered under $\hzero$, $t_{\lambda_0,\tau^*, \tau^*+ \ell^*}(1-\alpha) \leq \sqrt{\textrm{Var}_{\lambda_0}(T_{\tau^*, \tau^* + \ell^{*}}(N))/\alpha}$.
From the variance computation of Lemma \ref{MomentsT} under $\hzero$, we derive the upper bound $t_{\lambda_0,\tau^*, \tau^*+ \ell^*}(1-\alpha) \leq (\lambda_{0}/L)\sqrt{2/\alpha}$.

Moreover, still using Lemma \ref{MomentsT} but under $\hone$ now, one can see that $E_{\lambda}[T_{\tau^*, \tau^* + \ell^{*}}(N)] =   \delta^2 \ell^*$ (recall that $T_{\tau^*, \tau^* + \ell^{*}}(N)$ is an unbiased estimator of $d_2^2(\lambda,\calS_0[\lambda_0]= \delta^2 \ell^*$), and 
\[ \mathrm{Var}_{\lambda}(T_{\tau^*, \tau^* + \ell^{*}}(N)) =  \frac{4 (\lambda_{0}+\delta)\delta^2 \ell^*}{L}  + \frac{2(\lambda_{0} + \delta)^2}{L^2}\enspace.\]
Therefore,
 \begin{align*}
 P_{\lambda}(\phi_{2,\alpha}^{(2)}=0) 
 &=P_{\lambda} \pa{T_{\tau^*, \tau^* + \ell^{*}}(N) \leq t_{\lambda_0,\tau^*, \tau^*+ \ell^*}(1-\alpha)}\enspace,\\
 &\leq P_{\lambda} \pa{T_{\tau^*, \tau^* + \ell^{*}}(N) \leq \frac{\lambda_{0}}{L}\sqrt{\frac{2}{\alpha}}}\enspace,\\
   &\leq P_{\lambda}\pa{T_{\tau^*, \tau^* + \ell^{*}}(N)  - \delta^2 \ell^* \leq  \frac{\lambda_{0}}{L}\sqrt{\frac{2}{\alpha}} - \delta^2 \ell^* } \enspace.
\end{align*}
Assume now that $$d_2\pa{\lambda,\calS_{0}[\lambda_0]}  \geq \frac{C(\alpha,\beta,\lambda_0,\ell^*)}{\sqrt{L}}\enspace,$$ with 
\begin{equation*}
\begin{array}{c}
C(\alpha,\beta,\lambda_0,\ell^*)=\max \left(  \sqrt{3 \lambda_{0} \pa{\sqrt{\frac{2}{\alpha}}   +   \sqrt{\frac{2}{\beta}} }}~,~  \frac{6 \sqrt{ \lambda_{0}}}{\sqrt{\beta}}   +  \frac{3 \sqrt{2}}{\sqrt{\beta \ell^* L}}~,~\frac{36}{\beta \sqrt{\ell^* L}}     \right)\enspace. 
\end{array}
\end{equation*}
 This implies
 \[\delta^2 \ell^* \geq 3 \max \pa{\frac{\lambda_{0}}{L} \pa{\sqrt{\frac{2}{\alpha}} + \sqrt{\frac{2}{\beta}}},~ \vert \delta \vert \sqrt{\ell^*} \pa{\frac{2}{\sqrt{L}}  \sqrt{\frac{ \lambda_{0}}{\beta}} +  \frac{1}{L} \sqrt{\frac{2}{\beta \ell^*}}  },~  \frac{2\vert \delta \vert ^{3/2} {\sqrt{\ell^*}}}{\sqrt{\beta L}}  }\enspace,  \]
 and then
 \[ \delta^2 \ell^* \geq \frac{\lambda_{0}}{L}\sqrt{\frac{2}{\alpha}} +2 \sqrt{\frac{ \lambda_{0}}{\beta L}}\vert \delta \vert \sqrt{\ell^*}  + \sqrt{\frac{2}{\beta}}\pa{\frac{\lambda_{0}}{L} + \frac{1}{L \sqrt{\ell^*}} \vert \delta \vert \sqrt{\ell^*}}+\frac{2\vert \delta \vert ^{3/2} {\sqrt{\ell^*}}}{\sqrt{\beta L}} \enspace,\]
 hence, using $\sqrt{\lambda_{0} + \delta} \leq \sqrt{\lambda_{0}} + \sqrt{\vert \delta \vert}$,
  \begin{equation} \label{bcontrol2Pb1N2}
  \delta^2 \ell^* \geq \frac{\lambda_{0}}{L}\sqrt{\frac{2}{\alpha}} + 2\sqrt{\frac{\lambda_{0} + \delta}{L \beta}}\vert \delta \vert \sqrt{\ell^*} + \sqrt{\frac{2}{\beta}}\frac{\lambda_{0} +\delta}{L}\enspace.
 \end{equation}
Therefore,
 \begin{align*}
 P_{\lambda}(\phi_{2,\alpha}^{(2)}=0) 
  &\leq P_{\lambda}\pa{T_{\tau^*, \tau^* + \ell^{*}}(N)  - \delta^2 \ell^* \leq  \frac{\lambda_{0}}{L}\sqrt{\frac{2}{\alpha}} - \delta^2 \ell^* } \\
 &\leq P_{\lambda} \pa{T_{\tau^*, \tau^* + \ell^{*}}(N) - \delta^2 \ell^* \leq  -\frac{1}{\sqrt{\beta}} \sqrt{\frac{4 (\lambda_{0} + \delta) \delta^2 \ell^*}{L} + \frac{2(\lambda_{0}+ \delta)^2}{L^2}}     }\\
 &\leq P_{\lambda} \pa{T_{\tau^*, \tau^* + \ell^{*}}(N) - E_\lambda\cro{T_{\tau^*, \tau^* + \ell^{*}}(N)} \leq  -\sqrt{\frac{\textrm{Var}_\lambda\pa{T_{\tau^*, \tau^* + \ell^{*}}(N)}}{\beta}}}\\
   &\leq \beta\enspace.
 \end{align*}
This concludes the proof for the test $\phi_{2,\alpha}^{(2)}$.

\subsection{Proof of Proposition \ref{UBalt3}} Let us first give a short proof for the tests $\phi_{3,\alpha}^{(1)+}$ and $\phi_{3,\alpha}^{(1)-}$.

Start by remarking that the first kind error rates control of both tests is straightforward.

Since $g^{-1}(x) \leq 2x/3+ \sqrt{2x}$ for all $x> 0$ (see \eqref{UBginv}), Proposition \ref{bquantile_maxminNbis} leads to
\begin{equation} \label{bDeltaLengthProof1}
 p_{\lambda_0,\ell^*}^+(1- \alpha) \leq \lambda_{0} \ell^{*}L+ 2 \sqrt{  2 \lambda_{0}\log \pa{2/\alpha}L} +4\log \pa{2/\alpha}/3\enspace,
 \end{equation}
and
\begin{equation} \label{bDeltaLengthProof2}
 p_{\lambda_0,\ell^*}^-( \alpha) \geq \lambda_{0}\ell^{*}L  - 2 \sqrt{  2 \lambda_{0} \log \pa{2/\alpha}L} -4\log \pa{2/\alpha}/3\enspace.
  \end{equation}
Let us consider $\lambda=\lambda_0 +\delta^* \un{(\tau,\tau+\ell^*]}$ in $\calS_{\delta^*,\bbul\bbul,\ell^*}$ and assume that 
\begin{equation}\label{cond_alt3+-} 
d_2\pa{\lambda,\calS_{0}[\lambda_0]} \geq \frac{1}{\sqrt{L}} \pa{ 2 \sqrt{ \frac{ 2 \lambda_0\log \pa{2/\alpha}}{\ell^{*}}} + \sqrt{\frac{\lambda_0 + \delta^{*}}{\beta}}} + \frac{ 4 \log \pa{2/\alpha}}{3L\sqrt{\ell^{*}}}\enspace.
\end{equation}

If  $\delta^*>0$, the condition \eqref{cond_alt3+-} yields
\begin{equation} \label{bdistance1DeltaLengthN1}
\delta^*\ell^*L \geq 2 \sqrt{  2\lambda_{0} \log \pa{2/ \alpha}L} +\frac{4}{3} \log \pa{2/\alpha}+ \sqrt{\frac{(\lambda_{0} + \delta^{*})\ell^{*}L}{\beta}}\enspace,
\end{equation} 
which entails
\begin{align*}
P_\lambda \Bigg(&\max_{t \in [0, 1-\ell^*]} N(t,t+\ell^*] \leq p_{\lambda_0,\ell^*}^+(1- \alpha) \Bigg)\\
 &\leq P_\lambda \pa{\max_{t \in [0, 1-\ell^*]} N(t,t+\ell^*]  \leq (\lambda_{0}+\delta^*) \ell^*L - \sqrt{\frac{(\lambda_{0} + \delta^{*})\ell^{*}L}{\beta}} }\\
&\leq  P_{\lambda} \left(  N(\tau,\tau +\ell^*] -(\lambda_{0} + \delta^{*}{})L\ell^* \leq  - \sqrt{\frac{(\lambda_{0}+ \delta^{*}{}) \ell^* L}{\beta}} \right)\\
  &\leq \beta~~\text{with the Bienayme-Chebyshev inequality}\enspace.
\end{align*}
This concludes the proof for $\phi_{3,\alpha}^{(1)+}$.

\smallskip

If  $\delta^*$ belongs to $(-\lambda_{0}, 0)$, the condition \eqref{cond_alt3+-} yields
\begin{equation} \label{bdistance2DeltaLengthN1-2}
-\delta^*\ell^*L \geq 2 \sqrt{  2\lambda_{0} \log \pa{2/ \alpha}L} +\frac{4}{3} \log \pa{2/\alpha}+ \sqrt{\frac{(\lambda_{0} + \delta^{*})\ell^{*}L}{\beta}}\enspace.
\end{equation} 
We get then as above, with \eqref{bDeltaLengthProof2}, \eqref{bdistance2DeltaLengthN1-2} and the Bienayme-Chebyshev inequality,
\begin{align*}
P_\lambda \Bigg(&\min_{t \in [0, 1-\ell^*]} N(t,t+\ell^*] \geq  p_{\lambda_0,\ell^*}^-(\alpha) \Bigg)\\
 &\leq P_\lambda \pa{ \min_{t \in [0, 1-\ell^*]} N(t,t+\ell^*] \geq \lambda_{0} \ell^*L- 2 \sqrt{  2 \lambda_{0} \log \pa{2/\alpha}L} - \frac{4}{3} \log \pa{2/\alpha}}\\
  & \leq P_\lambda \pa{\min_{t \in [0, 1-\ell^*]} N(t,t+\ell^*]  \geq (\lambda_{0}+\delta^*) \ell^*L + \sqrt{\frac{(\lambda_{0} + \delta^{*})\ell^{*}L}{\beta}} }\\
  &\leq P_{\lambda} \left(  N(\tau,\tau +\ell^{*}] -(\lambda_{0} + \delta^{*})\ell^*L \geq   \sqrt{\frac{(\lambda_{0}+ \delta^{*}) \ell^{*}L}{\beta}} \right)\\
 &\leq \beta\enspace.
\end{align*}
This concludes the proof for $\phi_{3,\alpha}^{(1)-}$.

\smallskip

Now, let us turn to the test $\phi_{3/4,\alpha}^{(2)}$. As above, start by remarking that the first kind error rate control of this test straightforwardly follows from a basic union bound:
 \begin{align*}
P_{\lambda_0}\pa{\phi_{3/4,\alpha}^{(2)}(N)=1} &\leq \sum_{k=0}^{\lceil (1-\ell^*) M \rceil -1} P_{\lambda_0} \left(T_{\frac k M, \frac k M +\ell^*}(N) > t_{\frac k M, \frac k M +\ell^*}\left( 1-u_\alpha  \right)  \right) \\
 &\leq \sum_{k=0}^{\lceil (1-\ell^*) M \rceil -1} \frac{\alpha}{\lceil (1-\ell^*) M \rceil}\\
 &\leq \alpha\enspace.
 \end{align*}
 Let $\lambda$ in $\mathcal{S}_{\delta^{*},\bbul\bbul,\ell^{*}}$ such that $\lambda= \lambda_{0}+  \delta^{*} \un{(\tau, \tau + \ell^*]}$ with $\tau$ in $[0,1-\ell^*]$, and  assume that the following holds:
\begin{multline}\label{cond_alt3_N2} 
d_2\pa{\lambda,\calS_{0}[\lambda_0]}\geq
\frac{2}{\sqrt{L}} \max \Bigg(8 \sqrt{\frac{\lambda_{0} + \vert \delta^{*}\vert}{\beta}}~,~\\
 \Bigg(\frac{4\sqrt{2} (\lambda_{0}+ \vert \delta^{*} \vert)}{\sqrt{\beta}}+8 \lambda_{0}\pa{\frac{2}{3}\frac{\log \pa{3/u_\alpha}}{\sqrt{\lambda_0\ell^*L}}  + \sqrt{2  \log \pa{3/u_\alpha }}}^2      \Bigg)^{\frac{1}{2}} \Bigg)\enspace.
\end{multline}
This entails
  \begin{align}\label{cond_alt3_N2_bis}
   d_2^2(\lambda, \mathcal{S}_{0}[\lambda_0]) &\geq \frac{8 \lambda_{0}}{L}\pa{\frac{2}{3}\frac{\log \pa{3/u_\alpha}}{\sqrt{\lambda_0\ell^*L}}  + \sqrt{2  \log \pa{3/u_\alpha }}}^2  +\frac{4\sqrt{2} (\lambda_{0} + \vert \delta^{*} \vert) }{\sqrt{\beta}L}+ \\ &\quad \frac{8 d_2(\lambda,  \mathcal{S}_{0}[\lambda_0])}{\sqrt{L}} \sqrt{\frac{\lambda_{0} + \vert \delta^{*} \vert}{\beta}}  \nonumber\enspace.
\end{align} 
Noticing that
 \[ P_{\lambda}\pa{\phi_{3/4,\alpha}^{(2)}(N)=0} \leq \min_{k \in \lbrace  0,...,\lceil (1-\ell^*) M \rceil -1  \rbrace} P_{\lambda} \pa{T_{\frac k M, \frac k M +\ell^*}(N) \leq t_{\frac k M, \frac k M +\ell^*}( 1-u_\alpha)}\enspace,\]
we only need to exhibit some $k_{\tau}$ in $\lbrace  0,...,\lceil (1-\ell^*) M \rceil-1  \rbrace$ satisfying
 $$P_{\lambda}\pa{T_{\frac{k_{\tau}}{M}, \frac{k_{\tau}}{M} +\ell^*}(N) \leq t_{\frac{k_{\tau}}{M}, \frac{k_{\tau}}{M} +\ell^*}( 1-u_\alpha)} \leq \beta\enspace.$$ 
We set $k_\tau = \lfloor \tau M \rfloor$. Since $0< \tau<1-\ell^*$,  $k_\tau$ actually belongs to $\lbrace  0,...,\lceil (1-\ell^*) M \rceil-1  \rbrace$, and since $M=\lceil 2/\ell^*  \rceil$, $k_\tau/M \leq \tau < k_\tau/M +\ell^*/2$.
Therefore, using Lemma \ref{MomentsT} equation  \eqref{bpreliesperance}, we get on the one hand
$$ E_\lambda \cro{T_{\frac{k_\tau}{M},\frac{k_\tau}{M} +\ell^*}(N)} = {\delta^{*}}^2 \frac{(\ell^*+ k_\tau/M - \tau)^2 }{\ell^*}\geq \frac{{\delta^{*}}^2 \ell^*}{4}\enspace,$$
that is 
\begin{equation} \label{bproofDeltaLength1}
 E_\lambda \cro{T_{\frac{k_\tau}{M},\frac{k_\tau}{M} +\ell^*}(N)}\geq  \frac{d_2^2(\lambda, \mathcal{S}_{0}[\lambda_0]) }{4}\enspace,
 \end{equation}
 and on the other hand with Lemma \ref{MomentsT} equation \eqref{bprelivariance},
 \begin{align} \label{bproofDeltaLength2}
 \mathrm{Var}_\lambda \cro{T_{\frac{k_\tau}{M}, \frac{k_\tau}{M}+\ell^*}(N)} &= \frac{4 {\delta^{*}}^2 \pa{\lambda_0 \ell^* + \delta^{*}(k_\tau/M + \ell^* -\tau) }}{L} \frac{(k_\tau/M + \ell^* -\tau)^2}{{\ell^*}^2} \nonumber\\
&  + \frac{2}{L^2} \frac{\pa{\lambda_0 \ell^{*} + \delta^{*} (k_\tau/M + \ell^*-\tau) }^2}{{\ell^*}^2}
 \nonumber \\
 & \leq \frac{4 (\lambda_{0} + \vert \delta^{*}\vert )}{L} d_2^2(\lambda, \calS_0[\lambda_0]) + \frac{2(\lambda_{0} + \vert \delta^{*} \vert)^2}{L^2}\enspace.
 \end{align} 
 Moreover, Lemma \ref{QuantilesT} entails
 \[t_{\frac{k_\tau}{M}, \frac{k_\tau}{M} +\ell^*}( 1-u_\alpha)\leq 2\lambda_{0}^2 \ell^* \pa{g^{-1} \pa{\frac{\log \pa{3/u_\alpha}}{\lambda_{0} \ell^* L}}}^2\enspace,  \] with $g^{-1}(x) \leq 2x/3+ \sqrt{2x}$ (see \eqref{UBginv}), which implies, with \eqref{cond_alt3_N2_bis}, \eqref{bproofDeltaLength1} and \eqref{bproofDeltaLength2} that
\[ t_{\frac{k_\tau}{M}, \frac{k_\tau}{M} +\ell^*}( 1-u_\alpha)\leq  E_\lambda \cro{T_{\frac{k_\tau}{M}, \frac{k_\tau}{M} +\ell^*}(N)} -\sqrt{\mathrm{Var}_\lambda \cro{T_{\frac{k_\tau}{M}, \frac{k_\tau}{M} +\ell^*}(N)}/\beta}\enspace.\]
 We simply conclude the proof for $\phi_{3/4,\alpha}^{(2)}$ with 
 \begin{align*}
 &P_\lambda\pa{ T_{\frac{k_\tau}{M}, \frac{k_\tau}{M} +\ell^*}(N)\leq t_{\frac{k_\tau}{M}, \frac{k_\tau}{M} +\ell^*}( 1-u_\alpha)}\\
 &\leq  P_\lambda \pa{T_{\frac{k_\tau}{M}, \frac{k_\tau}{M} +\ell^*}(N) -  E_\lambda \cro{T_{\frac{k_\tau}{M}, \frac{k_\tau}{M} +\ell^*}(N)} \leq   -\sqrt{\mathrm{Var}_\lambda \cro{T_{\frac{k_\tau}{M}, \frac{k_\tau}{M} +\ell^*}(N)}/\beta}}\\
 &\leq \beta\enspace.
 \end{align*}

\subsection{Proof of Proposition \ref{UBalt4}}

The control of the first kind error rates of the two tests $\phi_{4,\alpha}^{(1)}$ and $\phi_{3/4,\alpha}^{(2)}$ is straightforward using simple union bounds. 

\medskip

\emph{$(i)$ Control of the second kind error rate of $\phi_{4,\alpha}^{(1)}$.}

\smallskip

Recall from the proof of Proposition \ref{UBalt3}, equations \eqref{bDeltaLengthProof1} and \eqref{bDeltaLengthProof2}, that 
\begin{equation}\label{malpha/2}
\left\{\begin{array}{rll}
p_{\lambda_0,\ell^*}^+(1- \alpha/2) &\leq &\lambda_{0} \ell^{*}L+ 2 \sqrt{  2 \lambda_{0}\log \pa{4/\alpha}L} +4 \log \pa{4/\alpha}/3\enspace,\\
p_{\lambda_0,\ell^*}^-( \alpha/2) &\geq &\lambda_{0}\ell^{*}L  - 2 \sqrt{  2 \lambda_{0} \log \pa{4/\alpha}L} -4 \log \pa{4/\alpha}/3\enspace.\end{array}\right.
\end{equation}

Let us first set $\lambda$ in $\mathcal{S}_{\bbul,\bbul\bbul,\ell^{*}}$ such that $\lambda= \lambda_{0}+  \delta\un{(\tau, \tau + \ell^*]}$ with $\delta>0$ or $\delta$ in $(-\lambda_0,0)$, $\tau$ in $(0,1-\ell^*)$, and  
\begin{multline} \label{bUBhdistance}
d_2(\lambda, \calS_0[\lambda_0]) \geq \frac{2\sqrt{\lambda_0}}{\sqrt{L}}\pa{\frac{1}{\sqrt{\beta}}+2\sqrt{\frac{2\log(4/\alpha)}{\ell^*}}}+\frac{1}{\sqrt{\ell^*}L}\pa{\frac{1}{\beta}+\frac{8}{3}\log(4/\alpha)}\enspace.
\end{multline}
The condition \eqref{bUBhdistance} entails
\[|\delta|\sqrt{\ell^*} \geq \frac{|\delta|\sqrt{\ell^*}}{2}+\frac{\sqrt{\lambda_0}}{\sqrt{L}}\pa{\frac{1}{\sqrt{\beta}}+2\sqrt{\frac{2\log(4/\alpha)}{\ell^*}}}+\frac{1}{\sqrt{\ell^*}L}\pa{\frac{1}{2\beta}+\frac{4}{3}\log(4/\alpha)}\enspace,\]
and therefore, using the inequalities $\sqrt{ab}\leq (a+b)/2$ and $\sqrt{a+b}\leq \sqrt{a}+\sqrt{b}$ for every $a,b\geq 0$,
\begin{eqnarray}
|\delta|\sqrt{\ell^*} &\geq &\sqrt{\frac{|\delta|}{\beta L}} +\frac{\sqrt{\lambda_0}}{\sqrt{L}}\pa{\frac{1}{\sqrt{\beta}}+2\sqrt{\frac{2\log(4/\alpha)}{\ell^*}}}+\frac{4}{3}\frac{\log(4/\alpha)}{\sqrt{\ell^*}L}\nonumber\\
 &\geq &\sqrt{\frac{|\delta|+\lambda_0}{\beta L}} +2\sqrt{\frac{2\log(4/\alpha)\lambda_0}{\ell^* L}}+\frac{4}{3}\frac{\log(4/\alpha)}{\sqrt{\ell^*}L}\enspace.\label{UBalt4ineq}
 \end{eqnarray}
Then, assuming that $\delta>0$, we conclude with the following inequalities: 
\begin{eqnarray*}
P_{\lambda}\pa{\phi_{4,\alpha}^{(1)}(N)=0} &\leq &P_{\lambda}\pa{\phi_{3,\alpha/2}^{(1)+}(N)=0}\\
&\leq &P_\lambda \pa{\max_{t \in [0, 1-\ell^{*}]} N(t,t+\ell^*] \leq p_{\lambda_0,\ell^*}^+(1- \alpha/2)} \\
&\leq &P_\lambda \pa{ N(\tau, \tau +\ell^*]  \leq  \lambda_{0} \ell^{*}L+ 2 \sqrt{  2 \lambda_{0}\log \pa{4/\alpha}L} +4 \log \pa{4/\alpha}/3} ~~\text{with \eqref{malpha/2}} \\
&\leq  &P_{\lambda} \left(  N(\tau,\tau +\ell^*] -(\lambda_{0}^{*}{} + \delta)\ell^* L \leq    -\sqrt{\frac{(\lambda_{0}+ \delta) \ell^* L}{\beta}} \right)~~\text{with \eqref{UBalt4ineq}}\\
  &\leq &\beta ~~\text{with the Bienayme-Chebyshev inequality}\enspace.
\end{eqnarray*}
Assuming now that $\delta$ is in $(-\lambda_0,0)$, 
\begin{eqnarray*}
P_{\lambda}\pa{\phi_{4,\alpha}^{(1)}(N)=0} &\leq &P_{\lambda}\pa{\phi_{3,\alpha/2}^{(1)-}(N)=0}\\
&\leq &P_\lambda \pa{\min_{t \in [0, 1-\ell^{*}]} N(t,t+\ell^*] \geq p_{\lambda_0,\ell^*}^-(\alpha/2)} \\
&\leq &P_\lambda \pa{ N(\tau, \tau +\ell^*]  \geq  \lambda_{0}\ell^{*}L  - 2 \sqrt{  2 \lambda_{0} \log \pa{4/\alpha}L} - 4\log \pa{4/\alpha}/3} ~~\text{with \eqref{malpha/2}} \\
&\leq  &P_{\lambda} \left(  N(\tau,\tau +\ell^*] -(\lambda_{0}+ \delta)\ell^* L \geq    \sqrt{\frac{(\lambda_{0}+ |\delta|) \ell^* L}{\beta}} \right)~~\text{with \eqref{UBalt4ineq}}\\
  &\leq &\beta\enspace.
\end{eqnarray*}

\medskip

\emph{$(ii)$ Control of the second kind error rate of $\phi_{3/4,\alpha}^{(2)}$.}

\smallskip

Let us now set $\lambda$ in $\mathcal{S}_{\bbul,\bbul\bbul,\ell^{*}}$ such that $\lambda= \lambda_{0}+  \delta\un{(\tau, \tau + \ell^*]}$ with $\delta>0$ or $\delta$ in $(-\lambda_0,0)$, $\tau$ in $(0,1-\ell^*)$, and  
\begin{multline} \label{UBdistance_alt4_2}
d_2(\lambda, \calS_0[\lambda_0]) \geq 2 \max \Bigg( 12  \frac{\sqrt{\lambda_{0}}}{\sqrt{\beta L}} +\frac{6 \sqrt{2}}{\sqrt{\beta \ell^*} L}~,~\frac{288}{\beta \sqrt{\ell^*}L}~,\\
\frac{\sqrt{3 \lambda_{0}(4 \log \pa{3/u_\alpha} +\sqrt{2/\beta})}} {\sqrt{L}}+ \frac{2 \sqrt{2} \log \pa{3/u_\alpha} }{\sqrt{3\ell^*}L} + \frac{2\sqrt{2} (2 \lambda_{0})^{1/4} \log^{3/4} \pa{3 /u_\alpha}}{ {\ell^{*}}^{1/4}L^{3/4} }\Bigg)\enspace.
\end{multline}
 Notice that \eqref{UBdistance_alt4_2} entails that
 \begin{multline*} 
  d_2^2(\lambda, \calS_0[\lambda_0]) \geq 3 \max \Bigg( d_2(\lambda, \calS_0[\lambda_0]) \left(  \frac{8 \sqrt{\lambda_{0}}}{\sqrt{ \beta L}} + \frac{4 \sqrt{2}}{\sqrt{\beta \ell^*}L} \right)~,~d_2^{3/2}(\lambda,  \calS_0[\lambda_0]) \frac{8}{ {\ell^*}^{1/4}\sqrt{\beta L}}~,\\
 \frac{16 \lambda_{0} \log \pa{3/u_\alpha}}{ L} +  \frac{32 \log^2 \pa{3/u_\alpha}}{9 \ell^*L^2 } + \frac{32 \sqrt{2\lambda_{0}}  \log^{3/2} \pa{3/u_\alpha}}{3\sqrt{\ell^*}L^{3/2}}  + \frac{4 \lambda_0\sqrt{2}}{\sqrt{\beta} L}\Bigg)\enspace,
 \end{multline*}
and therefore
\begin{multline*} 
  d_2^2(\lambda, \calS_0[\lambda_0]) \geq d_2(\lambda, \calS_0[\lambda_0]) \left(  \frac{8 \sqrt{\lambda_{0}}}{\sqrt{ \beta L}} + \frac{4 \sqrt{2}}{\sqrt{\beta \ell^*}L} \right)+d_2^{3/2}(\lambda,  \calS_0[\lambda_0]) \frac{8}{ {\ell^*}^{1/4}\sqrt{\beta L}}+\\
   \frac{16 \lambda_{0} \log \pa{3/u_\alpha}}{ L} +  \frac{32 \log^2 \pa{3/u_\alpha}}{9 \ell^* L^2} + \frac{32 \sqrt{2\lambda_{0}}  \log^{3/2} \pa{3/u_\alpha}}{3\sqrt{\ell^*}L^{3/2}}  + \frac{4 \lambda_0\sqrt{2}}{\sqrt{\beta} L}\enspace.
 \end{multline*}
Since $d_2(\lambda, \calS_0[\lambda_0])=|\delta|\sqrt{\ell^*}$ and $ \sqrt{\lambda_{0} + \vert \delta \vert } \leq \sqrt{\lambda_{0}} + \sqrt{\vert \delta \vert}$, this implies that
\begin{multline} \label{UBdistance_alt4_2bis}
  \frac{d_2^2(\lambda, \calS_0[\lambda_0])}{4} \geq \frac{2 \sqrt{\lambda_{0}+|\delta|}}{\sqrt{ \beta L}} d_2(\lambda, \calS_0[\lambda_0])  + \frac{\sqrt{2}(|\delta|+\lambda_0)}{\sqrt{\beta}L}+\\
   \frac{4 \lambda_{0} \log \pa{3/u_\alpha}}{ L} +  \frac{8 \log^2 \pa{3/u_\alpha}}{9 \ell^*L^2 } + \frac{8 \sqrt{2\lambda_{0}}  \log^{3/2} \pa{3/u_\alpha}}{3\sqrt{\ell^*}L^{3/2}} \enspace.
 \end{multline}
Let us now prove that $P_{\lambda}\pa{\phi_{3/4,\alpha}^{(2)}(N) =0} \leq \beta$. From the definition \eqref{testN2alt3-4}, we notice that
\[ P_{\lambda}\pa{\phi_{3/4,\alpha}^{(2)}(N) =0} \leq \min_{k\in \lbrace 0,\ldots, \lceil (1-\ell^*) M \rceil -1 \rbrace} P_{\lambda} \left( T_{\frac{k}{M}, \frac{k}{M}+\ell^*}(N)\leq t_{\frac{k}{M}, \frac{k}{M}+\ell^*} \pa{1 - u_\alpha} \right)\enspace,\]
so that we only need to exhibit some $k_{\tau}$ in $\lbrace  0,...,\lceil (1-\ell^*) M \rceil-1  \rbrace$ such that 
 $$P_{\lambda}\pa{T_{\frac{k_{\tau}}{M}, \frac{k_{\tau}}{M} +\ell^*}(N) \leq t_{\frac{k_{\tau}}{M}, \frac{k_{\tau}}{M} +\ell^*}( 1-u_\alpha)} \leq \beta\enspace.$$ 
As in the  proof of Proposition \ref{UBalt3}, we choose $k_\tau = \lfloor \tau M \rfloor$, which leads (see \eqref{bproofDeltaLength1}  and \eqref {bproofDeltaLength2}) to
\begin{equation} \label{bproofDeltaLength1_alt4}
 E_\lambda \cro{T_{\frac{k_\tau}{M}, \frac{k_\tau}{M}+\ell^*}(N)}\geq  \frac{d_2^2(\lambda, \mathcal{S}_{0}[\lambda_0]) }{4}\enspace,
 \end{equation}
and
\begin{equation}\label{bproofDeltaLength2_alt4}
 \mathrm{Var}_\lambda \cro{T_{\frac{k_\tau}{M}, \frac{k_\tau}{M}+\ell^*}(N)} \leq \frac{4 (\lambda_{0} + \vert \delta \vert )}{L} d_2^2(\lambda, \calS_0[\lambda_0]) + \frac{2(\lambda_{0} + \vert \delta \vert)^2}{L^2}\enspace.
 \end{equation}
From \eqref{UBdistance_alt4_2bis}, \eqref{bproofDeltaLength1_alt4} and \eqref{bproofDeltaLength2_alt4}, we derive that
\begin{multline*} 
 E_\lambda \cro{T_{\frac{k_\tau}{M}, \frac{k_\tau}{M}+\ell^*}(N)} \geq  \sqrt{\mathrm{Var}_\lambda \cro{T_{\frac{k_\tau}{M}, \frac{k_\tau}{M}+\ell^*}(N)}/{\beta}}+\\
     \frac{4 \lambda_{0} \log \pa{3/u_\alpha}}{ L} +  \frac{8 \log^2 \pa{3/u_\alpha}}{9 \ell^*L^2} + \frac{8 \sqrt{2\lambda_{0}}  \log^{3/2} \pa{3/u_\alpha}}{3\sqrt{\ell^*} L^{3/2}}\enspace.
 \end{multline*}
The conclusion then basically follows from Lemma \ref{QuantilesT}, supplemented by the upper bound \eqref{UBginv}, which allows to see that
\[  E_\lambda \cro{T_{\frac{k_\tau}{M}, \frac{k_\tau}{M}+\ell^*}(N)} \geq  \sqrt{\mathrm{Var}_\lambda \cro{T_{\frac{k_\tau}{M}, \frac{k_\tau}{M}+\ell^*}(N)}/{\beta}}+t_{\frac{k_\tau}{M}, \frac{k_\tau}{M} +\ell^*}( 1-u_\alpha)\enspace,\]
and the Bienayme-Chebyshev inequality, which entails
\[P_{\lambda}\pa{T_{\frac{k_{\tau}}{M}, \frac{k_{\tau}}{M} +\ell^*}(N) \leq t_{\frac{k_{\tau}}{M}, \frac{k_{\tau}}{M} +\ell^*}( 1-u_\alpha)} \leq \beta\enspace.\]

 \subsection{Proof of Proposition \ref{LBalt5}}

Let $C_{\alpha, \beta}=1+4(1- \alpha -\beta)^2$, $r=(\lambda_{0}  \log C_{\alpha,\beta}/L )^{1/2}$ and $\lambda_r$ defined  for all $t$ in $(0,1)$ by
\[ \lambda_r(t) = \lambda_{0} +  \delta^*  \mathds{1}_{(\tau^{*}, \tau^{*} + r^2/{\delta^{*}}^2]}(t)\enspace.\]
Notice that for all $L \geq \lambda_{0} \log C_{\alpha,\beta}/({\delta^{*}}^2 (1-\tau^{*})),$ $r^2/{\delta^{*}}^2 \leq 1-\tau^{*}$ and $\lambda_r$ belongs to $\pa{\calS_{\delta^*,\tau^*,\bbul\bbul\bbul}[\lambda_0]}_{r}$ in the notation of Lemma \ref{mSR}. We get now from Lemma \ref{lemmegirsanov} and Lemma \ref{momentPoisson}
\[ E_{\lambda_0} \left[\left( \frac{d P_{\lambda_r}}{dP_{0}} \right)^{2}(N)\right] = \exp \pa{  \frac{r^2 L}{\lambda_{0}} } = C_{\alpha,\beta}. \]
Lemmas \ref{lemmebayesien} and \ref{mSR} then entail  $\rho_\alpha\pa{\pa{\calS_{\delta^*,\tau^*,\bbul\bbul\bbul}[\lambda_0]}_{r}} \geq \beta$ and $\mSRab\pa{\calS_{\delta^*,\tau^*,\bbul\bbul\bbul}[\lambda_0]}\geq r.$

 \subsection{Proof of Proposition \ref{UBalt5}}

The first kind error rate control is straightforward. As for the second kind error rate control, let $\lambda= \lambda_{0}+  \delta^*\un{(\tau^*, \tau^* + \ell]}$ belonging to $\calS_{\delta^*,\tau^*,\bbul\bbul\bbul}[\lambda_0]$ with $\ell$ in $(0,1-\tau^*)$ and satisfying
\begin{multline}\label{distalt5}
d_2(\lambda, \calS_0[\lambda_0]) \geq \frac{2}{\sqrt{L}} \max \Bigg( 2 \sqrt{\frac{\lambda_{0} + \delta^{*}}{\beta}}~,~\sqrt{ \delta^{*} s_{\lambda_{0},\frac{\delta^{*}}{2}}^+\pa{ 1-\alpha}} \1{\delta^{*} >0} +\\
 \sqrt{\frac{\vert \delta^{*} \vert \log \pa{1/\alpha} }{\log \pa{\lambda_{0}/\pa{\lambda_{0} - \vert\delta^{*} \vert/ 2}}}} \1{-\lambda_{0} <\delta^{*} <0}\Bigg)\enspace.
 \end{multline}

Assume that $\delta^{*}{}>0$ and recall that $s_{\lambda_{0},\delta^*/2}^+\pa{ 1-\alpha}$ defined in Lemma \ref{QuantilessupShifted} is a constant which does not depend on $L$.
The assumption \eqref{distalt5} implies
$$ d_2(\lambda, \calS_0[\lambda_0]) \geq \frac{2}{\sqrt{L}}   \max \pa{ 2 \sqrt{\frac{\lambda_{0} + \delta^{*}{}}{\beta}} ~,~  \sqrt{\delta^{*}{} s_{\lambda_{0},\frac{\delta^{*}}{2}}^+\pa{ 1-\alpha} }}\enspace,$$
which yields
$$ \delta^{*} \ell  \geq 4 \max \pa{\sqrt{\frac{(\lambda_{0}+ \delta^{*})\ell }{\beta L}}  ~,~  \frac{s_{\lambda_{0},\frac{\delta^*}{2}}^+\pa{ 1-\alpha} }{L}  }\enspace,$$
hence
\begin{equation} \label{bineq33bislog}
\frac{\delta^{*}}{2}\ell L \geq  \sqrt{\frac{(\lambda_{0} + \delta^{*})\ell   L}{\beta}} + s_{\lambda_{0},\frac{\delta^{*}}{2}}^+\pa{ 1-\alpha }\enspace.
\end{equation}

We get then
\[P_{\lambda}\pa{\phi_{5,\alpha}(N)=0} \leq P_{\lambda} \pa{ \sup_{\ell' \in (0,1- \tau^{*})}S_{\delta^*,\tau^*,\tau^*+\ell'}(N) \leq s_{\lambda_0,\delta^*,\tau^*,L}^+(1-\alpha)}\enspace,\]
where 
\[S_{\delta^*,\tau^*,\tau^*+\ell'}(N)=  \mathrm{sgn}(\delta^{*}{}) \Big(N(\tau^{*}{},\tau^{*}{} +\ell'] - \lambda_{0} L\ell'\Big) - \vert \delta^{*}{} \vert L\ell'/2 \enspace,\]
as defined by \eqref{stat_alt5} and $s_{\lambda_0,\delta^*,\tau^*,L}^+(u)$ is the $u$-quantile of $\sup_{\ell' \in (0,1- \tau^{*})}S_{\delta^*,\tau^*,\tau^*+\ell'}(N)$  under $(H_0)$.
From the quantile upper bound \eqref{UBQuantilemaxShifted}, we deduce
\begin{align*}
P_{\lambda}\pa{\phi_{5,\alpha}(N)=0} &\leq P_{\lambda} \pa{\sup_{\ell' \in (0,1- \tau^{*})}S_{\delta^*,\tau^*,\tau^*+\ell'}(N)\leq   s_{\lambda_{0},\frac{\delta^*}{2}}^+\pa{ 1-\alpha}}\\
&\leq P_{\lambda} \pa{ N(\tau^{*}, \tau^{*} +\ell] - \pa{\lambda_{0} + \frac{\delta^{*}}{2}}\ell L     \leq  s_{\lambda_{0},\frac{\delta^*}{2}}^+\pa{ 1-\alpha} } \\
&\leq  P_{\lambda} \pa{  N(\tau^{*},\tau^{*}+\ell ] - (\lambda_{0} + \delta^{*}{})\ell  L \leq  s_{\lambda_{0},\frac{\delta^*}{2}}^+\pa{ 1-\alpha} - \frac{\delta^{*}{}}{2}  \ell  L  } \\
&\leq P_{\lambda} \pa{  N(\tau^{*},\tau^{*}+\ell ] - (\lambda_{0}+ \delta^{*}{}) \ell L \leq    - \sqrt{\frac{(\lambda_{0}+ \delta^{*}{}) \ell  L}{\beta}}  } ~~ \text{with \eqref{bineq33bislog}} \\
&\leq \beta\enspace,
\end{align*}
 with a last line simply following from the Bienayme-Chebyshev inequality.

\smallskip

Assume now that $\delta^{*}{}$ is in $(- \lambda_{0}, 0)$.
The assumption \eqref{distalt5} implies 
$$ d_2(\lambda, \calS_0[\lambda_0]) \geq \frac{2}{\sqrt{L}} \max \Bigg( 2 \sqrt{\frac{\lambda_{0} + \delta^{*}}{\beta}} ~,~
 \sqrt{\frac{\vert \delta^{*} \vert \log \pa{1/\alpha} }{\log \pa{\lambda_{0}/\pa{\lambda_{0} - \vert\delta^{*} \vert/ 2}}}} \Bigg)\enspace,$$
which yields
$$ \vert \delta^{*}{} \vert \ell  \geq  4 \max \pa{\sqrt{\frac{(\lambda_{0} + \delta^{*}{})\ell }{\beta L}}  ~,~  \frac{\log \pa{1/\alpha} }{L\log \pa{\lambda_{0}/\pa{\lambda_{0} - \vert\delta^{*} \vert/ 2} }}}\enspace.$$
Hence
$$ \frac{\vert \delta^{*}{} \vert}{2}\ell  L \geq   \sqrt{\frac{(\lambda_{0}+ \delta^{*}{})\ell  L}{\beta}} + \frac{\log \pa{1/\alpha} }{\log \pa{\lambda_{0}/\pa{\lambda_{0} - \vert\delta^{*} \vert/ 2} }}\enspace,$$
and then
\begin{equation}\label{bineq33-bislog}
\pa{\lambda_{0} - \frac{\vert \delta^{*}{} \vert}{2}}\ell  L -   \frac{\log \pa{1/\alpha} }{\log \pa{\lambda_{0}/\pa{\lambda_{0} - \vert\delta^{*} \vert/ 2} }}- (\lambda_{0} + \delta^{*}{})\ell  L \geq \sqrt{\frac{(\lambda_{0} + \delta^{*}{})\ell  L}{\beta}}\enspace.
\end{equation}

We conclude with the following inequalities:
\begin{align*}
P_{\lambda}\Big(&\phi_{5,\alpha}(N)=0\Big)\\
& \leq P_{\lambda} \pa{\sup_{\ell' \in (0,1- \tau^{*})}S_{\delta^*,\tau^*,\tau^*+\ell'}(N)\leq   s_{\lambda_0,\delta^*,\tau^*,L}^+(1-\alpha)}\\
&\leq P_{\lambda} \pa{ \pa{\lambda_{0} - \frac{\vert \delta^{*}{} \vert}{2}}\ell  L - N(\tau^{*}{},\tau^{*}{} +\ell]  \leq  \frac{\log \pa{1/\alpha} }{\log \pa{\lambda_{0}/\pa{\lambda_{0} - \vert\delta^{*} \vert/ 2} }} } ~~ \text{with \eqref{UBQuantilemaxShifted}} \\
&\leq P_\lambda \pa{ N(\tau^{*}{},\tau^{*}{} + \ell ]  \geq \pa{\lambda_{0} - \frac{\vert \delta^{*}{} \vert}{2}}\ell   L -  \frac{\log \pa{1/\alpha} }{\log \pa{\lambda_{0}/\pa{\lambda_{0} - \vert\delta^{*} \vert/ 2} }} }  \\
&\leq    P_\lambda \pa{ N(\tau^{*}{}, \tau^{*}{} + \ell] - (\lambda_{0}+ \delta^{*}{})\ell  L \geq \sqrt{\frac{(\lambda_{0} + \delta^{*}{})\ell  L}{\beta}}    }~~\text{with (\ref{bineq33-bislog})} \\
&\leq \beta~~\text{with the Bienayme-Chebyshev inequality} \enspace.
\end{align*}

\subsection{Proof of Lemma \ref{LBalt6-pre}}

Let $\lambda_0>0$, $\tau^*$ in $(0,1)$, and $\phi_{\alpha}$ a level-$\alpha$ test of 
  $\hzero \ "\lambda\in \calS_0[\lambda_0]=\{\lambda_0\}"$ versus $\hone\ "\lambda\in\calS_{\bbul,\tau^*,\bbul\bbul\bbul}[\lambda_0]"$, with 
$$\calS_{\bbul,\tau^*,\bbul\bbul\bbul}[\lambda_0]=
 \big\{ \lambda\!:\exists \delta \in (- \lambda_{0}, + \infty)\setminus\lbrace 0\rbrace ,\exists \ell \in (0,1-\tau^*),\\
 \lambda(t) = \lambda_{0} + \delta \mathds{1}_{(\tau^*,\tau^*+\ell]}(t) \big\}\enspace,$$
as defined by \eqref{alt6-pre}. 

Let $r>0$ and $\lambda$ in $\calS_{\bbul,\tau^*,\bbul\bbul\bbul}[\lambda_0]$  satisfying $d_2\pa{\lambda,\calS_0[\lambda_0]}\geq r$. We compute
 \begin{align*}
P_{\lambda}\pa{\phi_{\alpha}(N) =0} &= 1- P_{\lambda}\pa{\phi_{\alpha}(N)=1} + P_{\lambda_0}\pa{\phi_{\alpha}(N)=1}-P_{\lambda_0}\pa{\phi_{\alpha}(N)=1}\\
 &\geq 1- \alpha - \vert P_{\lambda_0}\pa{\phi_{\alpha}(N)=1} - P_{\lambda}\pa{\phi_{\alpha}(N)=1}   \vert \\
 &\geq 1- \alpha - V\pa{P_{\lambda}, P_{\lambda_0}}\enspace,
 \end{align*}
 where $V\pa{P_{\lambda}, P_{\lambda_0}}$ is the total variation distance between the probability measures $P_{\lambda}$ and $P_{\lambda_0}$.
 Then, using the Pinsker inequality (see for example Lemma 2.5 in \cite{Tsybakov2008}),
 \[ P_{\lambda}\pa{\phi_{\alpha} (N)=0} \geq 1- \alpha - \sqrt{\frac{K\pa{P_{\lambda}, P_{\lambda_0}}}{2}}\enspace,\]
 where $K\pa{P_{\lambda}, P_{\lambda_0}}$ is the Kullback divergence between the probability measures $P_{\lambda}$ and $P_{\lambda_0}$. We deduce from  Lemma \ref{mSR} that if there exists $\lambda$ in $\calS_{\bbul,\tau^*,\bbul\bbul\bbul}[\lambda_0]$ such that $d_2\pa{\lambda,\calS_0[\lambda_0]}\geq r$ satisfying $1- \alpha - \sqrt{K\pa{P_{\lambda}, P_{\lambda_0}}/2} \geq \beta$, then $\mSRab\pa{\calS_{\bbul,\tau^*,\bbul\bbul\bbul}[\lambda_0]}  \geq r$.

Let us introduce for all $\ell$ in $(0,1-\tau^{*}{})$, $\lambda_{r}= \lambda_{0} + r \ell^{-1/2} \mathds{1}_{(\tau^{*}{}, \tau^{*}{} +\ell]}$ in $\calS_{\bbul,\tau^*,\bbul\bbul\bbul}[\lambda_0]$ which satisfies $d_2(\lambda_{r},\calS_0[\lambda_0])=r$. Then, Lemma \ref{lemmegirsanov} entails
 \[ K\pa{P_{\lambda_{r}}, P_{\lambda_0}}= \int \log \left( \frac{dP_{\lambda_{r}}}{dP_{\lambda_0}}  \right) dP_{\lambda_{r}} = \log \left( 1+ \frac{r}{\lambda_{0} \sqrt{\ell}}  \right)\left( \lambda_{0}+ \frac{r}{\sqrt{\ell}} \right)\ell L-Lr \sqrt{\ell}\enspace.\]
Hence choosing $\ell$ close enough to $0$ -- which is allowed as long as $\lambda_r$ is not constrained to be upper bounded by some given constant, $K\pa{P_{\lambda_{r}}, P_{\lambda_0}} \leq 2(1- \alpha -\beta)^2$. This entails 
$$\mSRab\pa{\calS_{\bbul,\tau^*,\bbul\bbul\bbul}[\lambda_0]}  \geq r$$
for every $r>0$ and allows to conclude that $\mSRab\pa{\calS_{\bbul,\tau^*,\bbul\bbul\bbul}[\lambda_0]} = + \infty.$

\subsection{Proof of Proposition \ref{LBalt6}}

Assume that $L\geq 3$ and $\alpha+\beta<1/2$. 

Let $C'_{\alpha, \beta}=4(1- \alpha - \beta)^{2}$, $K_{\alpha,\beta,L}=\lceil (\log_2 L)/C'_{\alpha,\beta} \rceil$, and for $k$ in $\lbrace 1,\ldots, K_{\alpha,\beta,L}\rbrace$,
$\lambda_{k}= \lambda_{0}+ \delta_{k} \mathds{1}_{(\tau^{*}{}, \tau^{*}{} + \ell_{k}]}$  with $ \ell_{k}  = (1- \tau^{*}{})/2^{k}$
 and
 $ \delta_{k} = (\lambda_{0} \log \log L/ (\ell_k L))^{1/2}.$ Then  $d_2\pa{\lambda_{k},\calS_0[\lambda_0]} = \sqrt{\lambda_{0}\log \log L/L}$ for all $k$ in $\lbrace 1,\ldots,K_{\alpha,\beta,L} \rbrace$ and assuming that
\begin{equation}\label{assum_L_alt6_1}
\frac{\log \log L}{L^{1-1/C'_{\alpha,\beta}}} \leq (R- \lambda_{0})^2\frac{1-\tau^*}{2\lambda_0}\enspace,
\end{equation}
 $\lambda_k$ belongs to $\calS_{\bbul,\tau^*,\bbul\bbul\bbul}[\lambda_0,R]$.
Recall that for any $k$ in $\lbrace 1,\ldots,K_{\alpha,\beta,L} \rbrace$ $P_{\lambda_{k}}$ denotes the distribution of a Poisson process with intensity $\lambda_{k}$ with respect to the measure $\Lambda$, and consider $\kappa$, a random variable with uniform distribution on $\lbrace 1,\ldots, K_{\alpha,\beta,L} \rbrace$, which allows to define the probability distribution $\mu$ of $\lambda_{\kappa}$. From Lemma \ref{lemmebayesien}, we know that it is enough to prove $E_{\lambda_0} [\left( dP_{\mu}/dP_{\lambda_0}\right)^{2}  ] \leq 1+C'_{\alpha,\beta}$ to conclude that  $\mSRab\pa{\calS_{\bbul,\tau^*,\bbul\bbul\bbul}[\lambda_0,R]}\geq \sqrt{\lambda_{0}\log \log L/L}$.
 
 By definition, $ (dP_{\mu}/dP_{\lambda_0})(N)= E_{\kappa} \left[ (dP_{\lambda_{\kappa}}/dP_{\lambda_0})(N)  \right] $ (where $E_\kappa$ denotes the expectation w.r.t. the uniform variable $\kappa$) and therefore
\[\frac{dP_{\mu}}{dP_{\lambda_0}}(N)= \frac{1}{K_{\alpha,\beta,L}} \sum_{k=1}^{K_{\alpha,\beta,L}} \exp \left( \log \left( 1+ \frac{\delta_{k}}{\lambda_{0}} \right) N ( \tau^{*}{}, \tau^{*}{} + \ell_k ]  -L \ell_k \delta_{k}  \right)\enspace.\]
 Since $\ell_{k'}< \ell_{k}$ for all $k' >k$,
 \begin{multline*}
 \left( \frac{dP_{\mu}}{dP_{\lambda_0}}(N) \right)^{2} 
 = \frac{1}{K_{\alpha,\beta,L}^{2}} \sum_{k=1}^{K_{\alpha,\beta,L}} \exp \left( 2 \log \left( 1+ \frac{\delta_{k}}{\lambda_{0}} \right) N (\tau^{*}{}, \tau^{*}{} + \ell_k]  -2L\ell_k \delta_{k}  \right) \\
+ \frac{2}{K_{\alpha,\beta,L}^{2}} \sum_{k=1}^{K_{\alpha,\beta,L}-1} \sum_{k'=k+1}^{K_{\alpha,\beta,L}} \exp \left(\left( \log \left( 1+ \frac{\delta_{k}}{\lambda_{0}} \right) +\log\left( 1+ \frac{\delta_{k'}}{\lambda_{0}} \right) \right)N(\tau^{*}{}, \tau^{*}{}+\ell_{k'}]+\right.\\
 +\left.\log \left( 1+ \frac{\delta_{k}}{\lambda_{0}} \right) N(\tau^{*}{} + \ell_{k'}, \tau^{*}{}+\ell_{k}]-L\ell_k \delta_{k}  -L\ell_{k'} \delta_{k'} \right)\enspace.
\end{multline*}
Recall that under $\hzero$, $N$ is a homogeneous Poisson process with intensity $\lambda_{0}$ with respect to the measure $\Lambda$, so
\[E_{\lambda_0} \left[\left( \frac{dP_{\mu}}{dP_{\lambda_0}} \right)^{2}   \right]= \frac{1}{K_{\alpha,\beta,L}^{2}} \sum_{k=1}^{K_{\alpha,\beta,L}} \exp\pa{\frac{L \ell_k \delta_{k}^{2}}{\lambda_{0}}}+ \frac{2}{K_{\alpha,\beta,L}^{2}} \sum_{k=1}^{K_{\alpha,\beta,L}-1} \sum_{k'=k+1}^{K_{\alpha,\beta,L}} \exp\pa{ \frac{L \ell_{k'}\delta_{k} \delta_{k'}}{\lambda_{0}}}\enspace.\]
Then
\begin{equation}\label{eq_LBalt6}
E_{\lambda_0} \left[\left( \frac{dP_{\mu}}{dP_{\lambda_0}} \right)^{2}   \right] = \frac{\log L}{K_{\alpha,\beta,L}}  +\frac{2}{K_{\alpha,\beta,L}^{2}} \sum_{k=1}^{K_{\alpha,\beta,L}-1} \sum_{k'=k+1}^{K_{\alpha,\beta,L}} \exp \left( 2^{\frac{k-k'}{2} }(\log \log L) \right)\enspace,\\
\end{equation}
which entails
\begin{align*}
E_{\lambda_0} \left[\left( \frac{dP_{\mu}}{dP_{\lambda_0}} \right)^{2}   \right] &= \frac{\log L}{K_{\alpha,\beta,L}}  + \frac{2}{K_{\alpha,\beta,L}^{2}} \sum_{l=1}^{K_{\alpha,\beta,L}-1}  \left( K_{\alpha,\beta,L} - l \right)  \exp \left( 2^{-\frac{l}{2} }(\log \log L) \right)\\
&\leq C'_{\alpha,\beta}\log 2 +\frac{2}{K_{\alpha,\beta,L}^{2}} \sum_{l=1}^{K_{\alpha,\beta,L}-1}  \left( K_{\alpha,\beta,L} - l \right)  \exp \left( 2^{-\frac{l}{2} }(\log \log L) \right)\enspace.
\end{align*}
Now taking $\eta$ such that $0<\eta<1-1/\sqrt{2}$,
\begin{align}
E_{\lambda_0} \left[\left( \frac{dP_{\mu}}{dP_{\lambda_0}} \right)^{2}   \right] &\leq C'_{\alpha,\beta}\log 2 +\frac{2}{K_{\alpha,\beta,L}^{2}} \sum_{l=1}^{K_{\alpha,\beta,L}-1}  \left( K_{\alpha,\beta,L} - l \right)  \exp \left( 2^{-\frac{l}{2} }(\log \log L) \right) \nonumber\\
&= C'_{\alpha,\beta}\log 2 + \frac{2}{K_{\alpha, \beta,L}^{2}} \sum_{l=1}^{\lfloor (\log L)^{\eta} \rfloor}  \left( K_{\alpha,\beta,L}  - l \right)  \exp \left( 2^{-\frac{l}{2} }(\log \log L) \right) \nonumber\\
& \quad + \frac{2}{K_{\alpha, \beta,L}^{2}} \sum_{l=\lfloor (\log L)^{\eta} \rfloor +1}^{K_{\alpha, \beta,L}-1}  \left(K_{\alpha, \beta,L}- l \right)  \exp \left( 2^{-\frac{l}{2} }(\log \log L) \right) \nonumber\\
&\leq  C'_{\alpha,\beta}\log 2 + \frac{2C'_{\alpha,\beta}\log 2}{\log L} \left( \log L  \right)^{\eta+\frac{1}{\sqrt{2}}}+\exp \left( \frac{\log \log L}{2^{(\log L)^{\eta}/2}}\right) \label{upperboundExpectsquare}\enspace.
\end{align}
If we assume now that 
\begin{equation}\label{assum_L_alt6_2}
\exp \left( \frac{\log \log L}{2^{(\log L)^{\eta}/2}}\right) + \frac{2C'_{\alpha,\beta}\log 2}{\left( \log L  \right)^{1-\eta-\frac{1}{\sqrt{2}}}}\leq 1+ (1-\log 2) C'_{\alpha,\beta}\enspace,
\end{equation}
we finally obtain the expected result, i.e.
\[E_{\lambda_0} \left[\left( \frac{dP_{\mu}}{dP_{\lambda_0}} \right)^{2}   \right] \leq  1+C'_{\alpha,\beta}\enspace.\]
To end the proof, it remains to notice that there exists $L_0(\alpha,\beta,\lambda_0,R)\geq 3$ such that for all $L\geq L_0(\alpha,\beta,\lambda_0,R)$, both assumptions \eqref{assum_L_alt6_1} and \eqref{assum_L_alt6_2} hold.

\subsection{Proof of Proposition \ref{UBalt6}}

The control of the first kind error rates of the two tests $\phi_{6,\alpha}^{(1)}$ and $\phi_{6,\alpha}^{(2)}$ is straightforward using simple union bounds. 

\medskip

\emph{$(i)$ Control of the second kind error rate of $\phi_{6,\alpha}^{(1)}$.}

\smallskip

Let $\lambda$ in $\calS_{\bbul,\tau^*,\bbul\bbul\bbul}[\lambda_0,R]$ be such that $\lambda= \lambda_{0}+  \delta \mathds{1}_{(\tau^{*}{}, \tau^{*}{} + \ell]}$, with $\delta$ in $(-\lambda_{0}, R- \lambda_{0}]\setminus\{0\}$, $\ell$ in $(0,1-\tau^{*})$, and such that 
 \begin{align} \label{distance_alt6_N1}
  d_2\pa{\lambda,\calS_0[\lambda_0]}&\geq \sqrt{2} \max \left(2 \sqrt{\frac{R\log \left(2/u_\alpha \right)}{3 L}}~,~ 2  \sqrt{\frac{2 \lambda_{0}\log \left(2/u_\alpha\right)}{L}} + 2\sqrt{\frac{R}{\beta L}}  ~,~ \frac{R}{\sqrt{L}}\right)\enspace.
 \end{align}
Let us prove that $P_{\lambda}\pa{\phi_{6,\alpha}^{(1)}(N)=0} \leq \beta$. 

\smallskip

Assume first that $\delta$ belongs to $(0,R- \lambda_{0}]$. 
Noticing that
$$ P_{\lambda}\pa{\phi_{6,\alpha}^{(1)}(N)=0}  \leq \inf_{k\in \lbrace 1,\ldots, \lfloor \log_{2} L \rfloor \rbrace} P_{\lambda} \left( S_{\tau^{*}{}, \tau^{*}{} + \ell_{\tau^*,k}}(N)\leq  s_{\lambda_0,\tau^*,\tau^{*}{} + \ell_{\tau^*,k}} \pa{1 -u_\alpha} \right)\enspace,$$
one can see that it is enough to exhibit some $k$ in $\set{1,\ldots,\lfloor \log_{2} L \rfloor}$ satisfying 
\[P_{\lambda} \left( S_{\tau^{*}{}, \tau^{*}{} + \ell_{\tau^*,k}}(N)\leq  s_{\lambda_0,\tau^*,\tau^{*}{} + \ell_{\tau^*,k}} \pa{1 -u_\alpha} \right) \leq \beta\enspace.\]

We get from \eqref{distance_alt6_N1} that 
 $d_2^2\pa{\lambda,\calS_0[\lambda_0]} \geq 2 R^2/L$ which entails $\ell\geq 2/L$, so\\
    $ (1-\tau^{*}{}) 2^{- \lfloor \log_2 L \rfloor}\leq 2(1-\tau^*)/L< 2/{L}\leq \ell$ and  $\ell<(1-\tau^*) 2^{-1+1}$.  Therefore, one can find $k_{\tau^{*}{}}$ in $\lbrace 1,\ldots, \lfloor \log_{2} L \rfloor \rbrace$ satisfying $(1-\tau^{*}{}) 2^{-k_{\tau^{*}{}}} \leq \ell <  (1-\tau^{*}{}) 2^{-k_{\tau^{*}{}}+1}$.
Consider now
 $ \ell_{\tau^{*}{}}= (1-\tau^{*}{}) 2^{-k_{\tau^{*}{}}}=\ell_{\tau^*,k_{\tau^*}}$ such that $\ell/2<\ell_{\tau^{*}{}}\leq \ell $. We get
  \begin{equation} \label{eq1_alt6_N1}
   \delta^2 \ell_{\tau^{*}{}} > d_2^2\pa{\lambda,\calS_0[\lambda_0]}/2\enspace .
   \end{equation}

Moreover, we deduce from \eqref{distance_alt6_N1} that
\[d_2\pa{\lambda,\calS_0[\lambda_0]}\geq \sqrt{2} \max \left(2\sqrt{\frac{\delta\log \left(2/u_\alpha \right)}{3L}}~,~ 2  \sqrt{\frac{2 \lambda_{0}\log \left(2/u_\alpha\right)}{L}} + 2\sqrt{\frac{\lambda_0+\delta}{\beta L}} \right)\enspace,\]
and then with \eqref{eq1_alt6_N1},
\[\delta \sqrt{\ell_{\tau^{*}}} \geq    \max \left(2\sqrt{\frac{\delta\log \left(2/u_\alpha \right)}{3L}}~,~ 2  \sqrt{\frac{2 \lambda_{0}\log \left(2/u_\alpha\right)}{L}} + 2\sqrt{\frac{\lambda_0+\delta}{\beta L}} \right)\enspace.\]
This entails in particular
$\delta \ell_{\tau^{*}{}} \geq 4 \log \left( 2 /u_\alpha  \right)/(3L)$ as well as
\[\delta \ell_{\tau^{*}{}}  \geq 2 \sqrt{\ell_{\tau^{*}{}}}\pa{\sqrt{\frac{2 \lambda_{0}\log \left(2/u_\alpha\right)}{L}} + \sqrt{\frac{\lambda_0+\delta}{\beta L}} }\enspace.\]

Therefore,
\[ \delta \ell_{\tau^{*}{}} \geq 2  \max \left( \frac{2 \log \left(2 /u_\alpha  \right)}{3L} ~,~ \sqrt{\ell_{\tau^{*}{}}}\pa{\sqrt{\frac{2 \lambda_{0}\log \left(2/u_\alpha\right)}{L}} + \sqrt{\frac{\lambda_0+\delta}{\beta L}}} \right)\enspace,\]
hence
 \begin{equation} \label{eq2_alt6_N1}
    \delta \ell_{\tau^{*}{}} L \geq \frac{2}{3} \log \left( 2/u_\alpha \right) +  \sqrt{2 \lambda_{0}\ell_{\tau^{*}{}} L \log \left(2/u_\alpha\right) } +  \sqrt{\frac{(\lambda_{0}+ \delta)  \ell_{\tau^{*}{}} L}{\beta}}\enspace.
\end{equation}

On the one hand, since $\ell_{\tau^{*}{}} \leq \ell$, Lemma \ref{momentPoisson} gives $E_{\lambda}\cro{S_{\tau^{*}{}, \tau^{*}{} + \ell_{\tau^*}}(N)}= \delta \ell_{\tau^{*}{}} L$,
 and
$\mathrm{Var}_{\lambda}\cro{S_{\tau^{*}{}, \tau^{*}{} + \ell_{\tau^*}}(N) } =(\lambda_{0}+ \delta) \ell_{\tau^{*}{}} L$. On the other hand, Lemma \ref{QuantilesAbsShifted} gives
 \[s_{\lambda_0,\tau^*,\tau^{*}{} + \ell_{\tau^*}} \pa{1 -u_\alpha} \leq \lambda_{0} \ell_{\tau^{*}{}} L g^{-1} \pa{\frac{\log \pa{2/u_\alpha}}{\lambda_{0} L \ell_{\tau^{*}{}} }}\enspace,  \] with $g^{-1}(x) \leq 2x/3+ \sqrt{2x}$ for all $x>0$  (see \eqref{UBginv}), which leads to
  \[s_{\lambda_0,\tau^*,\tau^{*}{} + \ell_{\tau^*}} \pa{1 -u_\alpha} \leq \frac{2}{3} \log \left( 2/u_\alpha \right) +  \sqrt{2 \lambda_{0}\ell_{\tau^{*}{}} L \log \left(2/u_\alpha\right) } \enspace. \] 
 
Combined with \eqref{eq2_alt6_N1}, these computations yield
 \begin{equation} \label{eq3_alt6_N1}
E_{\lambda}\cro{S_{\tau^{*}{}, \tau^{*}{} + \ell_{\tau^*}}(N)} \geq s_{\lambda_0,\tau^*,\tau^{*}{} + \ell_{\tau^*}} \pa{1 -u_\alpha} +  \sqrt{\mathrm{Var}_{\lambda}\cro{S_{\tau^{*}{}, \tau^{*}{} + \ell_{\tau^*}}(N)}/\beta}\enspace.
\end{equation}

 We conclude with the Bienayme-Chebyshev inequality:
 \begin{align*}
P_{\lambda} \Big(& S_{\tau^{*}{}, \tau^{*}{} + \ell_{\tau^*}}(N)\leq  s_{\lambda_0,\tau^*,\tau^{*}{} + \ell_{\tau^*}} \pa{1 -u_\alpha} \Big)\\
 &\leq P_{\lambda} \left( S_{\tau^{*}{}, \tau^{*}{} + \ell_{\tau^*}}(N) - E_{\lambda}\cro{S_{\tau^{*}{}, \tau^{*}{} + \ell_{\tau^*}}(N)} \leq  - \sqrt{\mathrm{Var}_{\lambda}\cro{S_{\tau^{*}{}, \tau^{*}{} + \ell_{\tau^*}}(N) }/\beta} \right) ~~\text{with \eqref{eq3_alt6_N1}}\\
 &\leq \beta\enspace.
 \end{align*}

Assume now that $\delta$ belongs to $(-\lambda_0,0)$ and notice that
$$ P_{\lambda}\pa{\phi_{6,\alpha}^{(1)}(N)=0}  \leq \inf_{k\in \lbrace 1,\ldots, \lfloor \log_{2} L \rfloor \rbrace} P_{\lambda} \left( -S_{\tau^{*}{}, \tau^{*}{} + \ell_{\tau^*,k}}(N)\leq  s_{\lambda_0,\tau^*,\tau^{*}{} + \ell_{\tau^*,k}} \pa{1 -u_\alpha} \right)\enspace.$$
The same choice of $k_{\tau^{*}{}}$ and $\ell_{\tau^*}=\ell_{\tau^*,k_{\tau^*}}$ as in the above case where $\delta\in (0,R- \lambda_{0}]$ entails 
 \begin{equation} \label{eq4_alt6_N1}
    |\delta| \ell_{\tau^{*}{}} L \geq  s_{\lambda_0,\tau^*,\tau^{*}{} + \ell_{\tau^*}} \pa{1 -u_\alpha}    +  \sqrt{\frac{(\lambda_{0}+ \delta)  \ell_{\tau^{*}{}} L}{\beta}}\enspace,
    \end{equation}
and since $E_{\lambda}\cro{S_{\tau^{*}{}, \tau^{*}{} + \ell_{\tau^*}}(N)}  = - \vert \delta \vert \ell_{\tau^{*}{}} L$ and
$\mathrm{Var}_{\lambda}\cro{S_{\tau^{*}{}, \tau^{*}{} + \ell_{\tau^*}}(N)} =(\lambda_{0} + \delta) \ell_{\tau^{*}{}} L  $, we obtain in the same way
\[P_{\lambda} \Big(-S_{\tau^{*}{}, \tau^{*}{} + \ell_{\tau^*}}(N)\leq  s_{\lambda_0,\tau^*,\tau^{*}{} + \ell_{\tau^*}} \pa{1 -u_\alpha} \Big)\leq \beta\enspace.\]

 \smallskip
 
Finally \eqref{distance_alt6_N1} leads in both cases to $P_{\lambda}\pa{\phi_{6,\alpha}^{(1)}(N)=0} \leq \beta$, which allows to conclude that
\begin{multline*}
\SRb\Big(\phi_{6,\alpha}^{(1)},\calS_{\bbul,\tau^*,\bbul\bbul\bbul}[\lambda_0,R]\Big) \leq \sqrt{2} \max \Bigg(2\sqrt{\frac{R\log \left(2\lfloor \log_2 L\rfloor/\alpha  \right)}{3L}}~,\\
 2  \sqrt{\frac{2 \lambda_{0}\log \left(2\lfloor \log_2 L\rfloor/\alpha  \right)}{L}} + 2\sqrt{\frac{R}{\beta L}}  ~,~ \frac{R}{\sqrt{L}} \Bigg)\enspace.
\end{multline*}

\emph{$(ii)$ Control of the second kind error rate of $\phi_{6,\alpha}^{(2)}$.}

\smallskip

Let $\lambda$ in $\calS_{\bbul,\tau^*,\bbul\bbul\bbul}[\lambda_0,R]$ be such that $\lambda= \lambda_{0}+  \delta \mathds{1}_{(\tau^{*}{}, \tau^{*}{} + \ell]}$, with $\delta$ in $(-\lambda_{0}, R- \lambda_{0}]\setminus\{0\}$, $\ell$ in $(0,1-\tau^{*})$, and such that 
 \begin{multline} \label{distance_alt6_N2}
   d_2\pa{\lambda,\calS_0[\lambda_0]}\geq \max \Bigg(  4 \sqrt{\frac{2 \lambda_{0}\log \pa{3/u_\alpha}}{L}} +2 \sqrt{2\sqrt{\frac{2}{\beta}}\frac{R}{L}}~,~2\sqrt{\frac{2\sqrt{2}R\log \pa{3/u_\alpha}}{3L}}~,\\
4\pa{\frac{2}{3}}^{1/3} \lambda_{0}^{1/6}R^{1/3}\sqrt{\frac{\log \pa{3/u_\alpha}}{L}}~,~16 \sqrt{\frac{R}{\beta L}}~,~\frac{\sqrt{2}R}{\sqrt{L}}\Bigg)\enspace.
 \end{multline}
 
 Let us prove that $P_{\lambda}\pa{\phi_{6,\alpha}^{(2)}(N)=0} \leq \beta$. 

\smallskip

Noticing that
$$ P_{\lambda}\pa{\phi_{6,\alpha}^{(2)}(N)=0}  \leq \inf_{k\in \lbrace 1,\ldots,\lfloor \log_{2} L \rfloor \rbrace} P_{\lambda} \left(  T_{\tau^*, \tau^*+\ell_{\tau^*,k} }(N) \leq  t_{\lambda_0,\tau^*,\tau^{*}{} + \ell_{\tau^*,k}} \pa{1 -u_\alpha} \right)\enspace,$$
one can see that it is enough to exhibit some $k$ in $\set{1,\ldots,\lfloor \log_{2} L \rfloor}$ satisfying
\[P_{\lambda} \left( T_{\tau^{*}{}, \tau^{*}{} + \ell_{\tau^*,k}}(N)\leq  t_{\lambda_0,\tau^*,\tau^{*}{} + \ell_{\tau^*,k}} \pa{1 -u_\alpha} \right) \leq \beta\enspace.\]

From \eqref{distance_alt6_N2}, we deduce that 
 $d_2^2\pa{\lambda,\calS_0[\lambda_0]} \geq 2 R^2/L$ which entails $\ell \geq 2/L.$ Therefore, as in the above part $(i)$ of the proof, let $k_{\tau^{*}{}}$ in $\lbrace 1,\ldots, \lfloor \log_{2} L \rfloor\rbrace$  be such that $(1-\tau^{*}{}) 2^{-k_{\tau^{*}{}}} \leq \ell <  (1-\tau^{*}{}) 2^{-k_{\tau^{*}{}}+1}$,  and consider
 $ \ell_{\tau^{*}{}}= (1-\tau^{*}{}) 2^{-k_{\tau^{*}{}}}=\ell_{\tau^*,k_{\tau^*}}$. Then $\ell/2<\ell_{\tau^{*}{}}\leq \ell $ and
  \begin{equation} \label{eq1_alt6_N2}
   \delta^2 \ell_{\tau^{*}{}} > d_2^2\pa{\lambda,\calS_0[\lambda_0]}/2\enspace .
   \end{equation}

Moreover, we get from \eqref{distance_alt6_N2}
 \begin{multline*}
  d_2\pa{\lambda,\calS_0[\lambda_0]}\geq \max \Bigg(  4 \sqrt{\frac{2 \lambda_{0}\log \pa{3/u_\alpha}}{L}} + 2 \sqrt{2\sqrt{\frac{2}{\beta}}\frac{R}{L}}~,~2\sqrt{\frac{2\sqrt{2}R\log \pa{3/u_\alpha}}{3L}}~,\\
4\pa{\frac{2}{3}}^{1/3} \lambda_{0}^{1/6}R^{1/3}\sqrt{\frac{\log \pa{3/u_\alpha}}{L}}~,~16 \sqrt{\frac{R}{\beta L}}\Bigg)\enspace.
 \end{multline*}
On the one hand, this entails in particular
$d_2^4\pa{\lambda,\calS_0[\lambda_0]} \geq 128 R^2  \log^2 \pa{3/u_\alpha}/(9L^2)$, and then, with \eqref{eq1_alt6_N2}, $ d_2^4\pa{\lambda,\calS_0[\lambda_0]} \geq 64 d_2^2\pa{\lambda,\calS_0[\lambda_0]}  \log^{2} \pa{3/u_\alpha}/(9L^2  \ell
_{\tau^{*}{}})$.

On the other hand, this yields
\[d_2^3\pa{\lambda,\calS_0[\lambda_0]} \geq \frac{64}{3} d_2\pa{\lambda,\calS_0[\lambda_0]}  \sqrt{\frac{2\lambda_{0} \log^{3}  \pa{3/u_\alpha}}{L^3\ell_{\tau^{*}}}}\enspace,\]
using the same arguments.
Therefore
 \begin{multline*}
  d_2^2\pa{\lambda,\calS_0[\lambda_0]}\geq \max \Bigg( 32 \frac{ \lambda_{0}\log \pa{3/u_\alpha}}{L} + 8\sqrt{\frac{2}{\beta}} \frac{R}{L}~,~\frac{64 \log^2 \pa{3/u_\alpha}}{9L^2  \ell_{\tau^{*}{}}}~,\\
\frac{64}{3} \sqrt{\frac{2\lambda_{0} \log^{3}  \pa{3/u_\alpha}}{L^3\ell_{\tau^{*}}}}~,~16d_2(\lambda,\calS_0[\lambda_0])\sqrt{\frac{R}{\beta L}}\Bigg)\enspace.
 \end{multline*}
Hence
 \begin{multline} \label{eq2_alt6_N2}
  \frac{ d_2^2\pa{\lambda,\calS_0[\lambda_0]}}{2} \geq  \frac{4 \lambda_{0} \log \pa{3/u_\alpha}}{ L}   + \sqrt{\frac{2}{\beta}}\frac{ R}{L} + \frac{8 \log^2 \pa{3/u_\alpha}}{9L^2 \ell_{\tau^{*}{}}}\\
  +\frac{8}{3} \sqrt{\frac{2\lambda_{0} \log^{3}  \pa{3/u_\alpha}}{L^3\ell_{\tau^{*}}}}   + 2d_2(\lambda,\calS_0[\lambda_0])\sqrt{\frac{R}{\beta L}}\enspace.
 \end{multline}
Since $\ell_{\tau^{*}{}} \leq \ell$, Lemma \ref{MomentsT} gives  $E_{\lambda}\cro{T_{\tau^{*}{}, \tau^{*}{} + \ell_{\tau^*}}(N)}= \delta^2 \ell_{\tau^{*}{}}$
and
\[\mathrm{Var}_{\lambda}\cro{T_{\tau^{*}{}, \tau^{*}{} + \ell_{\tau^*}}(N) } = \frac{4 \delta^2(\lambda_{0}+\delta)\ell_{\tau^{*}{}}}{L}  + \frac{2(\lambda_{0}+ \delta)^2}{L^2} \enspace.\] 
This leads with \eqref{eq1_alt6_N2} to 
 \begin{equation}\label{eq3_alt6_N2}
E_{\lambda}\cro{T_{\tau^{*}{}, \tau^{*}{} + \ell_{\tau^*}}(N)} \geq \frac{d_2^2\pa{\lambda,\calS_0[\lambda_0]}}{2}\enspace,
 \end{equation}
 and 
 \begin{equation}\label{eq4_alt6_N2}
 \mathrm{Var}_{\lambda}\cro{T_{\tau^{*}{}, \tau^{*}{} + \ell_{\tau^*}}(N) } \leq \frac{4 d_2^2\pa{\lambda,\calS_0[\lambda_0]} R}{L} +  \frac{2 R^2}{L^2}\enspace.
 \end{equation}
 With \eqref{eq3_alt6_N2} and \eqref{eq4_alt6_N2}, the inequality \eqref{eq2_alt6_N2} yields
 \begin{multline} \label{bcontrol3Pb3}
E_{\lambda}\cro{T_{\tau^{*}{}, \tau^{*}{} + \ell_{\tau^*}}(N)} \geq  \frac{4 \lambda_{0} \log \pa{3/u_\alpha}}{ L} + \frac{8 \log^2 \pa{3/u_\alpha}}{9L^2 \ell_{\tau^{*}{}}}+\frac{8}{3} \sqrt{\frac{2\lambda_{0} \log^{3} \pa{3/u_\alpha}}{L^{3}\ell_{\tau^{*}}}}\\
 +  \sqrt{\frac{\mathrm{Var}_{\lambda}\cro{T_{\tau^{*}{}, \tau^{*}{} + \ell_{\tau^*}}(N) } }{\beta}}\enspace.
 \end{multline}
 Moreover, Lemma \ref{QuantilesT} gives
 \[t_{\lambda_0,\tau^*,\tau^{*}{} + \ell_{\tau^*}} \pa{1 -u_\alpha} \leq 2\lambda_{0}^2  \ell_{\tau^{*}{}} \pa{g^{-1} \pa{\frac{\log \pa{3/u_\alpha}}{\lambda_{0} \ell_{\tau^{*}{}} L}}}^2\enspace,  \] where  $g^{-1}(x) \leq 2x/3+ \sqrt{2x}$ for all $x>0$  (see \eqref{UBginv}), and then
 \begin{align*} 
t_{\lambda_0,\tau^*,\tau^{*}{} + \ell_{\tau^*}} \pa{1 -u_\alpha}  &\leq   \frac{4\lambda_{0}\log \pa{3/u_\alpha}}{ L} +  \frac{8 \log^2 \pa{3/u_\alpha}}{9L^2  \ell_{\tau^{*}{}}} +\frac{8}{3}\sqrt{\frac{2\lambda_{0} \log^{3}  \pa{3/u_\alpha}}{L^3\ell_{\tau^{*}}}}\\
&\leq E_{\lambda}\cro{T_{\tau^{*}{}, \tau^{*}{} + \ell_{\tau^*}}(N)} -\sqrt{\frac{\mathrm{Var}_{\lambda}\cro{T_{\tau^{*}{}, \tau^{*}{} + \ell_{\tau^*}}(N) } }{\beta}}\enspace.
 \end{align*}
We obtain with the Bienayme-Chebyshev inequality
 \begin{align*}
P_{\lambda} \Big(& T_{\tau^{*}{}, \tau^{*}{} + \ell_{\tau^*}}(N)\leq  t_{\lambda_0,\tau^*,\tau^{*}{}+ \ell_{\tau^*}} \pa{1 -u_\alpha} \Big)\\
 &\leq  P_{\lambda} \pa{T_{\tau^{*}{}, \tau^{*}{} + \ell_{\tau^*}}(N)\leq  E_{\lambda}\cro{T_{\tau^{*}{}, \tau^{*}{} + \ell_{\tau^*}}(N)} -\sqrt{\frac{\mathrm{Var}_{\lambda}\cro{T_{\tau^{*}{}, \tau^{*}{} + \ell_{\tau^*}}(N) } }{\beta}} }\\
 &\leq \beta \enspace,
 \end{align*}
which entails $P_{\lambda}\pa{\phi_{6,\alpha}^{(2)}(N)=0} \leq \beta$.

This finally allows to conclude that
\begin{multline*}
\SRb\Big(\phi_{6,\alpha}^{(2)},\calS_{\bbul,\tau^*,\bbul\bbul\bbul}[\lambda_0,R]\Big) \leq \max \Bigg(  4 \sqrt{\frac{2 \lambda_{0}\log \pa{3\lfloor \log_2 L\rfloor/\alpha }}{L}} +2 \sqrt{2\sqrt{\frac{2}{\beta}}\frac{R}{L}}~,\\
2\sqrt{\frac{2\sqrt{2}R \log \pa{3\lfloor \log_2 L\rfloor/\alpha }}{3L}}~,~4\pa{\frac{2}{3}}^{1/3} \lambda_{0}^{1/6}R^{1/3}\sqrt{\frac{\log \pa{3\lfloor \log_2 L\rfloor/\alpha }}{L}}~,~16 \sqrt{\frac{R}{\beta L}}~,~\frac{\sqrt{2}R}{\sqrt{L}}\Bigg)\enspace.
\end{multline*}

\subsection{Proof of Proposition \ref{LBalt7}}
Let $C_{\alpha, \beta}=1+4(1- \alpha -\beta)^2$, $r=(\lambda_{0}  \log C_{\alpha,\beta}/L )^{1/2}$ and $\lambda_r$ defined  for all $t$ in $(0,1)$ by
\[ \lambda_r(t) = \lambda_{0} +  \delta^*  \mathds{1}_{(1-r^2/{\delta^{*}}^2,1]}(t)\enspace.\]
Notice that for all $L \geq \lambda_{0} \log C_{\alpha,\beta}/{\delta^{*}}^2,$ we have $r\leq \abs{\delta^{*}}$ and $\lambda_r$ belongs to $\pa{\calS_{\delta^*,\bbul\bbul,1-\bbul\bbul}[\lambda_0]}_{r}$ in the notation of Lemma \ref{mSR}. We get now from Lemma \ref{lemmegirsanov} and Lemma \ref{momentPoisson}
\[ E_{\lambda_0} \left[\left( \frac{d P_{\lambda_r}}{dP_{0}} \right)^{2}(N)\right] = \exp \pa{  \frac{r^2 L}{\lambda_{0}} } = C_{\alpha,\beta}. \]
Lemmas \ref{lemmebayesien} and \ref{mSR} then entail  $\rho_\alpha\pa{\pa{\calS_{\delta^*,\bbul\bbul,1-\bbul\bbul}[\lambda_0]}_{r}} \geq \beta$ and $\mSRab\pa{\calS_{\delta^*,\bbul\bbul,1-\bbul\bbul}[\lambda_0]}\geq r.$

 \subsection{Proof of Proposition \ref{UBalt7}}

 The first kind error rate control is straightforward. As for the second kind error rate control, let $\lambda= \lambda_{0}+  \delta^*\un{(\tau,1]}$ belonging to $\calS_{\delta^*,\bbul\bbul,1-\bbul\bbul}[\lambda_0]$ with $\tau$ in $(0,1)$ and satisfying
\begin{multline}\label{distalt7}
d_2(\lambda, \calS_0[\lambda_0]) \geq \frac{2}{\sqrt{L}} \max \Bigg( 2 \sqrt{\frac{\lambda_{0} + \delta^{*}}{\beta}} ~,~ \sqrt{ \delta^{*} s_{\lambda_{0},\frac{\delta^{*}}{2}}^+\pa{ 1-\alpha}} \1{\delta^{*} >0} +\\
 \sqrt{\frac{\vert \delta^{*} \vert \log \pa{1/\alpha} }{\log \pa{\lambda_{0}/\pa{\lambda_{0} - \vert\delta^{*} \vert/ 2}}}} \1{-\lambda_{0} <\delta^{*} <0}\Bigg)\enspace.
 \end{multline}

Then the proof essentially follows the same line as the one of Proposition \ref{UBalt5} just replacing $\ell$ by $(1-\tau)$.

Assume that $\delta^{*}{}>0$ and recall that $s_{\lambda_{0},\delta^*/2}^+\pa{ 1-\alpha}$ defined in Lemma \ref{QuantilessupShifted} is a constant which does not depend on $L$.
The assumption \eqref{distalt7} implies that
$$ d_2(\lambda, \calS_0[\lambda_0]) \geq \frac{2}{\sqrt{L}}   \max \pa{ 2 \sqrt{\frac{\lambda_{0} + \delta^{*}{}}{\beta}} ~,~  \sqrt{\delta^{*}{} s_{\lambda_{0},\frac{\delta^{*}}{2}}^+\pa{ 1-\alpha} }}\enspace,$$
which entails
\begin{equation} \label{UBalt7eq1}
\frac{\delta^{*}}{2}(1-\tau) L \geq  \sqrt{\frac{(\lambda_{0} + \delta^{*})(1-\tau)   L}{\beta}} + s_{\lambda_{0},\frac{\delta^{*}}{2}}^+\pa{ 1-\alpha }\enspace.
\end{equation}

Then we get from the quantile upper bound \eqref{UBQuantilesupShifted2}
\begin{align*}
P_{\lambda}\pa{\phi_{7,\alpha}(N)=0} &\leq P_{\lambda} \pa{\sup_{\tau' \in (0,1)}S_{\delta^*,\tau',1}(N)\leq   s_{\lambda_{0},\frac{\delta^*}{2}}^+\pa{ 1-\alpha}}\\
&\leq P_{\lambda} \pa{ N(\tau, 1] - \pa{\lambda_{0} + \frac{\delta^{*}}{2}}(1-\tau) L     \leq  s_{\lambda_{0},\frac{\delta^*}{2}}^+\pa{ 1-\alpha} } \\
&\leq  P_{\lambda} \pa{  N(\tau, 1] - (\lambda_{0} + \delta^{*}{})(1-\tau) L \leq  s_{\lambda_{0},\frac{\delta^*}{2}}^+\pa{ 1-\alpha} - \frac{\delta^{*}{}}{2} (1-\tau)L  } \\
&\leq P_{\lambda} \pa{  N(\tau, 1]  - (\lambda_{0}+ \delta^{*}{})(1-\tau)L \leq    - \sqrt{\frac{(\lambda_{0}+ \delta^{*}{}) (1-\tau)  L}{\beta}}  } ~~ \text{with \eqref{UBalt7eq1}} \\
&\leq \beta\enspace.
\end{align*}

\smallskip

Assume now that $\delta^{*}{}$ is in $(- \lambda_{0}, 0)$.
The assumption \eqref{distalt7} implies that
$$ d_2(\lambda, \calS_0[\lambda_0]) \geq \frac{2}{\sqrt{L}} \max \Bigg( 2 \sqrt{\frac{\lambda_{0} + \delta^{*}}{\beta}} ~,~
 \sqrt{\frac{\vert \delta^{*} \vert \log \pa{1/\alpha} }{\log \pa{\lambda_{0}/\pa{\lambda_{0} - \vert\delta^{*} \vert/ 2}}}} \Bigg)\enspace,$$
which entails
$$ \frac{\vert \delta^{*}{} \vert}{2}(1-\tau) L \geq   \sqrt{\frac{(\lambda_{0}+ \delta^{*}{})(1-\tau)  L}{\beta}} + \frac{\log \pa{1/\alpha} }{\log \pa{\lambda_{0}/\pa{\lambda_{0} - \vert\delta^{*} \vert/ 2} }}\enspace,$$
and then
\begin{multline}\label{UBalt7eq2}
\pa{\lambda_{0} - \frac{\vert \delta^{*}{} \vert}{2}}(1-\tau) L -   \frac{\log \pa{1/\alpha} }{\log \pa{\lambda_{0}/\pa{\lambda_{0} - \vert\delta^{*} \vert/ 2} }}- (\lambda_{0} + \delta^{*}{})(1-\tau)  L \\
\geq \sqrt{\frac{(\lambda_{0} + \delta^{*}{})(1-\tau) L}{\beta}}\enspace.
\end{multline}

We conclude with the following inequalities, deduced from \eqref{UBQuantilesupShifted2}, \eqref{UBalt7eq2} and the Bienayme-Chebyshev inequality successively:
\begin{align*}
P_{\lambda}\Big(\phi_{7,\alpha}(N)=0\Big)& \leq P_{\lambda} \pa{\sup_{\tau' \in (0,1)}S_{\delta^*,\tau',1}(N)\leq   s_{\lambda_0,\delta^*,\tau^*,L}^+(1-\alpha)}\\
&\leq P_{\lambda} \pa{ \pa{\lambda_{0} - \frac{\vert \delta^{*}{} \vert}{2}}(1-\tau)  L - N(\tau,1]  \leq  \frac{\log \pa{1/\alpha} }{\log \pa{\lambda_{0}/\pa{\lambda_{0} - \vert\delta^{*} \vert/ 2} }} } \\
&\leq    P_\lambda \pa{ N(\tau, 1] - (\lambda_{0}+ \delta^{*}{})(1-\tau)  L \geq \sqrt{\frac{(\lambda_{0} + \delta^{*}{})(1-\tau)  L}{\beta}}    } \\
&\leq \beta\enspace.
\end{align*}
Coming back to the formulation of assumption \eqref{distalt7}, we therefore can take in the statement of Proposition \ref{UBalt7}
\begin{multline*}
C(\alpha, \beta, \lambda_{0},\delta^*)=2\max \Bigg(2 \sqrt{\frac{\lambda_{0} + \delta^{*}}{\beta}} ~,~ \sqrt{ \delta^{*} s_{\lambda_{0},\frac{\delta^{*}}{2}}^+\pa{ 1-\alpha}} \1{\delta^{*} >0} +\\
\sqrt{\frac{\vert \delta^{*} \vert \log \pa{1/\alpha} }{\log \pa{\lambda_{0}/\pa{\lambda_{0} - \vert\delta^{*} \vert/ 2}}}} \1{-\lambda_{0} <\delta^{*} <0}\Bigg)\enspace.
 \end{multline*}

 \subsection{Proof of Lemma \ref{LBalt8-pre}}

Let $\lambda_0>0$ and $\phi_{\alpha}$ a level-$\alpha$ test of the null hypothesis
  $\hzero \ "\lambda\in \calS_0[\lambda_0]=\{\lambda_0\}"$ versus the alternative $\hone\ "\lambda\in\calS_{\bbul,\bbul\bbul,1-\bbul\bbul}[\lambda_0]"$, with 
$\calS_{\bbul,\bbul\bbul,1-\bbul\bbul}[\lambda_0]$ defined by \eqref{alt8-pre}.  Let us fix some $r>0$. As in the proof of Lemma \ref{LBalt6-pre}, we can argue that
 if there exists $\lambda$ in $\calS_{\bbul,\bbul\bbul,1-\bbul\bbul}[\lambda_0]$ such that $d_2\pa{\lambda,\calS_0[\lambda_0]}\geq r$ satisfying $1- \alpha - \sqrt{K\pa{P_{\lambda}, P_{\lambda_0}}/2} \geq \beta$, then $\mSRab\pa{\calS_{\bbul,\bbul\bbul,1-\bbul\bbul}[\lambda_0]}  \geq r$.

Let us introduce for all $\tau$ in $(0,1)$, $\lambda_{r}= \lambda_{0} + r (1-\tau)^{-1/2} \mathds{1}_{(\tau,1]}$ in $\calS_{\bbul,\bbul\bbul,1-\bbul\bbul}[\lambda_0]$ which satisfies $d_2(\lambda_{r},\calS_0[\lambda_0])=r$. The end of the proof follows the same line as the one of Lemma~\ref{LBalt6-pre}, noticing that 
 \[ K\pa{P_{\lambda_{r}}, P_{\lambda_0}}= \log \left( 1+ \frac{r}{\lambda_{0} \sqrt{1-\tau}}  \right)\left( \lambda_{0}+ \frac{r}{\sqrt{1-\tau}} \right)(1-\tau) L-Lr \sqrt{1-\tau}\enspace, \] and choosing  $\tau$ close enough to $1$ in order to get $K\pa{P_{\lambda_{r}}, P_{\lambda_0}} \leq 2(1- \alpha -\beta)^2$, which yields $\mSRab\pa{\calS_{\bbul,\bbul\bbul,1-\bbul\bbul}[\lambda_0]}  \geq r$ for all $r>0$ and allows to conclude.

 \subsection{Proof of Proposition \ref{LBalt8}}

Assume that $L\geq 3$ and $\alpha+\beta<1/2$. As in the proof of Proposition \ref{LBalt6}, we consider $C'_{\alpha, \beta}=4(1- \alpha - \beta)^{2}$, $K_{\alpha,\beta,L}=\lceil (\log_2 L)/C'_{\alpha,\beta} \rceil$ and for $k$ in $\lbrace 1,\ldots, K_{\alpha,\beta,L}\rbrace$,
$\lambda_{k}= \lambda_{0}+ \delta_{k} \mathds{1}_{(\tau_k, 1]}$ with $ \tau_{k}  = 1- 2^{-k}$
 and
 $ \delta_{k} = (2^k\lambda_{0} \log \log L/ L)^{1/2}.$ Then, for every $k$ in $\lbrace 1,\ldots,K_{\alpha,\beta,L} \rbrace$, $d_2\pa{\lambda_{k},\calS_0[\lambda_0]} = \sqrt{\lambda_{0}\log \log L/L}$, and assuming that
\begin{equation}\label{assum_L_alt8}
\frac{\log \log L}{L^{1-1/C'_{\alpha,\beta}}} \leq \frac{(R- \lambda_{0})^2}{2\lambda_0}\enspace,
\end{equation}
 $\lambda_k$ belongs to $\calS_{\bbul,\bbul,1-\bbul\bbul}[\lambda_0,R]$. The proof then essentially follows the same arguments as the proof of Proposition \ref{LBalt6}. Thus, considering  a random variable $\kappa$ with uniform distribution on $\lbrace 1,\ldots, K_{\alpha,\beta,L} \rbrace$ and the probability distribution $\mu$ of $\lambda_{\kappa}$, we aim at proving that $E_{\lambda_0} [\left( dP_{\mu}/dP_{\lambda_0}\right)^{2}  ] \leq 1+C'_{\alpha,\beta}$, with $P_\mu$ defined as in Lemma \ref{lemmebayesien}, in order to conclude that  $\mSRab\pa{\calS_{\bbul,\bbul\bbul,1-\bbul\bbul}[\lambda_0,R]}\geq \sqrt{\lambda_{0}\log \log L/L}$.
 
 By definition,
\[\frac{dP_{\mu}}{dP_{\lambda_0}}(N)= \frac{1}{K_{\alpha,\beta,L}} \sum_{k=1}^{K_{\alpha,\beta,L}} \exp \left( \log \left( 1+ \frac{\delta_{k}}{\lambda_{0}} \right) N ( \tau_k, 1]  -L (1-\tau_k) \delta_{k}  \right)\enspace.\]
 Since $\tau_{k'}>\tau_{k}$ for all $k' >k$,
 \begin{multline*}
 \left( \frac{dP_{\mu}}{dP_{\lambda_0}}(N) \right)^{2} 
 = \frac{1}{K_{\alpha,\beta,L}^{2}} \sum_{k=1}^{K_{\alpha,\beta,L}} \exp \left( 2 \log \left( 1+ \frac{\delta_{k}}{\lambda_{0}} \right) N (\tau_k, 1]  -2L(1-\tau_k) \delta_{k}  \right) \\
+ \frac{2}{K_{\alpha,\beta,L}^{2}} \sum_{k=1}^{K_{\alpha,\beta,L}-1} \sum_{k'=k+1}^{K_{\alpha,\beta,L}} \exp \left(\left( \log \left( 1+ \frac{\delta_{k}}{\lambda_{0}} \right) +\log\left( 1+ \frac{\delta_{k'}}{\lambda_{0}} \right) \right)N(\tau_{k'}, 1]+\right.\\
 +\left.\log \left( 1+ \frac{\delta_{k}}{\lambda_{0}} \right) N(\tau_k, \tau_{k'}]-L(1-\tau_k) \delta_{k}  -L(1-\tau_{k'}) \delta_{k'} \right)\enspace,
\end{multline*}
hence
 \begin{multline*}
E_{\lambda_0} \left[\left( \frac{dP_{\mu}}{dP_{\lambda_0}} \right)^{2}   \right]= \frac{1}{K_{\alpha,\beta,L}^{2}} \sum_{k=1}^{K_{\alpha,\beta,L}} \exp\pa{\frac{L (1-\tau_k) \delta_{k}^{2}}{\lambda_{0}}}\\+ \frac{2}{K_{\alpha,\beta,L}^{2}} \sum_{k=1}^{K_{\alpha,\beta,L}-1} \sum_{k'=k+1}^{K_{\alpha,\beta,L}} \exp\pa{ \frac{L (1-\tau_{k'})\delta_{k} \delta_{k'}}{\lambda_{0}}}\enspace.
\end{multline*}

We then obtain the same expression of $E_{\lambda_0} \left[\left(dP_{\mu}/dP_{\lambda_0}\right)^{2}   \right]$ as in Equation \eqref{eq_LBalt6}, and the proof ends exactly as the one of Proposition \ref{LBalt6}, just replacing \eqref{assum_L_alt6_1} by  \eqref{assum_L_alt8} in the final argument.

\subsection{Proof of Proposition \ref{UBalt8}}

This proof is very similar to the one of Proposition \ref{UBalt6}. For the sake of completeness, we nevertheless detail it below.

 The control of the first kind error rates of the two tests $\phi_{8,\alpha}^{(1)}$ and $\phi_{8,\alpha}^{(2)}$ is straightforward using simple union bounds. 

\medskip

\emph{$(i)$ Control of the second kind error rate of $\phi_{8,\alpha}^{(1)}$.}

\smallskip

Let $\lambda$ in $\calS_{\bbul,\bbul\bbul,1-\bbul\bbul}[\lambda_0,R]$ be such that $\lambda= \lambda_{0}+  \delta \mathds{1}_{(\tau, 1]}$, with $\tau$ in $(0,1)$, $\delta$ in $(-\lambda_{0}, R- \lambda_{0}]\setminus\{0\}$, and such that 
 \begin{align} \label{distance_alt8_N1}
  d_2\pa{\lambda,\calS_0[\lambda_0]}&\geq \sqrt{2} \max \left(2\sqrt{\frac{R\log \left(2/u_\alpha \right)}{3L}}~,~ 2  \sqrt{\frac{2 \lambda_{0}\log \left(2/u_\alpha\right)}{L}} + 2\sqrt{\frac{R}{\beta L}}  ~,~ \frac{R}{\sqrt{L}}  \right)\enspace,
 \end{align}
as in \eqref{distance_alt6_N1}.

We prove here that $P_{\lambda}\pa{\phi_{8,\alpha}^{(1)}(N)=0} \leq \beta$, assuming first that $\delta$ belongs to $(0,R- \lambda_{0}]$. 

Noticing that
$$ P_{\lambda}\pa{\phi_{8,\alpha}^{(1)}(N)=0}  \leq \inf_{k\in \lbrace 1,\ldots, \lfloor \log_{2} L \rfloor \rbrace} P_{\lambda} \left( S_{\tau_k, 1}(N)\leq  s_{\lambda_0,\tau_k,1} \pa{1 -u_\alpha} \right)\enspace,$$
one can see that it is enough to exhibit some $k$ in $\set{1,\ldots, \lfloor \log_{2} L \rfloor}$ satisfying
\[P_{\lambda} \left( S_{\tau_k,1}(N)\leq  s_{\lambda_0,\tau_k,1} \pa{1 -u_\alpha} \right) \leq \beta\enspace.\]

Let $k_\tau=\lfloor -\log_2(1-\tau)\rfloor+1$. Since $0<1-\tau<1$, $k_\tau\geq 1$. Moreover, from \eqref{distance_alt8_N1}, we obtain 
 $d_2^2\pa{\lambda,\calS_0[\lambda_0]} \geq 2 R^2/L$ which entails $(1-\tau)\geq 2/L$ and $k_\tau\leq  \lfloor \log_2(L/2) \rfloor+1\leq \lfloor \log_2 L\rfloor$. Consider now $\tau_{k_\tau}=1-2^{-k_{\tau}},$ which satisfies  $(1-\tau)/2\leq 1-\tau_{k_\tau}<1-\tau $ as well as
  \begin{equation} \label{eq1_alt8_N1}
   \delta^2(1-\tau_{k_\tau}) \geq d_2^2\pa{\lambda,\calS_0[\lambda_0]}/2\enspace .
   \end{equation}

We get from \eqref{distance_alt8_N1}
\[d_2\pa{\lambda,\calS_0[\lambda_0]}\geq \sqrt{2} \max \left(2\sqrt{\frac{\delta\log \left(2/u_\alpha \right)}{3L}}~,~ 2  \sqrt{\frac{2 \lambda_{0}\log \left(2/u_\alpha\right)}{L}} + 2\sqrt{\frac{\lambda_0+\delta}{\beta L}} \right)\enspace,\]
which gives with \eqref{eq1_alt8_N1}
\[\delta \sqrt{1-\tau_{k_\tau}} \geq    \max \left(2\sqrt{\frac{\delta\log \left(2/u_\alpha \right)}{3L}}~,~ 2  \sqrt{\frac{2 \lambda_{0}\log \left(2/u_\alpha\right)}{L}} + 2\sqrt{\frac{\lambda_0+\delta}{\beta L}} \right)\enspace.\]
This entails in particular
$\delta \pa{1-\tau_{k_\tau}} \geq (4/3) \log \left( 2 /u_\alpha  \right)/L$ and also 
\[\delta\pa{1-\tau_{k_\tau}} \geq 2 \sqrt{1-\tau_{k_\tau}}\pa{\sqrt{\frac{2 \lambda_{0}\log \left(2/u_\alpha\right)}{L}} + \sqrt{\frac{\lambda_0+\delta}{\beta L}} }\enspace.\]

Hence,
 \begin{multline} \label{eq2_alt8_N1}
    \delta \pa{1-\tau_{k_\tau}} L \geq
     \frac{2}{3} \log \left( 2/u_\alpha \right) +  \sqrt{2 \lambda_{0}\pa{1-\tau_{k_\tau}} L \log \left(2/u_\alpha\right) } +\\  \sqrt{\frac{(\lambda_{0}+ \delta) \pa{1-\tau_{k_\tau}}L}{\beta}}\enspace.
\end{multline}

On the one hand, since $\tau_{k_\tau}>\tau$, Lemma \ref{momentPoisson} gives $E_{\lambda}\cro{S_{\tau_{k_\tau}, 1}(N)}= \delta \pa{1-\tau_{k_\tau}} L$
 and
$\mathrm{Var}_{\lambda}\cro{S_{\tau_{k_\tau},1}(N) } =(\lambda_{0}+ \delta)\pa{1-\tau_{k_\tau}} L$. On the other hand, Lemma \ref{QuantilesAbsShifted} with the inequality   \eqref{UBginv} give
   \[s_{\lambda_0,\tau_{k_\tau},1} \pa{1 -u_\alpha}  \leq \frac{2}{3} \log \left( 2/u_\alpha \right) +  \sqrt{2 \lambda_{0}\pa{1-\tau_{k_\tau}} L \log \left(2/u_\alpha\right) } \enspace. \] 
 
Combined with \eqref{eq2_alt8_N1}, these computations yield
 \begin{equation} \label{eq3_alt8_N1}
E_{\lambda}\cro{S_{\tau_{k_\tau},1}(N)} \geq s_{\lambda_0,\tau_{k_\tau},1} \pa{1 -u_\alpha} +  \sqrt{\mathrm{Var}_{\lambda}\cro{S_{\tau_{k_\tau},1}(N)}/\beta}\enspace.
\end{equation}

The Bienayme-Chebyshev then leads to
\[P_{\lambda} \Big(S_{\tau_{k_\tau},1}(N)\leq  s_{\lambda_0,\tau_{k_\tau},1} \pa{1 -u_\alpha} \Big)\leq \beta\enspace.\]

Assume now that $\delta$ belongs to $(-\lambda_0,0)$ and notice that
$$ P_{\lambda}\pa{\phi_{8,\alpha}^{(1)}(N)=0}  \leq \inf_{k\in \lbrace 1,\ldots, \lfloor \log_{2} L \rfloor \rbrace} P_{\lambda} \left( -S_{\tau_k,1}(N)\leq  s_{\lambda_0,\tau_k,1} \pa{1 -u_\alpha} \right)\enspace.$$
The same choice of $k_{\tau}$ as in the above case where $\delta\in (0,R- \lambda_{0}]$ entails 
 \begin{equation} \label{eq4_alt8_N1}
    |\delta| \pa{1-\tau_{k_\tau}}  L \geq  s_{\lambda_0,\tau_{k_\tau},1} \pa{1 -u_\alpha}    +  \sqrt{\frac{(\lambda_{0}+ \delta)  \pa{1-\tau_{k_\tau}} L}{\beta}}\enspace,
    \end{equation}
and since $E_{\lambda}\cro{S_{\tau_{k_\tau},1}(N)}  = - \vert \delta \vert \pa{1-\tau_{k_\tau}} L$ and
$\mathrm{Var}_{\lambda}\cro{S_{\tau_{k_\tau},1}(N)} =(\lambda_{0} + \delta) \pa{1-\tau_{k_\tau}}L  $, we obtain in the same way
\[P_{\lambda} \Big(-S_{\tau_{k_\tau},1}(N)\leq  s_{\lambda_0,\tau_{k_\tau},1} \pa{1 -u_\alpha} \Big)\leq \beta\enspace.\]

 \smallskip
 
Finally, \eqref{distance_alt8_N1} leads in both cases to $P_{\lambda}\pa{\phi_{8,\alpha}^{(1)}(N)=0} \leq \beta$, which allows to conclude that
\begin{multline*}
\SRb\Big(\phi_{8,\alpha}^{(1)},\calS_{\bbul,\bbul\bbul,1-\bbul\bbul}[\lambda_0,R]\Big) \leq \sqrt{2} \max \Bigg(2\sqrt{\frac{R\log \left(2\lfloor \log_2 L\rfloor/\alpha  \right)}{3L}}~,\\
 2  \sqrt{\frac{2 \lambda_{0}\log \left(2\lfloor \log_2 L\rfloor/\alpha  \right)}{L}} + 2\sqrt{\frac{R}{\beta L}}  ~,~ \frac{R}{\sqrt{L}} \Bigg)\enspace.
\end{multline*}

\emph{$(ii)$ Control of the second kind error rate of $\phi_{8,\alpha}^{(2)}$.}

\smallskip

Let $\lambda$ in $\calS_{\bbul,\bbul\bbul,1-\bbul\bbul}[\lambda_0,R]$ be such that $\lambda= \lambda_{0}+  \delta \mathds{1}_{(\tau, 1]}$, with $\tau$ in $(0,1)$, $\delta$ in $(-\lambda_{0}, R- \lambda_{0}]\setminus\{0\}$, and such that 
 \begin{multline} \label{distance_alt8_N2}
  d_2\pa{\lambda,\calS_0[\lambda_0]}\geq \max \Bigg(  4 \sqrt{\frac{2 \lambda_{0}\log \pa{3/u_\alpha}}{L}} +2 \sqrt{2\sqrt{\frac{2}{\beta}}\frac{R}{L}}~,~2\sqrt{\frac{2\sqrt{2}R\log \pa{3/u_\alpha}}{3L}}~,\\
4\pa{\frac{2}{3}}^{1/3} \lambda_{0}^{1/6}R^{1/3}\sqrt{\frac{\log \pa{3/u_\alpha}}{L}}~,~16 \sqrt{\frac{R}{\beta L}}~,~\frac{\sqrt{2}R}{\sqrt{L}}\Bigg)\enspace,
 \end{multline}
as in \eqref{distance_alt6_N2}.

Let us prove that this implies that  $P_{\lambda}\pa{\phi_{8,\alpha}^{(2)}(N)=0} \leq \beta$. 

\smallskip

Notice that
$$ P_{\lambda}\pa{\phi_{8,\alpha}^{(2)}(N)=0}  \leq \inf_{k\in \lbrace 1,\ldots,\lfloor \log_{2} L \rfloor \rbrace} P_{\lambda} \left(  T_{\tau_k, 1 }(N) \leq  t_{\lambda_0,\tau_k,1} \pa{1 -u_\alpha} \right)\enspace,$$
to see that one only needs to exhibit some $k$ in $\set{1,\ldots, \lfloor \log_{2} L \rfloor}$ satisfying
\[P_{\lambda} \left( T_{\tau_k,1}(N)\leq  t_{\lambda_0,\tau_k,1} \pa{1 -u_\alpha} \right) \leq \beta\enspace,\]
to obtain the expected result.

As in the above part $(i)$ of the proof, let $k_\tau=\lfloor -\log_2(1-\tau)\rfloor+1$ and $\tau_{k_\tau}=1-2^{-k_{\tau}}$. From \eqref{distance_alt8_N2} which in particular entails $d_2^2\pa{\lambda,\calS_0[\lambda_0]} \geq 2 R^2/L$, we get that $k_\tau$ actually belongs to $\set{1,\ldots, \lfloor \log_{2} L \rfloor}$. Furthermore, by definition, \begin{equation} \label{eq1_alt8_N2}
   \delta^2\pa{1-\tau_{k_\tau}}\geq d_2^2\pa{\lambda,\calS_0[\lambda_0]}/2\enspace .
   \end{equation}

Now, we also deduce from \eqref{distance_alt8_N2} that
 \begin{multline*}
  d_2\pa{\lambda,\calS_0[\lambda_0]}\geq \max \Bigg(  4 \sqrt{\frac{2 \lambda_{0}\log \pa{3/u_\alpha}}{L}} +2 \sqrt{2\sqrt{\frac{2}{\beta}}\frac{R}{L}}~,~\sqrt{\frac{\sqrt{2}R\log \pa{3/u_\alpha}}{3L}}~,\\
4\pa{\frac{2}{3}}^{1/3} \lambda_{0}^{1/6}R^{1/3}\sqrt{\frac{\log \pa{3/u_\alpha}}{L}}~,~16\sqrt{\frac{R}{\beta L}}\Bigg)\enspace.
 \end{multline*}
This entails on the one hand that
$d_2^4\pa{\lambda,\calS_0[\lambda_0]} \geq 128 R^2  \log^2 \pa{3/u_\alpha}/(9L^2)$, and with \eqref{eq1_alt8_N2}, $ d_2^4\pa{\lambda,\calS_0[\lambda_0]} \geq 64 d_2^2\pa{\lambda,\calS_0[\lambda_0]}  \log^{2} \pa{3/u_\alpha}/(9L^2 \pa{1-\tau_{k_\tau}})$.
On the other hand, we deduce that
\[d_2^3\pa{\lambda,\calS_0[\lambda_0]} \geq \frac{64}{3} d_2\pa{\lambda,\calS_0[\lambda_0]}  \sqrt{\frac{2\lambda_{0} \log^{3}  \pa{3/u_\alpha}}{L^3\pa{1-\tau_{k_\tau}}}}\enspace.\]
Therefore
 \begin{multline*}
  d_2^2\pa{\lambda,\calS_0[\lambda_0]}\geq \max \Bigg( 32 \frac{ \lambda_{0}\log \pa{3/u_\alpha}}{L} + 8\sqrt{\frac{2}{\beta}}\frac{ R}{L}~,~\frac{64 \log^2 \pa{3/u_\alpha}}{9L^2  \pa{1-\tau_{k_\tau}}}~,\\
\frac{64}{3}\sqrt{\frac{2\lambda_{0} \log^{3}  \pa{3/u_\alpha}}{L^3\pa{1-\tau_{k_\tau}}}}~,~16d_2(\lambda,\calS_0[\lambda_0])\sqrt{\frac{R}{\beta L}}\Bigg)\enspace.
 \end{multline*}
Hence,
 \begin{multline} \label{eq2_alt8_N2}
  \frac{ d_2^2\pa{\lambda,\calS_0[\lambda_0]}}{2} \geq  \frac{4 \lambda_{0} \log \pa{3/u_\alpha}}{ L}   + \sqrt{\frac{2}{\beta}}\frac{ R}{L} + \frac{8 \log^2 \pa{3/u_\alpha}}{9L^2 \pa{1-\tau_{k_\tau}}}\\
  +\frac{8}{3} \sqrt{\frac{2\lambda_{0} \log^{3}  \pa{3/u_\alpha}}{L^3\pa{1-\tau_{k_\tau}}}}   + 2d_2(\lambda,\calS_0[\lambda_0])\sqrt{\frac{R}{\beta L}}\enspace.
 \end{multline}
Since $\tau_{k_\tau}>\tau$, Lemma \ref{MomentsT} gives that $E_{\lambda}\cro{T_{\tau_{k_\tau}, 1}(N)}= \delta^2 \pa{1-\tau_{k_\tau}}$
and
\[\mathrm{Var}_{\lambda}\cro{T_{\tau_{k_\tau}, 1}(N) } = \frac{4 \delta^2(\lambda_{0}+\delta)\pa{1-\tau_{k_\tau}}}{L}  + \frac{2(\lambda_{0}+ \delta)^2}{L^2} \enspace.\] 
From \eqref{eq1_alt8_N2}, we get
 \begin{equation}\label{eq3_alt8_N2}
E_{\lambda}\cro{T_{\tau_{k_\tau}, 1}(N)} \geq \frac{d_2^2\pa{\lambda,\calS_0[\lambda_0]}}{2}\enspace.
 \end{equation}
Moreover,
 \begin{equation}\label{eq4_alt8_N2}
 \mathrm{Var}_{\lambda}\cro{T_{\tau_{k_\tau}, 1}(N) } \leq \frac{4 d_2^2\pa{\lambda,\calS_0[\lambda_0]} R}{L} +  \frac{2 R^2}{L^2}\enspace.
 \end{equation}
 With \eqref{eq3_alt8_N2} and \eqref{eq4_alt8_N2}, the inequality \eqref{eq2_alt8_N2} yields
 \begin{multline*}
E_{\lambda}\cro{T_{\tau_{k_\tau}, 1}(N)} \geq  \frac{4 \lambda_{0} \log \pa{3/u_\alpha}}{ L} + \frac{8 \log^2 \pa{3/u_\alpha}}{9L^2\pa{1-\tau_{k_\tau}}}+\frac{8}{3}\sqrt{\frac{2\lambda_{0} \log^{3} \pa{3/u_\alpha}}{L^{3}\pa{1-\tau_{k_\tau}}}}\\
 +  \sqrt{\frac{\mathrm{Var}_{\lambda}\cro{T_{\tau_{k_\tau}, 1}(N) } }{\beta}}\enspace.
 \end{multline*}

Furthermore, Lemma \ref{QuantilesT} gives
 \begin{align*} 
t_{\lambda_0,\tau_{k_\tau},1} \pa{1 -u_\alpha}  &\leq   \frac{4\lambda_{0}\log \pa{3/u_\alpha}}{ L} +  \frac{ 8 \log^2 \pa{3/u_\alpha}}{9L^2  \pa{1-\tau_{k_\tau}}} +\frac{8}{3}\sqrt{\frac{2\lambda_{0} \log^{3}  \pa{3/u_\alpha}}{L^3\pa{1-\tau_{k_\tau}}}}\\
&\leq E_{\lambda}\cro{T_{\tau_{k_\tau}, 1}(N)} -\sqrt{\frac{\mathrm{Var}_{\lambda}\cro{T_{\tau_{k_\tau}, 1}(N) } }{\beta}}\enspace, 
 \end{align*}
and the Bienayme-Chebyshev leads to
\[
P_{\lambda} \Big(T_{\tau_{k_\tau}, 1}(N)\leq  t_{\lambda_0,\tau_{k_\tau},1} \pa{1 -u_\alpha} \Big)\leq \beta \enspace.
\]
This entails the expected result $P_{\lambda}\pa{\phi_{8,\alpha}^{(2)}(N)=0} \leq \beta$, 
 and finally allows to conclude that
\begin{multline*}
\SRb\Big(\phi_{8,\alpha}^{(2)},\calS_{\bbul,\bbul\bbul,1-\bbul\bbul}[\lambda_0,R]\Big) \leq \max \Bigg(  4 \sqrt{\frac{2 \lambda_{0}\log \pa{3\lfloor \log_2 L\rfloor/\alpha }}{L}} +2 \sqrt{2\sqrt{\frac{2}{\beta}}\frac{R}{L}}~,\\
2\sqrt{\frac{2\sqrt{2}R\log \pa{3\lfloor \log_2 L\rfloor/\alpha }}{3L}}~,~4\pa{\frac{2}{3}}^{1/3}\lambda_{0}^{1/6}R^{1/3}\sqrt{\frac{\log \pa{3\lfloor \log_2 L\rfloor/\alpha }}{L}}~,~16 \sqrt{\frac{R}{\beta L}}~,~\frac{\sqrt{2}R}{\sqrt{L}}\Bigg)\enspace.
\end{multline*}

\subsection{Proof of Proposition \ref{LBalt9}}

Let $L \geq 2$. For all $k$ in $\set{1,\ldots,\lceil L^{3/4} \rceil}$, let us define $\lambda_k(t) = \lambda_{0} + \delta^{*}{} \mathds{1}_{(\tau_k, \tau_k + \ell]}(t)$ with
$\tau_k =k/L$, and $\ell =\lambda_{0} \log L /(2{\delta^*}^2 L)$. Then $\lambda_k$ belongs to $\calS_{\delta^*,\bbul\bbul,\bbul\bbul\bbul} [\lambda_0]$ for all $k$ in $\set{1,\ldots,\lceil L^{3/4} \rceil}$ as soon as
\begin{equation} \label{LBalt9eq1}
\frac{\lceil L^{3/4} \rceil}{L} + \frac{ \lambda_{0} \log L}{2 {\delta^*}^2 L}< 1\enspace,
\end{equation}
and it satisfies
\[d_2^2\pa{\lambda_k,\calS_0[\lambda_0]} = \lambda_{0} \log L/(2L)\enspace.\]
Considering a random variable $J$ uniformly distributed on $\set{1,\ldots, \lceil L^{3/4} \rceil}$ and
the distribution $\mu$ of $\lambda_{J}$ and using Lemma \ref{lemmebayesien}, one can see that it is enough to prove that $E_{\lambda_0} [\left( dP_{\mu}/dP_{\lambda_0}\right)^{2}  ] \leq 1+4(1- \alpha -\beta)^{2}$ to obtain the expected lower bound.
By definition, $ (dP_{\mu}/dP_{\lambda_0})(N)= E_{J} \left[ (dP_{\lambda_{J}}/dP_{\lambda_0})(N)  \right] $, therefore
 \[ \frac{dP_{\mu}}{dP_{\lambda_0}}(N)= \frac{1}{\lceil L^{3/4} \rceil} \sum_{k=1}^{\lceil L^{3/4} \rceil} \exp \left( \log \left( 1+ \frac{\delta^{*}{}}{\lambda_{0}} \right) N (\tau_k, \tau_k +\ell]  -L \delta^* \ell  \right)\enspace.\]
 We then expand the square as
  \begin{align*}
\left( \frac{dP_{\mu}}{dP_{\lambda_{0}}} \right)^{2} (N)
 = &\frac{1}{ \lceil L^{3/4} \rceil^{2}} \sum_{k=1}^{ \lceil L^{3/4} \rceil} \exp \left( 2 \log \left( 1+ \frac{\delta^{*}{}}{\lambda_{0}} \right) N(\tau_k, \tau_k + \ell]  -2 L \delta^{*}{} \ell  \right) \\
  &+ \frac{2}{\lceil L^{3/4} \rceil^{2}} \sum_{k=1}^{\lceil L^{3/4} \rceil -1} \sum_{k'=k+1}^{\lceil L^{3/4} \rceil} \exp \Bigg( 
   \log \left( 1+ \frac{\delta^{*}{}}{\lambda_{0}} \right) N (\tau_k, \tau_k+\ell]\\
  &+ \log \left(1+ \frac{\delta^{*}{}}{\lambda_{0}} \right) N (\tau_{k'}, \tau_{k'} +\ell]
 -2 L \delta^{*}{} \ell 
  \Bigg)\enspace.
  \end{align*}
  For $k$ in $\set{1,\ldots,\lceil L^{3/4} \rceil-1}$, setting $K_0(k)= \max \left( k' \in \set{k+1,\ldots,\lceil L^{3/4} \rceil}:\tau_{k'} < \tau_k +\ell \right)$, we may write
    \begin{align*}
\left( \frac{dP_{\mu}}{dP_{\lambda_{0}}} \right)^{2} (N)
 = &\frac{1}{ \lceil L^{3/4} \rceil^{2}} \sum_{k=1}^{ \lceil L^{3/4} \rceil} \exp \left( 2 \log \left( 1+ \frac{\delta^{*}{}}{\lambda_{0}} \right) N(\tau_k, \tau_k + \ell]  -2 L \delta^{*}{} \ell \right) \\
  &+ \frac{2}{\lceil L^{3/4} \rceil^{2}} \sum_{k=1}^{\lceil L^{3/4} \rceil -1} \sum_{k'=K_0(k)+1}^{\lceil L^{3/4} \rceil} \exp \Bigg( 
   \log \left( 1+ \frac{\delta^{*}{}}{\lambda_{0}} \right) N (\tau_k, \tau_k+\ell]\\
  &\qquad+ \log \left(1+ \frac{\delta^{*}{}}{\lambda_{0}} \right) N (\tau_{k'}, \tau_{k'} +\ell]
 -2 L \delta^{*}{} \ell
  \Bigg) \\
   &+ \frac{2}{\lceil L^{3/4} \rceil^{2}} \sum_{k=1}^{\lceil L^{3/4} \rceil -1} \sum_{k'=k+1}^{K_0(k)} \exp \Bigg( 
   \log \left( 1+ \frac{\delta^{*}{}}{\lambda_{0}} \right) \left( N (\tau_k, \tau_{k'}] + N (\tau_{k'}, \tau_k+\ell] \right)\\
  &\qquad+ \log \left(1+ \frac{\delta^{*}{}}{\lambda_{0}} \right) ( N (\tau_{k'}, \tau_{k} +\ell]+N (\tau_k +\ell, \tau_{k'}+\ell] )-2 L \delta^{*}{} \ell \Bigg)\enspace.
  \end{align*}
  Under $\hzero$, $N$ is a homogeneous Poisson process with intensity $\lambda_{0}$ with respect to the measure $\Lambda$. Thus
  \begin{align*}
   E_{\lambda_0} \left[  \left( \frac{dP_{\mu}}{dP_{\lambda_0}} \right)^{2} (N)  \right] = &\frac{\sqrt{L}}{\lceil L^{3/4} \rceil} + \frac{2}{\lceil L^{3/4} \rceil^2} \sum_{k=1}^{\lceil L^{3/4} \rceil-1} (\lceil L^{3/4} \rceil-K_0(k)) \\
   &+  \frac{2}{\lceil L^{3/4} \rceil^{2}} \sum_{k=1}^{\lceil L^{3/4} \rceil -1} \sum_{k'=k+1}^{K_0(k)} \exp \left(  \frac{\delta^{*}{}^2 L}{\lambda_{0}}(\tau_k - \tau_{k'} +\ell)  \right) \\
  = & \frac{\sqrt{L}}{\lceil L^{3/4} \rceil} + \frac{2}{\lceil L^{3/4} \rceil^2} \sum_{k=1}^{\lceil L^{3/4} \rceil-1} (\lceil L^{3/4} \rceil-K_0(k))\\
  & +  \frac{2 \sqrt{L}}{\lceil L^{3/4} \rceil^{2}} \sum_{k=1}^{\lceil L^{3/4} \rceil -1} \sum_{k'=k+1}^{K_0(k)} \exp \left(  -\frac{\delta^{*}{}^2}{\lambda_{0} } (k' - k)  \right) \\
   \leq  &  \frac{\sqrt{L}}{\lceil L^{3/4} \rceil} + \frac{2}{\lceil L^{3/4} \rceil^2} \sum_{k=1}^{\lceil L^{3/4} \rceil-1} (\lceil L^{3/4} \rceil-K_0(k))\\
   & +  \frac{2 \sqrt{L}}{\lceil L^{3/4} \rceil} \left(  e^{\delta^{*}{}^2 /\lambda_{0} }-1 \right)^{-1}\enspace.
   \end{align*}
   Since $K_0(k) \geq k+1$ for all $k$ in $\set{1,\ldots,\lceil L^{3/4} \rceil-1}$, notice that  
  \[\sum_{k=1}^{\lceil L^{3/4} \rceil-1} (\lceil L^{3/4} \rceil-K_0(k)) \leq \sum_{k=1}^{\lceil L^{3/4} \rceil-2} k \leq \frac{\lceil L^{3/4} \rceil^2}{2}\enspace,\]
  hence
  \[E_{\lambda_0} \left[  \left( \frac{dP_{\mu}}{dP_{\lambda_0}} \right)^{2} (N)  \right]  \leq   1+ \frac{\sqrt{L}}{\lceil L^{3/4} \rceil}\frac{e^{\delta^{*}{}^2 /\lambda_{0} }+1}{e^{\delta^{*}{}^2 /\lambda_{0} }-1}\enspace.\]
   Finally, assuming that 
 \begin{equation} \label{LBalt9eq2}
\frac{\sqrt{L}}{\lceil L^{3/4} \rceil}\frac{e^{\delta^{*}{}^2 /\lambda_{0} }+1}{e^{\delta^{*}{}^2 /\lambda_{0} }-1} \leq 4(1- \alpha - \beta)^2\enspace,
\end{equation}
we get
\[E_{\lambda_0} \left[  \left( \frac{dP_{\mu}}{dP_{\lambda_0}} \right)^{2} (N)  \right]  \leq   1+ 4(1- \alpha - \beta)^2\enspace.\]
Noticing that there exists $L_0(\alpha,\beta,\lambda_0,\delta^*)\geq 2$ such that for all $L\geq L_0(\alpha,\beta,\lambda_0,\delta^*)$, both assumptions \eqref{LBalt9eq1} and \eqref{LBalt9eq2} hold then allows to end the proof.

\subsection{Proof of Proposition \ref{UBalt10}}

 The control of the first kind error rates of the two tests $\phi_{9/10,\alpha}^{(1)}$ and $\phi_{9/10,\alpha}^{(2)}$ is straightforward using simple union bounds. 

\medskip

\emph{$(i)$ Control of the second kind error rate of $\phi_{9/10,\alpha}^{(1)}$.}

\smallskip

 Let $\lambda$ in $\calS_{\bbul,\bbul\bbul,\bbul\bbul\bbul} [\lambda_0,R]$. We may fix $\delta$ in $(-\lambda_0,R-\lambda_0]\setminus\{0\}$, $\tau$ in $(0,1)$, $\ell$ in $(0,1-\tau)$ such that $\lambda= \lambda_{0}+  \delta \mathds{1}_{(\tau, \tau + \ell]}.$ We assume by now that
 \begin{multline} \label{UBalt10eq1}
  d_2\pa{\lambda, \calS_{0}[\lambda_0]} \geq \sqrt{3} \max \Bigg(   \sqrt{\frac{4 R\log \left(2/u_\alpha)  \right)}{3 L}}~,\\
 2 \sqrt{\frac{2\lambda_{0}\log \left(2/u_\alpha \right)}{L}} + 2 \sqrt{\frac{R}{\beta L}}~,~\frac{R}{\sqrt{ L}}\Bigg)\enspace,
 \end{multline}
and we prove the inequality $P_{\lambda}\pa{\phi_{9/10,\alpha}^{(1)}(N)=0} \leq \beta.$ 

\smallskip

Assume first that $\delta$ belongs to $(0,R- \lambda_{0}].$  Noticing that
\begin{multline*}
 P_{\lambda}\pa{\phi_{9/10,\alpha}^{(1)}(N)=0}\\
  \leq\inf_{k \in \lbrace 0,\ldots,\lceil L \rceil-1 \rbrace} \inf_{k' \in \lbrace 1,\ldots,\lceil L \rceil-k \rbrace}  P_{\lambda} \left( S_{\frac{k}{\lceil L\rceil },\frac{k+k'}{\lceil L\rceil }}(N)\leq  s_{\lambda_0,\frac{k}{\lceil L\rceil },\frac{k+k'}{\lceil L\rceil }} \pa{1 -u_\alpha} \right)\enspace,
  \end{multline*}
one can see that it is enough to exhibit some $k_0$ in $\lbrace 0,\ldots,\lceil L \rceil-1 \rbrace$ and $k_0'$ in $\lbrace 1,\ldots,\lceil L \rceil-k_{0} \rbrace$ satisfying
\[P_{\lambda} \left( S_{\frac{k_0}{\lceil L\rceil },\frac{k_0+k_0'}{\lceil L\rceil }}(N)\leq  s_{\lambda_0,\frac{k_0}{\lceil L\rceil },\frac{k_0+k_0'}{\lceil L\rceil }} \pa{1 -u_\alpha} \right)\leq \beta\enspace.\]

  We get from \eqref{UBalt10eq1} that 
 \begin{equation}\label{UBalt10eq2}
 d_2^2\pa{\lambda, \calS_{0}[\lambda_0]}  \geq 3 R^2/ L >3 \delta^2/\lceil L \rceil\enspace,
 \end{equation}
  which entails
 \begin{equation}\label{UBalt10eq2bis}
\ell > 3/ \lceil L \rceil \textrm{ and }\tau<1-{3}/{\lceil L\rceil}\enspace.
\end{equation}  
 We therefore can define
 $ k_{0}= \min (  k \in \lbrace 0,\ldots, \lceil L \rceil-1 \rbrace,~ \tau \leq k/\lceil L \rceil  )$
 and $ k'_{0} = \max (k' \in \lbrace 1,\ldots,\lceil L \rceil-k_{0} \rbrace,~ (k_{0} + k')/\lceil L \rceil \leq \tau + \ell )$,
so that $\tau \leq k_{0}/\lceil L \rceil < (k_{0} + k'_{0})/\lceil L \rceil \leq \tau + \ell$.

  Since by definition $k_{0}/\lceil L \rceil - \tau < 1/\lceil L \rceil$ and $ \tau + \ell- (k_{0}+k_{0}')/\lceil L \rceil < 1/\lceil L \rceil$, notice that
\[\frac{k'_{0} }{\lceil L \rceil} = \ell - \left(  \left( \frac{k_{0}}{\lceil L \rceil}- \tau    \right) + \left( \tau + \ell - \frac{k_{0}+k_{0}'}{\lceil L \rceil}    \right)  \right) > \ell - \frac{2}{\lceil L \rceil}\enspace.\]
This, combined with \eqref{UBalt10eq2bis} and the expression of $d_2^2\pa{\lambda, \calS_{0}[\lambda_0]}=\delta^2\ell$, implies that
 \begin{equation} \label{UBalt10eq3}
 \delta^2  \frac{k'_{0} }{\lceil L \rceil}
> d_2^2\pa{\lambda, \calS_{0}[\lambda_0]}-  \frac{2\delta^2 }{\lceil L \rceil} > \frac{d_2^2\pa{\lambda, \calS_{0}[\lambda_0]}}{3}\enspace,
  \end{equation}
which yields with \eqref{UBalt10eq1}
 \begin{equation}\label{UBalt10eq4}
 \delta \sqrt{ \frac{k'_{0} }{\lceil L \rceil}} > \max \Bigg(   \sqrt{\frac{4 \delta \log \left(2/u_\alpha)  \right)}{3 L}}~,~
 2 \sqrt{\frac{2\lambda_{0}\log \left(2/u_\alpha \right)}{L}} + 2 \sqrt{\frac{\lambda_0+\delta}{\beta L}}\Bigg)\enspace.
 \end{equation}
We then deduce on the one hand
\[\frac{\delta  k'_{0}}{\lceil L \rceil} > \frac{4  \log \left( 2/u_\alpha \right)}{3 L}\enspace,\]
and on the other hand
\[\frac{\delta k'_{0}}{\lceil L \rceil} \geq \sqrt{ \frac{k'_{0}}{\lceil L \rceil}} \pa{2  \sqrt{ \frac{2 \lambda_{0}  \log \left(2/u_\alpha \right)}{L}}   + 2 \sqrt{\frac{\lambda_{0} + \delta}{\beta L}}}\enspace, \]
which together can be synthesized in
\[
\frac{\delta k'_{0} }{\lceil L \rceil}> 2 \max \left(    \frac{2 \log \left(2/u_\alpha\right)}{3L}~,~ \sqrt{ \frac{ k'_{0}} {\lceil L \rceil}} \left(\sqrt{\frac{2\lambda_{0}\log \left( 2/u_\alpha  \right)}{L}} + \sqrt{\frac{\lambda_{0} + \delta}{\beta L}} \right) \right)\enspace.
\]
Hence
\begin{equation} \label{UBalt10eq5}
\frac{\delta k'_{0} L }{\lceil L \rceil} >\frac{2\log \left(2/u_\alpha\right) }{3}  + \sqrt{\frac{2\log \left( 2/u_\alpha  \right)\lambda_{0} k'_{0} L}{\lceil L \rceil}} + \sqrt{\frac{\pa{\lambda_{0} + \delta}k'_{0} L}{\beta \lceil L \rceil}}\enspace.
\end{equation}

From Lemma \ref{momentPoisson}, we easily deduce that $E_\lambda \cro{S_{{k_0}/{\lceil L\rceil },{(k_0+k_0')}/{\lceil L\rceil }}(N)} = \delta k'_{0} L /\lceil L \rceil$ and $\mathrm{Var}_\lambda \pa{S_{{k_0}/{\lceil L\rceil },{(k_0+k_0')}/{\lceil L\rceil }}(N)}= (\lambda_{0} + \delta) k'_{0} L/\lceil L \rceil$, and from Lemma \ref{QuantilesAbsShifted}
\[ s_{\lambda_0, \frac{k_0}{\lceil L\rceil },\frac{k_0+k_0'}{\lceil L\rceil }}\pa{1-u_{\alpha}}  \leq \frac{\lambda_{0} k_{0}' L }{\lceil L \rceil} g^{-1} \left( \frac{ \log \left(2/u_\alpha\right)\lceil L \rceil}{\lambda_{0} k_{0}'  L}   \right)\enspace.\]
Using the upper bound  \eqref{UBginv}, this leads to 
  \begin{equation} \label{UBalt10eq6}
s_{\lambda_0, \frac{k_0}{\lceil L\rceil },\frac{k_0+k_0'}{\lceil L\rceil }}\pa{1-u_{\alpha}} \leq \frac{2\log \left(2/u_\alpha\right)}{3}  +  \sqrt{\frac{2\log \left(2/u_\alpha\right) \lambda_{0} k'_{0} L} {\lceil L \rceil}}\enspace.
 \end{equation}
The inequality \eqref{UBalt10eq5} therefore entails
\begin{equation} \label{UBalt10eq7}
E_\lambda \cro{S_{\frac{k_0}{\lceil L\rceil },\frac{k_0+k_0'}{\lceil L\rceil }}(N)}> s_{\lambda_0, \frac{k_0}{\lceil L\rceil },\frac{k_0+k_0'}{\lceil L\rceil }}\pa{1-u_{\alpha}}
 + \sqrt{\mathrm{Var}_\lambda\pa{S_{\frac{k_0}{\lceil L\rceil },\frac{k_0+k_0'}{\lceil L\rceil }}(N)}/{\beta}}\enspace.
\end{equation}
We conclude with \eqref{UBalt10eq7} and the Bienayme-Chebyshev inequality:
\begin{align*}
&P_\lambda \pa{S_{\frac{k_0}{\lceil L\rceil },\frac{k_0+k_0'}{\lceil L\rceil }}(N)\leq s_{\lambda_0, \frac{k_0}{\lceil L\rceil },\frac{k_0+k_0'}{\lceil L\rceil }}\pa{1-u_{\alpha}}}\\
&\leq P_\lambda \pa{S_{\frac{k_0}{\lceil L\rceil },\frac{k_0+k_0'}{\lceil L\rceil }}(N)-E_\lambda \cro{S_{\frac{k_0}{\lceil L\rceil },\frac{k_0+k_0'}{\lceil L\rceil }}(N)}\leq  - \sqrt{\mathrm{Var}_\lambda\pa{S_{\frac{k_0}{\lceil L\rceil },\frac{k_0+k_0'}{\lceil L\rceil }}(N)}/{\beta}}}\\
&\leq \beta\enspace.
\end{align*}
Assume now that $\delta$ belongs to $(-\lambda_{0}, 0)$ and notice that we have here
 \begin{multline*}
 P_{\lambda}\pa{\phi_{9/10,\alpha}^{(1)}(N)=0}\\
  \leq\inf_{k \in \lbrace 0,\ldots,\lceil L \rceil-1 \rbrace} \inf_{k' \in \lbrace 1,\ldots,\lceil L \rceil-k \rbrace}  P_{\lambda} \left( -S_{\frac{k}{\lceil L\rceil },\frac{k+k'}{\lceil L\rceil }}(N)\leq  s_{\lambda_0,\frac{k}{\lceil L\rceil },\frac{k+k'}{\lceil L\rceil }} \pa{1 -u_\alpha} \right)\enspace,
  \end{multline*}
 The same choice of $k_{0}$ and $k'_{0}$ as in the previous case yields $E_\lambda \cro{-S_{{k_0}/{\lceil L\rceil },{(k_0+k_0')}/{\lceil L\rceil }}(N)} = |\delta| k'_{0} L /\lceil L \rceil$ and $\mathrm{Var}_\lambda \pa{S_{{k_0}/{\lceil L\rceil },{(k_0+k_0')}/{\lceil L\rceil }}(N)}= (\lambda_{0} + \delta) k'_{0} L/\lceil L \rceil$, and we obtain
  \[P_{\lambda} \left( -S_{\frac{k_0}{\lceil L\rceil },\frac{k_0+k_0'}{\lceil L\rceil }}(N)\leq  s_{\lambda_0,\frac{k_0}{\lceil L\rceil },\frac{k_0+k_0'}{\lceil L\rceil }} \pa{1 -u_\alpha} \right)
   \leq \beta\enspace,\] 
in the same way,  notably replacing  $\delta$ by $\vert\delta\vert$ (except when it is involved in $\lambda_0+\delta$) in the rest of the proof.
 \smallskip
 
Coming back to the assumption \eqref{UBalt10eq1} and the definition of $u_\alpha$, one can finally claim that
\begin{multline*}
 \SRb\Big(\phi_{9/10,\alpha}^{(1)},\calS_{\bbul,\bbul\bbul,\bbul\bbul\bbul}[\lambda_0,R]\Big) \leq \sqrt{3} \max \Bigg(   \sqrt{\frac{4  R\log \left(\lceil L\rceil (\lceil L\rceil +1)/\alpha  \right)}{3 L}}~,\\
 2 \sqrt{\frac{2\lambda_{0}\log \left(\lceil L\rceil (\lceil L\rceil +1)/\alpha \right)}{L}} + 2 \sqrt{\frac{R}{\beta L}}~  ,~ \frac{R}{\sqrt{ L}}\Bigg)\enspace,
 \end{multline*}
which ends the proof of $(i)$.

\medskip

\emph{$(ii)$ Control of the second kind error rate of $\phi_{9/10,\alpha}^{(2)}$.}

\smallskip

  Let $\lambda$ in $\calS_{\bbul,\bbul\bbul,\bbul\bbul\bbul} [\lambda_0,R]$. There exist $\delta$ in $(-\lambda_0,R-\lambda_0]\setminus\{0\}$, $\tau$ in $(0,1)$, $\ell$ in $(0,1-\tau)$ such that $\lambda= \lambda_{0}+  \delta \mathds{1}_{(\tau, \tau + \ell]}.$ Let us assume now that
 \begin{multline} \label{UBalt10-2eq1}
  d_2\pa{\lambda, \calS_{0}[\lambda_0]} \geq \max \Bigg(R\sqrt{\frac{3}{ L}}~,~  4  \sqrt{\frac{3\lambda_{0}\log \pa{3/u_\alpha}}{L}} +2 \sqrt{\frac{3\sqrt{2}R}{\sqrt{\beta}L}}~,\\
 2 \sqrt{\frac{\sqrt{2} R \log \pa{3/u_\alpha}}{L}}~,~
32^{1/3} (6\lambda_{0})^{1/6} R^{1/3}\sqrt{\frac{\log \pa{3/u_ \alpha}}{L}}~,~ 24\sqrt{\frac{R}{\beta L}} \Bigg)\enspace,
 \end{multline}
and  prove that it entails $P_{\lambda}\pa{\phi_{9/10,\alpha}^{(2)}(N)=0} \leq \beta$. 

As in $(i)$, we begin by noticing that
 \begin{multline*}
 P_{\lambda}\pa{\phi_{9/10,\alpha}^{(2)}(N)=0}\\
  \leq\inf_{k \in \lbrace 0,...,\lceil L \rceil-1 \rbrace} \inf_{k' \in \lbrace 1,...,\lceil L \rceil-k \rbrace}  P_{\lambda} \left( T_{\frac{k}{\lceil L\rceil },\frac{k+k'}{\lceil L\rceil }}(N)\leq  t_{\lambda_0,\frac{k}{\lceil L\rceil },\frac{k+k'}{\lceil L\rceil }} \pa{1 -u_\alpha} \right)\enspace,
  \end{multline*}
so it suffices to find $k_{0}$ in $\lbrace 0,\ldots,\lceil L \rceil-1 \rbrace$ and $k_{0}'$ in $\lbrace 1,\ldots,\lceil L \rceil-k_{0} \rbrace$ satisfying
\[
P_{\lambda} \left( T_{\frac{k_0}{\lceil L\rceil },\frac{k_0+k_0'}{\lceil L\rceil }}(N)\leq  t_{\lambda_0,\frac{k_0}{\lceil L\rceil },\frac{k_0+k_0'}{\lceil L\rceil }} \pa{1 -u_\alpha} \right)\leq \beta\enspace,\] to obtain $P_{\lambda}\pa{\phi_{9/10,\alpha}^{(2)}(N)=0} \leq \beta$.

We get from \eqref{UBalt10-2eq1} that 
 $d_2^2\pa{\lambda, \calS_{0}[\lambda_0]} \geq 3 R^2/ L  >3 \delta^2/\lceil L \rceil 
 $ which entails
 \begin{equation}\label{UBalt10-2eq2}
\ell > 3/ \lceil L \rceil \textrm{ and }\tau<1-{3}/{\lceil L\rceil}\enspace.
\end{equation}  
 We therefore can define
 $ k_{0}= \min (  k \in \lbrace 0,\ldots, \lceil L \rceil-1 \rbrace,~ \tau \leq k/\lceil L \rceil  )$
 and $ k'_{0} = \max (k' \in \lbrace 1,\ldots,\lceil L \rceil-k_{0} \rbrace,~ (k_{0} + k')/\lceil L \rceil \leq \tau + \ell )$,
so that $\tau \leq k_{0}/\lceil L \rceil < (k_{0} + k'_{0})/\lceil L \rceil \leq \tau + \ell$.

As in $(i)$ above,   starting from the remark that $k_{0}/\lceil L \rceil - \tau < 1/\lceil L \rceil$ and $ \tau + \ell- (k_{0}+k_{0}')/\lceil L \rceil < 1/\lceil L \rceil$, which implies $k'_{0}/{\lceil L \rceil} > \ell -{2}/{\lceil L \rceil},$
 we obtain, combined with \eqref{UBalt10-2eq2} and the expression of $d_2^2\pa{\lambda, \calS_{0}[\lambda_0]}=\delta^2\ell$:
 \begin{equation} \label{UBalt10-2eq3}
 \delta^2  \frac{k'_{0} }{\lceil L \rceil}
> d_2^2\pa{\lambda, \calS_{0}[\lambda_0]}-  \frac{2\delta^2 }{\lceil L \rceil} > \frac{d_2^2\pa{\lambda, \calS_{0}[\lambda_0]}}{3}\enspace.
  \end{equation}
Moreover, we get from \eqref{UBalt10-2eq1}
 \begin{multline*}
d_2\pa{\lambda, \calS_{0}[\lambda_0]} > \max \Bigg(4  \sqrt{\frac{3\lambda_{0}\log \pa{3/u_\alpha}}{L}} +2 \sqrt{\frac{3\sqrt{2} (\lambda_0+\delta)}{\sqrt{\beta}L}}~,~
 2\sqrt{\frac{\sqrt{2}\vert\delta\vert \log \pa{3/u_\alpha}}{L}}~,\\
32^{1/3} (6\lambda_{0})^{1/6} \vert \delta\vert ^{1/3}\sqrt{\frac{\log \pa{3/u_ \alpha}}{L}}~,~24\sqrt{\frac{\lambda_0+\delta}{\beta L}}  \Bigg)\enspace,
 \end{multline*}
which entails
  \begin{multline*}
 d_2^2\pa{\lambda, \calS_{0}[\lambda_0]}> \max \Bigg( \frac{48 \lambda_{0} \log \pa{3/u_\alpha}}{ L}+ \frac{12 \sqrt{2} (\lambda_{0}+ \delta)}{\sqrt{\beta} L} ~,~   \frac{ 32 \delta^2 \log^2 \pa{3/u_\alpha}}{L^2  d_2^{2}\pa{\lambda, \calS_{0}[\lambda_0]}}  ~,\\
  \frac{32\sqrt{6 \lambda_{0}} \vert \delta \vert  \log^{3/2} \pa{3/u_\alpha}}{L^{3/2} d_2\pa{\lambda, \calS_{0}[\lambda_0]}}    ~,~ \frac{24 \sqrt{\lambda_{0} + \delta} d_2\pa{\lambda, \calS_{0}[\lambda_0]}}{\sqrt{\beta L}} \Bigg)\enspace.
 \end{multline*}
 Then with \eqref{UBalt10-2eq3},
  \begin{multline*}
 d_2^2\pa{\lambda, \calS_{0}[\lambda_0]}> \max \Bigg( \frac{48 \lambda_{0} \log \pa{3/u_\alpha}}{ L}+ \frac{12 \sqrt{2} (\lambda_{0}+ \delta)}{\sqrt{\beta} L} ~,~   \frac{32\log^2 \pa{3/u_\alpha}\lceil L \rceil}{ 3 k_0'  L^2}  ~,\\
   \frac{32 \log^{3/2} \pa{3/u_\alpha}}{L^{3/2}} \sqrt{\frac{2\lambda_{0} \lceil L \rceil}{k_{0}'}}  ~,~ \frac{24 \sqrt{\lambda_{0} + \delta} d_2\pa{\lambda, \calS_{0}[\lambda_0]}}{\sqrt{\beta L}} \Bigg)\enspace,
 \end{multline*}
hence
  \begin{multline}\label{UBalt10-2eq4}
 \frac{d_2^2\pa{\lambda, \calS_{0}[\lambda_0]} }{3} > \frac{4 \lambda_{0} \log \pa{3/u_\alpha}}{ L}+ \frac{ \sqrt{2} (\lambda_{0}+ \delta)}{\sqrt{\beta} L}+  \frac{8 \log^2 \pa{3/u_\alpha}\lceil L \rceil}{9k_0'  L^2} +\\
 \frac{8   \log^{3/2} \pa{3/u_\alpha}}{3 L^{3/2}} \sqrt{\frac{2\lambda_{0} \lceil L \rceil}{k_{0}'}}+ \frac{2\sqrt{\lambda_{0} + \delta} d_2\pa{\lambda, \calS_{0}[\lambda_0]}}{\sqrt{\beta L}} \Bigg)\enspace.
 \end{multline}  
  Using Lemma \ref{MomentsT}, we compute
  \begin{equation*} 
  E_{\lambda}\cro{T_{\frac{k_0}{\lceil L\rceil },\frac{k_0+k_0'}{\lceil L\rceil }}(N)} =  \delta^2 \frac{k'_{0}}{\lceil L \rceil} > \frac{d_2^2\pa{\lambda, \calS_{0}[\lambda_0]} }{3}~~ \text{with }\eqref{UBalt10-2eq3}\enspace, 
  \end{equation*}
  and 
  \begin{align*} 
   \mathrm{Var}_{\lambda}\pa{T_{\frac{k_0}{\lceil L\rceil },\frac{k_0+k_0'}{\lceil L\rceil }}(N)} &=  \frac{4(\lambda_{0}+\delta) \delta^2}{L} \frac{k_{0}'}{\lceil L \rceil}  + \frac{2(\lambda_{0} + \delta)^2}{L^2}\\
   & \leq \frac{4 (\lambda_{0} + \delta)d_2^2\pa{\lambda, \calS_{0}[\lambda_0]}}{L} + \frac{2(\lambda_{0} + \delta)^2}{L^2}\enspace.
  \end{align*}
 These computations combined with \eqref{UBalt10-2eq4} leads to 
  \begin{multline}\label{UBalt10-2eq5}
E_{\lambda}\cro{T_{\frac{k_0}{\lceil L\rceil },\frac{k_0+k_0'}{\lceil L\rceil }}(N)} > \frac{4 \lambda_{0} \log \pa{3/u_\alpha}}{ L}+  \frac{8 \log^2 \pa{3/u_\alpha}\lceil L \rceil}{9 k_0'  L^2} +\\
 \frac{8   \log^{3/2} \pa{3/u_\alpha}}{3L^{3/2}} \sqrt{\frac{2\lambda_{0} \lceil L \rceil}{k_{0}'}}
   + \sqrt{{ \mathrm{Var}_{\lambda}\pa{T_{\frac{k_0}{\lceil L\rceil },\frac{k_0+k_0'}{\lceil L\rceil }}(N)}}/{\beta}}\enspace.
 \end{multline}  
Furthermore, Lemma \ref{QuantilesT} gives
\[t_{\lambda_0,\frac{k_0}{\lceil L\rceil },\frac{k_0+k_0'}{\lceil L\rceil }}(1-u_\alpha) \leq \frac{2 \lambda_{0}^2 k_{0}'}{\lceil L \rceil} \pa{g^{-1} \pa{ \frac{ \log \pa{3/u_\alpha}\lceil L \rceil}{\lambda_{0}k_{0}' L } }}^2\enspace,\]
  where $g^{-1}(x) \leq 2x/3+ \sqrt{2x}$ for all $x>0$  (see \eqref{UBginv}), and then
 \begin{equation} \label{UBalt10-2eq6}
t_{\lambda_0,\frac{k_0}{\lceil L\rceil },\frac{k_0+k_0'}{\lceil L\rceil }}(1-u_\alpha) \leq   \frac{8 \log^2 \pa{3/u_\alpha}\lceil L \rceil}{9k_{0}'L^2} + \frac{8   \log^{3/2} \pa{3/u_\alpha}}{3L^{3/2}} \sqrt{\frac{2\lambda_{0} \lceil L \rceil}{k_{0}'}}+\frac{4\lambda_{0} \log \pa{3/u_\alpha}}{ L}\enspace.
 \end{equation}
It follows from \eqref{UBalt10-2eq5} and \eqref{UBalt10-2eq6} that
 \begin{multline*}
E_{\lambda}\cro{T_{\frac{k_0}{\lceil L\rceil },\frac{k_0+k_0'}{\lceil L\rceil }}(N)} >  t_{\lambda_0,\frac{k_0}{\lceil L\rceil },\frac{k_0+k_0'}{\lceil L\rceil }}(1-u_\alpha)  + \sqrt{{ \mathrm{Var}_{\lambda}\pa{T_{\frac{k_0}{\lceil L\rceil },\frac{k_0+k_0'}{\lceil L\rceil }}(N)}}/{\beta}}\enspace,
 \end{multline*}  
whereby
 \begin{multline*}
P_{\lambda} \left( T_{\frac{k_0}{\lceil L\rceil },\frac{k_0+k_0'}{\lceil L\rceil }}(N)\leq  t_{\lambda_0,\frac{k_0}{\lceil L\rceil },\frac{k_0+k_0'}{\lceil L\rceil }} \pa{1 -u_\alpha} \right)
\\
\leq P_{\lambda} \left( T_{\frac{k_0}{\lceil L\rceil },\frac{k_0+k_0'}{\lceil L\rceil }}(N)\leq E_{\lambda}\cro{T_{\frac{k_0}{\lceil L\rceil },\frac{k_0+k_0'}{\lceil L\rceil }}(N)} -\sqrt{\mathrm{Var}_{\lambda}\pa{T_{\frac{k_0}{\lceil L\rceil },\frac{k_0+k_0'}{\lceil L\rceil }}(N)}/\beta} \right)\enspace.
\end{multline*}

The Bienayme-Chebyshev inequality allows to conclude that 
\[P_{\lambda} \left( T_{\frac{k_0}{\lceil L\rceil },\frac{k_0+k_0'}{\lceil L\rceil }}(N)\leq  t_{\lambda_0,\frac{k_0}{\lceil L\rceil },\frac{k_0+k_0'}{\lceil L\rceil }} \pa{1 -u_\alpha} \right)
\leq \beta\enspace.\]

Coming back to the assumption \eqref{UBalt10-2eq1} and the definition of $u_\alpha$, one can finally claim that
\begin{multline*}
 \SRb\Big(\phi_{9/10,\alpha}^{(2)},\calS_{\bbul,\bbul\bbul,\bbul\bbul\bbul}[\lambda_0,R]\Big) \leq \max \Bigg(4  \sqrt{\frac{3\lambda_{0}\log \left(3\lceil L\rceil (\lceil L\rceil +1)/(2\alpha) \right)}{L}} +2 \sqrt{\frac{3\sqrt{2} R}{\sqrt{\beta}L}}~,\\
  2\sqrt{\frac{\sqrt{2}R \log \left(3\lceil L\rceil (\lceil L\rceil +1)/2\alpha \right)}{L}}~,~24\sqrt{\frac{R}{\beta L}} ~,~ R \sqrt{\frac{3}{ L}}~,\\
32^{1/3} (6\lambda_{0})^{1/6} R^{1/3}\sqrt{\frac{\log \left(3\lceil L\rceil (\lceil L\rceil +1)/2\alpha \right)}{L}}\Bigg)
\enspace,
 \end{multline*}
which ends the proof of $(ii)$.

\subsection{Proof of Proposition \ref{bNP1U}}

By definition of $b_{n,\ell^*}(u)$ as the $u$-quantile of a binomial distribution with  parameters $(n,\ell^{*}{})$, the Bienayme-Chebyshev inequality easily gives \[b_{n,\ell^*}(1-\alpha)\leq n\ell^{*}{} + \sqrt{n \ell^{*}{} (1-\ell^{*}{})/\alpha}\enspace,\]
for all $n$ in $\mathbb{N}.$ It also gives for every $n$ in $\mathbb{N}$ and every $\varepsilon>0$, $b_{n,\ell^*}(\alpha) > n\ell^{*}{} - \sqrt{n \ell^{*}{} (1-\ell^{*}{})/(\alpha - \varepsilon)}$. Therefore, letting $\varepsilon$ tending to $0$,  
\[b_{n,\ell^*}(\alpha)\geq n \ell^{*}{} - \sqrt{{n \ell^{*}{} (1-\ell^{*}{})}/{\alpha}}\enspace.\]  

\smallskip

(i) Assume first that $0<\delta^{*}{}<R$ and let $\lambda$ in $\mathcal{S}^u_{ \delta^{*}{}, \tau^{*}{}, \ell^{*}{}}[R].$

Setting $I(\lambda)= \int_0^1 \lambda(t) dt\leq R$,
the assumption \eqref{cond_alt1U} leads to
$$\delta^{*}{} \sqrt{\ell^{*}{} (1-\ell^{*}{})} \geq \frac{1}{\sqrt{L}} \pa{\sqrt{\frac{2(\lambda_0 + \delta^{*}{})}{\beta (1-\ell^{*}{})}} + 2\sqrt{\frac{I(\lambda)\ell^{*}{}}{\beta (1-\ell^{*}{})}} + \sqrt{ \frac{1}{\alpha} \pa{I(\lambda) +2\sqrt{\frac{I(\lambda)}{\beta L}} }    }        }\enspace.$$
Hence
$$   \delta^{*}{}  \ell^{*}{} (1-\ell^{*}{})L \geq \sqrt{\frac{2(\lambda_0 + \delta^{*}{}) \ell^{*}{} L}{\beta}} + 2\ell^{*}{} \sqrt{\frac{I(\lambda)L}{\beta }} + \sqrt{ \frac{\ell^{*}{} (1-\ell^{*}{})}{\alpha} \pa{I(\lambda)L +2\sqrt{\frac{I(\lambda)L}{\beta}} }    }        \enspace ,$$
and since $I(\lambda)= \lambda_0 + \delta^{*}{} \ell^{*}{}$,
\begin{multline}\label{bNP1U_eq1}
\pa{I(\lambda)L + 2\sqrt{\frac{I(\lambda) L}{\beta}}    }\ell^{*}{}  + \sqrt{\frac{\ell^{*}{}  (1-\ell^{*}{} )}{\alpha} \pa{I(\lambda)L +2\sqrt{\frac{I(\lambda)L}{\beta}}}   }\\ - (\lambda_0 + \delta^{*}{})\ell^{*}{}  L \leq -\sqrt{\frac{2(\lambda_0 + \delta^{*}{}) \ell^{*}{} L}{\beta}}\enspace.
\end{multline}
We get then the following inequalities
\begin{align*}
P_\lambda &\pa{\phi_{1,\alpha}^{u,+}(N) =0} \\
&=P_\lambda \pa{N(\tau^{*}{}, \tau^{*}{} + \ell^{*}{}]  < b_{N_1,\ell^*}(1-\alpha)} \\ 
&= P_\lambda \pa{N(\tau^{*}{}, \tau^{*}{} + \ell^{*}{}] < b_{N_1,\ell^*}(1-\alpha)~ ,~  |N_1-I(\lambda) L|\leq 2 \sqrt{{ I(\lambda)L}/{\beta}}}  \\
&\quad +   P_\lambda \pa{N_1<  I(\lambda) L - 2 \sqrt{{ I(\lambda)L}/{\beta}}         }+ P_\lambda \pa{N_1>    I(\lambda) L + 2 \sqrt{{ I(\lambda)L}/{\beta}}} \\
&\leq P_\lambda \pa{N(\tau^{*}{}, \tau^{*}{} + \ell^{*}{}] < b_{N_1,\ell^*}(1-\alpha)~ ,~  |N_1-I(\lambda) L|\leq 2 \sqrt{{ I(\lambda)L}/{\beta}}} + \frac{\beta}{2}\\
&\leq P_\lambda \pa{N(\tau^{*}{}, \tau^{*}{} + \ell^{*}{}] < N_1 \ell^{*}{}  + \sqrt{\frac{N_1 \ell^{*}{}  (1-\ell^{*}{} )}{\alpha}}~ ,~  |N_1-I(\lambda) L|\leq 2 \sqrt{{ I(\lambda)L}/{\beta}}} + \frac{\beta}{2}\\
&\leq P_\lambda \pa{N(\tau^{*}{}, \tau^{*}{} + \ell^{*}{}] <      \pa{ I(\lambda) L + 2 \sqrt{\frac{ I(\lambda)L}{\beta}}  }  \ell^{*}{}  + \sqrt{\frac{ \ell^{*}{}  (1-\ell^{*}{} )}{\alpha}   \pa{  I(\lambda) L + 2 \sqrt{\frac{ I(\lambda)L}{\beta}}}    }          } + \frac{\beta}{2}\\
&\leq P_\lambda \pa{N(\tau^{*}{}, \tau^{*}{} + \ell^{*}{}] - (\lambda_0 + \delta^{*}{})\ell^{*}{}  L  <  -\sqrt{\frac{2(\lambda_0 + \delta^{*}{}) \ell^{*}{}  L}{\beta}}    } + \frac{\beta}{2} ~~ \text{(thanks to \eqref{bNP1U_eq1})}\\
&\leq \beta ~~ \text{(Bienayme-Chebyshev)} \enspace.
\end{align*}
This concludes the proof of $(i)$.

$(ii)$ Assume now that $-R<\delta^{*}{}<0$ and let $\lambda$ in $\mathcal{S}^u_{ \delta^{*}{}, \tau^{*}{}, \ell^{*}{}}[R]$. The assumption \eqref{cond_alt1U} entails
$$ \vert \delta^{*}{} \vert \sqrt{\ell^{*}{} (1-\ell^{*}{})} \geq \frac{1}{\sqrt{L}} \pa{\sqrt{\frac{2(\lambda_0 + \delta^{*}{})}{\beta (1-\ell^{*}{})}} + 2\sqrt{\frac{I(\lambda)\ell^{*}{}}{\beta (1-\ell^{*}{})}} + \sqrt{ \frac{1}{\alpha} \pa{I(\lambda) +2\sqrt{\frac{I(\lambda)}{\beta L}} }    }        } \enspace.$$
Hence
$$  - \delta^{*}{}  \ell^{*}{} (1-\ell^{*}{})L \geq \sqrt{\frac{2(\lambda_0 + \delta^{*}{}) \ell^{*}{} L}{\beta}} + 2 \ell^{*}{} \sqrt{\frac{I(\lambda)L}{\beta }} + \sqrt{ \frac{ \ell^{*}{} (1-\ell^{*}{})}{\alpha} \pa{I(\lambda)L +2\sqrt{\frac{I(\lambda)L}{\beta}} }    }         \enspace,$$
and then
\begin{multline}\label{bNP1U_eq2}
\pa{I(\lambda)L - 2\sqrt{\frac{I(\lambda) L}{\beta}}    } \ell^{*}{} - \sqrt{\frac{ \ell^{*}{} (1-\ell^{*}{} )}{\alpha} \pa{I(\lambda)L +2\sqrt{\frac{I(\lambda)L}{\beta}}}   }\\- (\lambda_0 + \delta^{*}{}) \ell^{*}{} L \geq \sqrt{\frac{2(\lambda_0 + \delta^{*}{}) \ell^{*}{} L}{\beta}} \enspace.
\end{multline}

We get then
\begin{align*}
P_\lambda &\pa{\phi_{1,\alpha}^{u,-}(N)=0}\\
&= P_\lambda \pa{N(\tau^{*}{}, \tau^{*}{} + \ell^{*}{}] > b_{N_1,\ell^*}(\alpha)} \\ 
&= P_\lambda \pa{N(\tau^{*}{}, \tau^{*}{} + \ell^{*}{}]> b_{N_1,\ell^*}(\alpha)~ ,~   |N_1-I(\lambda) L|\leq 2 \sqrt{{ I(\lambda)L}/{\beta}}}  \\
&\quad +   P_\lambda \pa{N_1<  I(\lambda) L - 2 \sqrt{{ I(\lambda)L}/{\beta}}         }+ P_\lambda \pa{ N_1>    I(\lambda) L + 2 \sqrt{{ I(\lambda)L}/{\beta}}             } \\
&\leq P_\lambda \pa{N(\tau^{*}{}, \tau^{*}{} + \ell^{*}{}]> b_{N_1,\ell^*}(\alpha)~ ,~   |N_1-I(\lambda) L|\leq 2 \sqrt{{ I(\lambda)L}/{\beta}}} + \frac{\beta}{2} \\
&\leq P_\lambda     \left(     N(\tau^{*}{}, \tau^{*}{} + \ell^{*}{}] >      N_1 \ell^{*}{} - \sqrt{\frac{N_1 \ell^{*}{} (1-\ell^{*}{})}{\alpha}} ,~   |N_1-I(\lambda) L|\leq 2 \sqrt{{ I(\lambda)L}/{\beta}}\right)+ \frac{\beta}{2} \\
&\leq P_\lambda \pa{ N(\tau^{*}{}, \tau^{*}{} + \ell^{*}{}]>      \pa{ I(\lambda) L - 2 \sqrt{\frac{ I(\lambda)L}{\beta}}  } \ell^{*}{} - \sqrt{\frac{ \ell^{*}{} (1-\ell^{*}{})}{\alpha}   \pa{  I(\lambda) L + 2 \sqrt{\frac{ I(\lambda)L}{\beta}}}    }          }  + \frac{\beta}{2}\\
&\leq P_\lambda \pa{N(\tau^{*}{}, \tau^{*}{} + \ell^{*}{}]- (\lambda_0 + \delta^{*}{}) \ell^{*}{} L  >  \sqrt{\frac{2(\lambda_0 + \delta^{*}{}) \ell^{*}{} L}{\beta}}    } + \frac{\beta}{2} ~~ \text{(thanks to \eqref{bNP1U_eq2})}\\
&\leq \beta ~~ \text{(Bienayme-Chebyshev)}\enspace.
\end{align*}
This concludes the proof of $(ii)$.

\subsection{Proof of Proposition \ref{LB_alt2_u}}

Let us set $\lambda_0=R/2$ and introduce for $r>0$ the Poisson intensity $\lambda_r$ defined for all $t$ in $[0,1]$ by
\[ \lambda_r(t) = \lambda_{0} + \frac{r}{\sqrt{\ell^{*}{}(1- \ell^{*}{})}}  \mathds{1}_{(\tau^{*}{}, \tau^{*}{}+ \ell^{*}{}]}(t)\enspace.\]
Notice that when $0<{r}/{\sqrt{\ell^{*}{}(1- \ell^{*}{})}}\leq R/2$,   $\lambda_r$ belongs to 
$$( \mathcal{S}^u_{ \bbul, \tau^{*}{}, \ell^{*}{} }[R])_r = \lbrace \lambda \in \mathcal{S}^u_{\bbul, \tau^{*}{}, \ell^{*}{}}[R],~d_2(\lambda, \mathcal{S}_0[R]) \geq r \rbrace\enspace,$$ as defined in Lemma \ref{mSR}. We get from  Lemma \ref{lemmegirsanov} and Lemma \ref{momentPoisson} that 
\[ E_{\lambda_0} \left[\left( \frac{d P_{\lambda}}{dP_{\lambda_0}} \right)^{2}(N)\right] = \exp \pa{\frac{r^2 L}{\lambda_0 (1-\ell^{*}{})} }  \enspace. \]
Choosing $r=(\lambda_0 (1-\ell^{*}{}) \log C_{\alpha,\beta}/L )^{1/2}$ then leads to $ E_{\lambda_0} \left[ (d P_{\lambda}/dP_{\lambda_0})^2(N)\right] =  C_{\alpha,\beta}.$ 

For $L\geq 1$ such that $2\log C_{\alpha,\beta}/(\ell^* L )\leq R$, we obtain $0<{r}/{\sqrt{\ell^{*}{}(1- \ell^{*}{})}}\leq R/2$ whereby $\lambda_r$ belongs to 
$( \mathcal{S}^u_{ \bbul, \tau^{*}{}, \ell^{*}{} }[R])_r$ and Lemma \ref{lemmebayesien} allows us to conclude that $\rho \pa{ \pa{\mathcal{S}^u_{\bbul, \tau^{*}{}, \ell^{*}{}}[R]}_r } \geq \beta$ and $\mathrm{mSR}_{\alpha, \beta}\pa{ \mathcal{S}^u_{\bbul, \tau^{*}{}, \ell^{*}{}}[R]} \geq r.$

\subsection{Proof of Proposition \ref{UBalt2_u}}

The first statement of Proposition \ref{UBalt2_u} is straightforward, just noticing that for every $\lambda_0$ in $\calS_0^u[R]$
$$E_{\lambda_0}\cro{\phi_{2,\alpha}^{u(1/2)}(N)} =  E_{\lambda_0}\cro{E_{\lambda_0}\cro{\phi_{2,\alpha}^{u(1/2)}(N)\Big| N_1}}\leq \alpha\enspace.$$

Let us assume that $\lambda \in \mathcal{S}^u_{\bbul, \tau^{*}{}, \ell^{*}{}}[R]$, that is there exists $\lambda_0$ in $(0,R)$ and $\delta$ in $(-\lambda_0,R-\lambda_0] \setminus \lbrace 0 \rbrace$ satisfying $\lambda(t) = \lambda_{0} + \delta \mathds{1}_{(\tau^{*}{},\tau^{*}{} +\ell^{*}{}]}(t)$ for all $t$ in $[0,1]$.  

\smallskip

Let us first consider the test  $\phi_{2,\alpha}^{u(1)}(N)$ and assume 
\begin{equation} \label{UBalt2_u_eq1}
d_2(\lambda,\calS^u_0[R]) \geq \frac{1}{\sqrt{L}} \pa{ \sqrt{\frac{2R}{\beta (1-\ell^{*}{} )}}   +    2\sqrt{\frac{R \ell^{*}{} }{\beta (1-\ell^{*}{} )}}   + \sqrt{\frac{1}{\alpha_1 \wedge \alpha_2} \pa{R+2\sqrt{\frac{R}{\beta L}}}    }       } \enspace.
\end{equation}

We may write  $\phi_{2,\alpha}^{u(1)}(N) = \phi_{1,\alpha_2}^{u,-}(N) \vee \phi_{1,\alpha_1}^{u,+}(N)$ by  definition of the tests  $\phi_{1,\alpha}^{u,-}$ and $\phi_{1,\alpha}^{u,+}$ in \eqref{test_alt1_u}.
We therefore obtain  $P_\lambda \left(\phi_{2,\alpha}^{u(1)}(N) = 0\right) = P_\lambda \left( \phi_{1,\alpha_2}^{u,-}(N)=0,~ \phi_{1,\alpha_1}^{u,+}(N)=0\right)$. 

From the assumption \eqref{UBalt2_u_eq1} and the same computations as in the proof of Proposition \ref{bNP1U}, we get
$$P_\lambda \pa{\phi_{1,\alpha_2}^{u,-}(N) =0}  \leq P_\lambda \pa{ N(\tau^{*}{} ,\tau^{*}{}+\ell^{*}{} ]> b_{N_1,\ell^*}(\alpha_2) } \leq \beta\enspace,$$ when $- \lambda_0 < \delta <0$ and $$P_\lambda \pa{ \phi_{1,\alpha_1}^{u,+}(N) =0}  \leq P_\lambda \pa{ N(\tau^{*}{} ,\tau^{*}{}+\ell^{*}{} ]< b_{N_1,\ell^*}(1-\alpha_1)} \leq \beta\enspace,$$ when $0< \delta \leq R-\lambda_0 $.

The result of Proposition \ref{UBalt2_u} for the test  $\phi_{2,\alpha}^{u(1)}(N)$ follows with $$C(\alpha, \beta,R, \ell^{*}{})= \sqrt{\frac{2R}{\beta (1-\ell^{*}{})}}   +    2\sqrt{\frac{R \ell^{*}{}}{\beta (1-\ell^{*}{})}}   + \sqrt{\frac{1}{\alpha_1 \wedge \alpha_2} \pa{R+2\sqrt{\frac{R}{\beta L}}}    }\enspace.$$

\smallskip

Let us consider then the test  $\phi_{2,\alpha}^{u(2)}(N)$. We get from  Lemma \ref{quantile_T'}
$$t'_{n,\tau^*, \tau^*+\ell^*}(1-\alpha) \leq \frac{C}{L^2}  \left( 5n  \log \! \left(\!  \frac{2.77 }{\alpha} \! \right)\!\! + 3\max \left( \frac{1-\ell^*}{\ell^*},\frac{\ell^*}{1- \ell^*}   \right)    \log^2\!\! \left(\!\frac{2.77 }{\alpha} \! \right)  \right)  \enspace.$$
Now, the Bienayme-Chebyshev inequality and the bound $\int_{0}^{1} \lambda(x) Ldx\leq R L$ give
$$ P_{\lambda} \left( N_{1} \geq RL + \sqrt{\frac{2RL}{\beta}} \right) \leq \frac{\beta}{2} \enspace.$$
This yields $P_{\lambda} \left( t'_{N_1,\tau^*, \tau^*+\ell^*}(1-\alpha) \geq   C'(\alpha, \beta, R, \ell^{*}{},L)  \right) \leq {\beta}/{2}$ with
\begin{multline*}
C'(\alpha, \beta, R, \ell^{*}{},L)=C \Bigg( 5R \frac{ \log \left(  \frac{2.77 }{\alpha} \right)}{L} + 5 \sqrt{\frac{2R}{\beta}} \frac{ \log \left(  \frac{2.77 }{\alpha} \right)}{L^{3/2}}\\
 + 3 \max \left( \frac{1-\ell^*}{\ell^*}~,~ \frac{\ell^*}{1- \ell^*}   \right) \frac{ \log^2 \left(  \frac{2.77 }{\alpha} \right)}{L^2}    \Bigg)\enspace.
\end{multline*}
Noticing that
\begin{multline*}
P_\lambda\pa{\phi_{2,\alpha}^{u(2)}(N)  = 0}\leq P_\lambda \pa{T'_{\tau^{*}{}, \tau^{*}{} + \ell^{*}{}}(N) \leq C'(\alpha, \beta, R, \ell^{*}{},L) } \\
+ P_\lambda \pa{t'_{N_1,\tau^*, \tau^* + \ell^*}(1-\alpha) > C'(\alpha, \beta, R, \ell^{*}{},L) }\enspace,
\end{multline*}
this enables to write
\begin{equation} \label{UBalt2_eq2}
P_\lambda\pa{\phi_{2,\alpha}^{u(2)}(N)  = 0} \leq P_\lambda \pa{T'_{\tau^{*}{}, \tau^{*}{} + \ell^{*}{}}(N) \leq C'(\alpha, \beta, R, \ell^{*}{},L) } +\frac{\beta}{2}
\enspace.
\end{equation}

Assume now
\begin{equation} \label{UBalt2_eq3}
d_2(\lambda, \calS^u_0[R]) \geq \frac{1}{\sqrt{L}} \max \Bigg(4 \sqrt{\frac{2R}{\beta}}~,~  \Bigg( \frac{4R}{\sqrt{\beta}} + 2LC'(\alpha, \beta, R, \ell^{*}{},L)  \Bigg)^{1/2}  \Bigg) \enspace,
\end{equation}
which ensures
\begin{multline*}
 \vert \delta \vert \sqrt{\ell^{*}{} (1-\ell^{*}{})} \geq \frac{1}{\sqrt{L}} \max \Bigg(4 \sqrt{\frac{2( \lambda_0 + \delta (1-\ell^{*}{}))}{\beta}}~,\\
 ~  \left( \frac{4( \lambda_0 + \delta (1-\ell^{*}{}))}{\sqrt{\beta}} + 2LC'(\alpha, \beta, R, \ell^{*}{},L)  \right)^{1/2} \Bigg)\enspace.
 \end{multline*}
We get then
\begin{multline*}
 \delta^2 \ell^{*}{} (1-\ell^{*}{}) \geq 2 \max \Bigg(2 \sqrt{\frac{2( \lambda_0 + \delta (1-\ell^{*}{}))}{\beta L}}  \vert \delta \vert \sqrt{\ell^{*}{} (1-\ell^{*}{})}  ~,\\
 ~   \frac{2( \lambda_0 + \delta (1-\ell^{*}{}))}{\sqrt{\beta} L} + C'(\alpha, \beta, R, \ell^{*}{},L)  \Bigg)\enspace,
 \end{multline*}
and using the simple facts that $a+b \leq 2 \max(a,b)$ and $\sqrt{a+b} \leq \sqrt{a}+\sqrt{b}$ for all $a,b \geq 0$, 
  \begin{multline} \label{UBalt2_eq4}
 \delta^2 \ell^{*}{} (1-\ell^{*}{}) \geq  \sqrt{\frac{2}{\beta} \pa{  \frac{4( \lambda_0 + \delta (1-\ell^{*}{}))}{ L}   \delta^2 \ell^{*}{} (1-\ell^{*}{})  +   \frac{2( \lambda_0 + \delta (1-\ell^{*}{}))^2}{ L^2}}}\\
  + C'(\alpha, \beta, R, \ell^{*}{},L)\enspace.
 \end{multline}
Furthermore, Lemma \ref{MomentsT'} gives $E_\lambda [T'_{\tau^{*}{}, \tau^{*}{} + \ell^{*}{}}(N)] = \delta^2 \ell^{*}{} (1-\ell^{*}{})$ and
$$ \mathrm{Var}_\lambda \pa{T'_{\tau^{*}{}, \tau^{*}{} + \ell^{*}{}}(N)} =\frac{2(\lambda_0 + \delta (1-\ell^{*}{}))^2}{L^2} + \frac{4(\lambda_0 + \delta (1- \ell^{*}{})  )}{L} \delta^2 \ell^{*}{} (1-\ell^{*}{})\enspace,$$
so \eqref{UBalt2_eq4} leads to
\begin{equation*}
E_\lambda [T'_{\tau^{*}{}, \tau^{*}{} + \ell^{*}{}}(N)] \geq  \sqrt{{2 \mathrm{Var}_\lambda \pa{T'_{\tau^{*}{}, \tau^{*}{} + \ell^{*}{}}(N) }}/{\beta}} + C'(\alpha, \beta, R, \ell^*,L)\enspace.
 \end{equation*}
Combined with \eqref{UBalt2_eq2}, this inequality entails
\begin{multline*}
P_\lambda\pa{\phi_{2,\alpha}^{u(2)}(N)  = 0} \leq P_\lambda \Bigg(T'_{\tau^{*}{}, \tau^{*}{} + \ell^{*}{}}(N)-E_\lambda [T'_{\tau^{*}{}, \tau^{*}{} + \ell^{*}{}}(N)]\\
  \leq  \sqrt{{2 \mathrm{Var}_\lambda \pa{T'_{\tau^{*}{}, \tau^{*}{} + \ell^{*}{}}(N) }}/{\beta}}\Bigg)
 +\frac{\beta}{2}\enspace,
\end{multline*}
and the proof ends with the Bienayme-Chebyshev inequality, thus giving
\[P_\lambda\pa{\phi_{2,\alpha}^{u(2)}(N)  = 0} \leq \beta\enspace.\]
The result of Proposition \ref{UBalt2_u} for the test  $\phi_{2,\alpha}^{u(2)}(N)$ then follows with 
\begin{multline*}
C(\alpha, \beta,R, \ell^{*}{})= \max \Bigg(4 \sqrt{\frac{2R}{\beta}}~,~  \Bigg( \frac{4R}{\sqrt{\beta}}
+ 2C \Bigg( 5R\log \left(  \frac{2.77 }{\alpha} \right)+ 5 \sqrt{\frac{2R}{\beta} }\log \left(  \frac{2.77 }{\alpha} \right) +\\
     3 \max \left( \frac{\ell^*}{1-\ell^*}, \frac{1-\ell^*}{\ell^*}    \right)\log^2 \left(  \frac{2.77 }{\alpha} \right)   \Bigg)\Bigg)^{1/2}  \Bigg)  \enspace,
\end{multline*}
where the constant $C$ is defined in Lemma \ref{quantile_T'}.

\subsection{Proof of Proposition \ref{UBalt3_u}}

Start by remarking that the control of the first kind error rates of the three tests $\phi_{3,\alpha}^{u(1)+}$, $\phi_{3,\alpha}^{u(1)-}$ and $\phi_{3/4,\alpha}^{u(2)}$ is straightforward, considering the same conditioning trick as in the beginning of the proof of Proposition \ref{UBalt2_u} above.

\medskip

Let us first address the statement of Proposition~\ref{UBalt3_u} for $\phi_{3,\alpha}^{u(1)+}$. 

Let $L \geq 1$ and let us consider $\lambda= \lambda_{0}+  \delta^{*}{} \mathds{1}_{(\tau, \tau + \ell^{*}{}]}$ in $\mathcal{S}^u_{\delta^{*}{}, \bbul \bbul, \ell^{*}{}}[R]$ with $\delta^*> 0$.
Notice that
\begin{align*}
P_\lambda \pa{\phi_{3,\alpha}^{u(1)+}(N)  =0}&=P_\lambda \pa{\max_{t \in [0,1-\ell^*\wedge(1/2)]} N(t,t+\ell^*\wedge(1/2)] \leq b_{N_1,\ell^*\wedge(1/2)}^+(1- \alpha)}\\
&\leq P_\lambda \pa{N(\tau,\tau+\ell^*\wedge(1/2)] \leq b_{N_1,\ell^*\wedge(1/2)}^+(1- \alpha)}\enspace.
\end{align*}
One deduces from Lemma \ref{bquantile_maxminNbis_u} that for every $n$ in $\N\setminus\{0\}$,
$$ b_{n,\ell^*\wedge(1/2)}^+(1- \alpha) \leq (\ell^*\wedge(1/2)) n + \frac{n}{2} g^{-1} \pa{\frac{32}{n} \log \pa{\frac{320}{\alpha}}}\enspace,$$
with $g$ defined by \eqref{defg}, and then from the inequality $g^{-1}(x) \leq 2x/3 + \sqrt{2x}$ for all $x>0$ (see \eqref{UBginv}),
$$
 b_{n,\ell^*\wedge(1/2)}^+(1- \alpha) \leq (\ell^*\wedge(1/2))n + 4 \sqrt{n \log \pa{\frac{320}{\alpha}} } + \frac{32}{3}\log \pa{\frac{320}{\alpha}}\enspace.
$$
Since $b_{0,\ell^*\wedge(1/2)}^+(1- \alpha) =0$, the above control holds in fact for every $n$ in $\N$, whereby
$$
b_{N_1,\ell^*\wedge(1/2)}^+(1- \alpha) \leq (\ell^*\wedge(1/2))N_1 + 4 \sqrt{N_1 \log \pa{\frac{320}{\alpha}} } + \frac{32}{3}\log \pa{\frac{320}{\alpha}}\enspace.
$$

Setting $I_{\lambda,L}(\beta)=I(\lambda )L+\sqrt{\frac{2I(\lambda)L}{\beta}}$ with $I(\lambda)=\int_0^1 \lambda(t) dt$, and
\begin{equation} \label{UBalt3_u_eq1}
Q(\alpha, \beta,\ell,L )= \ell I_{\lambda,L}(\beta)  + 4 \sqrt{I_{\lambda,L}(\beta) \log \pa{\frac{320}{\alpha}}}+ \frac{32}{3}\log \pa{\frac{320}{\alpha}}\enspace,
\end{equation}
we obtain for $\ell$ in $(0,1/2]$
\begin{multline*}
P_\lambda \pa{\phi_{3,\alpha}^{u(1)+}(N)  =0} \leq  P_\lambda \pa{N(\tau,\tau+\ell^*\wedge(1/2)] \leq Q(\alpha, \beta,\ell^*\wedge(1/2),L )}\\
+ P_\lambda \pa{ N_1 > I_{\lambda,L}(\beta)}
\enspace.
\end{multline*}
Therefore
\begin{equation}\label{UBalt3_u_eq2}
P_\lambda \pa{\phi_{3,\alpha}^{u(1)+}(N)  =0} \leq P_\lambda  \pa{N(\tau,\tau+\ell^*\wedge(1/2)] \leq Q(\alpha, \beta,\ell^*\wedge(1/2),L )}  + \frac{\beta}{2}\enspace,
\end{equation}
with
\begin{multline}\label{UBalt3_u_eq3}
Q(\alpha, \beta,\ell^*\wedge(1/2),L ) \leq \pa{\ell^*\wedge\frac{1}{2}} (\lambda_0 +\delta^*\ell^*)L+ \pa{\ell^*\wedge\frac{1}{2}}\sqrt{\frac{2RL}{\beta}} \\
  + 4 \sqrt{\pa{RL+\sqrt{\frac{2RL}{\beta}}}  \log \pa{\frac{320}{\alpha}}}+ \frac{32}{3}\log \pa{\frac{320}{\alpha}}\enspace,
\end{multline}
since $I(\lambda )=\lambda_0 +\delta^*\ell^* \leq R$.
Let us now assume that 
\begin{multline}\label{UBalt3_u_cond1}
d_2\pa{\lambda,\calS^u_0[R]}=\delta^{*}{} \sqrt{\ell^{*}{} (1-\ell^{*}{})} \geq \sqrt{ \frac{\ell^*}{(1-\ell^*)L}} \Bigg(\sqrt{\frac{2R}{\beta}} + \sqrt{\frac{2R}{\beta\pa{\ell^*\wedge\frac{1}{2}}} }\\
+  \frac{4}{\ell^*\wedge\frac{1}{2}} \sqrt{\pa{R+ \sqrt{\frac{2R}{\beta L}}} \log \pa{\frac{320}{\alpha}}}  + \frac{32 \log \pa{320/\alpha}}{3\pa{\ell^*\wedge\frac{1}{2}}\sqrt{L}} \Bigg) \enspace.\end{multline}
Then
\begin{multline*}
\delta^{*}{} (1-\ell^{*}{})\pa{\ell^*\wedge\frac{1}{2}} L \geq \pa{\ell^*\wedge \frac{1}{2}}\sqrt{\frac{2RL}{\beta}} + \sqrt{\frac{2RL}{\beta} \pa{\ell^*\wedge\frac{1}{2}}}\\
+  4 \sqrt{\pa{RL+ \sqrt{\frac{2RL}{\beta}}} \log \pa{\frac{320}{\alpha}}}  + \frac{32}{3} \log \pa{\frac{320}{\alpha}} \enspace.
\end{multline*}
With \eqref{UBalt3_u_eq3} and using  $R\geq \lambda_0 +\delta^*$, this implies
\begin{multline*}
\delta^{*}{} (1-\ell^{*}{})\pa{\ell^*\wedge\frac{1}{2}} L \geq Q(\alpha, \beta,\ell^*\wedge(1/2),L )+ \sqrt{\frac{2\pa{\lambda_0 +\delta^*}L}{\beta} \pa{\ell^*\wedge\frac{1}{2}}}\\
-\pa{\ell^*\wedge\frac{1}{2}} (\lambda_0 +\delta^*\ell^*)L \enspace,
\end{multline*}
and
\[Q(\alpha, \beta,\ell^*\wedge(1/2),L )\leq \pa{\ell^*\wedge\frac{1}{2}} (\lambda_0 +\delta^*)L  
- \sqrt{\frac{2\pa{\lambda_0 +\delta^*}L}{\beta} \pa{\ell^*\wedge\frac{1}{2}}} \enspace.\]
By simply using the exact computation of $E_\lambda [N(\tau,\tau+\ell^*\wedge(1/2)]]$ and $\textrm{Var}_\lambda [N(\tau,\tau+\ell^*\wedge(1/2)]]$ which both equal $(\lambda_0 + \delta^{*}{})\pa{\ell^*\wedge (1/2)} L$, one can notice that
this is equivalent to
\[Q(\alpha, \beta,\ell^*\wedge(1/2),L )\leq E_\lambda [N(\tau,\tau+\ell^*\wedge(1/2)]]
- \sqrt{\frac{2\textrm{Var}_\lambda [N(\tau,\tau+\ell^*\wedge(1/2)]] }{\beta}} \enspace.\]
Coming back to \eqref{UBalt3_u_eq2}, one finally deduces from the Bienayme-Chebyshev inequality that
\[P_\lambda \pa{\phi_{3,\alpha}^{u(1)+}(N)  =0} \leq P_\lambda  \pa{N(\tau,\tau+\ell^*\wedge(1/2)] \leq Q(\alpha, \beta,\ell^*\wedge(1/2),L )}  + \frac{\beta}{2}\leq\beta\enspace.\]

\smallskip

Let us now address the statement of Proposition~\ref{UBalt3_u} for $\phi_{3,\alpha}^{u(1)-}$. 

Let $L \geq 1$ and let us consider again $\lambda= \lambda_{0}+  \delta^{*}{} \mathds{1}_{(\tau, \tau + \ell^{*}{}]}$ in $\mathcal{S}^u_{\delta^{*}{}, \bbul \bbul, \ell^{*}{}}[R]$, but with $-\lambda_0<\delta^*<0$ here.
Notice first that
\begin{align*}
P_\lambda \pa{\phi_{3,\alpha}^{u(1)-}(N)  =0}&=P_\lambda \pa{\min_{t \in [0,1-\ell^*\wedge(1/2)]} N(t,t+\ell^*\wedge(1/2)] \geq b_{N_1,\ell^*\wedge(1/2)}^-(\alpha)}\\
&\leq P_\lambda \pa{N(\tau,\tau+\ell^*\wedge(1/2)] \geq b_{N_1,\ell^*\wedge(1/2)}^-(\alpha)}\enspace.
\end{align*}

From Lemma \ref{bquantile_maxminNbis_u}, one deduces that 
$$ b_{N_1,\ell^*\wedge(1/2)}^-(\alpha) \geq N_1\pa{\ell^*\wedge(1/2)} - 4\sqrt{2N_1 \log \pa{\frac{320}{\alpha}} } \enspace.$$
Setting for $\ell$ in $(0,1/2]$,
\begin{equation} \label{UBalt3_u_eq4}
Q'(\alpha, \beta,\ell,L )= \ell I(\lambda )L-2 \ell \sqrt{\frac{I(\lambda)L}{\beta}} -  4\sqrt{I(\lambda)L + 2\sqrt{\frac{I(\lambda)L}{\beta}}} \sqrt{2\log \pa{\frac{320}{\alpha}}}\enspace,
\end{equation}
with $I(\lambda)=\int_0^1 \lambda(t) dt$ as above, 
this entails
\begin{multline*}
P_\lambda \pa{\phi_{3,\alpha}^{u(1)-}(N)  =0} \leq  P_\lambda \pa{N(\tau,\tau+\ell^*\wedge(1/2)] \geq Q'(\alpha, \beta,\ell^*\wedge(1/2),L )}\\
+P_\lambda \pa{ N_1\not\in \cro{I(\lambda)L -2 \sqrt{{I(\lambda)L}/{\beta}} ; I(\lambda)L +2 \sqrt{{I(\lambda)L}/{\beta}}}}
\enspace.
\end{multline*}

The Bienayme-Chebyshev inequality therefore leads to
\begin{equation}\label{UBalt3_u_eq5}
P_\lambda \pa{\phi_{3,\alpha}^{u(1)-}(N)  =0} \leq P_\lambda \pa{N(\tau,\tau+\ell^*\wedge(1/2)] \geq Q'(\alpha, \beta,\ell^*\wedge(1/2),L )} + \frac{\beta}{2}\enspace,
\end{equation}
with
\begin{multline}\label{UBalt3_u_eq6}
Q'(\alpha, \beta,\ell^*\wedge(1/2),L )\geq  \pa{\ell^*\wedge\frac{1}{2}} (\lambda_0 +\delta^*\ell^*)L-2  \pa{\ell^*\wedge\frac{1}{2}} \sqrt{\frac{RL}{\beta}} \\-  4\sqrt{RL + 2\sqrt{\frac{RL}{\beta}}} \sqrt{2\log \pa{\frac{320}{\alpha}}}\enspace.
\end{multline}
since $I(\lambda )=\lambda_0 +\delta^*\ell^*$ belongs to $(0, R]$.

Let us furthermore assume that 
\begin{multline}\label{UBalt3_u_cond2}
d_2\pa{\lambda,\calS^u_0[R]} \geq \sqrt{ \frac{\ell^*}{(1-\ell^*)L}} \Bigg(2\sqrt{\frac{R}{\beta}} + \sqrt{\frac{2R}{\beta\pa{\ell^*\wedge\frac{1}{2}}} }\\
+  \frac{4}{\ell^*\wedge\frac{1}{2}} \sqrt{R+ 2\sqrt{\frac{R}{\beta L}}}\sqrt{2 \log \pa{\frac{320}{\alpha}}}\Bigg) \enspace.
\end{multline}

Then
\begin{multline*}
|\delta^{*}{}| (1-\ell^{*}{})\pa{\ell^*\wedge\frac{1}{2}} L \geq \sqrt{\frac{2R\pa{\ell^*\wedge(1/2)} L}{\beta}}  +2\pa{\ell^*\wedge \frac{1}{2}}\sqrt{\frac{RL}{\beta}} \\
+  4 \sqrt{RL+2 \sqrt{\frac{RL}{\beta}}}\sqrt{2 \log \pa{\frac{320}{\alpha}}}\enspace.
\end{multline*}
With \eqref{UBalt3_u_eq6} and the bound  $R\geq \lambda_0 +\delta^*$, this yields
\begin{multline*}
|\delta^{*}{}| (1-\ell^{*}{})\pa{\ell^*\wedge\frac{1}{2}} L \geq  \sqrt{\frac{2(\lambda_0 +\delta^*)\pa{\ell^*\wedge(1/2)} L}{\beta}} +\pa{\ell^*\wedge\frac{1}{2}} (\lambda_0 +\delta^*\ell^*)L\\
-Q'(\alpha, \beta,\ell^*\wedge(1/2),L ) \enspace.
\end{multline*}
From $E_\lambda [N(\tau,\tau+\ell^*\wedge(1/2)]]=\textrm{Var}_\lambda [N(\tau,\tau+\ell^*\wedge(1/2)]]=(\lambda_0 + \delta^{*}{})\pa{\ell^*\wedge (1/2)} L$, we then deduce that  \begin{multline*}
Q'(\alpha, \beta,\ell^*\wedge(1/2),L )  \geq  E_\lambda [N(\tau,\tau+\ell^*\wedge(1/2)]] +\sqrt{\frac{2\textrm{Var}_\lambda [N(\tau,\tau+\ell^*\wedge(1/2)]]}{\beta}}\enspace.
\end{multline*}
Inserting this inequality in \eqref{UBalt3_u_eq5} and using the Bienayme-Chebyshev inequality again, we finally obtain
\[P_\lambda \pa{\phi_{3,\alpha}^{u(1)-}(N)  =0} \leq P_\lambda  \pa{N(\tau,\tau+\ell^*\wedge(1/2)] \geq Q'(\alpha, \beta,\ell^*\wedge(1/2),L )}  + \frac{\beta}{2}\leq\beta\enspace.\]
This concludes the proof for the test $\phi_{3,\alpha}^{u(1)-}$.

\smallskip

Let us finally turn to the test $\phi_{3/4,\alpha}^{u(2)}$.

Let $L \geq 1$ and let us consider $\lambda= \lambda_{0}+  \delta^{*}{} \mathds{1}_{(\tau, \tau + \ell^{*}{}]}$ in $\mathcal{S}^u_{\delta^{*}{}, \bbul \bbul, \ell^{*}{}}[R]$ satisfying
\begin{multline}\label{UBalt3_u_cond3}
 d_2(\lambda, \calS_0^u[R]) \geq 2\max \Bigg(   \frac{1}{\sqrt{L}} \left(   \sqrt{10C  R \log \left(  \frac{2.77  }{u_\alpha}   \right)} +  2\sqrt{\frac{R}{\sqrt{\beta}}} \right) \\+ \frac{1}{L^{3/4}}  \sqrt{10C \sqrt{\frac{2R}{\beta}} \log \left(  \frac{2.77  }{u_\alpha}   \right)}  +\frac{1}{L} \sqrt{6C \max \left( \frac{\ell^{*}{}}{1-\ell^{*}{}}, \frac{1-\ell^{*}{}}{\ell^{*}{}}   \right)} \log \left(  \frac{2.77 }{u_\alpha}   \right)~,~ 8\sqrt{\frac{2R}{\beta L}}
  \Bigg) \enspace,
 \end{multline}
$C$ being the constant defined in Lemma \ref{quantile_T'}.

 In order to prove $P_{\lambda}(\phi_{3/4,\alpha}^{u(2)}(N)=0)  \leq \beta$, noticing first that
\begin{multline*} 
  P_{\lambda}(\phi_{3/4,\alpha}^{u(2)}(N)=0)  \leq \inf_{k \in \lbrace  0,...,\lceil (1-\ell^{*}{}) M \rceil-1  \rbrace} P_{\lambda} \left( T'_{\frac{k}{M}, \frac{k}{M}+\ell^{*}{}}(N) \leq t'_{N_1,\frac{k}{M}, \frac{k}{M}+\ell^{*}{}}(1-u_{\alpha}) \right) \enspace,
  \end{multline*}
 we only need to exhibit some $k_{\tau}$ in $\lbrace  0,...,\lceil (1-\ell^{*}{}) M \rceil-1  \rbrace$ such that $$P_{\lambda} \left( T'_{\frac{k_\tau}{M}, \frac{k_\tau}{M}+\ell^{*}{}}(N) \leq t'_{N_1\frac{k_\tau}{M}, \frac{k_\tau}{M}+\ell^{*}{}}(1-u_{\alpha}) \right) \leq \beta\enspace.$$
Let 
\begin{multline*}
C'(u_\alpha, \beta, R, \ell^{*}{},L)=C \Bigg( 5R \frac{ \log \left(  \frac{2.77 }{u_\alpha} \right)}{L} + 5 \sqrt{\frac{2R}{\beta}} \frac{ \log \left(  \frac{2.77 }{u_\alpha} \right)}{L^{3/2}}\\
 + 3 \max \left( \frac{1-\ell^*}{\ell^*}~,~ \frac{\ell^*}{1- \ell^*}   \right) \frac{ \log^2 \left(  \frac{2.77 }{u_\alpha} \right)}{L^2}    \Bigg)\enspace.
\end{multline*}  
Since Lemma \ref{quantile_T'} with the Bienayme-Chebyshev inequality together entail
 $$P_{\lambda} \pa{ t'_{N_1\frac{k_\tau}{M}, \frac{k_\tau}{M}+\ell^{*}{}}(1-u_{\alpha}) \geq C'(u_\alpha,\beta,R,\ell^*,L)}   \leq \frac{\beta}{2} \enspace,$$
$k_{\tau}$ only needs  to satisfy
 \begin{equation} \label{UBalt3_u_eq7}
P_{\lambda} \left( T'_{\frac{k_\tau}{M}, \frac{k_\tau}{M}+\ell^{*}{}}(N) \leq C'(u_\alpha,\beta,R,\ell^*,L) \right) \leq \frac{\beta}{2}\enspace.
\end{equation}
Using now the simple facts that $(a+b)^2 \geq a^{2}+b^{2}$ and $a+b \leq 2 \max(a,b)$ for all $a,b \geq 0$, \eqref{UBalt3_u_cond3} entails
 \begin{align*} 
  d_2^2(\lambda, \mathcal{S}^u_0[R]) &\geq 
 4C \left( 5R \frac{ \log \left(  \frac{2.77}{u_\alpha}   \right) }{L} + 5 \sqrt{\frac{2R}{\beta} } \frac{ \log \left(  \frac{2.77 }{u_\alpha}   \right) }{L^{3/2}} +    3 \max \left( \frac{\ell^{*}{} }{1- \ell^{*}{}}, \frac{1- \ell^{*}{}}{\ell^{*}{}}    \right)   \frac{ \log^2 \left(  \frac{2.77  }{u_\alpha}   \right)}{L^2}    \right) \nonumber \\ &
   + \frac{8R}{L\sqrt{\beta}} + \frac{8 \sqrt{2R}}{\sqrt{L \beta}} \vert \delta^{*}{} \vert \sqrt{\ell^{*}{}(1-\ell^{*}{})}  \enspace.
 \end{align*}
Further using $\sqrt{a+b} \leq \sqrt{a} + \sqrt{b}$ for all $a,b \geq 0$, by definition of $C'(u_\alpha, \beta, R, \ell^{*}{},L)$, this yields
 \begin{equation} \label{UBalt3_u_eq8}
  \frac{\delta^{*}{}^{2} \ell^{*}{}(1-\ell^{*}{})}{4} \geq C'(u_\alpha, \beta, R, \ell^{*}{},L)+\sqrt{\frac{4R^2}{L^2 \beta}  + \frac{8R \delta^{*}{}^2 \ell^{*}{}(1-\ell^{*}{})  }{L \beta} }\enspace.
 \end{equation}
Let us set $k_\tau = \lfloor \tau M \rfloor$. Since $0< \tau < 1-\ell^{*}{}$, $k_\tau$ actually belongs to $\lbrace  0,...,\lceil (1-\ell^{*}{}) M \rceil-1 \rbrace$, and since $M=\lceil 2/(\ell^{*}{} (1-\ell^{*}{} ))  \rceil$, $ k_\tau/M \leq \tau < k_\tau/M + (\ell^{*}{}(1-\ell^{*}{})/2$. Therefore, since we get in particular $k_\tau/M \leq \tau < k_\tau/M +\ell^{*}{}$, using Lemma \ref{MomentsT'} (equations \eqref{MomentsT'_E} and \eqref{MomentsT'_V}) with $x=\lambda_0 k_\tau /M,$ $y=\lambda_0 \ell^* + \delta (k_{\tau}/M+\ell^*-\tau) $ and $z= \lambda_0(1-\ell^*-k_{\tau}/M)+\delta( \tau-k_{\tau}/M)$, we get on the one hand
 $$ E_\lambda \left[T'_{\frac{k_\tau}{M}, \frac{k_\tau}{M}+\ell^{*}{}}(N) \right] = \delta^{*}{}^2 \frac{\pa{\ell^{*}{}(1-\ell^{*}{}) + k_\tau/M -\tau}^2}{\ell^{*}{}(1-\ell^{*}{})} \geq \frac{\delta^{*}{}^2 \ell^{*}{}(1-\ell^{*}{})}{4}\enspace,$$
 and on the other hand
\begin{multline*}
 \mathrm{Var}_\lambda \left[ T'_{\frac{k_\tau}{M}, \frac{k_\tau}{M}+\ell^{*}{}}(N) \right] = \frac{2}{L^2} \left( \lambda_0 + \delta^{*}{} \left( 1-\ell^* + \left( \tau - \frac{k_\tau}{M} \right) \frac{2\ell^{*}{} -1}{\ell^{*}{} (1-\ell^{*}{})}    \right)   \right)^2\\+ \frac{4}{L}  \delta^{*}{}^{2}   \frac{1-\ell^{*}{}}{\ell^{*}{}} \left( \ell^{*}{} - \frac{\tau -k_\tau/M}{1-\ell^{*}{}} \right)^2\left( \lambda_0  + \delta^{*}{} \left( 1- \ell^{*}{} + \left( \tau - \frac{k_\tau}{M}   \right) \frac{2\ell^{*}{} -1}{\ell^{*}{} (1-\ell^{*}{})}   \right)   \right)\enspace.
\end{multline*}
Using again the fact that $\tau - k_\tau/M < (\ell^{*}{}(1-\ell^{*}{})/2,$ we obtain
\begin{equation*}
 0 \leq  1-\ell^* + \left( \tau - \frac{k_\tau}{M} \right) \frac{2\ell^{*}{} -1}{\ell^{*}{} (1-\ell^{*}{})}    \leq 1
\end{equation*}
for all $\ell^*$ in $(0,1),$ leading to 
\begin{multline*}
 \mathrm{Var}_\lambda \left[ T'_{\frac{k_\tau}{M}, \frac{k_\tau}{M}+\ell^{*}{}}(N) \right] \leq \pa{ \frac{2}{L^2} \left( \lambda_0 +  \delta^{*}{}  \right)^2 + \frac{4}{L}\left( \lambda_0  +  \delta^{*}{}   \right)  \delta^{*}{}^{2}   \ell^{*}{} (1- \ell^{*}{} )} \mathds{1}_{\delta^{*}{}>0}\\
+ \pa{ \frac{2}{L^2} \lambda_0^2 + \frac{4}{L}\lambda_0 \delta^{*}{}^{2}\ell^{*}{} (1-\ell^{*}{}) }\mathds{1}_{\delta^{*}{}<0}
\enspace,
\end{multline*}
whereby
\[ \mathrm{Var}_\lambda \left[ T'_{\frac{k_\tau}{M}, \frac{k_\tau}{M}+\ell^{*}{}}(N) \right] \leq   \frac{2R^2}{L^2}  + \frac{4R}{L}   \delta^{*}{}^{2}   \ell^{*}{} (1- \ell^{*}{} )\enspace.
\]
Combined with these computations,  \eqref{UBalt3_u_eq8} leads to
 \begin{equation*}
E_\lambda \left[T'_{\frac{k_\tau}{M}, \frac{k_\tau}{M}+\ell^{*}{}}(N) \right] \geq C'(u_\alpha, \beta, R, \ell^{*}{},L)+\sqrt{\frac{2 \mathrm{Var}_\lambda \left[ T'_{\frac{k_\tau}{M}, \frac{k_\tau}{M}+\ell^{*}{}}(N) \right]}{\beta}} \enspace.
 \end{equation*}
The Bienayme-Chebyshev inequality finally allows to obtain \eqref{UBalt3_u_eq7}, which ends the proof.

\subsection{Proof of Proposition \ref{UBalt4_u}}

As for the other Bonferroni type aggregated tests, the control of the first kind error rates of the two tests $\phi_{4,\alpha}^{u(1)}$ and $\phi_{3/4,\alpha}^{u(2)}$ is straightforward using simple union bounds. 

\medskip

\emph{$(i)$ Control of the second kind error rate of $\phi_{4,\alpha}^{u(1)}$.}

\smallskip

Let us first set $\lambda$ in $\mathcal{S}^u_{\bbul,\bbul\bbul,\ell^{*}}[R]$ such that $\lambda= \lambda_{0}+  \delta\un{(\tau, \tau + \ell^*]}$ with $\tau$ in $(0,1-\ell^*)$, $\lambda_0$ in $(0,R]$, $\delta$ in $(0,R-\lambda_0]$ or $\delta$ in $(-\lambda_0,0)$, and 
\begin{multline*}
d_2(\lambda, \calS_0^u[R])\geq \sqrt{ \frac{\ell^*}{(1-\ell^*)L}} \Bigg(2\sqrt{\frac{R}{\beta}} + \sqrt{\frac{2R}{\beta\pa{\ell^*\wedge\frac{1}{2}}} }\\
+  \frac{4}{\ell^*\wedge\frac{1}{2}} \sqrt{2\pa{R+ 2\sqrt{\frac{R}{\beta L}}} \log \pa{\frac{640}{\alpha}}}  + \frac{32 \log \pa{640/\alpha}}{3\pa{\ell^*\wedge\frac{1}{2}}\sqrt{L}} \Bigg) \enspace.
\end{multline*}
This condition ensures that \eqref{UBalt3_u_cond1} and \eqref{UBalt3_u_cond2} both hold, but with $\alpha$ replaced by $\alpha/2$.
Then,  it suffices to notice that if $\delta\in (0,R-\lambda_0]$,
$$P_{\lambda}\pa{\phi_{4,\alpha}^{u(1)}(N)=0} \leq P_{\lambda}\pa{\phi_{3,\alpha/2}^{(1)+}(N)=0}\enspace,$$
and if $\delta$ in $(-\lambda_0,0)$,
$$P_{\lambda}\pa{\phi_{4,\alpha}^{u(1)}(N)=0} \leq P_{\lambda}\pa{\phi_{3,\alpha/2}^{(1)-}(N)=0}\enspace.$$
Since \eqref{UBalt3_u_cond1} and \eqref{UBalt3_u_cond2} hold with $\alpha/2$ instead of $\alpha$, by using exactly the same arguments as in the proof of Proposition \ref{UBalt3_u}, we obtain  $P_{\lambda}\pa{\phi_{3,\alpha/2}^{(1)+}(N)=0}\leq \beta$
when $\delta\in (0,R-\lambda_0]$ on the one hand, and $P_{\lambda}\pa{\phi_{3,\alpha/2}^{(1)-}(N)=0}\leq \beta$ when
$\delta$ in $(-\lambda_0,0)$ on the other hand.

In any case, whatever the value of $\delta$ in $(-\lambda_0,R-\lambda_0]\setminus\{0\}$, one has
\[P_{\lambda}\pa{\phi_{4,\alpha}^{u(1)}(N)=0}\leq \beta\enspace.\]

\medskip

\emph{$(ii)$ Control of the second kind error rate of $\phi_{3/4,\alpha}^{u(2)}$.}

\smallskip

Let us set now $\lambda$ in $\mathcal{S}^u_{\bbul,\bbul\bbul,\ell^{*}}[R]$ such that $\lambda= \lambda_{0}+  \delta\un{(\tau, \tau + \ell^*]}$ with $\tau$ in $(0,1-\ell^*)$, $\lambda_0$ in $(0,R]$, $\delta$ in $(-\lambda_0,R-\lambda_0]\setminus\{0\}$ as above, but with
\begin{multline*}
d_2(\lambda, \calS_0^u[R])\geq 2\max \Bigg(   \frac{1}{\sqrt{L}} \left(   \sqrt{10C  R \log \left(  \frac{2.77  }{u_\alpha}   \right)} +2\sqrt{\frac{R}{\sqrt{\beta}}} \right) + \frac{1}{L^{3/4}}  \sqrt{10C \sqrt{\frac{2R}{\beta}} \log \left(  \frac{2.77  }{u_\alpha}   \right)} \\ +\frac{1}{L} \sqrt{6C \max \left( \frac{\ell^{*}{}}{1-\ell^{*}{}}, \frac{1-\ell^{*}{}}{\ell^{*}{}}   \right)} \log \left(  \frac{2.77 }{u_\alpha}   \right)~,~ 8\sqrt{\frac{2R}{\beta L}}
  \Bigg) \enspace,
\end{multline*}
so that \eqref{UBalt3_u_cond3} holds.

Following the same arguments as in the proof of Proposition \ref{UBalt3_u} (with the only change of $\delta$ instead of $\delta^*$), we prove that
$$P_{\lambda}\pa{\phi_{3/4,\alpha}^{u(2)}(N)=0} \leq \beta\enspace.$$
This ends the proof, just taking for instance $C(\alpha,\beta,R,\ell^*)$ as the maximum between 
\begin{multline*}
\sqrt{ \frac{\ell^*}{(1-\ell^*)}} \Bigg(2\sqrt{\frac{R}{\beta}} + \sqrt{\frac{2R}{\beta\pa{\ell^*\wedge\frac{1}{2}}} }\\
+  \frac{4}{\ell^*\wedge\frac{1}{2}} \sqrt{2\pa{R+ 2\sqrt{\frac{R}{\beta}}} \log \pa{\frac{640}{\alpha}}}  + \frac{32 \log \pa{640/\alpha}}{3\pa{\ell^*\wedge\frac{1}{2}}} \Bigg)\enspace,
\end{multline*}
and
\begin{multline*}
2\max \Bigg(  \left(   \sqrt{10C  R \log \left(  \frac{2.77  }{u_\alpha}   \right)} +2\sqrt{\frac{R}{\sqrt{\beta}}} \right) +  \sqrt{10C \sqrt{\frac{2R}{\beta}} \log \left(  \frac{2.77  }{u_\alpha}   \right)} \\ + \sqrt{6C \max \left( \frac{\ell^{*}{}}{1-\ell^{*}{}}, \frac{1-\ell^{*}{}}{\ell^{*}{}}   \right)} \log \left(  \frac{2.77 }{u_\alpha}   \right)~,~ 8\sqrt{\frac{2R}{\beta }}
  \Bigg) \enspace.
\end{multline*}

\subsection{Proof of Proposition \ref{LBalt5_u}}

Assume that  
 \begin{equation} \label{LBalt5_u_eq1}
L > \frac{((R- \delta^{*}{}) \wedge R) \log C_{\alpha,\beta}}{2 \delta^{*}{}^2 \tau^{*}{}(1-\tau^{*}{})} \enspace,
\end{equation}
and set
 \begin{equation} \label{LBalt5_u_eq2}
 r=\sqrt{  \frac{((R- \delta^{*}{}) \wedge R)  \log C_{\alpha,\beta}}{2L}} \enspace.
 \end{equation}
The assumption \eqref{LBalt5_u_eq1} ensures 
\begin{equation} \label{LBalt5_u_eq3}
 r^2 < \delta^{*}{}^2 \tau^{*}{}(1- \tau^{*}{})\leq \delta^{*}{}^2/4\enspace,
  \end{equation}
which enables us to define $\lambda_r$ for $t$ in $(0,1)$ by
 $$ \lambda_r(t) = \lambda_0 + \delta^{*}{}  \mathds{1}_{( \tau^{*}{}, \tau^{*}{} + \ell_r ]}(t)\textrm{ with } \lambda_0=((R- \delta^{*}{}) \wedge R) \textrm{ and } \ell_r = \frac{1}{2} \pa{1- \frac{\sqrt{\delta^{*}{}^{2} - 4r^2}}{\vert \delta^{*}{} \vert}}\enspace.$$
First, $\ell_r$ belongs to $(0,1-\tau^{*}{})$ for all $\tau^{*}$ in $(0,1).$ Indeed, if $\tau^{*}{} \leq 1/2$ the result is straightforward by definition of $\ell_r$ and if $\tau^{*}{} >1/2$ the result follows from \eqref{LBalt5_u_eq3}. 

Moreover, the definition of $\ell_r$ implies $\delta^{*}{}^2 \ell_r(1-\ell_r)=r^2$ and ensures that $\lambda_r$ belongs to $(\mathcal{S}^u_{\delta^{*}{}, \tau^{*}{}, \bbul \bbul \bbul}[R])_r$.
Furthermore, we get from \eqref{LBalt5_u_eq1} that $L > (2 \lambda_0 \log C_{\alpha,\beta} )/ \delta^{*}{}^{2}$ and then $ 1- \lambda_0 \log C_{\alpha,\beta}/(L \delta^{*}{}^2) > 1/2$, which leads to
\begin{equation*} 
  r^2 < \frac{\lambda_0 \log C_{\alpha,\beta}}{L} \pa{1-  \frac{\lambda_0 \log C_{\alpha,\beta}}{L \delta^{*}{}^2}}= \frac{\delta^{*}{}^2}{4} \pa{ 1- \pa{1- \frac{2 \lambda_0 \log C_{\alpha,\beta}}{L\delta^{*}{}^2}}^2}\enspace,
\end{equation*}
 hence $\delta^{*}{}^2 \ell_r < \lambda_0  \log C_{\alpha,\beta}/L$. We then obtain from Lemma \ref{lemmegirsanov} and Lemma \ref{momentPoisson}
 \[ E_{\lambda_0} \left[\left( \frac{d P_{\lambda_r}}{dP_{\lambda_0}} \right)^{2}(N)\right] = \exp \left( \frac{L \delta^{*}{}^2 \ell_r }{\lambda_0}   \right) < C_{\alpha,\beta}\enspace. \]
Lemmas \ref{mSR} and \ref{lemmebayesien} then entail $\rho \pa{(\mathcal{S}^u_{\delta^{*}{}, \tau^{*}{}, \bbul \bbul \bbul}[R])_r} \geq \beta$ and $\mathrm{mSR}_{\alpha, \beta}(\mathcal{S}^u_{\delta^{*}{}, \tau^{*}{}, \bbul \bbul \bbul}[R])\geq r$.

\subsection{Proof of Proposition \ref{UBalt5_u}}

The control of the first kind error rate is straightforward, and even more strong by using the same conditioning trick as in the proof of Proposition \ref{UBalt2_u}:
in fact, for every $\lambda_0$ in $\calS_0^u[R]$ and $n$ in $\N$, $E_{\lambda_0}\cro{\phi_{5,\alpha}^{u}(N)\Big| N_1=n}=P_{\lambda_0}\pa{\phi_{5,\alpha}^{u}(N)=1\Big| N_1=n}\leq \alpha$, so
$$\forall \lambda_0\in \calS_0^u[R],~ P_{\lambda_0}\pa{\phi_{5,\alpha}^{u}(N)=1} =  E_{\lambda_0}\cro{P_{\lambda_0}\pa{\phi_{5,\alpha}^{u}(N)\Big| N_1}}\leq \alpha\enspace.$$ 
Let us turn to the control of the second kind error rate of $\phi^u_{5,\alpha}$.

Set $\lambda$ in $\calS^u_{\delta^*,\tau^*,\bbul\bbul\bbul}[R]$ such that $\lambda =\lambda_0 + \delta^* \mathds{1}_{(\tau^*, \tau^* + \ell]}$ with $\lambda_0$ in $(-\delta^{*}{} \vee 0,(R- \delta^{*}{}) \wedge R]$ and $\ell$ in $(0,1-\tau^{*}{})$, and satisfying
\begin{equation} \label{UBalt5_u_cond}
d_2(\lambda, \calS_0[R]) \geq \frac{2}{\sqrt{L}} \max\pa{\sqrt{\vert \delta^* \vert Q(2R, \delta^{*}{}, \alpha) },~ 2\sqrt{\frac{2R}{\beta}},~ \frac{\vert \delta^* \vert}{2\sqrt{2 \beta R}}} \enspace,
\end{equation} 
where $Q(2R, \delta^{*}{}, \alpha)$ is the quantile upper bound defined by Lemma \ref{QuantilessupShifted_u}, and which does not depend on $L$. The condition \eqref{UBalt5_u_cond} ensures that $  \vert \delta^{*}{} \vert  \sqrt{\ell (1-\ell)} \geq \vert\delta^{*}{} \vert /\sqrt{2 \beta R L}$, that is, using the fact that $\ell (1-\ell) \leq 1/4$,
 \begin{equation} \label{UBalt5_u_eq1}
 L \geq \frac{2}{\beta R}\enspace.
\end{equation}  
Setting $I(\lambda)= \int_0^{1} \lambda(t) dt$ as in the proof of Proposition~\ref{UBalt3_u}, with $I(\lambda )\leq R$ (and therefore obviously $2R - I(\lambda) \geq R$), since \eqref{UBalt5_u_eq1} also entails  $L \geq 2R/(\beta R^2)$ we obtain
\begin{equation} \label{UBalt5_u_eq2}
L \geq \frac{2I(\lambda)}{(2R - I(\lambda))^2 \beta}\enspace.
\end{equation}
Notice now that \eqref{UBalt5_u_eq2} and the Bienayme-Chebyshev inequality yield
\begin{align*} 
P_\lambda (N_1 > 2RL) &= P_\lambda \big(N_1 - I(\lambda)L > L(2R - I(\lambda)) \big)\\
&\leq P_\lambda \pa{N_1 - I(\lambda)L > \sqrt{\frac{2 I(\lambda)L}{\beta}} }\leq  \frac{\beta}{2} \enspace.
\end{align*}
This leads to 
\[P_\lambda \pa{ \phi^u_{5,\alpha}(N) =0} \leq P_\lambda \pa{\sup_{\ell' \in (0,1- \tau^{*})}S'_{\delta^*,\tau^*,\tau^*+\ell'}(N) \leq s_{N_1,\delta^*,\tau^*,L}^{'+}(1-\alpha),~ N_1 \leq 2RL}+ \frac{\beta}{2}\enspace,\]
and Lemma \ref{QuantilessupShifted_u} allows to write that
\begin{align*}
P_\lambda \pa{ \phi^u_{5,\alpha}(N) =0} &\leq P_\lambda \pa{\sup_{\ell' \in (0,1- \tau^{*})}S'_{\delta^*,\tau^*,\tau^*+\ell'}(N) \leq Q(2R, \delta^{*}{}, \alpha)} + \frac{\beta}{2}\\
&\leq P_\lambda \pa{S'_{\delta^*,\tau^*,\tau^*+\ell}(N) \leq Q(2R, \delta^{*}{}, \alpha)} + \frac{\beta}{2} \enspace.
\end{align*}
Moreover, the assumption \eqref{UBalt5_u_cond} implies 
$$ \vert \delta^* \vert \sqrt{\ell (1-\ell)} \geq \frac{2}{\sqrt{L}} \max\pa{\sqrt{\vert \delta^* \vert Q(2R, \delta^{*}{}, \alpha) },~ 2\sqrt{\frac{2R}{\beta}}} \enspace,$$
which entails
$$ \frac{\vert \delta^* \vert}{2} \ell (1-\ell) L \geq 2 \max\pa{ Q(2R, \delta^{*}{}, \alpha) ,~ \sqrt{\frac{2R \ell (1-\ell) L}{\beta}}} \enspace,$$
hence
$$ \frac{\vert \delta^* \vert}{2} \ell (1-\ell) L \geq 2 \max\pa{ Q(2R, \delta^{*}{}, \alpha) ,~ \sqrt{\frac{2(\lambda_0 + \delta^* (1- \ell)) \ell (1-\ell) L}{\beta}}} \enspace.$$
Noticing that $E_\lambda [S'_{\delta^*,\tau^*,\tau^*+\ell}(N)] = \vert\delta^*\vert \ell (1-\ell)L/2$ and $\mathrm{Var}_\lambda (S'_{\delta^*,\tau^*,\tau^*+\ell}(N)) = (\lambda_0 + \delta^{*}{} (1-\ell)) \ell (1-\ell)L$, we get
\begin{equation} \label{bdistanceDelta+bis}
E_\lambda [S'_{\delta^*,\tau^*,\tau^*+\ell}(N)] \geq  Q(2R, \delta^{*}{}, \alpha)+ \sqrt{\frac{2\mathrm{Var}(S'_{\delta^*,\tau^*,\tau^*+\ell}(N))}{\beta}} \enspace.
\end{equation}
Therefore,
\begin{align*}
P_\lambda \pa{ \phi^u_{5,\alpha}(N) =0} &\leq P_\lambda \pa{S'_{\delta^*,\tau^*,\tau^*+\ell}(N) \leq Q(2R, \delta^{*}{}, \alpha)} + \frac{\beta}{2} \\
&\leq P_\lambda \pa{S'_{\delta^*,\tau^*,\tau^*+\ell}(N) - E_\lambda [S'_{\delta^*,\tau^*,\tau^*+\ell}(N)] \leq -\sqrt{\frac{2\mathrm{Var}(S'_{\delta^*,\tau^*,\tau^*+\ell}(N))}{\beta}}} + \frac{\beta}{2}\\
&\leq \beta \enspace,
\end{align*}
with a last line simply following from the Bienayme-Chebyshev inequality.

\subsection{Proof of Proposition \ref{LBalt6_u}}

Assume that $L\geq 3$ and $\alpha+\beta<1/2$, and set  $\lambda_0=R/2$. As in the proof of Proposition \ref{LBalt6}, we consider $C'_{\alpha, \beta}=4(1- \alpha - \beta)^{2}$, $K_{\alpha,\beta,L}=\lceil (\log_2 L)/C'_{\alpha,\beta} \rceil$, and for $k$ in $\lbrace 1,\ldots, K_{\alpha,\beta,L}\rbrace$,
$\lambda_{k}= \lambda_{0}+ \delta_{k} \mathds{1}_{(\tau^{*}{}, \tau^{*}{} + \ell_{k}]}$ with $ \ell_{k}  = (1- \tau^{*}{})/2^{k}$
 and
 $ \delta_{k} = (\lambda_{0} \log \log L/ (\ell_k L))^{1/2}.$ Then, for all $k$ in $\lbrace 1,\ldots,K_{\alpha,\beta,L} \rbrace$, noticing that $\ell_{k} < 1- \tau^{*}{}$, we get $d_2\pa{\lambda_{k},\mathcal{S}_0[R]} > \sqrt{\lambda_{0} \tau^{*}{} \log \log L/L}$. Furthermore, assuming that 
\begin{equation}\label{LBalt6_u_eq1}
\frac{\log \log L}{L^{1+1/C'_{\alpha,\beta}}} \leq  \frac{R(1-\tau^*)}{4}\enspace,
\end{equation}
 one obtains that $\lambda_k$ belongs to $\mathcal{S}^u_{\bbul, \tau^{*}{}, \bbul \bbul \bbul}[R]$.

Recall also that for all $k$ in $\lbrace 1,\ldots, K_{\alpha,\beta,L} \rbrace$, $P_{\lambda_{k}}$ denotes the distribution of a Poisson process with intensity $\lambda_{k}$ with respect to the measure $\Lambda$, and consider $\kappa$, a random variable with uniform distribution on $\lbrace 1,\ldots, K_{\alpha,\beta,L} \rbrace$, which allows to define the probability distribution $\mu$ of $\lambda_{\kappa}$. From Lemma \ref{lemmebayesien}, we know that it is enough to prove $E_{\lambda_0} [\left( dP_{\mu}/dP_{\lambda_0}\right)^{2}  ] \leq 1+C'_{\alpha,\beta}$ to conclude that  $\mathrm{mSR}_{\alpha, \beta}(\mathcal{S}^u_{\bbul, \tau^{*}{}, \bbul \bbul \bbul}[R]) \geq \sqrt{R\tau^{*}{} \log \log L/(2L)}.$

\smallskip

 The same calculation as in the proof of Proposition \ref{LBalt6}  (see \eqref{upperboundExpectsquare}) gives for $\eta$ such that $0<\eta<1-1/\sqrt{2},$
\[ E_{\lambda_0} \left[\left( \frac{dP_{\mu}}{dP_{\lambda_0}} \right)^{2}   \right] \leq C'_{\alpha,\beta}\log 2 + \frac{2C'_{\alpha,\beta}\log 2}{\log L} \left( \log L  \right)^{\eta+\frac{1}{\sqrt{2}}}+\exp \left( \frac{\log \log L}{2^{(\log L)^{\eta}/2}}\right)\enspace.\]
If we assume now that 
\begin{equation}\label{LBalt6_u_eq2}
\exp \left( \frac{\log \log L}{2^{(\log L)^{\eta}/2}}\right) + \frac{2C'_{\alpha,\beta}\log 2}{\left( \log L  \right)^{1-\eta-\frac{1}{\sqrt{2}}}}\leq 1+ (1-\log 2) C'_{\alpha,\beta}\enspace,
\end{equation}
we finally obtain the expected result
\[E_{\lambda_0} \left[\left( \frac{dP_{\mu}}{dP_{\lambda_0}} \right)^{2}   \right] \leq  1+C'_{\alpha,\beta}\enspace.\]
To end the proof, it remains to notice that there exists $L_0(\alpha,\beta,R,\tau^{*}{})\geq 3$ such that for all $L\geq L_0(\alpha,\beta,R,\tau^{*}{}),$ both assumptions \eqref{LBalt6_u_eq1} and \eqref{LBalt6_u_eq2} hold.
\bigskip

\subsection{Proof of Proposition \ref{UBalt6_u}}

As for all our Bonferroni type aggregated tests, the control of the first kind error rates of the two tests $\phi_{6,\alpha}^{u(1)}$ and $\phi_{6,\alpha}^{u(2)}$ is straightforward using simple union bounds and the conditioning trick of the above proofs for upper bounds. 

\medskip

\emph{$(i)$ Control of the second kind error rate of $\phi_{6,\alpha}^{u(1)}$.}

\smallskip

Let $\lambda$ in $\mathcal{S}^u_{\bbul, \tau^{*}{}, \bbul \bbul \bbul}[R]$ be such that $\lambda = \lambda_{0} + \delta \mathds{1}_{(\tau^{*}{}, \tau^{*}{} +\ell]}$ with $\lambda_{0}$ in $(0,R]$, $\delta$ in $(-\lambda_{0}, R- \lambda_0] \setminus \lbrace 0 \rbrace$ and $\ell$ in $(0,1-\tau^{*}{})$ and assume that

\begin{multline} \label{UBalt6_u_cond1}
  d_2(\lambda, \calS^u_0[R]) \geq 2 \max \Bigg(  \sqrt{\frac{R}{3}} \sqrt{ \frac{1+ \tau^{*}{}}{\tau^{*}{}}} \sqrt{\frac{\log \left(  2 /u_\alpha   \right)}{L}}~,\\
  ~  \frac{1+ \tau^{*}{}}{\tau^{*}{}}\left(  \sqrt{\frac{2 \log \left( 2/u_\alpha   \right)}{L}} \sqrt{R + \sqrt{\frac{2R}{ \beta L}}}+\sqrt{ \frac{2R}{\beta L}}  \right)~,
~ \frac{\sqrt{1-\tau^{*}{}} R}{\sqrt{2L}}
\Bigg)\enspace.
\end{multline}

Let us prove the inequality $P_{\lambda}(\phi_{6,\alpha}^{u(1)}(N)=0) \leq \beta.$
Applying Lemma \ref{momentPoisson}, we get for all $k$ in $\lbrace 1,\ldots,\lfloor \log_{2} L \rfloor \rbrace$ and $\ell_{\tau^*,k}=\pa{1 - \tau^{*}{}}{2^{-k}}$
\begin{equation} \label{alt6_esp}
E_\lambda [ S'_{\tau^*,\tau^*+\ell_{\tau^*,k}}(N)] = \delta  (\ell_{\tau^*,k} \wedge \ell) (1- \ell_{\tau^*,k} \vee \ell) L \enspace ,
\end{equation}
and
\begin{equation}\label{alt6_var}
\mathrm{Var}_\lambda \cro{ S'_{\tau^*,\tau^*+\ell_{\tau^*,k}}(N)} = \begin{cases}
(\lambda_0 (1- \ell_{\tau^*,k}) + \delta (1- 2 \ell_{\tau^*,k} +\ell_{\tau^*,k} \ell))  \ell_{\tau^*,k} L~~\text{if } \ell_{\tau^*,k} \leq \ell\enspace, \\
(\lambda_0 \ell_{\tau^*,k} + \delta \ell (1- \ell_{\tau^*,k})) (1- \ell_{\tau^*,k} )L~~\text{if }\ell_{\tau^*,k} \geq \ell\enspace.
\end{cases}
\end{equation}

Assume first that $\delta$ belongs to $(0, R - \lambda_0].$
Noticing that  
\[ P_{\lambda}(\phi_{6,\alpha}^{u(1)}(N)=0) \leq    \inf_{k\in \lbrace 1,\ldots, \lfloor \log_{2} L \rfloor\rbrace}  P_{\lambda} \left(  S'_{\tau^*,\tau^*+\ell_{\tau^*,k}}(N) \leq s'_{N_1,\tau^*,\tau^*+\ell_{\tau^*,k}} \pa{1 - u_\alpha} \right)  \enspace,\]
 one can see it is enough to exhibit some $k$ in $\lbrace 1,\ldots, \lfloor \log_{2} L \rfloor\rbrace$ satisfying
\[ P_{\lambda} \left( S'_{\tau^*,\tau^*+\ell_{\tau^*,k}}(N)\leq s'_{N_1,\tau^*,\tau^*+\ell_{\tau^*,k}} \pa{1 - u_\alpha}\right) \leq \beta\enspace.\]

Under the condition \eqref{UBalt6_u_cond1}, we have $d_2^{2}(\lambda, \mathcal{S}^u_0[R]) \geq 2R^2 (1-\tau^{*}{})/L$ which ensures the inequality $\ell(1- \ell) > 2(1-\tau^{*}{})/L > (1-\tau^{*}{}) 2^{- \lfloor \log_2 L  \rfloor}$ and then
\begin{equation} \label{UBalt6_u_1_eq1}
 1- \tau^{*}{} >\ell > \frac{1- \tau^{*}{}}{2^{\lfloor \log_2 L  \rfloor}}\enspace .
  \end{equation}
From \eqref{UBalt6_u_1_eq1}, we deduce the existence of  $k_{\ell}$ in $ \lbrace 1,\ldots, \lfloor \log_{2} L \rfloor\rbrace$ satisfying $(1-\tau^{*}{}) 2^{-k_{\ell}} \leq \ell < (1-\tau^{*}{}) 2^{-k_{\ell}+1}$, that is
\begin{equation} \label{UBalt6_u_1_eq2}
\ell_{\tau^*,k_\ell}\leq \ell< \ell_{\tau^*,k_\ell-1}\enspace.
\end{equation} 
  Then
 \[ \frac{\ell_{\tau^*,k_\ell}}{1-\ell_{\tau^*,k_\ell}}\leq \frac{\ell}{1-\ell} < \frac{\ell_{\tau^*,k_\ell-1}}{1-\ell_{\tau^*,k_\ell-1}}\enspace,\]
 and
 \[  \frac{\ell_{\tau^*,k_\ell}}{1-\ell_{\tau^*,k_\ell}} =  \underbrace{\frac{\ell_{\tau^*,k_\ell-1}}{1-\ell_{\tau^*,k_\ell-1}}}_{>\frac{\ell}{1- \ell}} \frac{1-\ell_{\tau^*,k_\ell-1}}{1-\ell_{\tau^*,k_\ell}} \underbrace{\frac{\ell_{\tau^*,k_\ell}}{\ell_{\tau^*,k_\ell-1}}}_{=\frac{1}{2}} >\frac{\ell}{2(1-\ell)}  \frac{1-\ell_{\tau^*,k_\ell-1}}{1-\ell_{\tau^*,k_\ell}}\enspace.\]
But
\[ \frac{1-\ell_{\tau^*,k_\ell-1}}{1-\ell_{\tau^*,k_\ell}} = 1-\frac{1-\tau^*}{2^{k_{\ell}}-(1-\tau^{*}{})} \geq 2 \frac{\tau^{*}}{1+\tau^{*}}\enspace,\] so we finally obtain
 \begin{equation} \label{UBalt6_u_1_eq3}
\frac{\ell_{\tau^*,k_\ell}}{1-\ell_{\tau^*,k_\ell}} >  \frac{\tau^{*}}{1+\tau^*}\frac{\ell}{1-\ell}\enspace.
\end{equation}

The condition \eqref{UBalt6_u_cond1} then gives on the one hand
\[ \delta \sqrt{\ell(1-\ell)} \geq \frac{2}{\sqrt{3}} \sqrt{R} \sqrt{ \frac{1+ \tau^{*}{}}{\tau^{*}{}}} \sqrt{\frac{\log \left(  2 /u_\alpha   \right)}{L}}\enspace ,\]
and using the fact that $\delta <R$
$$ \delta \ell(1-\ell)  \frac{\tau^{*}{}}{1+ \tau^{*}{}} \geq  \frac{4 \log \left( 2 /u_\alpha   \right)}{3L}\enspace,$$
which entails with  \eqref{UBalt6_u_1_eq2} and  \eqref{UBalt6_u_1_eq3}
\begin{equation} \label{UBalt6_u_1_eq4}
\delta \ell_{\tau^*,k_\ell} (1-\ell) \geq \frac{4 \log \left( {2 }/{u_\alpha}   \right)}{3L}\enspace.
\end{equation}
On the other hand, \eqref{UBalt6_u_cond1} yields
\[ \delta \sqrt{\ell (1-\ell)} \geq 2 \frac{1+ \tau^{*}{}}{\tau^{*}{}}\left(  \sqrt{\frac{2\log \left(  2 /u_\alpha   \right)}{L}} \sqrt{R + \sqrt{\frac{2R}{ \beta L}}} + \sqrt{\frac{2R}{\beta L}}  \right)\enspace,\]
which entails with \eqref{UBalt6_u_1_eq2} and \eqref{UBalt6_u_1_eq3}
\[ \delta (1-\ell) \sqrt{\frac{\ell_{\tau^*,k_\ell}}{1-\ell_{\tau^*,k_\ell}}} \geq 2 \left(  \sqrt{\frac{2\log \left(  2 /u_\alpha   \right)}{L}} \sqrt{R + \sqrt{\frac{2R}{ \beta L}}} + \sqrt{\frac{2R}{\beta L}}  \right)\enspace,\]
whereby
\begin{equation} \label{UBalt6_u_1_eq5}
\delta \ell_{\tau^*,k_\ell} (1-\ell) \geq  2 \sqrt{\ell_{\tau^*,k_\ell} (1-\ell_{\tau^*,k_\ell})} \left(  \sqrt{\frac{2\log \left(  2 /u_\alpha   \right)}{L}} \sqrt{R + \sqrt{\frac{2R}{ \beta L}}} + \sqrt{\frac{2R}{\beta L}}  \right)\enspace.
\end{equation}
Thus, with \eqref{UBalt6_u_1_eq4} and \eqref{UBalt6_u_1_eq5} the condition \eqref{UBalt6_u_cond1} leads to
$$ \delta \ell_{\tau^*,k_\ell} (1-\ell) \geq \max \left( \frac{4 \log \left(  2 /u_\alpha   \right)}{3 L}~,~    2 \sqrt{\ell_{\tau^*,k_\ell} (1-\ell_{\tau^*,k_\ell})} \left( \sqrt{\frac{2\log \left(  2 /u_\alpha   \right)}{L}} \sqrt{R + \sqrt{\frac{2R}{ \beta L}}} + \sqrt{\frac{2R}{\beta L}}  \right) \right)\enspace.$$ 
Using the fact that $2\max(a,b) \geq a+b$ for all $a,b \geq 0$, we get
\begin{multline} \label{UBalt6_u_1_eq6}
\delta \ell_{\tau^*,k_\ell} (1-\ell) L \geq \frac{2}{3} \log \left(  \frac{2 }{u_\alpha}   \right)\\
 +  \sqrt{\ell_{\tau^*,k_\ell} (1-\ell_{\tau^*,k_\ell})} \left(  \sqrt{2 \log \left(  \frac{2 }{u_\alpha}   \right)} \sqrt{L I(\lambda) + \sqrt{\frac{2 L I(\lambda)}{ \beta }}} + \sqrt{\frac{2RL}{\beta }}  \right)\enspace,
\end{multline}
where we recall that $I(\lambda)$ stands for $ \int_0^1 \lambda(t) dt.$

Moreover, with \eqref{UBalt6_u_1_eq2}, the equations \eqref{alt6_esp} and \eqref{alt6_var}  lead to $E_\lambda [ S'_{\tau^*,\tau^*+\ell_{\tau^*,k_\ell}}(N)]  = \delta \ell_{\tau^*,k_\ell} (1- \ell) L$ and $\mathrm{Var}_\lambda \cro{ S'_{\tau^*,\tau^*+\ell_{\tau^*,k_\ell}}(N)}=(\lambda_0 (1- \ell_{\tau^*,k}) + \delta (1- 2 \ell_{\tau^*,k} +\ell_{\tau^*,k} \ell))  \ell_{\tau^*,k} L  \leq \ell_{\tau^*,k_\ell}(1-\ell_{\tau^*,k_\ell}) R L$. So \eqref{UBalt6_u_1_eq6} entails
  \begin{multline} \label{UBalt6_u_1_eq7}
E_\lambda [ S'_{\tau^*,\tau^*+\ell_{\tau^*,k_\ell}}(N)] \geq \frac{2}{3} \log \left( \frac{2 }{u_\alpha}  \right) +  \sqrt{\ell_{\tau^*,k_\ell}(1- \ell_{\tau^*,k_\ell})} \sqrt{2 \log \left(  \frac{2}{u_\alpha}  \right)} \sqrt{L I(\lambda) + \sqrt{\frac{2 L I(\lambda)}{ \beta }}}\\
   + \sqrt{{2 \mathrm{Var}_\lambda \cro{ S'_{\tau^*,\tau^*+\ell_{\tau^*,k_\ell}}(N)}}/{\beta}}\enspace.
  \end{multline}
The total probability formula then ensures that
\begin{align*}
P_{\lambda} &\left(  S'_{\tau^*,\tau^*+\ell_{\tau^*,k_\ell}}(N) \leq s'_{N_1,\tau^*,\tau^*+\ell_{\tau^*,k_\ell}} \pa{1 - u_\alpha}\right) \\
\leq  & P_\lambda \left( S'_{\tau^*,\tau^*+\ell_{\tau^*,k_\ell}}(N) \leq E_\lambda [ S'_{\tau^*,\tau^*+\ell_{\tau^*,k_\ell}}(N)]  - \sqrt{{2 \mathrm{Var}_\lambda \cro{ S'_{\tau^*,\tau^*+\ell_{\tau^*,k_\ell}}(N)}}/{\beta}} \right)
  \\ & +  P_\lambda \left( s'_{N_1,\tau^*,\tau^*+\ell_{\tau^*,k_\ell}} \pa{1 - u_\alpha}> \frac{2}{3}\log \left( \frac{2 }{u_\alpha}  \right) +  \sqrt{\ell_{\tau^*,k_\ell}(1- \ell_{\tau^*,k_\ell})} \sqrt{2 \log \left(  \frac{2}{u_\alpha}  \right)} \sqrt{L I(\lambda) + \sqrt{\frac{2 L I(\lambda)}{ \beta }}}    \right) \enspace.
 \end{align*} 
From the Bienayme-Chebyshev inequality, we deduce on the one hand that the first right hand side term is upper bounded by $\beta/2,$ i.e. 
\[P_\lambda \left( S'_{\tau^*,\tau^*+\ell_{\tau^*,k_\ell}}(N) \leq E_\lambda [ S'_{\tau^*,\tau^*+\ell_{\tau^*,k_\ell}}(N)]  - \sqrt{{2 \mathrm{Var}_\lambda \cro{ S'_{\tau^*,\tau^*+\ell_{\tau^*,k_\ell}}(N)}}/{\beta}} \right) \leq \frac{\beta}{2}\enspace,\]
and on the other hand that
\[P_\lambda \left(  N_1 \geq L I(\lambda) + \sqrt{\frac{2L I(\lambda)}{\beta}}   \right) \leq \frac{\beta}{2}\enspace.\]
This last inequality, combined with Lemma \ref{QuantilesAbsS_u}, which follows from a simple application of Bennett's inequality, leads to
\begin{multline*}
P_\lambda \Bigg( s'_{N_1,\tau^*,\tau^*+\ell_{\tau^*,k_\ell}} \pa{1 - u_\alpha}> \frac{2}{3}\log \left( \frac{2 }{u_\alpha}  \right)\\
+  \sqrt{\ell_{\tau^*,k_\ell}(1- \ell_{\tau^*,k_\ell})} \sqrt{2 \log \left(  \frac{2}{u_\alpha}  \right)} \sqrt{L I(\lambda) + \sqrt{\frac{2 L I(\lambda)}{ \beta }}}    \Bigg)
 \leq \frac{\beta}{2}\enspace.
\end{multline*}

We therefore conclude that
\[P_{\lambda}\left(  S'_{\tau^*,\tau^*+\ell_{\tau^*,k_\ell}}(N) \leq s'_{N_1,\tau^*,\tau^*+\ell_{\tau^*,k_\ell}} \pa{1 - u_\alpha}\right) \leq \beta\enspace,\]
so, as expected, $P_{\lambda}(\phi_{6,\alpha}^{u(1)}(N)=0) \leq   \beta$.

\smallskip

Assume now that $\delta$ belongs to $(- \lambda_0, 0)$ and notice that
\[ P_{\lambda}(\phi_{6,\alpha}^{u(1)}(N)=0) \leq    \inf_{k\in \lbrace 1,\ldots, \lfloor \log_{2} L \rfloor \rbrace}  P_{\lambda} \left( -S'_{\tau^*,\tau^*+\ell_{\tau^*,k}}(N) \leq s'_{N_1,\tau^*,\tau^*+\ell_{\tau^*,k}} \pa{1 - u_\alpha} \right)  \enspace.\]
 One can see it is enough to exhibit some $k$ in $\lbrace 1,\ldots, \lfloor \log_{2} L \rfloor \rbrace$ satisfying
\[ P_{\lambda} \left(  -S'_{\tau^*,\tau^*+\ell_{\tau^*,k}}(N)\leq s'_{N_1,\tau^*,\tau^*+\ell_{\tau^*,k}} \pa{1 - u_\alpha}\right) \leq \beta\enspace.\]
The same choice of $k_\ell$ as above leads, with  \eqref{alt6_esp} and \eqref{alt6_var}, to $E_\lambda [ -S'_{\tau^*,\tau^*+\ell_{\tau^*,k_\ell}}(N)] = \vert \delta \vert \ell_{\tau^*,k_\ell} (1- \ell) L$ and $\mathrm{Var}_\lambda [-S'_{\tau^*,\tau^*+\ell_{\tau^*,k_\ell}}(N) ] \leq  \ell_{\tau^*,k_\ell} (1-\ell_{\tau^*,k_\ell}) \lambda_0 L\leq  \ell_{\tau^*,k_\ell} (1-\ell_{\tau^*,k_\ell}) R L$. Using very similar arguments and calculations (mainly replacing $\delta$ by $|\delta|$), we also conclude that condition  \eqref{UBalt6_u_cond1}
implies 
\[ P_{\lambda} \left(  -S'_{\tau^*,\tau^*+\ell_{\tau^*,k_\ell}}(N)\leq s'_{N_1,\tau^*,\tau^*+\ell_{\tau^*,k_\ell}} \pa{1 - u_\alpha}\right) \leq \beta\enspace,\]
and $P_{\lambda}(\phi_{6,\alpha}^{u(1)}(N)=0) \leq \beta.$

Since $\log \left({2 }/{u_\alpha}  \right)=\log (2\lfloor\log_2 L\rfloor/\alpha)$ and $L\geq 3$, there exists
$C(\alpha, \beta, R,\tau^*)>0$ such that 
\begin{multline*}
2 \max \Bigg(  \sqrt{\frac{R}{3}} \sqrt{ \frac{1+ \tau^{*}{}}{\tau^{*}{}}} \sqrt{\frac{\log \left(  2 /u_\alpha   \right)}{L}}~,
  ~  \frac{1+ \tau^{*}{}}{\tau^{*}{}}\Bigg(  \sqrt{\frac{2 \log \left( 2/u_\alpha   \right)}{L}} \sqrt{R + \sqrt{\frac{2R}{ \beta L}}}+\sqrt{ \frac{2R}{\beta L}}  \Bigg)~,
~ \frac{\sqrt{1-\tau^{*}{}} R}{\sqrt{2L}}
\Bigg)\\
\leq C(\alpha, \beta, R,\tau^*) \sqrt{\frac{\log \log L}{L}}\enspace,
\end{multline*}
which allows to conclude the proof.

\medskip

\emph{$(ii)$ Control of the second kind error rate of $\phi_{6,\alpha}^{u(2)}$.}

\medskip

Let $\lambda$ in $\mathcal{S}^u_{\bbul, \tau^{*}{}, \bbul \bbul \bbul}[R]$ such that $\lambda = \lambda_{0} + \delta \mathds{1}_{(\tau^{*}{}, \tau^{*}{} +\ell]}$ with $\lambda_{0}$ in $(0,R]$, $\delta$ in $(-\lambda_{0}, R- \lambda_0] \setminus \lbrace 0 \rbrace$ and $\ell$ in $(0,1-\tau^{*}{})$ and assume that

\begin{multline} \label{UBalt6_u_cond2}
  d_2(\lambda, \calS^u_0[R]) \geq 
 \sqrt{\frac{1+ \tau^{*}{}}{\tau^{*}{}}} \max \Bigg( \sqrt{2C} \Bigg( \sqrt{5R} \sqrt{\frac{ \log \left(  2.77 /u_\alpha \right)}{L}} + \sqrt{5} \pa{\frac{2R}{\beta} }^{1/4} \frac{ \sqrt{\log \left(  2.77  /u_\alpha \right)}}{L^{3/4}} \\ +     \frac{\sqrt{3}\log \left(  2.77 /u_\alpha \right)}{2\sqrt{2 L \log \log L}} \Bigg)+  \frac{2 \sqrt{R}}{\sqrt{L\sqrt{\beta}}}  ~, ~ 4 \sqrt{\frac{1+\tau^{*}{}}{\tau^{*}{}}} \sqrt{\frac{2R}{ L\beta }}~,~ 4R \sqrt{\frac{\tau^*}{1+\tau^*}}  \sqrt{\frac{\log \log L}{L}} \Bigg) 
 \enspace,
\end{multline}
where $C$ is the constant defined in Lemma \ref{quantile_T'}, and let us prove the inequality $P_{\lambda}(\phi_{6,\alpha}^{u(2)}(N)=0) \leq \beta.$

Noticing that  
\[ P_{\lambda}(\phi_{6,\alpha}^{u(2)}(N)=0) \leq    \inf_{k\in \lbrace 1,\ldots, \lfloor \log_{2} L \rfloor \rbrace}  P_{\lambda} \left(  T'_{\tau^*,\tau^*+\ell_{\tau^*,k}}(N) \leq t'_{N_1,\tau^*,\tau^*+\ell_{\tau^*,k}} \pa{1 - u_\alpha} \right)  \enspace,\]
one only needs to exhibit some $k$ in $\lbrace 1,\ldots, \lfloor \log_{2} L \rfloor \rbrace$ satisfying 
\[ P_{\lambda} \left( T'_{\tau^*,\tau^*+\ell_{\tau^*,k}}(N)\leq t'_{N_1,\tau^*,\tau^*+\ell_{\tau^*,k}} \pa{1 - u_\alpha}\right) \leq \beta\enspace,\]
in order to prove the result.

Under the condition \eqref{UBalt6_u_cond2}, we first obtain $d_2^{2}(\lambda, \mathcal{S}^u_0[R]) \geq 16 R^2 (\log \log L)/L$ which ensures $\ell(1- \ell) > 16 (\log \log L)/L$ as well as
\begin{equation} \label{UBalt6_u_2_eq0}
\ell > \frac{16 \log \log L}{L}~\text{and}~ 1-\ell > \frac{16 \log \log L}{L} \enspace .
  \end{equation}
  Moreover, since
$16 (\log \log L)/L >2^{- \lfloor \log_2 L  \rfloor} > (1-\tau^{*}{}) 2^{- \lfloor \log_2 L  \rfloor}$, we obtain
\begin{equation} \label{UBalt6_u_2_eq1}
 1- \tau^{*}{} >\ell > \frac{1- \tau^{*}{}}{2^{\lfloor \log_2 L  \rfloor}}\enspace .
  \end{equation}
From \eqref{UBalt6_u_2_eq1}, as in the above part $(i)$, we deduce that there exists $k_{\ell}$ in $ \lbrace 1,\ldots, \lfloor \log_{2} L \rfloor \rbrace$ satisfying $(1-\tau^{*}{}) 2^{-k_{\ell}} \leq \ell < (1-\tau^{*}{}) 2^{-k_{\ell}+1}$, that is
\begin{equation} \label{UBalt6_u_2_eq2}
\ell_{\tau^*,k_\ell}\leq \ell< \ell_{\tau^*,k_\ell-1}\enspace.
\end{equation} 
As above again, this leads to the same inequality as \eqref{UBalt6_u_1_eq3}, i.e.
 \begin{equation} \label{UBalt6_u_2_eq3}
\frac{\ell_{\tau^*,k_\ell}}{1-\ell_{\tau^*,k_\ell}} >  \frac{\tau^{*}}{1+\tau^*}\frac{\ell}{1-\ell}\enspace.
\end{equation}

Applying Lemma \ref{MomentsT'}, we get with \eqref{UBalt6_u_2_eq2}
\begin{equation} \label{alt6_esp2}
E_\lambda [ T'_{\tau^*,\tau^*+\ell_{\tau^*,k_\ell}}(N)] = \delta^2 (1-\ell)^2 \frac{\ell_{\tau^*,k_\ell}}{1-\ell_{\tau^*,k_\ell}} \enspace ,
\end{equation}
and
\begin{multline}\label{alt6_var2}
\mathrm{Var}_\lambda \cro{ T'_{\tau^*,\tau^*+\ell_{\tau^*,k_\ell}}(N)} = \frac{2}{L^2} \left( \lambda_0 + \delta \left( 1- (1-\ell) \frac{\ell_{\tau^*,k_\ell}}{1-\ell_{\tau^*,k_\ell}} \right)  \right)^2 \\ + \frac{4}{L} \delta^2 (1-\ell)^2  \frac{\ell_{\tau^*,k_\ell}}{1-\ell_{\tau^*,k_\ell}}  \left(  \lambda_0 + \delta \left( 1- (1-\ell) \frac{\ell_{\tau^*,k_\ell}}{1-\ell_{\tau^*,k_\ell}} \right) \right) \enspace.
\end{multline}

Using \eqref{UBalt6_u_2_eq3},
\begin{equation} \label{alt6_esp2_eq}
E_\lambda [ T'_{\tau^*,\tau^*+\ell_{\tau^*,k_\ell}}(N)] > \delta^2 \ell (1-\ell) \frac{\tau^*}{1+ \tau^*} \enspace ,
\end{equation}
and with \eqref{UBalt6_u_2_eq2} we obtain
\begin{equation}\label{alt6_var2_eq}
\mathrm{Var}_\lambda \cro{ T'_{\tau^*,\tau^*+\ell_{\tau^*,k_\ell}}(N)} \leq \frac{2R^2}{L^2} + \frac{4R}{L} \delta^2 \ell (1-\ell)   \enspace.
\end{equation}

On the one hand, using  $a^2+b^2 \leq (a+b)^2$ for $a,b\geq 0$, the condition \eqref{UBalt6_u_cond2} ensures that 
\begin{multline*}
 \delta^2  \ell (1-\ell) \geq
  2 \frac{1+ \tau^{*}{}}{\tau^{*}{}}  \Bigg( C \left( 5R \frac{ \log \left(  2.77  /u_\alpha \right)}{L} + 5 \sqrt{\frac{2R}{\beta} } \frac{ \log \left(  2.77  /u_\alpha \right)}{L^{3/2}} +     \frac{ 3 \log^2 \left(  2.77 /u_\alpha \right)}{8L \log \log L} \right)  +  \frac{2R}{L\sqrt{\beta}}
\Bigg) \enspace,
\end{multline*}
and on the other hand
\begin{align*}
 \delta^2  \ell (1-\ell) &\geq 2 \frac{1+ \tau^{*}{}}{\tau^{*}{}} \left(   2 \sqrt{\frac{2R}{L \beta}} \vert \delta \vert \sqrt{\ell (1-\ell)}  \right)\enspace,
 \end{align*}
 hence 
\begin{multline*}
 \delta^2  \ell (1-\ell) \geq
 2 \frac{1+ \tau^{*}{}}{\tau^{*}{}} \max  \Bigg( C \left( 5R \frac{ \log \left(  2.77 /u_\alpha \right)}{L} + 5 \sqrt{\frac{2R}{\beta} } \frac{ \log \left(  2.77 /u_\alpha \right)}{L^{3/2}} +     \frac{3 \log^2 \left(  2.77 /u_\alpha \right)}{ 8 L \log \log L} \right)\\+  \frac{2R}{L\sqrt{\beta}} ~,~  2  \sqrt{\frac{2R}{L\beta}} \vert \delta \vert \sqrt{\ell (1-\ell)} 
\Bigg)\enspace.
\end{multline*}
Finally, with the inequality $a+b \leq 2 \max(a,b)$, we get
\begin{multline*}
 \delta^2  \ell (1-\ell) \frac{\tau^{*}{}}{1+\tau^{*}{}}  \geq
 C \left( 5R \frac{ \log \left(  2.77 /u_\alpha \right)}{L} + 5 \sqrt{\frac{2R}{\beta} } \frac{ \log \left(  2.77  /u_\alpha \right)}{L^{3/2}} +  \frac{3 \log^2 \left(  2.77 /u_\alpha \right)}{ 8 L \log \log L} \right)  +  \frac{2 R}{L\sqrt{\beta}} \\+   2  \sqrt{\frac{2R}{L\beta}} \vert \delta \vert \sqrt{\ell (1-\ell)}\enspace,
 \end{multline*}
that is, with \eqref{alt6_esp2_eq}, \eqref{alt6_var2_eq} and using  $\sqrt{a+b} \leq \sqrt{a}+ \sqrt{b}$,
\begin{multline} \label{UBalt6_u_cond2bis}
E_\lambda [ T'_{\tau^*,\tau^*+\ell_{\tau^*,k_\ell}}(N)]  \geq
 C \left( 5R \frac{ \log \left(  2.77 /u_\alpha \right)}{L} + 5 \sqrt{\frac{2R}{\beta} } \frac{ \log \left(  2.77 /u_\alpha \right)}{L^{3/2}}   +\frac{ 3 \log^2 \left(  2.77 /u_\alpha \right)}{8 L \log \log L} \right) \\ + \sqrt{2 \mathrm{Var}_\lambda \cro{ T'_{\tau^*,\tau^*+\ell_{\tau^*,k_\ell}}(N)} /\beta}\enspace.
 \end{multline}

Furthermore, \eqref{UBalt6_u_2_eq0} and \eqref{UBalt6_u_2_eq2} lead to
$$\frac{1}{\ell_{\tau^*,k_\ell}} < \frac{L}{8 \log \log L} \textrm{ and } \frac{1}{1-\ell_{\tau^*,k_\ell}} < \frac{L}{16 \log \log L} \enspace,$$
and therefore
 \begin{equation} \label{UBalt6_u_2_eq5}
 \max \left( \frac{1-\ell_{\tau^*,k_\ell}}{\ell_{\tau^*,k_\ell}}~,~\frac{\ell_{\tau^*,k_\ell}}{1-\ell_{\tau^*,k_\ell}}    \right) < \frac{L}{8 \log \log L} \enspace.
 \end{equation}
We set now 
\begin{multline*} 
Q(\alpha, \beta, L,R, \ell_{\tau^*,k_\ell})=C \Bigg( 5R \frac{ \log \left(  2.77  /u_\alpha \right)}{L} + 5 \sqrt{\frac{2R}{\beta} } \frac{ \log \left(  2.77  /u_\alpha \right)}{L^{3/2}} \\+    3 \max \left( \frac{\ell_{\tau^*,k_\ell}}{1- \ell_{\tau^*,k_\ell}}, \frac{1- \ell_{\tau^*,k_\ell}}{\ell_{\tau^*,k_\ell}}    \right) \frac{ \log^2 \left(  2.77/u_\alpha \right)}{L^2}    \Bigg) \enspace,
\end{multline*}  
and with \eqref{UBalt6_u_2_eq5}, the condition \eqref{UBalt6_u_cond2bis} ensures that 
\begin{equation} \label{UBalt6_u_cond2ter}
E_\lambda [ T'_{\tau^*,\tau^*+\ell_{\tau^*,k_\ell}}(N)]  \geq Q(\alpha, \beta, L,R,k_\ell)  + \sqrt{2 \mathrm{Var}_\lambda \cro{ T'_{\tau^*,\tau^*+\ell_{\tau^*,k_\ell}}(N)} /\beta}\enspace.
 \end{equation}
 
 The Bienayme-Chebyshev inequality leads to
$$P_\lambda \left(  N_1 \geq L I(\lambda) + \sqrt{\frac{2L I(\lambda)}{\beta}}   \right) \leq \frac{\beta}{2}\enspace,$$
and combined with Lemma \ref{quantile_T'} and the fact that $I(\lambda) \leq R$, we obtain 
\begin{equation}\label{UBalt6_u_2_eq6}
  P_\lambda \left( t'_{N_1,\tau^{*}{}, \tau^{*}{} + \ell_{\tau^*,k_\ell} }(1-u_{\alpha})   \geq Q(\alpha, \beta, L,R,\ell_{\tau^*,k_\ell})   \right) \leq \frac{\beta}{2} \enspace.
 \end{equation}

As a consequence,
\begin{multline*}
P_\lambda ( T'_{\tau^{*}{}, \tau^{*}{} + \ell_{\tau^*,k_\ell} }(N)  \leq t'_{N_1,\tau^{*}{}, \tau^{*}{} + \ell_{\tau^*,k_\ell} }(1-u_{\alpha}) ) \\
\leq P_\lambda \left( T'_{\tau^{*}{}, \tau^{*}{} + \ell_{\tau^*,k_\ell} }(N) < Q(\alpha, \beta, L,R,\ell_{\tau^*,k_\ell})   \right) + \frac{\beta}{2}\enspace,
\end{multline*}
whereby we finally obtain with \eqref{UBalt6_u_cond2ter} and the Bienayme-Chebyshev inequality
\[P_\lambda ( T'_{\tau^{*}{}, \tau^{*}{} + \ell_{\tau^*,k_\ell} }(N)  \leq t'_{N_1,\tau^{*}{}, \tau^{*}{} + \ell_{\tau^*,k_\ell} }(1-u_{\alpha}) ) \leq \beta\enspace.\]

The proof is ended by noticing that there exists
$C(\alpha, \beta, R,\tau^*)>0$ such that 
\begin{multline*}
  \sqrt{\frac{1+ \tau^{*}{}}{\tau^{*}{}}} \max \Bigg( \sqrt{2C} \Bigg( \sqrt{5R} \sqrt{\frac{ \log \left(  2.77 /u_\alpha \right)}{L}} + \sqrt{5} \pa{\frac{2R}{\beta} }^{1/4} \frac{ \sqrt{\log \left(  2.77  /u_\alpha \right)}}{L^{3/4}} +   \frac{\sqrt{3}\log \left(  2.77 /u_\alpha \right)}{ 2\sqrt{2 L \log \log L}} \Bigg) \\+  \frac{2 \sqrt{R}}{\beta^{1/4}}\frac{1}{\sqrt{L}}  ~, ~ 4 \sqrt{\frac{1+\tau^{*}{}}{\tau^{*}{}}} \sqrt{\frac{2R}{\beta }} \frac{1}{\sqrt{L}}~,~ 4R \sqrt{\frac{\tau^*}{1+\tau^*}}  \sqrt{\frac{\log \log L}{L}} \Bigg) \leq  C(\alpha, \beta, R,\tau^*) \sqrt{\frac{\log \log L}{L}}\enspace,
\end{multline*}
as $\log \left({2.77 }/{u_\alpha}  \right)=\log (2.77 \lfloor\log_2 L\rfloor/\alpha)$ and $L\geq 3$.

 \subsection{Proof of Proposition \ref{LBalt7_u}}
 
Assume that
 \begin{equation} \label{LBalt7_eq1_u}
  L > \frac{2((R- \delta^{*}{}) \wedge R) \log C_{\alpha,\beta}}{\delta^{*}{}^2}\enspace,
  \end{equation} 
  and set
 \begin{equation} \label{LBalt7_eq2_u}
 r=\sqrt{  \frac{((R - \delta^{*}{}) \wedge R) \log C_{\alpha,\beta}}{2L}} \enspace.
 \end{equation}
The assumption \eqref{LBalt7_eq1_u} entails
\begin{equation}  \label{LBalt7_eq3_u}
   r^{2} < \frac{\delta^{*}{}^2}{4}\enspace,
   \end{equation}
 which enables us to define $\lambda_r$ by $\lambda_r(t) = \lambda_0 + \delta^{*}{}  \mathds{1}_{( \tau_r,1 ]}(t)$ for $t$ in $(0,1),$ with
$\lambda_{0} = (R - \delta^{*}{}) \wedge R$ and
  \[ \tau_r = \frac{1}{2} \pa{1+ \frac{\sqrt{\delta^{*}{}^2 - 4r^2}}{\vert \delta^{*}{} \vert}}\enspace.\]
Thanks to  \eqref{LBalt7_eq3_u}, $\tau_r$ belongs to $(0,1)$ and satisfies  $\delta^{*}{}^2 \tau_r (1- \tau_r) = r^2,$ that is $ \lambda_r$ belongs to $\pa{ \calS^u_{\delta^*,\bbul\bbul,1-\bbul\bbul}[R]}_r.$ Moreover, the inequality $1- \lambda_0  \log C_{\alpha,\beta}/(L \delta^{*}{}^2)  > 1/2$ comes from \eqref{LBalt7_eq1_u} and yields
  \begin{align*}
  r^2 &< \frac{\lambda_0 \log C_{\alpha,\beta}}{L} \pa{1-  \frac{\lambda_0 \log C_{\alpha,\beta}}{L \delta^{*}{}^2}} \\
&= \frac{\delta^{*}{}^2}{4} \pa{ 1- \pa{1- \frac{2 \lambda_0 \log C_{\alpha,\beta}}{L \delta^{*}{}^2}}^2}\enspace,
\end{align*}
  hence $\delta^{*}{}^2 (1- \tau_r) < \lambda_0 \log C_{\alpha,\beta}/L$. We get now from Lemma \ref{lemmegirsanov} and Lemma \ref{momentPoisson}
 \[ E_{\lambda_0} \left[\left( \frac{d P_{\lambda}}{dP_{\lambda_0}} \right)^{2}(N)\right] = \exp \left( \frac{L (1- \tau_r) \delta^{*}{}^2}{\lambda_0}   \right) < C_{\alpha,\beta}\enspace. \]
Lemmas \ref{mSR} and \ref{lemmebayesien} then entail $\rho\left( (\calS^u_{\delta^*,\bbul\bbul,1-\bbul\bbul}[R])_r \right) \geq \beta$ and $\mathrm{mSR}_{\alpha, \beta}( \calS^u_{\delta^*,\bbul\bbul,1-\bbul\bbul}[R])\geq r.$

\subsection{Proof of Proposition \ref{UBalt7_u}}
The control of the first kind error rate is straightforward, and even more strong by using the same conditioning trick as in the proof of Proposition \ref{UBalt5_u}:
in fact, for every $\lambda_0$ in $\calS_0^u[R]$ and $n$ in $\N$, $E_{\lambda_0}\cro{\phi_{7,\alpha}^{u}(N)\Big| N_1=n}=P_{\lambda_0}\pa{\phi_{7,\alpha}^{u}(N)=1\Big| N_1=n}\leq \alpha$. 

Let us turn to the control of the second kind error rate of $\phi^u_{7,\alpha}$.

Set $\lambda$ in $\calS^u_{\delta^*,\bbul\bbul,1-\bbul\bbul}[R]$ such that $\lambda =\lambda_0 + \delta^* \mathds{1}_{(\tau, 1]}$ with $\lambda_0$ in $(-\delta^{*}{} \vee 0,(R- \delta^{*}{}) \wedge R]$ and $\tau$ in $(0,1)$, and satisfying
\begin{equation} \label{UBalt7_u_cond1}
d_2(\lambda, \calS^u_0[R]) \geq \frac{2}{\sqrt{L}} \max\pa{\sqrt{\vert \delta^* \vert Q(2R, \delta^{*}{}, \alpha) },~ 2\sqrt{\frac{2R}{\beta}},~ \frac{\vert \delta^* \vert}{2\sqrt{2 \beta R}}} \enspace,
\end{equation} 
where $Q(2R, \delta^{*}{}, \alpha)$ is the constant, not depending on $L$, defined in Lemma \ref{QuantilessupShifted_u} and used in Lemma \ref{QuantilessupShifted_ubis}.
The condition \eqref{UBalt7_u_cond1} ensures that $  \vert \delta^{*}{} \vert  \sqrt{\tau (1-\tau)} \geq \vert\delta^{*}{} \vert /\sqrt{2 \beta R L}$, and therefore, using the fact that $\tau (1-\tau) \leq 1/4$,
 \begin{equation} \label{UBalt7_u_eq1}
 L \geq \frac{2}{\beta R}\enspace.
\end{equation}  
Setting $I(\lambda)= \int_0^{1} \lambda(t) dt$,  we obtain with $I(\lambda )\leq R$ (and therefore obviously $2R - I(\lambda) \geq R$) and the fact that (\ref{UBalt7_u_eq1}) entails  $L \geq 2R/(\beta R^2)$
\begin{equation} \label{UBalt7_u_eq2}
L \geq \frac{2I(\lambda)}{(2R - I(\lambda))^2 \beta}\enspace.
\end{equation}
Notice now that \eqref{UBalt7_u_eq2} and the Bienayme-Chebyshev inequality yield
\begin{align*} 
P_\lambda (N_1 > 2RL) &= P_\lambda \big(N_1 - I(\lambda)L > L(2R - I(\lambda)) \big) \nonumber \\
&\leq P_\lambda \pa{N_1 - I(\lambda)L > \sqrt{\frac{2 I(\lambda)L}{\beta}} } \leq \frac{\beta}{2} \enspace.
\end{align*}
This leads to
\begin{align*}
P_\lambda \pa{ \phi^u_{7,\alpha}(N) =0} &\leq P_\lambda \pa{\sup_{\tau' \in (0,1)}S'_{\delta^*,\tau',1}(N) \leq s_{N_1,\delta^*,L}^{'+}(1-\alpha),~ N_1 \leq 2RL} + \frac{\beta}{2} \\
&\leq P_\lambda \pa{\sup_{\tau' \in (0,1)}S'_{\delta^*,\tau',1}(N) \leq Q(2R, \delta^{*}{}, \alpha)} + \frac{\beta}{2}~~\text{with Lemma \ref{QuantilessupShifted_ubis}} \\
&\leq P_\lambda \pa{S'_{\delta^*,\tau,1}(N) \leq Q(2R, \delta^{*}{}, \alpha)} + \frac{\beta}{2} \enspace.
\end{align*}
Moreover, the assumption \eqref{UBalt7_u_cond1} implies 
$$ \vert \delta^* \vert \sqrt{\tau (1-\tau)} \geq \frac{2}{\sqrt{L}} \max\pa{\sqrt{\vert \delta^* \vert Q(2R, \delta^{*}{}, \alpha) },~ 2\sqrt{\frac{2R}{\beta}}} \enspace,$$
which entails
$$ \frac{\vert \delta^* \vert}{2} \tau (1-\tau) L \geq 2 \max\pa{ Q(2R, \delta^{*}{}, \alpha) ,~ \sqrt{\frac{2R \tau (1-\tau) L}{\beta}}} \enspace,$$
whereby
$$ \frac{\vert \delta^* \vert}{2} \tau (1-\tau) L \geq 2 \max\pa{ Q(2R, \delta^{*}{}, \alpha) ,~ \sqrt{\frac{2(\lambda_0 + \delta^* \tau) \tau (1-\tau) L}{\beta}}} \enspace.$$
Noticing that $E_\lambda [S'_{\delta^*,\tau,1}(N)] =\vert \delta^{*}{} \vert \tau (1-\tau)L/2$ and $\mathrm{Var}(S'_{\delta^*,\tau,1}(N)) = (\lambda_0 + \delta^* \tau) \tau (1-\tau) L$, this leads to
\begin{equation} \label{UBalt7_u_cond2}
E_\lambda [S'_{\delta^*,\tau,1}(N)] \geq  Q(2R, \delta^{*}{}, \alpha)+ \sqrt{\frac{2\mathrm{Var}(S'_{\delta^*,\tau,1}(N))}{\beta}} \enspace.
\end{equation}
Therefore,
\begin{align*}
P_\lambda \pa{ \phi^u_{7,\alpha}(N) =0}& \leq P_\lambda \pa{S'_{\delta^*,\tau,1}(N) - E_\lambda [S'_{\delta^*,\tau,1}(N)] \leq -\sqrt{\frac{2\mathrm{Var}(S'_{\delta^*,\tau,1}(N))}{\beta}}} + \frac{\beta}{2}~~\text{with \eqref{UBalt7_u_cond2}} \\
&\leq \beta \enspace,
\end{align*}
with a last line simply deduced from the Bienayme-Chebyshev inequality.

Setting 
$$C(\alpha,\beta,R,\delta^*)=2\max\pa{\sqrt{\vert \delta^* \vert Q(2R, \delta^{*}{}, \alpha) },~ 2\sqrt{\frac{2R}{\beta}},~ \frac{\vert \delta^* \vert}{2\sqrt{2 \beta R}}} \enspace,$$
allows to conclude the proof.

\subsection{Proof of Proposition \ref{LBalt8_u}} Assume that $L \geq 3$ and $\alpha + \beta < 1/2$. We consider $C'_{\alpha, \beta} = 4(1- \alpha - \beta)^2$, $K_{\alpha,\beta,L}=\lceil (\log_2 L)/C'_{\alpha,\beta} \rceil$, $\lambda_0 = R/2$ and for $k$ in $\lbrace 1,\ldots, K_{\alpha,\beta,L}\rbrace$, 
$\lambda_{k}= \lambda_{0}+ \delta_{k} \mathds{1}_{(\tau_k, 1]}$, with $ \tau_k = 1- 2^{-k}$
 and
 $ \delta_{k} = (2^k \lambda_{0} \log \log L/ L)^{1/2}.$ Then, for every $k$ in $\lbrace 1,\ldots,K_{\alpha,\beta,L} \rbrace$, $d_2\pa{\lambda_{k},\mathcal{S}_0^u[R]} \geq \sqrt{R  \log \log L/(4L)}$, and $\lambda_k$ belongs to $\calS^u_{\bbul,\bbul\bbul,1-\bbul\bbul}[R]$ assuming
\begin{equation}\label{LBalt8_u_eq1}
\frac{\log \log L}{L^{1-1/C'_{\alpha,\beta}}} \leq  \frac{R}{4}\enspace.
\end{equation}
 The proof then essentially follows the same arguments as the proof of Proposition \ref{LBalt8}. 

Considering a random variable $\kappa$ with uniform distribution on $\lbrace 1,\ldots, K_{\alpha,\beta,L} \rbrace$ and the probability distribution $\mu$ of $\lambda_{\kappa}$, we aim at proving that $E_{\lambda_0} [\left( dP_{\mu}/dP_{\lambda_0}\right)^{2}  ] \leq 1+C'_{\alpha,\beta}$, with $P_\mu$ defined as in Lemma \ref{lemmebayesien}, in order to conclude that $\mathrm{mSR}_{\alpha, \beta}(\calS^u_{\bbul,\bbul\bbul,1-\bbul\bbul}[R]) \geq \sqrt{R \log \log L/(4L)}$. 
 
 The same calculation as in the proof of Proposition \ref{LBalt6} and Proposition \ref{LBalt8} gives for $\eta$ such that $0<\eta<1-1/\sqrt{2}$,
$$ E_{\lambda_0} \left[\left( \frac{dP_{\mu}}{dP_{\lambda_0}} \right)^{2}   \right] \leq C'_{\alpha,\beta}\log 2 + \frac{2C'_{\alpha,\beta}\log 2}{\log L} \left( \log L  \right)^{\eta+\frac{1}{\sqrt{2}}}+\exp \left( \frac{\log \log L}{2^{(\log L)^{\eta}/2}}\right)\enspace.$$
If we assume now that 
\begin{equation}\label{LBalt8_u_eq2}
\exp \left( \frac{\log \log L}{2^{(\log L)^{\eta}/2}}\right) + \frac{2C'_{\alpha,\beta}\log 2}{\left( \log L  \right)^{1-\eta-\frac{1}{\sqrt{2}}}}\leq 1+ (1-\log 2) C'_{\alpha,\beta}\enspace,
\end{equation}
we finally obtain the expected result, that is
\[E_{\lambda_0} \left[\left( \frac{dP_{\mu}}{dP_{\lambda_0}} \right)^{2}   \right] \leq  1+C'_{\alpha,\beta}\enspace.\]
To end the proof, it remains to notice that there exists $L_0(\alpha,\beta,R)\geq 3$ such that for all $L\geq L_0(\alpha,\beta,R)$, both assumptions \eqref{LBalt8_u_eq1} and \eqref{LBalt8_u_eq2} hold.

\subsection{Proof of Proposition \ref{UBalt8_u}}~\\
As for all our Bonferroni type aggregated tests, the control of the first kind error rates of the two tests $\phi_{8,\alpha}^{u(1)}$ and $\phi_{8,\alpha}^{u(2)}$ is straightforward using union bounds and the conditioning trick of the above proofs for upper bounds. Let us now turn to the second kind error rates.

\medskip

\emph{$(i)$ Control of the second kind error rate of $\phi_{8,\alpha}^{u(1)}$.}

\smallskip

Let $\lambda$ in $\calS^u_{\bbul,\bbul\bbul,1-\bbul\bbul}[R]$ such that $\lambda = \lambda_{0} + \delta \mathds{1}_{(\tau,1]}$, with $\lambda_0$ in $(0,R]$, $\tau$ in $(0,1)$, $\delta$ in $(-\lambda_{0}, R- \lambda_0] \setminus \lbrace 0 \rbrace$, and assume that
\begin{multline} \label{UBalt8_u_cond1}
d_2(\lambda, \mathcal{S}^u_0[R]) \geq \max \Bigg( \sqrt{\frac{\log \left(    2 /u_\alpha   \right)}{L}}  2 \sqrt{R}  ~  ,\\
~\sqrt{\frac{\log \left(   2 /u_\alpha  \right)}{L}} 2 \sqrt{6} \sqrt{R + \sqrt{\frac{2R}{\beta L}}} + \frac{2}{\sqrt{L}} \sqrt{\frac{6R}{\beta}} ~,~   \sqrt{\frac{2}{L}}R  \Bigg) \enspace .
\end{multline}
Let us prove that under this assumption, $P_{\lambda}(\phi_{8,\alpha}^{u(1)}(N)=0) \leq \beta$.

 Applying Lemma \ref{momentPoisson}, we get for all $\tau'$ in $\mathcal{D}_L$
\begin{equation} \label{alt8_u_esp}
E_\lambda [ S'_{\tau',1}(N)] = \delta  (\tau' \wedge \tau) (1- \tau' \vee \tau) L \enspace,
\end{equation}
and
\begin{equation}\label{alt8_u_var}
\mathrm{Var}_\lambda \cro{S'_{\tau',1}(N)}= \begin{cases}
(\lambda_0 (1- \tau') + \delta \tau' (1- \tau)) \tau' L~~\text{if } \tau' \leq \tau \\
(\lambda_0 \tau'+ \delta (\tau' - \tau + \tau' \tau)) (1- \tau' )L~~\text{if } \tau' \geq \tau \enspace.
\end{cases}
\end{equation}

Assume first that $\delta$ belongs to  $(0, R - \lambda_0]$.

Noticing that 
 \[ P_{\lambda}(\phi_{8,\alpha}^{u(1)}(N)=0) \leq   \inf_{\tau' \in \mathcal{D}_L} P_{\lambda} \left(   S'_{\tau',1}(N) \leq s'_{N_1,\tau',1} (1-u_{\alpha}) \right) \enspace,\]
one can see that it is enough to exhibit some $\tau'$ in $\mathcal{D}_L$ such that 
  $$P_{\lambda} \left(  S'_{\tau',1}(N) \leq s'_{N_1,\tau',1} (1-u_{\alpha})\right) \leq \beta \enspace.$$

Under the condition (\ref{UBalt8_u_cond1}), we have $d_2^{2}(\lambda, \mathcal{S}^u_0[R]) \geq 2R^2 /L$ which entails
\begin{equation} \label{UBalt8_u_1_eq1}
 \tau(1- \tau) > \frac{2}{L} \enspace.
  \end{equation}

On the one hand, if $\tau$ belongs to $(0, 1/2)$, the condition (\ref{UBalt8_u_1_eq1}) implies the inequalities $2^{-1}>\tau > 2/L>2^{- \lfloor \log_2 L \rfloor}$ and the existence of $k_{\tau}$ in $\lbrace 2,..., \lfloor \log_{2}L  \rfloor \rbrace$ satisfying $2^{-k_{\tau}} \leq \tau < 2^{-k_{\tau}+1}.$ We set $\tau_{k_{\tau}}= 2^{-k_{\tau}}$ and then 
\begin{equation} \label{UBalt8_u_1_eq2}
\tau_{k_{\tau}} \leq \tau < \tau_{k_{\tau}-1} \enspace,
\end{equation}
whereby
 $$ \frac{\tau_{k_{\tau}}}{1- \tau_{k_{\tau}}} \leq \frac{\tau}{1- \tau} < \frac{\tau_{k_{\tau}-1}}{1- \tau_{k_{\tau}-1}} \enspace,$$
and therefore
 $$  \frac{\tau_{k_{\tau}}}{1- \tau_{k_{\tau}}}=  \underbrace{\frac{\tau_{k_{\tau}-1}}{1-\tau_{k_{\tau}-1}}}_{>\frac{\tau}{1- \tau}} \frac{1-\tau_{k_{\tau}-1}}{1- \tau_{k_{\tau}}} \underbrace{\frac{\tau_{k_{\tau}}}{\tau_{k_{\tau}-1}}}_{=\frac{1}{2}} > \frac{1}{2} \frac{\tau}{1-\tau}  \frac{1-\tau_{k_{\tau}-1}}{1- \tau_{k_{\tau}}} \enspace.$$
 But
$$  
  \frac{1-\tau_{k_{\tau}-1}}{1- \tau_{k_{\tau}}} = 1 - \frac{1}{2^{k_\tau} -1} \geq \frac{2}{3} \enspace,$$
 so we finally obtain
\begin{equation} \label{UBalt8_u_1_eq3}
 \frac{\tau_{k_{\tau}}}{1-\tau_{k_{\tau}}}>  \frac{\tau}{3(1-\tau)} \enspace.
  \end{equation}

On the other hand, if $\tau$ belongs to $[1/2, 1)$, the condition (\ref{UBalt8_u_1_eq1}) implies the inequalities $1-2^{-1}\leq \tau < 1- 2/L< 1-2^{-\lfloor \log_2 L \rfloor}$ and the existence of $k_{\tau}$ in $\lbrace 1,..., \lfloor \log_{2}L  \rfloor -1 \rbrace$ satisfying $1-2^{-k_{\tau}} \leq \tau < 1- 2^{-k_{\tau}-1}.$ We set $\tau_{k_\tau}=1- 2^{-k_{\tau}}$ and we obtain
\begin{equation} \label{UBalt8_u_1_eq4}
\tau_{k_{\tau}} \leq \tau < \tau_{k_{\tau}+1} \enspace,
\end{equation}
whereby
 $$ \frac{\tau_{k_\tau}}{1- \tau_{k_\tau}} \leq \frac{\tau}{1- \tau} < \frac{\tau_{{k_\tau}+1}}{1-\tau_{k_\tau+1}} \enspace,$$
and therefore
 $$  \frac{\tau_{k_\tau}}{1- \tau_{k_\tau}}=  \underbrace{\frac{\tau_{k_\tau+1}}{1-\tau_{k_\tau+1}}}_{>\frac{\tau}{1- \tau}} \underbrace{\frac{1-\tau_{k_\tau+1}}{1-\tau_{k_\tau}}}_{=\frac{1}{2}} \frac{\tau_{k_\tau}}{\tau_{k_\tau+1}} >\frac{1}{2} \frac{ \tau}{ 1- \tau}  \frac{\tau_{k_\tau}}{\tau_{k_\tau+1}} \enspace.$$
But
$$ \frac{\tau_{k_\tau}}{\tau_{k_\tau+1}} = 1- \frac{1}{2^{k_\tau+1}-1} \geq \frac{2}{3} \enspace,  $$
so we finally get
\begin{equation} \label{UBalt8_u_1_eq5}
\frac{\tau_{k_\tau}}{1- \tau_{k_\tau}} > \frac{\tau}{3(1-\tau)} \enspace.
\end{equation}

Recall that $I(\lambda)=\int_0^1 \lambda(t) dt \leq R$ and notice that the assumption \eqref{UBalt8_u_cond1} leads to
\begin{multline} \label{UBalt8_u_cond2}
\delta \sqrt{\tau (1-\tau)} \geq  \max \Bigg( \sqrt{\frac{\log \left(    2 /u_\alpha   \right)}{L}}  2 \sqrt{R} ~    ,\\~\sqrt{\frac{\log \left(    2 /u_\alpha   \right)}{L}} 2 \sqrt{6} \sqrt{I(\lambda) + \sqrt{\frac{2 I(\lambda)}{\beta L}}}   + \frac{2}{\sqrt{L}} \sqrt{\frac{6(\lambda_0 + \delta \tau_{k_\tau})}{\beta}} \Bigg) \enspace.
\end{multline}
In particular, \eqref{UBalt8_u_cond2} gives on the one hand
$$ \delta^2 \tau (1-\tau) \geq 4R\frac{\log \left(  2 /u_\alpha  \right)}{L} \enspace ,$$
which entails, using \eqref{UBalt8_u_1_eq2}, \eqref{UBalt8_u_1_eq4} and the inequality $\delta < R$
$$ \delta \tau (1-\tau_{k_\tau}) \geq 4 \frac{\log \left(    2 /u_\alpha   \right)}{L}\enspace .$$
Then, \eqref{UBalt8_u_1_eq3} and \eqref{UBalt8_u_1_eq5} ensure that $ \tau (1- \tau_{k_\tau})/(3 (1-\tau))< \tau_{k_\tau}$ and thus
\begin{align} \label{UBalt8_u_cond3}
\delta \tau_{k_\tau} (1- \tau) \geq \frac{4}{3} \frac{\log \left(    2 /u_\alpha   \right)}{L} \enspace.
\end{align}

On the other hand, it follows from \eqref{UBalt8_u_cond2} that
\begin{align*}
\delta \sqrt{\tau (1-\tau)} &\geq  \sqrt{\frac{\log \left(    2 /u_\alpha  \right)}{L}} 2 \sqrt{6} \sqrt{I(\lambda) + \sqrt{\frac{2 I(\lambda)}{\beta L}}} + \frac{2}{\sqrt{L}} \sqrt{\frac{6(\lambda_0 + \delta \tau_{k_\tau})}{\beta}} \enspace,
\end{align*}
which entails with \eqref{UBalt8_u_1_eq3} and \eqref{UBalt8_u_1_eq5}
\begin{align*}
\delta (1-\tau) \sqrt{\frac{\tau_{k_\tau}}{1- \tau_{k_\tau}}} &\geq  \sqrt{\frac{\log \left(    2 /u_\alpha   \right)}{L}} 2\sqrt{2} \sqrt{I(\lambda) + \sqrt{\frac{2 I(\lambda)}{\beta L}}} + \frac{2}{\sqrt{L}} \sqrt{\frac{2(\lambda_0 + \delta \tau_{k_\tau})}{\beta}} \enspace,
\end{align*}
that is
\begin{multline}  \label{UBalt8_u_cond4}
\delta \tau_{k_\tau} (1- \tau) L \geq  2 \sqrt{2 \tau_{k_\tau} (1- \tau_{k_\tau})} \sqrt{L \log \left(   2 /u_\alpha  \right)} \sqrt{I(\lambda) + \sqrt{\frac{2 I(\lambda)}{\beta L}}} \\
+  2 \sqrt{\frac{2\tau_{k_\tau}(1- \tau_{k_\tau})(\lambda_0 + \delta \tau_{k_\tau})L}{\beta}} \enspace.
\end{multline}
Thereby, with \eqref{UBalt8_u_cond3} and \eqref{UBalt8_u_cond4}, 
\begin{multline*}
\delta\tau_{k_\tau} (1- \tau) L \geq 2 \max \Bigg( \frac{2}{3}\log \left(  \frac{2}{u_\alpha}  \right) ~,~  \sqrt{\tau_{k_\tau} (1- \tau_{k_\tau})} \sqrt{2 \log \left(  \frac{2 }{u_\alpha}  \right)} \sqrt{L I(\lambda) + \sqrt{\frac{2L I(\lambda)}{\beta }}} \\   +  \sqrt{\frac{2\tau_{k_\tau}(1- \tau_{k_\tau})(\lambda_0 + \delta \tau_{k_\tau})L}{\beta}} \Bigg) \enspace,
\end{multline*}
hence
\begin{equation}  \label{UBalt8_u_cond5}
\delta \tau_{k_\tau} (1- \tau) L \geq  Q(\alpha, \beta,L, \tau_{k_\tau}) +  \sqrt{\frac{2\tau_{k_\tau}(1- \tau_{k_\tau})(\lambda_0 + \delta \tau_{k_\tau})L}{\beta}} \enspace,
\end{equation}
with
$$ Q(\alpha, \beta,L, \tau_{k_\tau}) = \frac{2}{3} \log \left(  \frac{2 }{u_\alpha}  \right)   + \sqrt{\tau_{k_\tau} (1- \tau_{k_\tau})} \sqrt{2 \log \left(  \frac{2 }{u_\alpha}  \right)} \sqrt{L I(\lambda) + \sqrt{\frac{2L I(\lambda)}{\beta }}}\enspace .$$
Furthermore, with \eqref{UBalt8_u_1_eq2} and \eqref{UBalt8_u_1_eq4}, the expressions of $E_\lambda \cro{ S'_{\tau',1}(N)}$ and $\mathrm{Var}_\lambda \cro{ S'_{\tau',1}(N)} $ given in \eqref{alt8_u_esp} and \eqref{alt8_u_var} entail
\begin{equation} \label{alt8_u_moment}
E_\lambda [ S'_{\tau_{k_\tau},1}(N)] = \delta  \tau_{k_\tau} (1-  \tau) L \enspace,
~
\mathrm{Var}_\lambda \cro{S'_{\tau_{k_\tau},1}(N) } \leq
(\lambda_0 + \delta \tau_{k_\tau} ) \tau_{k_\tau} (1- \tau_{k_\tau}) L\enspace.
\end{equation}
The Bienayme-Chebyshev inequality leads to
$$P_\lambda \left(  N_1 \geq L I(\lambda) + \sqrt{\frac{2L I(\lambda)}{\beta}}   \right) \leq \frac{\beta}{2}\enspace,$$
and combined with Lemma \ref{QuantilesAbsS_u}, which follows from a simple application of Bennett's inequality, we get that 
\begin{equation}\label{UBalt8_u_1_eq6}
 P_\lambda \left(   s'_{N_1,\tau_{k_\tau},1} (1-u_{\alpha})     \geq Q(\alpha, \beta,L,\tau_{k_\tau})   \right) \leq \frac{\beta}{2} \enspace.
 \end{equation}
We conclude with the following inequalities:
\begin{align*}
P_\lambda &( S'_{\tau_{k_\tau},1}(N) \leq s'_{N_1,\tau_{k_\tau},1} (1-u_{\alpha}) ) \\
&\leq P_\lambda \left( S'_{\tau_{k_\tau},1}(N) <  Q(\alpha, \beta,L,\tau_{k_\tau})  \right) + \frac{\beta}{2}~ \text{with \eqref{UBalt8_u_1_eq6}} \\
&\leq P_\lambda \left( S'_{\tau_{k_\tau},1}(N) -  E_\lambda [S'_{\tau_{k_\tau},1}(N)]<  - \sqrt{\frac{2 \mathrm{Var}_\lambda (S'_{\tau_{k_\tau},1}(N) )}{\beta}} \right) + \frac{\beta}{2} ~~\text{with \eqref{UBalt8_u_cond5} and \eqref{alt8_u_moment}} \\
&\leq \beta ~~\text{with the Bienayme-Chebyshev inequality again}\enspace.
\end{align*}

Assume now that  $\delta$ belongs to $(-\lambda_0, 0)$ and notice that we have also 
\[ P_{\lambda}(\phi_{8,\alpha}^{u(1)}(N)=0) \leq \inf_{\tau' \in \mathcal{D}_L} P_{\lambda} \left( -S'_{\tau',1}(N) \leq s'_{N_1,\tau',1} (1-u_{\alpha}) \right) \enspace.\]
The same choices of $\tau_{k_\tau}$ as above lead, with \eqref{alt8_u_esp} and \eqref{alt8_u_var}, to $E_\lambda [-S'_{\tau_{k_\tau},1}(N) ] = \vert \delta \vert  \tau_{k_\tau} (1-  \tau) L$ and
$
\mathrm{Var}_\lambda \cro{- S'_{\tau_{k_\tau},1}(N) } \leq
\lambda_0 \tau_{k_\tau} (1- \tau_{k_\tau}) L$
and we obtain, using very similar arguments and calculations (mainly replacing $\delta$ by $\vert \delta \vert$), that the assumption \eqref{UBalt8_u_cond1} implies
$$P_\lambda ( -S'_{\tau_{k_\tau},1}(N)  \leq s'_{N_1,\tau_{k_\tau},1} (1-u_{\alpha}) ) \leq \beta \enspace,$$
so $P_{\lambda}(\phi_{8,\alpha}^{u(1)}(N)=0) \leq \beta$.

Since $\log(2/u_\alpha)= \log(2(2 \lfloor \log_2 L \rfloor -1)/\alpha)$ and $L \geq 3$, there exists $C(\alpha, \beta,R)>0$ such that 
\begin{multline}
\max \left( \sqrt{\frac{\log \left(    2 /u_\alpha   \right)}{L}}  2 \sqrt{R}    ,~\sqrt{\frac{\log \left(   2 /u_\alpha  \right)}{L}} 2 \sqrt{6} \sqrt{R + \sqrt{\frac{2R}{\beta L}}} + \frac{2}{\sqrt{L}} \sqrt{\frac{6R}{\beta}}~,~   \sqrt{\frac{2}{L}}R  \right) \\ \leq C(\alpha, \beta,R) \sqrt{\frac{\log \log L}{L}} \enspace,
\end{multline}
which allows to conclude the proof.

\medskip

\emph{$(ii)$ Control of the second kind error rate of $\phi_{8,\alpha}^{u(2)}$.}

\smallskip

Let $\lambda$ in $\calS^u_{\bbul,\bbul\bbul,1-\bbul\bbul}[R]$ such that $\lambda= \lambda_{0} + \delta \mathds{1}_{(\tau,1]}$ with $\lambda_0$ in $ (0,R]$, $\tau$ in $(0,1)$ and $\delta$ in  $(-\lambda_{0}, R- \lambda_0] \setminus \lbrace 0 \rbrace$ and assume that 
\begin{multline} \label{UBalt8_u2_cond1}
d_2(\lambda, \mathcal{S}^u_0[R]) \geq \max \Bigg( \sqrt{\frac{30CR\log \left(   2.77/u_\alpha \right)}{L}} +  \frac{ \sqrt{30C \log \left(   2.77/u_\alpha \right)}}{L^{3/4}} \pa{\frac{2R}{\beta}}^{1/4}\\+ \frac{3\sqrt{ C}\log \left(   2.77/u_\alpha \right)}{2\sqrt{ L\log \log L}} + 2\sqrt{\frac{3R}{L\sqrt{\beta}}}~,~12 \sqrt{\frac{2R}{L \beta}}~,~4R \sqrt{\frac{\log \log L}{L}}  \Bigg) \enspace ,
\end{multline}
where $C$ is the constant defined in Lemma \ref{quantile_T'}. 

Let us prove the inequality $P_{\lambda}(\phi_{8,\alpha}^{u(2)}(N) =0) \leq \beta$. 
Notice first that
\[ P_{\lambda}(\phi_{8,\alpha}^{u(2)}(N)  =0)  \leq \inf_{\tau' \in \mathcal{D}_L} P_{\lambda} \pa{ T'_{\tau',1}(N) \leq t'_{N_1,\tau',1}(1-u_{\alpha}) } \enspace ,\]
so one only needs to exhibit some $\tau'$ in $\mathcal{D}_L$ satisfying \[P_{\lambda} \pa{ T'_{\tau',1}(N)\leq t'_{N_1,\tau',1}(1-u_{\alpha}) } \leq \beta \enspace, \]
to obtain the expected result.

Under the assumption \eqref{UBalt8_u2_cond1}, we get $d_2^{2}(\lambda, \mathcal{S}^u_0[R]) \geq 16R^2 (\log \log L)/L$ which entails
\begin{equation} \label{UBalt8_u_2_eq1} 
  \tau(1- \tau) > \frac{16\log \log L}{L} \enspace.
  \end{equation}

Assume first that $\tau$ belongs to $(0, 1/2)$. 
 The condition (\ref{UBalt8_u_2_eq1}) leads to $\tau > 16 (\log \log L)/L$ and since $L \geq 3$, we get the inequality $16 (\log \log L)/L > 2^{- \lfloor \log_2 L \rfloor}$ and the existence of  
 $k_{\tau}$ in $\lbrace 2,..., \lfloor \log_{2}L  \rfloor \rbrace$ satisfying $2^{-k_{\tau}} \leq \tau < 2^{-k_{\tau}+1}.$ Setting $\tau_{k_\tau}= 2^{-k_{\tau}}$ we can prove as in the above case $(i)$ that
 \begin{equation} \label{UBalt8_u_2_eq2}
\tau_{k_{\tau}} \leq \tau < \tau_{k_{\tau}-1} \enspace,
\end{equation}
and
\begin{equation} \label{UBalt8_u_2_eq3}
 \frac{\tau_{k_{\tau}}}{1-\tau_{k_{\tau}}}>  \frac{\tau}{3(1-\tau)} \enspace.
  \end{equation}
Moreover, since $\tau_{k_\tau} < 1/2$ and by definition  of $k_\tau$ ,
\[ \max \left( \frac{\tau_{k_\tau}}{1- \tau_{k_\tau}}, \frac{1- \tau_{k_\tau}}{\tau_{k_\tau}}    \right) = \frac{1- \tau_{k_\tau}}{\tau_{k_\tau}}< \frac{2- \tau}{\tau} \enspace.\]
Since (\ref{UBalt8_u_2_eq1}) implies $\tau > 16 (\log \log L)/L$, we get 
  \begin{equation}  \label{UBalt8_u_2_eq4}
  \max \left( \frac{\tau_{k_\tau}}{1- \tau_{k_\tau}}, \frac{1- \tau_{k_\tau}}{\tau_{k_\tau}}    \right)< \frac{L}{8 \log \log L}\enspace.
    \end{equation}

Assume now that $\tau$ belongs to $[1/2, 1)$.
 The condition (\ref{UBalt8_u_2_eq1}) entails $\tau < 1- 16 (\log \log L)/L$ and since $L \geq 3$, $1-2^{-1}\leq \tau< 1-2^{- \lfloor \log_2 L  \rfloor}$. Hence, there exists $k_{\tau}$ in $\lbrace 1,..., \lfloor \log_{2}L  \rfloor -1 \rbrace$ satisfying $1-2^{-k_{\tau}} \leq \tau < 1- 2^{-k_{\tau}-1}.$ We set $\tau_{k_\tau}=1- 2^{-k_{\tau}}$ and we obtain as in the above case $(i)$
\begin{equation} \label{UBalt8_u_2_eq5}
\tau_{k_{\tau}} \leq \tau < \tau_{k_{\tau}+1} \enspace,
\end{equation}
and
\begin{equation} \label{UBalt8_u_2_eq6}
\frac{\tau_{k_\tau}}{1- \tau_{k_\tau}} > \frac{\tau}{3(1-\tau)} \enspace.
\end{equation}

Moreover, since $\tau_{k_\tau}$ belongs to $[1/2,1),$ we get using \eqref{UBalt8_u_2_eq5}
 \begin{align*} 
 \max \left( \frac{\tau_{k_\tau}}{1- \tau_{k_\tau}}, \frac{1- \tau_{k_\tau}}{\tau_{k_\tau}}    \right) &= \frac{\tau_{k_\tau}}{1- \tau_{k_\tau}} 
 \leq \frac{ \tau}{1-\tau} \enspace.
  \end{align*}
 Since (\ref{UBalt8_u_2_eq1}) yields $1-\tau > 16 (\log \log L)/L$, we get also
 \begin{equation} \label{UBalt8_u_2_eq7}
 \max \left( \frac{\tau_{k_\tau}}{1- \tau_{k_\tau}}, \frac{1- \tau_{k_\tau}}{\tau_{k_\tau}}    \right) < \frac{L}{16 \log \log L} \enspace.
 \end{equation}
 
 Applying Lemma \ref{MomentsT'}, we obtain with \eqref{UBalt8_u_2_eq2} and \eqref{UBalt8_u_2_eq5}
 \begin{equation}\label{UBalt8_u_2_esp}
 E_{\lambda}[T'_{\tau_{k_\tau},1}(N)] = \delta^2 (1- \tau)^2 \frac{\tau_{k_\tau}}{1- \tau_{k_\tau}} \enspace,
 \end{equation}
 and 
  \begin{equation}\label{UBalt8_u_2_var}
  \mathrm{Var}_{\lambda}\cro{T'_{\tau_{k_\tau},1}(N) } = \frac{4}{L} \delta^2 (1-\tau)^2 \frac{\tau_{k_\tau}}{1- \tau_{k_\tau}} \Bigg( \lambda_0 + \delta(1-\tau) \frac{\tau_{k_\tau}}{1- \tau_{k_\tau}}    \Bigg) + \frac{2}{L^2} \Bigg(  \lambda_0 + \delta (1-\tau) \frac{\tau_{k_\tau}}{1- \tau_{k_\tau}} \Bigg)^2 \enspace.
  \end{equation} 
  On the one hand, we finally get with \eqref{UBalt8_u_2_eq3} and \eqref{UBalt8_u_2_eq6},
   \begin{equation}\label{UBalt8_u_2_esp1}
 E_{\lambda}[T'_{\tau_{k_\tau},1}(N)] > \frac{\delta^2 \tau(1- \tau)}{3} \enspace,
 \end{equation}
 and on the other hand, using \eqref{UBalt8_u_2_eq2} and \eqref{UBalt8_u_2_eq5},
   \begin{equation}\label{UBalt8_u_2_var1}
  \mathrm{Var}_{\lambda}\cro{T'_{\tau_{k_\tau},1}(N) }
\leq \frac{4 \delta^2 \tau (1-\tau)R}{L} + \frac{2R^2}{L^2}  \enspace.
  \end{equation}

 Notice that the assumption \eqref{UBalt8_u2_cond1} entails
\begin{multline*}
\delta^2 \tau (1-\tau) \geq \max \Bigg( \frac{30 CR \log \left(   2.77/u_\alpha  \right)}{L}  +  \frac{   30 C \log \left(   2.77/u_\alpha  \right)}{L^{3/2}} \sqrt{\frac{2R}{\beta}}  + \frac{9C \log^2 \left(   2.77/u_\alpha \right)}{ 4 L\log \log L}   \\ +  \frac{12R}{L \sqrt{\beta}} ~,~12 \vert \delta \vert \sqrt{\tau (1-\tau)} \sqrt{\frac{2R}{L \beta}}\Bigg) \enspace,
\end{multline*}
hence  \begin{multline} \label{UBalt8_u2_cond2}
 \delta^2 \tau (1-\tau) \geq  \frac{15CR \log \left(   2.77/u_\alpha  \right)}{L}  +  \frac{   15 C \log \left(   2.77/u_\alpha  \right)}{L^{3/2}} \sqrt{\frac{2R}{\beta}} \\ + \frac{9C \log^2 \left(   2.77/u_\alpha \right)}{ 8 L\log \log L}  +  \frac{6R}{L \sqrt{\beta}} +6 \vert \delta \vert \sqrt{\tau (1-\tau)} \sqrt{\frac{2R}{L \beta}} \enspace.
 \end{multline}

Using \eqref{UBalt8_u_2_esp1} and \eqref{UBalt8_u_2_var1}, the condition \eqref{UBalt8_u2_cond2} then leads to 
\begin{multline} \label{UBalt8_u2_cond3}
E_{\lambda}[T'_{\tau_{k_\tau},1}(N) ]  - \sqrt{2 \mathrm{Var}_{\lambda} \cro{T'_{\tau_{k_\tau},1}(N) }/\beta}  \\ \geq  C \left( 5R \frac{ \log \left(  2.77/u_\alpha \right)}{L} + 5 \sqrt{\frac{2R}{\beta} } \frac{ \log \left(  2.77/u_\alpha \right)}{L^{3/2}} +    \frac{3}{8} \frac{ \log^2 \left( 2.77/u_\alpha \right)}{L \log \log L}    \right) \enspace.
\end{multline}

 Now, if we set
 \begin{multline}
 Q(\alpha, \beta,L,R, \tau_{k_\tau})= C \Bigg( 5R \frac{ \log \left( 2.77/u_\alpha \right)}{L} + 5 \sqrt{\frac{2R}{\beta} } \frac{ \log \left(  2.77/u_\alpha \right)}{L^{3/2}} \\+    3 \max \left( \frac{\tau_{k_\tau}}{1- \tau_{k_\tau}}, \frac{1- \tau_{k_\tau}}{\tau_{k_\tau}}    \right) \frac{ \log^2 \left(  2.77/u_\alpha \right)}{L^2}    \Bigg) \enspace,
 \end{multline}
combined with \eqref{UBalt8_u_2_eq4} and \eqref{UBalt8_u_2_eq7}, the condition \eqref{UBalt8_u2_cond3} yields
 \begin{equation} \label{UBalt8_u2_cond4}
E_{\lambda}[T'_{\tau_{k_\tau},1}(N) ]  - \sqrt{2 \mathrm{Var}_{\lambda}\cro{T'_{\tau_{k_\tau},1}(N) )}/\beta}   \geq  Q(\alpha, \beta,L,R, \tau_{k_\tau}) \enspace.
\end{equation}

Furthermore, from the inequality
$$P_\lambda \left(  N_1 \geq L I(\lambda) + \sqrt{\frac{2L I(\lambda)}{\beta}}   \right) \leq \frac{\beta}{2}\enspace,$$
Lemma \ref{quantile_T'}, and the fact that $I(\lambda) \leq R$, we deduce that
\begin{equation}\label{UBalt8_u_2_eq8}
 P_\lambda \left(   t'_{N_1,\tau_{k_\tau},1} (1-u_{\alpha})     \geq Q(\alpha, \beta,L,R, \tau_{k_\tau})   \right) \leq \frac{\beta}{2} \enspace.
 \end{equation}

We finally obtain using successively \eqref{UBalt8_u_2_eq8}, \eqref{UBalt8_u2_cond4} and the Bienayme-Chebyshev inequality
\begin{align*}
P_{\lambda}(T'_{\tau_{k_\tau},1}(N) \leq t'_{N_1,\tau_{k_\tau},1}(1-u_{\alpha})   ) 
&\leq P_{\lambda} \left( T'_{\tau_{k_\tau},1}(N) < Q(\alpha, \beta,L,R, \tau_{k_\tau}) \right) + \frac{\beta}{2}\\
&\leq P_{\lambda} \left( T'_{\tau_{k_\tau},1}(N) -E_{\lambda}[T'_{\tau_{k_\tau},1}(N) ]  < - \sqrt{2 \mathrm{Var}_{\lambda}\cro{T'_{\tau_{k_\tau},1}(N) )}/\beta}\right) + \frac{\beta}{2}\\
&\leq \beta\enspace.
\end{align*}
Since $\log(2/u_\alpha)= \log((4 \lfloor \log_2 L \rfloor -2)/\alpha)$ and $L \geq 3$, there exists $C(\alpha, \beta,R)>0$ such that 
\eqref{UBalt8_u2_cond1} is implied by
$$d_2(\lambda, \mathcal{S}^u_0[R]) \geq C(\alpha, \beta,R) \sqrt{\frac{\log \log L}{L}} \enspace,$$
which concludes the proof.

\subsection{Proof of Proposition \ref{LBalt9_u}}

Let $L \geq 2$ and set $\lambda_0= (R-\delta^{*}{}) \wedge R$.

 For all $k$ in $\set{1,\ldots,\lceil L^{3/4} \rceil}$, let us define $\lambda_k(t) = \lambda_{0} + \delta^{*}{} \mathds{1}_{(\tau_k, \tau_k + \ell]}(t)$ with
$\tau_k =k/L$ and $\ell =\lambda_{0} \log L /(2{\delta^*}^2 L)$. Then, as soon as
\begin{equation} \label{LBalt9_u_eq1}
\frac{\lceil L^{3/4} \rceil}{L} + \frac{  ((R-\delta^{*}{}) \wedge R) \log L}{2 {\delta^*}^2 L}< 1\enspace,
\end{equation}
$\lambda_k$ belongs to $\calS^u_{\delta^*,\bbul\bbul,\bbul\bbul\bbul} [R]$ for any $k$ in $\set{1,\ldots,\lceil L^{3/4} \rceil}.$ If in addition
\begin{equation} \label{LBalt9_u_eq2}
\frac{\log L}{L} \leq \frac{\delta^{*}{}^2}{ (R-\delta^{*}{}) \wedge R} \enspace,
\end{equation}
$\lambda_k$ satisfies for all $k$ in $\set{1,\ldots,\lceil L^{3/4} \rceil}$,
\[d_2^2\pa{\lambda_k,\calS^u_{0} [R]} \geq \frac{((R-\delta^{*}{}) \wedge R) \log L}{4L}\enspace.\]
Using Lemma \ref{lemmebayesien} and considering a random variable $J$ uniformly distributed on $\set{1,\ldots, \lceil L^{3/4} \rceil}$ and
the distribution $\mu$ of $\lambda_{J}$, one can see that it is enough to prove that $E_{\lambda_0} [\left( dP_{\mu}/dP_{\lambda_0}\right)^{2}  ] \leq 1+4(1- \alpha -\beta)^{2}$ to obtain the expected lower bound.
The same calculation as in the proof of Proposition \ref{LBalt9} leads to
 \[E_{\lambda_0} \left[  \left( \frac{dP_{\mu}}{dP_{\lambda_0}} \right)^{2} (N)  \right]  \leq   1+ \frac{\sqrt{L}}{\lceil L^{3/4} \rceil}\frac{e^{\delta^{*}{}^2 /\lambda_{0} }+1}{e^{\delta^{*}{}^2 /\lambda_{0} }-1}\enspace.\]
Therefore,  assuming that
 \begin{equation} \label{LBalt9_u_eq3}
\frac{\sqrt{L}}{\lceil L^{3/4} \rceil}\frac{e^{\delta^{*}{}^2 /(R-\delta^{*}{}) \wedge R }+1}{e^{\delta^{*}{}^2 /(R-\delta^{*}{}) \wedge R }-1} \leq 4(1- \alpha - \beta)^2\enspace,
\end{equation}
we get the expected result, that is
\[E_{\lambda_0} \left[  \left( \frac{dP_{\mu}}{dP_{\lambda_0}} \right)^{2} (N)  \right]  \leq   1+ 4(1- \alpha - \beta)^2\enspace.\]
We then conclude the proof noticing that there exists $L_0(\alpha,\beta,\delta^*,R)\geq 2$ such that  the assumptions \eqref{LBalt9_u_eq1}, \eqref{LBalt9_u_eq2} and \eqref{LBalt9_u_eq3} hold for all $L\geq L_0(\alpha,\beta,\delta^*,R).$

\subsection{Proof of Proposition \ref{UBalt10_u}}

As for all our Bonferroni type aggregated tests of Section~~\ref{Sec:unknownbaseline} dedicated to change detection from an unknown baseline intensity, the control of the first kind error rates of the two tests $\phi_{9/10,\alpha}^{u(1)}$ and $\phi_{9/10,\alpha}^{u(2)}$ is easily deduced from union bounds and the conditioning trick of the above proofs for upper bounds. We therefore focus here on the second kind error rates.

\medskip

\emph{$(i)$ Control of the second kind error rate of $\phi_{9/10,\alpha}^{u(1)}$.}

\smallskip

Let $\lambda$ in $\calS^u_{\bbul,\bbul\bbul,\bbul\bbul\bbul}[R]$ such that $\lambda = \lambda_{0} + \delta \mathds{1}_{(\tau,\tau+ \ell]}$, with $\lambda_0$ in $(0,R]$, $\tau$ in $(0,1)$, $\delta$ in $(-\lambda_{0}, R- \lambda_0] \setminus \lbrace 0 \rbrace$ and $\ell$ in $(0,1-\tau)$. By now, we aim at proving that  $P_{\lambda}\pa{ \phi_{9/10,\alpha}^{u(1)}(N)=0} \leq \beta$ as soon as we assume that
 \begin{multline} \label{UBalt10_u1_cond1}
  d_2(\lambda, \mathcal{S}^u_0[R]) \geq  \max \Bigg(  2\sqrt{\frac{R\log \left(  2/u_\alpha^{(1)}  \right)}{L}}~,~ 6  \sqrt{\frac{2 \log \left(  2/u_\alpha^{(1)} \right)}{L}} \sqrt{R + \sqrt{\frac{2R}{ \beta L}}}\\ + 6 \sqrt{\frac{2R}{\beta L}}~  ,~ \frac{\sqrt{3}R}{\sqrt{ L}}           \Bigg) \enspace.
 \end{multline}

Let us first consider the case where $\delta$ belongs to $(0,R-\lambda_0]$.

Noticing that
  \begin{multline*} 
  P_{\lambda}\pa{ \phi_{9/10,\alpha}^{(1)}(N)=0}\\\leq \inf_{k \in \lbrace 0,\ldots,\lceil L\rceil-1 \rbrace} \inf_{k' \in \lbrace 1,\ldots,\lceil L\rceil-k \rbrace} P_{\lambda} \left(   S'_{\frac{k}{\lceil L\rceil },\frac{k+k'}{\lceil L\rceil }}(N)  \leq s'_{N_1,\frac{k}{\lceil L\rceil },\frac{k+k'}{\lceil L\rceil }}\left( 1-u_\alpha^{(1)} \right)  \right)\enspace ,
  \end{multline*}
one can see that it is enough to exhibit some $k_0$ in $ \lbrace 0,\ldots,\lceil L\rceil-1 \rbrace$ and $k'_0$ in $\lbrace 1,\ldots,\lceil L\rceil-k_0 \rbrace$ such that 
$$P_{\lambda} \left( S'_{\frac{k_0}{\lceil L\rceil },\frac{k_0+k'_0}{\lceil L\rceil }}(N)   \leq s'_{N_1,\frac{k_0}{\lceil L\rceil },\frac{k_0+k'_0}{\lceil L\rceil }}\left( 1-u_\alpha^{(1)} \right)   \right) \leq \beta \enspace.$$
  We get from (\ref{UBalt10_u1_cond1}) that 
 $
  d_2^2(\lambda, \mathcal{S}^u_0[R]) \geq 3 R^2/ L \geq 3R^2/\lceil L\rceil $ which entails
 \begin{equation}\label{UBalt10_u1_eq1}
  \ell(1-\ell) > 3/ \lceil L\rceil \enspace. 
\end{equation}  

 Assume first that $\ell \leq 1/2.$ The condition (\ref{UBalt10_u1_eq1}) leads to
 \begin{equation} \label{UBalt10_u1_eq2}
  \ell > 3/ \lceil L\rceil~ \textrm{and}~ \tau < 1-3/ \lceil L\rceil  \enspace.
 \end{equation}
 Therefore, we can define
 $ k_{0} = \min (  k \in \lbrace 0,\ldots,\lceil L\rceil-1 \rbrace,~ \tau \leq k/\lceil L\rceil  )   $ and $ k'_{0} = \max (k' \in \lbrace 1,\ldots,\lceil L\rceil-k_0 \rbrace,~ (k_0 + k')/\lceil L\rceil \leq \tau + \ell )$, so that $\tau \leq k_0/\lceil L\rceil < (k_0 + k'_0)/\lceil L\rceil \leq \tau + \ell$.
  Since by definition $k_{0}/\lceil L\rceil - \tau < 1/\lceil L\rceil$ and $ \tau + \ell - (k_{0}+k'_{0})/\lceil L\rceil < 1/\lceil L\rceil$, we get that
 $$ \frac{k'_{0} }{\lceil L\rceil} = \ell - \left(  \left( \frac{k_{0}}{\lceil L\rceil}- \tau    \right) + \left( \tau + \ell - \frac{k_{0}+k'_{0}}{\lceil L\rceil}    \right)  \right) > \ell - \frac{2}{\lceil L\rceil},$$
 and then, combining with \eqref{UBalt10_u1_eq2},
 \begin{equation} \label{UBalt10_u1_eq3}
   \frac{k'_{0} }{\lceil L\rceil}
  > \frac{\ell}{3} \enspace.
  \end{equation}
Lemma \ref{momentPoisson} leads to $E_\lambda \left[S'_{k_0/\lceil L\rceil ,(k_0+k'_0)/\lceil L\rceil }(N)  \right] = \delta (1-\ell)L  k'_{0}/\lceil L\rceil $ and $$ \mathrm{Var}_\lambda \cro{ S'_{\frac{k_0}{\lceil L\rceil },\frac{k_0+k'_0}{\lceil L\rceil }}(N)  } = \pa{\lambda_0 \pa{ 1- \frac{k'_{0}}{\lceil L\rceil}}  + \delta \pa{1- (2-\ell) \frac{k'_{0}}{\lceil L\rceil} }} \frac{k'_{0}}{\lceil L\rceil}L \enspace.$$
With \eqref{UBalt10_u1_eq3}, we obtain on the one hand
\begin{equation}  \label{UBalt10_u1_esp}
E_\lambda \left[ S'_{\frac{k_0}{\lceil L\rceil },\frac{k_0+k'_0}{\lceil L\rceil }}(N)  \right]  > \frac{\delta}{3} \ell(1-\ell)L \enspace.
\end{equation}
On the other hand, $k'_{0}/\lceil L\rceil \leq \ell<1$ yields $0< 1-(2-\ell)k'_{0}/ \lceil L\rceil < 1- k'_{0}/\lceil L\rceil.$
Moreover, using $k'_{0}/\lceil L\rceil \leq \ell$ again and the fact that $\ell \leq 1/2$ one obtains
\begin{equation}  \label{UBalt10_u1_eq4}
\frac{k'_{0}}{\lceil L\rceil} \pa{1- \frac{k'_{0}}{\lceil L\rceil}} \leq \ell(1-\ell) \enspace,
\end{equation}  
 and we finally get
\begin{equation}  \label{UBalt10_u1_var}
\mathrm{Var}_\lambda \cro{ S'_{\frac{k_0}{\lceil L\rceil },\frac{k_0+k'_0}{\lceil L\rceil }}(N)  } \leq (\lambda_0 + \delta) \ell(1-\ell)L \enspace.
\end{equation}

Recall that $I(\lambda)=\int_0^1 \lambda(t) dt$ and notice that (\ref{UBalt10_u1_cond1}) entails
$$    \delta \sqrt{\ell(1-\ell)} \geq  \max \left(  2 \sqrt{\frac{\delta \log \left(   2/u_\alpha^{(1)} \right)}{L}}~,~ 6  \sqrt{\frac{2\log \left(   2/u_\alpha^{(1)} \right)}{L}} \sqrt{I(\lambda) + \sqrt{\frac{2 I(\lambda)}{ \beta L}}} + 6\sqrt{\frac{2(\lambda_0 + \delta)}{\beta L}}          \right)\enspace ,$$
thereby
$$   \delta \ell(1-\ell) \geq  2 \max \left(    \frac{2\log \left(   2/u_\alpha^{(1)} \right)}{L}~,~ 3 \sqrt{\ell(1-\ell)} \left(   \sqrt{\frac{2\log \left(   2/u_\alpha^{(1)}  \right)}{L}} \sqrt{I(\lambda) + \sqrt{\frac{2 I(\lambda)}{ \beta L}}} +  \sqrt{\frac{2(\lambda_0 + \delta)}{\beta L}}  \right)        \right) \enspace,$$
and then \begin{equation} \label{UBalt10_u1_cond2}
 \frac{  \delta }{3}  \ell(1-\ell) L \geq    \frac{2}{3}  \log \left(  \frac{2}{ u_\alpha^{(1)}}  \right) +   \sqrt{\ell(1-\ell)}    \sqrt{2 L \log \left(  \frac{2}{ u_\alpha^{(1)}}  \right)} \sqrt{I(\lambda) + \sqrt{\frac{2 I(\lambda)}{ \beta L}}} +  \sqrt{\frac{2 \ell(1-\ell) L(\lambda_0 + \delta)}{\beta}} \enspace .
\end{equation}
  
Thus, with (\ref{UBalt10_u1_esp}) and (\ref{UBalt10_u1_var}), the inequality (\ref{UBalt10_u1_cond2}) ensures that
\begin{multline} \label{UBalt10_u1_cond3}
E_\lambda \left[ S'_{\frac{k_0}{\lceil L\rceil },\frac{k_0+k'_0}{\lceil L\rceil }}(N)   \right] \geq \frac{2}{3}  \log \left(  \frac{2}{ u_\alpha^{(1)}}  \right) +   \sqrt{\ell(1-\ell)}    \sqrt{ 2L \log \left(  \frac{2}{ u_\alpha^{(1)}}  \right)} \sqrt{I(\lambda) + \sqrt{\frac{2 I(\lambda)}{ \beta L}}}\\ +\sqrt{\frac{2 \mathrm{Var}_\lambda \cro{ S'_{k_0/\lceil L\rceil ,\pa{k_0+k'_0}/\lceil L\rceil }(N)   }}{\beta}} \enspace.
\end{multline}

Furthermore, we deduce from the inequality 
$$P_\lambda \left(  N_1 \geq L I(\lambda) + \sqrt{\frac{2L I(\lambda)}{\beta}}   \right) \leq \frac{\beta}{2}\enspace,$$
combined with Lemma \ref{QuantilesAbsS_u} that
\begin{equation} \label{UBalt10_u1_eq5}
 P_\lambda \left( s'_{N_1,\frac{k_0}{\lceil L\rceil },\frac{k_0+k'_0}{\lceil L\rceil }}\left( 1-u_\alpha^{(1)} \right)    \geq  Q(\alpha, \beta,L, k'_{0})  \right) \leq \frac{\beta}{2} \enspace, 
 \end{equation}
with
\begin{equation} \label{UBalt10_u1_eq6}
 Q(\alpha, \beta,L, k'_{0})= \frac{2}{3}  \log \left(  \frac{2}{ u_\alpha^{(1)}}  \right) +  \sqrt{\frac{k'_{0}}{\lceil L\rceil} \pa{1- \frac{k'_{0}}{\lceil L\rceil}}} \sqrt{2 \log \left(  \frac{2}{ u_\alpha^{(1)}}  \right)  } \sqrt{L I(\lambda) +  \sqrt{\frac{2L I(\lambda)}{ \beta}} } \enspace.
 \end{equation}
Using (\ref{UBalt10_u1_eq4}), the inequality (\ref{UBalt10_u1_cond3}) leads to
\begin{equation} \label{UBalt10_u1_cond4}
E_\lambda \left[ S'_{\frac{k_0}{\lceil L\rceil },\frac{k_0+k'_0}{\lceil L\rceil }}(N)  \right] \geq Q(\alpha, \beta,L,k'_{0})+ \sqrt{2 \mathrm{Var}_\lambda \cro{ S'_{\frac{k_0}{\lceil L\rceil },\frac{k_0+k'_0}{\lceil L\rceil }}(N) }/\beta} \enspace,
\end{equation}
and we conclude with the following inequalities
\begin{align*}
P_\lambda &\left( S'_{\frac{k_0}{\lceil L\rceil },\frac{k_0+k'_0}{\lceil L\rceil }}(N)   \leq s'_{N_1,\frac{k_0}{\lceil L\rceil },\frac{k_0+k'_0}{\lceil L\rceil }}\left( 1-u_\alpha^{(1)} \right)   \right) \\ &\leq P_\lambda \left( S'_{\frac{k_0}{\lceil L\rceil },\frac{k_0+k'_0}{\lceil L\rceil }}(N)   < Q(\alpha, \beta,L, k'_{0}) \right) + P_\lambda \left(s'_{N_1,\frac{k_0}{\lceil L\rceil },\frac{k_0+k'_0}{\lceil L\rceil }}\left( 1-u_\alpha^{(1)} \right)   \geq Q(\alpha, \beta,L, k'_{0}) \right) \\
&\leq P_\lambda \left( S'_{\frac{k_0}{\lceil L\rceil },\frac{k_0+k'_0}{\lceil L\rceil }}(N)   < Q(\alpha, \beta,L, k'_{0}) \right)  +\frac{\beta}{2}~~\text{with (\ref{UBalt10_u1_eq5}) } \\
&\leq P_\lambda \left( S'_{\frac{k_0}{\lceil L\rceil },\frac{k_0+k'_0}{\lceil L\rceil}}(N)   - E_\lambda \left[ S'_{\frac{k_0}{\lceil L\rceil},\frac{k_0+k'_0}{\lceil L\rceil }}(N)   \right] <  -\sqrt{2 \mathrm{Var}_\lambda \left(S'_{\frac{k_0}{\lceil L\rceil },\frac{k_0+k'_0}{\lceil L\rceil }}(N)  \right)/\beta}   \right) +\frac{\beta}{2}~~\text{with (\ref{UBalt10_u1_cond4}) } \\
&\leq \beta~~\text{with the Bienayme-Chebyshev inequality}\enspace.
\end{align*}
~\\
Assume now that $\ell >1/2$. 
We define
 $ k_{0} = \max (  k \in \lbrace 0,\ldots,\lceil L\rceil-1 \rbrace,~ \tau \geq k/\lceil L\rceil  )   $ and $ k'_{0} = \min (k' \in\lbrace 1,\ldots,\lceil L\rceil-k_0 \rbrace,~ (k_{0} + k')/\lceil L\rceil \geq \tau + \ell )$ so that $ k_{0}/\lceil L\rceil \leq \tau< \tau+\ell \leq (k_{0} + k'_{0})/\lceil L\rceil$. Notice that the condition (\ref{UBalt10_u1_eq1}) entails
\begin{equation} \label{UBalt10_u1_eq7}
1-\ell > {3}/{\lceil L\rceil} \enspace.
\end{equation}
Since by definition of $k_0$ and $k_0'$, $\tau-k_0/\lceil L\rceil<1/\lceil L\rceil$ and $(k_0+k_0')/\lceil L\rceil-(\tau+\ell)<1/\lceil L\rceil$, then, with \eqref{UBalt10_u1_eq7},
$$\frac{k_0'}{\lceil L\rceil}=\frac{k_{0}+k'_{0}}{\lceil L\rceil}   - (\tau +\ell) +\ell + \tau-\frac{k_0}{\lceil L\rceil} <1/\lceil L\rceil+(1-3/\lceil L\rceil)+1/\lceil L\rceil \leq 1-1/\lceil L\rceil<1\enspace.$$
In the same way,  we can also obtain with \eqref{UBalt10_u1_eq7} that
 \begin{equation} \label{UBalt10_u1_eq8}
 1- \frac{k'_{0} }{\lceil L\rceil} = 1- \ell - \left(  \left( \tau - \frac{k_{0}}{\lceil L\rceil}    \right) + \left( \frac{k_{0}+k'_{0}}{\lceil L\rceil}   - (\tau +\ell) \right)  \right) > 1- \ell - \frac{2}{\lceil L\rceil}> \frac{1-\ell}{3} \enspace.
  \end{equation}
Furthermore, by definition of $k_{0}$ and $k'_{0}$, one has $k'_{0}/\lceil L\rceil \geq \ell$ and since $\ell > 1/2$,
\begin{equation}  \label{UBalt10_u1_eq9}
\frac{k'_{0}}{\lceil L\rceil} \pa{1- \frac{k'_{0}}{\lceil L\rceil}} \leq \ell(1-\ell)\enspace.
\end{equation}  
From Lemma \ref{momentPoisson}, we therefore deduce
\begin{align*}
 E_\lambda \left[ S'_{\frac{k_0}{\lceil L\rceil },\frac{k_0+k'_0}{\lceil L\rceil }}(N)  \right] &= \delta \ell \pa{1- \frac{k'_{0}}{\lceil L\rceil}} L \\
 & > \frac{\delta}{3 }\ell(1-\ell)L ~~\text{with (\ref{UBalt10_u1_eq8})} \enspace,
 \end{align*}
 and since $k'_0/\lceil L\rceil \geq \ell$,
\begin{align*}
 \mathrm{Var}_\lambda \cro{S'_{\frac{k_0}{\lceil L\rceil },\frac{k_0+k'_0}{\lceil L\rceil }}(N)  } &=\pa{\lambda_0 \frac{k'_{0}}{\lceil L\rceil} + \delta \ell \pa{1- \frac{k'_{0}}{\lceil L\rceil}}  }   \pa{1- \frac{k'_{0}}{\lceil L\rceil}} L \\
 &\leq \pa{\lambda_0  + \delta  \pa{1- \frac{k'_{0}}{L}}  }   \pa{1- \frac{k'_{0}}{L}} \frac{k'_{0}}{\lceil L\rceil}L\\
 &\leq (\lambda_0 + \delta) \ell(1-\ell)L ~~\text{with (\ref{UBalt10_u1_eq9})}\enspace.
 \end{align*}

Thus, the condition (\ref{UBalt10_u1_cond2}) also yields
 $$E_\lambda \left[S'_{\frac{k_0}{\lceil L\rceil },\frac{k_0+k'_0}{\lceil L\rceil}}(N)  \right] \geq Q(\alpha, \beta,L,k'_{0})+ \sqrt{2 \mathrm{Var}_\lambda \cro{ S'_{\frac{k_0}{\lceil L\rceil },\frac{k_0+k'_0}{\lceil L\rceil }}(N)  }/\beta}\enspace,$$ where $Q(\alpha, \beta,L,k'_{0})$ is defined by (\ref{UBalt10_u1_eq6}) and we conclude that $$P_\lambda \left(S'_{\frac{k_0}{\lceil L\rceil },\frac{k_0+k'_0}{\lceil L\rceil }}(N)  \leq s'_{N_1,\frac{k_0}{\lceil L\rceil},\frac{k_0+k'_0}{\lceil L\rceil }}\left( 1-u_\alpha^{(1)} \right)   \right) \leq \beta \enspace,$$ with the same arguments as above.

\medskip
  
Let us then treat the case where $\delta$ belongs to $(-\lambda_0,0).$ We start by noticing that 
   \begin{multline*} 
  P_{\lambda}\pa{ \phi_{9/10,\alpha}^{(1)}(N)=0}\\
  \leq \inf_{k \in \lbrace 0,\ldots,\lceil L\rceil-1 \rbrace} \inf_{k' \in \lbrace 1,\ldots,\lceil L\rceil-k \rbrace}  P_{\lambda} \left(  - S'_{\frac{k}{\lceil L\rceil },\frac{k+k'}{\lceil L\rceil}}(N) \leq s'_{N_1,\frac{k}{\lceil L\rceil },\frac{k+k'}{\lceil L\rceil}}\left( 1-u_\alpha^{(1)} \right)  \right) \enspace.
  \end{multline*}
The same choice of $k_{0}$ and $k'_{0}$ as in the case where $\delta$ belongs to $(0,R-\lambda_0]$  leads to $$ E_\lambda \left[-S'_{\frac{k_0}{\lceil L\rceil },\frac{k_0+k'_0}{\lceil L\rceil }}(N)  \right] > \frac{\vert \delta \vert \ell(1-\ell)L}{3}$$ and 
 $$\mathrm{Var}_\lambda \cro{ S'_{\frac{k_0}{\lceil L\rceil },\frac{k_0+k'_0}{\lceil L\rceil }}(N)  } \leq \lambda_{0} \frac{k'_{0}}{\lceil L\rceil} \pa{1-\frac{k'_{0}}{\lceil L\rceil}}L \leq R \frac{k'_{0}}{\lceil L\rceil} \pa{1-\frac{k'_{0}}{\lceil L\rceil}}L \enspace,$$
 and we obtain
  $$P_{\lambda} \left( -S'_{\frac{k_0}{\lceil L\rceil },\frac{k_0+k'_0}{\lceil L\rceil }}(N)   \leq s'_{N_1,\frac{k_0}{\lceil L\rceil },\frac{k_0+k'_0}{\lceil L\rceil }}\left( 1-u_\alpha^{(1)} \right)   \right) \leq \beta \enspace,$$
following the same lines of proof as above, but with $\delta$ replaced by $\vert \delta \vert$ except when it is involved in $\lambda_0 + \delta$. 

Coming back to the assumption \eqref{UBalt10_u1_cond1} and the definition of $u_\alpha^{(1)}$, one can finally claim that
\begin{multline}
\SRb\Big(\phi_{9/10,\alpha}^{u(1)},\calS^u_{\bbul,\bbul\bbul,\bbul\bbul\bbul}[R]\Big) \leq \max \Bigg(  2\sqrt{\frac{R\log \left(  \lceil L\rceil ( \lceil L\rceil +1)/\alpha  \right)}{L}}~,\\
6  \sqrt{\frac{2 \log \left( \lceil L\rceil(\lceil L\rceil +1)/\alpha \right)}{L}} \sqrt{R + \sqrt{\frac{2R}{ \beta L}}} + 6 \sqrt{\frac{2R}{\beta L}}~  ,~ \frac{\sqrt{3}R}{\sqrt{ L}}           \Bigg) \enspace,
\end{multline}
 which leads to the result stated in Proposition  \ref{UBalt10_u} for $\phi_{9/10,\alpha}^{u(1)}$ and the set $\calS^u_{\bbul,\bbul\bbul,\bbul\bbul\bbul}[R]$ of the alternative $\bold{[Alt^{u}.10]}$.
Since $\calS^u_{\delta^*,\bbul\bbul,\bbul\bbul\bbul}[R]\subset \calS^u_{\bbul,\bbul\bbul,\bbul\bbul\bbul}[R]$ for any $\delta^*$ in $(-R,R) \setminus \lbrace 0 \rbrace$,  the same result holds for $\phi_{9/10,\alpha}^{u(1)}$ and the set $\calS^u_{\delta^*,\bbul\bbul,\bbul\bbul\bbul}[R]$ of the alternative $\bold{[Alt^{u}.9]}$. Notice that in this case, the constant $C(\alpha,\beta,R)$ involved in the upper bound can be refined, benefiting from  the knowledge of $\delta^*$, which explains the formulation with $C(\alpha,\beta,R,\delta^*)$ instead of $C(\alpha,\beta,R)$ in Proposition~\ref{UBalt10_u}.

\medskip

\emph{$(ii)$ Control of the second kind error rate of $\phi_{9/10,\alpha}^{u(2)}$.}

\smallskip

Let $\lambda$ in $\calS^u_{\bbul,\bbul\bbul,\bbul\bbul\bbul}[R]$ such that $\lambda = \lambda_{0} + \delta \mathds{1}_{(\tau,\tau+ \ell]}$, with $\lambda_0$ in $(0,R]$, $\tau$ in $(0,1)$, $\delta$ in $(-\lambda_{0}, R- \lambda_0] \setminus \lbrace 0 \rbrace$ and $\ell$ in $(0,1-\tau)$. We assume by now that
 \begin{multline} \label{UBalt10_u2_cond1}
 d_2(\lambda, \mathcal{S}^u_0[R]) \geq \max \Bigg( \sqrt{\frac{60CR\log \left(   2.77/u_\alpha^{(2)} \right)}{L}} +   \pa{\frac{2R}{\beta}}^{1/4} \frac{ \sqrt{60C \log \left(   2.77/u_\alpha^{(2)} \right)}}{L^{3/4}}\\+ 6\sqrt{ C}\frac{\log \left(   2.77/u_\alpha^{(2)} \right)}{\sqrt{ L \log L}} + 2\sqrt{\frac{6R}{L\sqrt{\beta}}}~,~24 \sqrt{\frac{2R}{ \beta L}}~,~R \sqrt{\frac{3\log L}{L}}  \Bigg) \enspace ,
\end{multline}
where $C$ is the constant defined in Lemma \ref{quantile_T'}. 

Let us prove that this assumption implies  $P_{\lambda} \pa{\phi_{9/10,\alpha}^{u(2)}(N)=0} \leq \beta$.

As above, we invoke that
\begin{multline*} 
  P_{\lambda} \pa{\phi_{9/10,\alpha}^{u(2)}(N)=0} \\
   \leq \inf_{k \in \lbrace 0,\ldots,M_L-1 \rbrace} 
\inf_{\substack{k' \in \lbrace 1,\ldots,M_L-k \rbrace \\ (k,k') \neq (0,M_L)}} P_{\lambda} \left( T'_{\frac{k}{M_L}, \frac{k+k'}{M_L}}(N) \leq t'_{N_1,\frac{k}{M_L}, \frac{k+k'}{M_L}}\left( 1- u_{\alpha}^{(2)} \right) \right) \enspace,
  \end{multline*}
 to argue that we only need  to exhibit some $k_{0}$ in $\lbrace 0,\ldots,M_L-1 \rbrace$ and $k'_{0}$ in $\lbrace 1,\ldots,M_L-k_{0} \rbrace$ such that $(k_0,k'_0) \neq (0,M_L)$ satisfying $$P_{\lambda} \left( T'_{\frac{k_0}{M_L}, \frac{k_0+k'_0}{M_L}}(N) \leq t'_{N_1,\frac{k_0}{M_L}, \frac{k_0+k'_0}{M_L}}\left( 1- u_{\alpha}^{(2)} \right)  \right) \leq \beta \enspace.$$
Recalling that $M_L=\lceil L/\log L \rceil$, we get from (\ref{UBalt10_u2_cond1}) that 
 $
  d_2^2(\lambda, \mathcal{S}^u_0[R]) \geq 3 R^2 \log L/ L  \geq 3 R^2/M_L 
 $ which entails
 \begin{equation}\label{UBalt10_u2_eq1}
  \ell(1-\ell) > 3/ M_L \enspace. 
\end{equation}  

\medskip

Let us first consider the case where $\ell \leq 1/2$. 

The condition (\ref{UBalt10_u2_eq1}) leads to
 \begin{equation} \label{UBalt10_u2_eq2}
  \ell > 3/ M_L \enspace,
 \end{equation}
 and therefore, we can define
 $ k_{0} = \min (  k \in \lbrace 1,\ldots, M_L-1 \rbrace,~ \tau \leq k/M_L  )   $ and $ k'_{0}= \max (k' \in \lbrace 1,\ldots,M_L-k_{0} \rbrace,~ (k_{0} + k')/M_L \leq \tau + \ell )$, so that $\tau \leq k_{0}/M_L < (k_{0} + k'_{0})/M_L \leq \tau + \ell$.
  Since by definition $k_{0}/M_L - \tau < 1/M_L$ and $ \tau + \ell - (k_{0}+k'_{0})/M_L < 1/M_L$, notice that
 $$ \frac{k'_{0} }{M_L} = \ell - \left(  \left( \frac{k_{0}}{M_L}- \tau    \right) + \left( \tau + \ell - \frac{k_{0}+k'_{0}}{M_L}    \right)  \right) > \ell - \frac{2}{M_L}\enspace,$$
 and then, combined with (\ref{UBalt10_u2_eq2}),
 \begin{equation} \label{UBalt10_u2_eq3}
   \frac{k'_{0} }{M_L}
  > \frac{\ell}{3} \enspace.
  \end{equation}
Lemma \ref{MomentsT'} gives
$E_\lambda \left[ T'_{k_0/M_L, \pa{k_0+k'_0}/M_L}(N) \right] = \delta^2 (1-\ell)^2   (k'_{0} /M_L)/(1-k'_{0} /M_L)$, and with (\ref{UBalt10_u2_eq3}) and the facts that $1-\ell \geq 1/2$ and $1/(1-k'_{0} /M_L) >1$, we get
\begin{equation}  \label{UBalt10_u2_esp}
E_\lambda \left[ T'_{\frac{k_0}{M_L}, \frac{k_0+k'_0}{M_L}}(N) \right]  > \frac{\delta^2}{6} \ell(1-\ell)\enspace.
\end{equation}

Lemma \ref{MomentsT'} also gives
\begin{multline}
 \mathrm{Var}_\lambda \cro{  T'_{\frac{k_0}{M_L}, \frac{k_0+k'_0}{M_L}}(N)  } =\frac{2}{L^2} \pa{\lambda_0   + \delta \frac{1-(2-\ell)k'_{0}/M_L}{1-k'_{0}/M_L} }^2 \\ + \frac{4}{L} \delta^{2} (1-\ell)^2 \frac{k'_{0}/M_L}{1-k'_{0}/M_L} \pa{\lambda_0   + \delta \frac{1-(2-\ell)k'_{0}/M_L}{1-k'_{0}/M_L} } \enspace,
 \end{multline}
and since $k'_0/M_L \leq \ell<1$, notice that $0< 1-(2-\ell)k'_{0}/ M_L < 1- k'_{0}/M_L$.
Moreover using $k'_{0}/M_L \leq \ell$ again, one obtains
\begin{equation}  \label{UBalt10_u2_eq4}
\frac{k'_{0}/M_L} {1- k'_{0}/M_L} \leq \frac{\ell}{1-\ell} \enspace.
\end{equation}  
 Therefore, using (\ref{UBalt10_u2_eq4}),
\begin{multline*}
\mathrm{Var}_\lambda \cro{ T'_{\frac{k_0}{M_L}, \frac{k_0+k'_0}{M_L}}(N)  } \leq
\pa{ \frac{2}{L^2}(\lambda_0 + \delta)^2 + \frac{4}{L}(\lambda_0 + \delta) \delta^2 \ell(1-\ell)} \mathds{1}_{\delta>0}\\
 +
\pa{ \frac{2}{L^2} \lambda_0^2 + \frac{4}{L} \lambda_0 \delta^2 \ell(1-\ell)} \mathds{1}_{\delta<0} \enspace,
\end{multline*}
hence
\begin{align}   \label{UBalt10_u2_var1}
\mathrm{Var}_\lambda \cro{ T'_{\frac{k_0}{M_L}, \frac{k_0+k'_0}{M_L}}(N)  } & \leq
 \frac{2R^2}{L^2} + \frac{4R}{L} \delta^2 \ell(1-\ell) \enspace.
\end{align}
    
 Now, on the one hand, notice that the assumption \eqref{UBalt10_u2_cond1} ensures that
 \begin{multline*}
 \vert \delta \vert \sqrt{\ell (1-\ell)}  \geq \sqrt{\frac{60CR\log \left(   2.77/u_\alpha^{(2)} \right)}{L}} +   \pa{\frac{2R}{\beta}}^{1/4} \frac{ \sqrt{60C \log \left(   2.77/u_\alpha^{(2)} \right)}}{L^{3/4}}\\+ 6\sqrt{ C(M_L-1)}\frac{\log \left(   2.77/u_\alpha^{(2)} \right)}{L} + 2\sqrt{\frac{6R}{L\sqrt{\beta}}} \enspace ,
\end{multline*}
 which entails
 \begin{multline} \label{UBalt10_u2_cond2}
  \delta^{2} \ell (1-\ell) \geq  60 C R\frac{ \log \left(  2.77/u_{\alpha}^{(2)}  \right)}{L} + 60C\sqrt{\frac{2R}{\beta}} \frac{\log \left(  2.77/u_{\alpha}^{(2)}  \right)}{L^{3/2}} \\
   + 36 C (M_L -1)\frac{\log^2 \left(  2.77/u_{\alpha}^{(2)} \right)}{ L^2  } + \frac{24 R}{\sqrt{\beta} L} \enspace.
 \end{multline}
 On the other hand, the same assumption \eqref{UBalt10_u2_cond1} ensures that
$
  \vert \delta \vert \sqrt{\ell (1-\ell)} \geq  24 \sqrt{2R/(\beta L)},
 $
 which leads to 
 \begin{equation} \label{UBalt10_u2_cond2bis}
 \delta^{2} \ell (1-\ell) \geq 24 \sqrt{\frac{2R}{\beta L}}  \vert \delta \vert \sqrt{\ell (1-\ell)} \enspace.
  \end{equation}
Therefore, \eqref{UBalt10_u2_cond2} and \eqref{UBalt10_u2_cond2bis} imply
 \begin{multline*} 
  \delta^{2} \ell (1-\ell) \geq  2 \max \Bigg( C \Bigg( 30 R \frac{\log \left(  2.77/u_{\alpha}^{(2)} \right)}{L} + 30 \sqrt{\frac{2R}{\beta}} \frac{\log \left(  2.77/u_{\alpha}^{(2)} \right)}{L^{3/2}} \\ + 18 (M_L -1)\frac{\log^2 \left(  2.77/u_{\alpha}^{(2)}  \right)}{ L^2  } \Bigg) + \frac{12 R}{\sqrt{\beta} L}~,  ~
   12 \sqrt{\frac{2R}{\beta L}}  \vert \delta \vert \sqrt{\ell (1-\ell)}          \Bigg) \enspace,
 \end{multline*}
so
\begin{multline} \label{UBalt10_u2_cond3}
  \frac{\delta^{2}}{6} \ell (1-\ell) \geq   C \Bigg( 5 R \frac{\log \left(  2.77/u_{\alpha}^{(2)}  \right)}{L} + 5 \sqrt{\frac{2R}{\beta}} \frac{\log \left(  2.77/u_{\alpha}^{(2)}  \right)}{L^{3/2}}   + 3 (M_L -1)\frac{\log^2 \left(  2.77/u_{\alpha}^{(2)}  \right)}{ L^2  } \Bigg) \\ + \frac{2 R}{\sqrt{\beta}L}+ 2 \sqrt{\frac{2R}{\beta L}}  \vert \delta \vert \sqrt{\ell (1-\ell)} \enspace .
 \end{multline}

Thus, with (\ref{UBalt10_u2_esp}) and (\ref{UBalt10_u2_var1}), the inequality (\ref{UBalt10_u2_cond3}) ensures that
\begin{multline} \label{UBalt10_u2_cond4}
E_\lambda \left[ T'_{\frac{k_0}{M_L}, \frac{k_0+k'_0}{M_L}}(N) \right] \geq  C \Bigg( 5 R \frac{\log \left(  2.77/u_{\alpha}^{(2)}  \right)}{L} + 5 \sqrt{\frac{2R}{\beta}} \frac{\log \left(  2.77/u_{\alpha}^{(2)}  \right)}{L^{3/2}}  \\ + 3 (M_L -1)\frac{\log^2 \left(  2.77/u_{\alpha}^{(2)}  \right)}{ L^2  } \Bigg) + \sqrt{2 \mathrm{Var}_\lambda \cro{ T'_{\frac{k_0}{M_L}, \frac{k_0+k'_0}{M_L}}(N)   }/\beta} \enspace.
\end{multline}

Furthermore, the inequality
$$P_\lambda \left(  N_1 \geq L I(\lambda) + \sqrt{\frac{2L I(\lambda)}{\beta}}   \right) \leq \frac{\beta}{2}\enspace,$$
combined with Lemma \ref{quantile_T'} and the upper bound $I(\lambda) \leq R$, leads to
\begin{align} \label{UBalt10_u2_eq5}
 P_{\lambda} &\left( t'_{N_1,\frac{k_0}{M_L}, \frac{k_0+k'_0}{M_L}} \left( 1- u_{\alpha}^{(2)} \right)  \geq Q(\alpha, \beta,L, k'_{0}) \right)   \leq \frac{\beta}{2} \enspace,
 \end{align}
with
\begin{multline} \label{UBalt10_u2_eq6}
 Q(\alpha, \beta,L, k'_{0})= C \Bigg(  5R \frac{ \log \left(  2.77/u_{\alpha}^{(2)} \right)}{L} + 5 \sqrt{\frac{2R}{\beta} } \frac{\log \left(  2.77/u_{\alpha}^{(2)}  \right)}{L^{3/2}}  \\ +    3 \max \left( \frac{k'_{0}/M_L}{1- k'_{0}/M_L}, \frac{1- k'_{0}/M_L}{k'_{0}/M_L}    \right)   \frac{ \log^2 \left(  2.77/u_{\alpha}^{(2)}  \right)}{L^2}    \Bigg) \enspace.
 \end{multline}
Since $k'_{0}/M_L \leq \ell \leq 1/2$, one has
$$\max \left( \frac{k'_{0}/M_L}{1- k'_{0}/M_L}, \frac{1- k'_{0}/M_L}{k'_{0}/M_L}    \right) = \frac{1- k'_{0}/M_L}{k'_{0}/M_L} = \frac{M_L}{k'_{0}} -1 \leq M_L -1 \enspace,$$
and  the inequality (\ref{UBalt10_u2_cond4}) leads to
\begin{equation} \label{UBalt10_u2_cond5}
E_\lambda \left[ T'_{\frac{k_0}{M_L}, \frac{k_0+k'_0}{M_L}}(N) \right] \geq Q(\alpha, \beta,L,k'_{0})+ \sqrt{2 \mathrm{Var}_\lambda \cro{ T'_{\frac{k_0}{M_L}, \frac{k_0+k'_0}{M_L}}(N)   } /\beta} \enspace.
\end{equation}
We conclude with the following inequalities
\begin{align*}
P_\lambda &\pa{ T'_{\frac{k_0}{M_L}, \frac{k_0+k'_0}{M_L}}(N)  \leq t'_{N_1,\frac{k_0}{M_L}, \frac{k_0+k'_0}{M_L}} \left( 1- u_{\alpha}^{(2)} \right)  } \\
 &\leq P_\lambda \left( T'_{\frac{k_0}{M_L}, \frac{k_0+k'_0}{M_L}}(N)  < Q(\alpha, \beta,L, k'_{0}) \right) + P_\lambda \left( t'_{N_1,\frac{k_0}{M_L}, \frac{k_0+k'_0}{M_L}}\left( 1- u_{\alpha}^{(2)} \right)  \geq Q(\alpha, \beta,L, k'_{0}) \right) \\
&\leq P_\lambda \left( T'_{\frac{k_0}{M_L}, \frac{k_0+k'_0}{M_L}}(N)  < Q(\alpha, \beta,L, k'_{0}) \right)  +\frac{\beta}{2}~~\text{with (\ref{UBalt10_u2_eq5}) } \\
&\leq P_\lambda \left( T'_{\frac{k_0}{M_L}, \frac{k_0+k'_0}{M_L}}(N)  - E_\lambda \left[ T'_{\frac{k_0}{M_L}, \frac{k_0+k'_0}{M_L}}(N) \right] <  -\sqrt{2 \mathrm{Var}_\lambda \cro{ T'_{\frac{k_0}{M_L}, \frac{k_0+k'_0}{M_L}}(N)  }/\beta}   \right) +\frac{\beta}{2}~~\text{with (\ref{UBalt10_u2_cond5}) } \\
&\leq \beta~~\text{with the Bienayme-Chebyshev inequality}\enspace.
\end{align*}

\medskip

Let us then consider the case where $\ell> 1/2$.

We define
 $ k_{0} = \max (  k \in \lbrace 0,\ldots, M_L-1 \rbrace,~ \tau \geq k/M_L  )   $ and $ k'_{0} = \min (k' \in \lbrace 1,\ldots,M_L-k_{0} \rbrace,~ (k_{0} + k')/M_L \geq \tau + \ell )$ so that $ k_{0}/M_L \leq \tau< \tau+\ell \leq (k_{0} + k'_{0})/M_L$.
 
  Notice that the condition (\ref{UBalt10_u2_eq1}) entails
\begin{equation} \label{UBalt10_u2_eq7}
1-\ell > \frac{3}{M_L}\enspace.
\end{equation}

Since by definition of $k_0$ and $k_0'$, $\tau-k_0/M_L<1/M_L$ and $(k_0+k_0')/M_L-(\tau+\ell)<1/M_L$, then, with \eqref{UBalt10_u2_eq7},
 \begin{equation} \label{UBalt10_u2_eq8}
 \frac{k_0'}{M_L}=\frac{k_{0}+k'_{0}}{M_L}   - (\tau +\ell) +\ell + \tau-\frac{k_0}{M_L} <1/M_L+(1-3/M_L)+1/M_L<1\enspace.
 \end{equation}
In the same way,  we can also get with \eqref{UBalt10_u2_eq7}
 \begin{equation} \label{UBalt10_u2_eq9}
 1- \frac{k'_{0} }{M_L} = 1- \ell - \left(  \left( \tau - \frac{k_{0}}{M_L}    \right) + \left( \frac{k_{0}+k'_{0}}{M_L}   - (\tau +\ell) \right)  \right) > 1- \ell - \frac{2}{M_L}> \frac{1-\ell}{3} \enspace.
  \end{equation}
Furthermore, by definition of $k_{0}$ and $k'_{0}$ again, one has $k'_{0}/M_L \geq \ell$, so
\begin{equation}  \label{UBalt10_u2_eq10}
\frac{1-k'_{0}/M_L}{k'_{0}/M_L} \leq \frac{1-\ell}{\ell} \enspace.
\end{equation}  

From Lemma \ref{MomentsT'}, we deduce the expectation
$$E_\lambda \left[  T'_{k_0/M_L, \pa{k_0+k'_0}/M_L}(N) \right] = \delta^2 \ell^2 \pa{1-k'_{0}/M_L}/\pa{k'_{0}/M_L} \enspace,$$
 which satisfies with \eqref{UBalt10_u2_eq9} and $\ell>1/2$
$$E_\lambda \left[  T'_{k_0/M_L, \pa{k_0+k'_0}/M_L}(N) \right] > \frac{\delta^2}{6 }\ell(1-\ell)\enspace.$$
Lemma \ref{MomentsT'} also gives the variance
 \begin{align*}
  \mathrm{Var}_\lambda \cro{ T'_{\frac{k_0}{M_L}, \frac{k_0+k'_0}{M_L}}(N)  } &= \frac{2}{L^2} \pa{\lambda_0 + \delta \ell \frac{1-k'_{0}/M_L}{k'_{0}/M_L}}^2 + \frac{4}{L}\delta^2 \ell^2 \frac{1-k'_{0}/M_L}{k'_{0}/M_L} \pa{\lambda_0 + \delta \ell \frac{1-k'_{0}/M_L}{k'_{0}/M_L}} \enspace,
  \end{align*}
 which satisfies with \eqref{UBalt10_u2_eq10} 
  \begin{multline*}
 \mathrm{Var}_\lambda \cro{  T'_{\frac{k_0}{M_L}, \frac{k_0+k'_0}{M_L}}(N) } \leq
  \pa{\frac{2}{L^2} (\lambda_0 +  \delta )^2 + \frac{4}{L} \delta^2 \ell(1-\ell)(\lambda_0 +  \delta ) } \mathds{1}_{\delta>0}\\
   + \pa{
  \frac{2}{L^2} \lambda_0^2 + \frac{4}{L} \lambda_0 \delta^2 \ell(1-\ell)} \mathds{1}_{\delta<0} \enspace,
  \end{multline*}
  thereby
  $$  \mathrm{Var}_\lambda \cro{  T'_{\frac{k_0}{M_L}, \frac{k_0+k'_0}{M_L}}(N) }
   \leq
  \frac{2R^2}{L^2} + \frac{4R}{L} \delta^2 \ell(1-\ell) \enspace.
 $$
Finally, since $k'_{0}/M_L \geq \ell >1/2$, \eqref{UBalt10_u2_eq8} leads to
 $$ \max \pa{\frac{k'_{0}/M_L}{1-k'_{0}/M_L},\frac{1-k'_{0}/M_L}{k'_{0}/M_L}} = \frac{k'_{0}/M_L}{1-k'_{0}/M_L} = \frac{k'_{0}}{M_L - k'_{0}} \leq M_L -1 \enspace.$$

As in the above case where $\ell \leq 1/2$, we can therefore prove that the assumption \eqref{UBalt10_u2_cond1}  leads to \eqref{UBalt10_u2_cond5}, that is
 \begin{equation*}
E_\lambda \left[ T'_{\frac{k_0}{M_L}, \frac{k_0+k'_0}{M_L}}(N) \right] \geq Q(\alpha, \beta,L,k'_{0})+ \sqrt{2 \mathrm{Var}_\lambda \cro{ T'_{\frac{k_0}{M_L}, \frac{k_0+k'_0}{M_L}}(N)   } /\beta} \enspace,
\end{equation*}
with $Q(\alpha, \beta,L,k'_{0})$ is defined by \eqref{UBalt10_u2_eq6}.
We then use the same final arguments as in this case, to conclude that $$P_\lambda \left( T'_{\frac{k_0}{M_L}, \frac{k_0+k'_0}{M_L}}(N)  \leq t'_{N_1,\frac{k_0}{M_L}, \frac{k_0+k'_0}{M_L}} \left( 1- u_{\alpha}^{(2)} \right)  \right) \leq \beta \enspace.$$
Coming back to the assumption \eqref{UBalt10_u2_cond1}, we can thus affirm that
\begin{multline}
\SRb\Big(\phi_{9/10,\alpha}^{u(2)},\calS^u_{\bbul,\bbul\bbul,\bbul\bbul\bbul}[R]\Big) \leq  \max \Bigg( \sqrt{\frac{60CR\log \left(   2.77/u_\alpha^{(2)} \right)}{L}}\\
 +   \pa{\frac{2R}{\beta}}^{1/4} \frac{ \sqrt{60C \log \left(   2.77/u_\alpha^{(2)} \right)}}{L^{3/4}}+ 6\sqrt{ C}\frac{\log \left(   2.77/u_\alpha^{(2)} \right)}{\sqrt{ L \log L}} + 2\sqrt{\frac{6R}{L\sqrt{\beta}}}~,\\~24 \sqrt{\frac{2R}{ \beta L}}~,~R \sqrt{\frac{3\log L}{L}}  \Bigg) \enspace ,
\end{multline}
 which, since $M_L=\lceil L/\log L \rceil$ and $u_\alpha^{(2)}=2\alpha/( M_L (  M_L+1)-2)$, leads
 to the result stated in Proposition  \ref{UBalt10_u} for $\phi_{9/10,\alpha}^{u(2)}$ and the set $\calS^u_{\bbul,\bbul\bbul,\bbul\bbul\bbul}[R]$ of the alternative $\bold{[Alt^{u}.10]}$.

Since $\calS^u_{\delta^*,\bbul\bbul,\bbul\bbul\bbul}[R]\subset \calS^u_{\bbul,\bbul\bbul,\bbul\bbul\bbul}[R]$ for any $\delta^*$ in $(-R,R) \setminus \lbrace 0 \rbrace$,  the same result holds for $\phi_{9/10,\alpha}^{u(2)}$ and the set $\calS^u_{\delta^*,\bbul\bbul,\bbul\bbul\bbul}[R]$ of the alternative $\bold{[Alt^{u}.9]}$. Again, notice that in this case, the constant $C(\alpha,\beta,R)$ could be refined thanks to the knowledge of $\delta^*$, which justifies the formulation with $C(\alpha,\beta,R,\delta^*)$ instead of $C(\alpha,\beta,R)$ in Proposition~\ref{UBalt10_u}.

\section{Key arguments for minimax separation rates lower bounds and technical results for minimax separation rates upper bounds}\label{Sec:FTresults}

\subsection{Key arguments for nonasymptotic minimax separation rates lower bounds}\label{Keys}

The arguments given below are repeated from \cite{Baraud2002}, who adapted the asymptotic Bayesian approach of Ingster \cite{Ingster1993} to obtain lower bounds for minimax rates of testing in a nonasymptotic perspective, and \cite{FLRB2011} who derived these lower bounds in the Poisson framework. We only recall these arguments without the proofs for the sake of clarity and completeness.

Let us recall the notation of the introduction, where $N=(N_t)_{t\in[0,1]}$ is a Poisson process observed on $[0,1]$, with intensity $\lambda$ w.r.t.  some measure $\Lambda$ on $[0,1]$, and whose distribution is denoted by $P_\lambda$, and where  $\calS_0$ is either the set  $\calS_0[\lambda_0]=\{\lambda_0\}$ of a single known constant intensity $\lambda_0$ on $[0,1]$, or the set $\calS_0^u[R]$ of all constant intensities on $[0,1]$ bounded by $R$.

We consider the problem of testing $\hzero$ $"\lambda \in \calS_0"$ versus $\hone$ $"\lambda\in \calS"$, where  $\calS$ is a set of possible alternative intensities, from the nonasymptotic minimax point of view based on the definition of $\mSRab(\calS)$ given in \eqref{defmSR}.

\begin{lemma} \label{mSR} 
For $r>0$ and any subspace $\mathcal{S}$ of $\bbL_{2}([0,1])$, set $\mathcal{S}_{r}= \lbrace \lambda \in \mathcal{S},~d_2(\lambda, \calS_{0}) \geq r   \rbrace$. For  $\alpha$ in $(0,1)$, we define
\[ \rho_\alpha(\calS_r) = \inf_{\phi_{\alpha}, \sup_{\lambda\in \calS_0} P_{\lambda}(\phi_{\alpha}(N)=1) \leq \alpha} \quad \sup_{\lambda \in \mathcal{S}_r} P_{\lambda}(\phi_{\alpha}(N)=0)\enspace.\]

$(i)$ If $\rho_\alpha( \mathcal{S}_{r}) \geq \beta$ then $\mSRab(\calS) \geq r$.

$(ii)$ For all  subsets $\mathcal{S}'$ and $\mathcal{S}$ of $\bbL_{2}([0,1])$ such that $\mathcal{S}' \subset \mathcal{S}$, then $\rho_\alpha( \mathcal{S}'_{r})\leq \rho_\alpha( \mathcal{S}_{r})$ and $\mSRab(\mathcal{S}') \leq \mSRab(\mathcal{S})$.
\end{lemma}

\begin{lemma}[Minimax lower bounds and Bayesian approach] \label{lemmebayesien}

Let $\mu$ be a probability measure on $\mathcal{S}_{r}$ and define $P_{\mu}$ the mixture probability by
\[ P_{\mu}= \int P_{\lambda} d\mu(\lambda).\]
Then \begin{align*}
 \rho_\alpha(\mathcal{S}_{r})  \geq  1- \alpha- \frac{1}{2}\pa{\inf_{\lambda_0\in \calS_0} E_{\lambda_0} \left[\left( \frac{dP_{\mu}}{dP_{\lambda_{0}}} \right)^{2}(N)\right]-1  }^{1/2}\enspace.
 \end{align*}
  
 As a consequence, if $\alpha + \beta <1$ and $\inf_{\lambda_0\in \calS_0} E_{\lambda_0} \left[\left( dP_{\mu}/dP_{\lambda_{0}} \right)^{2}(N)\right] \leq 1+4(1- \alpha - \beta)^{2}$ then  \[\rho_\alpha(\mathcal{S}_{r}) \geq \beta \quad \textrm{and}\quad \mSRab(\mathcal{S}) \geq r\enspace.\]
\end{lemma}

\subsection{Technical results for minimax separation rates upper bounds }\label{TechnicalRes}

\begin{lemma}[Moments of a Poisson distribution] \label{momentPoisson}
Let $X$ be a Poisson distributed random variable with constant intensity $\xi>0$.
\begin{enumerate}
\item The Laplace transform for $X$ is given by
\[ \forall c\in \R,~~\mathbb{E} \left[ \exp(cX)  \right] = \exp [\xi(\exp(c)-1)].\]
\item The first moments of $X$ are given by the following formulas:
\begin{align*}
\mathbb{E}[X]= \xi,~
\mathbb{E}[X^2]= \xi^2 +\xi, ~
\mathbb{E}[X^3]= \xi^3 + 3 \xi^2 + \xi, ~
\mathbb{E}[X^4]= \xi^4 +6 \xi^3 +7 \xi^2 + \xi, 
\end{align*}
and its central moments are given by $
\mathbb{E}[(X-\mathbb{E}[X])^3]= \xi$ and $
\mathbb{E}[(X-\mathbb{E}[X])^4]= 3\xi^2+ \xi.$ 
\end{enumerate}
\end{lemma}

 \begin{lemma}[Quantile bounds for the Poisson distribution] \label{lemmequant1}
The $u$-quantile  $p_\xi(u)$ of the Poisson distribution with parameter $\xi$ satisfies
\begin{equation}\label{controlquantilePb1}
 -\sqrt{{\xi}/{u}} +\xi\leq p_\xi(u) \leq \sqrt{\xi/(1-u)} +\xi\enspace.
 \end{equation}
 \end{lemma}
\begin{proof}
Let $N$ be a Poisson random variable with parameter $\xi.$ Using the Bienayme-Chebyshev inequality, we obtain for all $\varepsilon>0$
\[ 
\P\left( N  \leq  -\sqrt{\frac{\xi}{u-\varepsilon}} +\xi \right) \leq u-\varepsilon<u \quad \textrm{and} \quad  \P\left( N >  \sqrt{\frac{\xi}{u}} +\xi \right) \leq u,\]
which leads to the expected result by letting $\varepsilon$ tend to zero.
 \end{proof}

\begin{lemma}[Expectation and variance of $T_{\tau_1,\tau_2}(N)$] \label{MomentsT}
Let $\tau_1,\tau_2$ be in $(0,1)$ such that $0 \leq \tau_1 < \tau_2 \leq 1,$  $\lambda_0>0$ and $T_{\tau_1, \tau_2}(N)$ be defined by \eqref{estimateurSB}. Assume that $N(\tau_1 ,\tau_2]$ follows a Poisson distribution with parameter $Lx$ with $x>0.$ Then
\begin{equation} \label{bpreliesperance}
\mathbb{E}[T_{\tau_1, \tau_2}(N)]= \pa{ \frac{x}{\sqrt{\tau_2 - \tau_1}} - \lambda_{0} \sqrt{\tau_2 - \tau_1}}^2\enspace,
\end{equation}
and
 \begin{equation} \label{bprelivariance}
\mathrm{Var}(T_{\tau_1, \tau_2}(N))= \frac{4x}{L}  \left(\frac{x}{\tau_2 - \tau_1}- \lambda_{0}\right)^2 + \frac{2}{L^2} \frac{x^2}{(\tau_2 - \tau_1)^2}\enspace.
\end{equation}
\end{lemma}

\begin{proof}
Set $m_{i}$ the moment of order $i$ of $N(\tau_1,\tau_2]$. Since $N(\tau_1,\tau_2]$ follows a Poisson distribution, $m_2=m_1+m_1^2$ and
\[ \mathbb{E}[T_{\tau_1, \tau_2}(N)]=\frac{1}{L^2(\tau_2-\tau_1)}m_{1}^2 - \frac{2 \lambda_{0}}{L} m_{1}    + \lambda_{0}^2(\tau_2-\tau_1),\]
 with $m_1=Lx$. This straightforwardly  leads  to the first statement of Lemma \ref{MomentsT}. 
Notice now that
\begin{align*}
T_{\tau_1, \tau_2}(N)=& \frac{1}{L^2(\tau_2-\tau_1)} \left( (N(\tau_1,\tau_2]-m_{1})^2 + m_{1}^2 +(2m_{1}-1) (N(\tau_1,\tau_2]-m_{1}) -m_{1} \right)\\
& - \frac{2 \lambda_{0}}{L} (N(\tau_1,\tau_2]-m_{1})- \frac{2 \lambda_{0}}{L} m_{1}    + \lambda_{0}^2(\tau_2-\tau_1)\enspace,
\end{align*}
which entails
\begin{align*}
 T_{\tau_1, \tau_2}(N) - \mathbb{E}[T_{\tau_1, \tau_2}(N)] &= \frac{1}{L^2(\tau_2-\tau_1)} \left( (N(\tau_1,\tau_2]-m_{1})^2  +(2m_{1}-1) (N(\tau_1,\tau_2] -m_{1}) -m_{1}   \right)\\
 & - \frac{2 \lambda_{0}}{L} (N(\tau_1,\tau_2]-m_{1}).
\end{align*}
Considering the moment  $c_{i}$ of order $i$ of the centered variable $N(\tau_1,\tau_2]-m_{1}$, one obtains
\begin{align*}
\mathrm{Var}(T_{\tau_1, \tau_2}(N))&= \frac{1}{L^4(\tau_2-\tau_1)^2} \left( c_{4} +(2m_1-1)^2 c_2+ m_1^2 + 2(2m_1-1)c_3 - 2m_1c_2\right) + \frac{4 \lambda_{0}^2}{L^2}c_{2}\\ &\quad - \frac{4 \lambda_{0}}{L^3(\tau_2-\tau_1)} \left( c_{3} +(2m_{1}-1)c_{2}   \right).
\end{align*}
Applying Lemma \ref{momentPoisson} finally leads to the second statement of Lemma \ref{MomentsT}.
\end{proof}

\begin{lemma}[Quantile bound for $T_{\tau_1,\tau_2}(N)$] \label{QuantilesT}
Let $\lambda_{0}>0$, $u$ in $(0,1)$ and assume that $N$ is a homogeneous Poisson process of intensity $\lambda_{0}$ with respect to the measure $\Lambda$. Let $t_{\lambda_0,\tau_1, \tau_2}(1-u)$ be the $(1-u)$-quantile of $T_{\tau_1,\tau_2}(N)$  defined by \eqref{estimateurSB}. Then for all $0 \leq \tau_1 < \tau_2 \leq 1$, 
\[ t_{\lambda_0,\tau_1, \tau_2}\left( 1- u\right) \leq 2\lambda_{0}^2 (\tau_2 - \tau_1) \pa{g^{-1} \pa{\frac{\log \pa{3/u}}{\lambda_{0}L (\tau_2 - \tau_1) }}}^2\enspace,   \]
where $g$ is defined by \eqref{defg}.
\end{lemma}

\begin{proof}
Notice first that under the assumption of Lemma \ref{QuantilesT}, $T_{\tau_1,\tau_2}(N)$ can be written as
\[ T_{\tau_1,\tau_2}(N)= \left( \int_{0}^{1} \frac{ \varphi_{(\tau_1,\tau_2]}(t)}{L} (dN_{t}-\lambda_{0}Ldt)  \right)^2 - \int_{0}^{1} \pa{\frac{ \varphi_{(\tau_1,\tau_2]}(t)}{L}}^{2} dN_{t}\]
with $ \varphi_{(\tau_1,\tau_2]}=  \mathds{1}_{(\tau_1, \tau_2 ]}/\sqrt{\tau_2 - \tau_1}.$ Applying the exponential inequality of Theorem 2 in \cite{LeGuevel2021}, we obtain for all $x>0$
\begin{equation*}
\P(T_{\tau_1,\tau_2}(N)> x) \leq 3 \exp \pa{-\frac{\|H_{\tau_1,\tau_2}\|_{2,L}^2}{\|H_{\tau_1,\tau_2}\|_{\infty}^2}g\pa{  \frac{\|H_{\tau_1,\tau_2}\|_{\infty}}{\|H_{\tau_1,\tau_2}\|_{2,L}^2} \sqrt{\frac{x}{2}} } } \enspace,
\end{equation*}
where $H_{\tau_1,\tau_2}(t) = \varphi_{(\tau_1,\tau_2]}(t)/L,$ $g$ is defined by \eqref{defg} and $\|H_{\tau_1,\tau_2}\|_{2,L}$ is the $\bbL_2$-norm of $H_{\tau_1,\tau_2}$ in $\bbL^2([0,1],\lambda_0 Ldt),$ that is $\|H_{\tau_1,\tau_2}\|_{2,L}^2 = \int_0^1 |H_{\tau_1,\tau_2}(t)|^2 \lambda_{0} L dt = \lambda_0/L.$
This leads to
\[ \P \pa{T_{\tau_1,\tau_2}(N)> x} \leq 3 \exp \pa{- \lambda_{0}L(\tau_2-\tau_1) g \pa{\frac{1}{\lambda_{0}\sqrt{\tau_2-\tau_1}} \sqrt{\frac{x}{2}}}}\enspace.\]
We therefore obtain $ \P \pa{T_{\tau_1,\tau_2}(N)> x} \leq u$ if $ x \geq 2 \lambda_{0}^2 (\tau_2-\tau_1) \pa{g^{-1} \pa{\log \pa{3/u}/(\lambda_{0}L(\tau_2-\tau_1)) }}^2$ and the result follows.
\end{proof}

\begin{proposition}[Quantile bound for $\max/\min_{\tau\in [0,1-\ell^*]}{N(\tau, \tau + \ell^*]}$] \label{bquantile_maxminNbis} $\ $\\
Let $L \geq 1$. The $(1-\alpha)$-quantile $p_{\lambda_0,\ell^*}^+(1-\alpha)$ of $\max_{\tau\in [0,1-\ell^*]}{N(\tau, \tau + \ell^*]}$ under $\hzero$ satisfies
$$ p_{\lambda_0,\ell^*}^+(1- \alpha) \leq  \lambda_{0} L\ell^*+ 2  \lambda_{0}L g^{-1} \left( \frac{\log (2/ \alpha)}{\lambda_{0} L}   \right)\enspace.$$
The $\alpha$-quantile $p_{\lambda_0,\ell^*}^-(\alpha)$ of $\min_{\tau\in [0,1-\ell^*]}{N(\tau, \tau+ \ell^*]}$ under $\hzero$ satisfies
$$ p_{\lambda_0,\ell^*}^-(\alpha) \geq  \lambda_{0} L\ell^* - 2  \lambda_{0} L g^{-1} \left( \frac{\log (2/ \alpha)}{\lambda_{0} L}   \right)\enspace,$$
where $g$ is defined by \eqref{defg}.
\end{proposition}

\begin{proof}
We define $M_s^t = \int_s^t (dN_u - \lambda_{0} L du)$ for all $s,t >0.$ Let $x>0$ be such that
\begin{equation} \label{bproofDeltaLength5N1}
x \geq \lambda_{0} L\ell^*+ 2  \lambda_{0}L g^{-1} \left( \frac{\log (2/ \alpha)}{\lambda_{0} L}   \right)\enspace.\end{equation}
Since
 \begin{align*}
 P_{\lambda_0}\left(\max_{\tau \in [0, 1-\ell^*]} N(\tau,\tau+\ell^*] >x\right) &= P_{\lambda_0} \left(  \max_{\tau \in [0, 1-\ell^*]} M_\tau^{\tau+\ell^*}  > x- \lambda_{0} L\ell^*  \right) \\
 &\leq P_{\lambda_0} \left( \max_{s,t \in [0,1]} \vert M_s^t \vert > x- \lambda_{0} L\ell^*    \right) \enspace,
 \end{align*}
Theorem 4 in \cite{LeGuevel2021} ensures that 
 \[P_{\lambda_0} \left( \max_{s,t \in [0,1]} \vert M_s^t \vert > x- \lambda_{0} L\ell^*    \right) 
 \leq 2 \exp \left(  -\lambda_{0} L g \left(  \frac{x - \lambda_{0} L\ell^*}{2 \lambda_{0} L}  \right)   \right)\enspace.\]
Then (\ref{bproofDeltaLength5N1}) entails \[2 \exp \left(  -\lambda_{0} L g \left(  \frac{x - \lambda_{0} L\ell^*}{2 \lambda_{0} L}  \right)   \right) \leq \alpha\enspace,
 \]
leading to $P_{\lambda_0}(\max_{\tau \in [0, 1-\ell^*]} N(\tau,\tau+\ell^*] >x) \leq \alpha.$ 
The $(1-\alpha)$-quantile $p_{\lambda_0,\ell^*}^+(\alpha)$ of $\max_{\tau\in [0,1-\ell^*]}{N(\tau, \tau+ \ell^*]}$ under $\hzero$ therefore satisfies $ p_{\lambda_0,\ell^*}^+(1- \alpha) \leq  x$ for every $x$ such that  \eqref{bproofDeltaLength5N1} holds. In particular, 
 $$ p_{\lambda_0,\ell^*}^+(1- \alpha) \leq  \lambda_{0} L\ell^*+ 2 \lambda_{0} L g^{-1} \left( \frac{\log (2/ \alpha)}{\lambda_{0} L}   \right)\enspace.$$
 Let us consider now $x$ in $\R$ and $\varepsilon$ in $(0,1)$ satisfying
\begin{equation} \label{bproofDeltaLength7N1}
x \leq  \lambda_{0} L\ell^*- 2 \lambda_{0} Lg^{-1} \left( \frac{\log (2/ (\alpha(1-\varepsilon)))}{\lambda_{0} L}   \right)\enspace  .
\end{equation}
 Using \eqref{bproofDeltaLength7N1} and Theorem 4  in \cite{LeGuevel2021} again, we obtain
 \begin{align*}
 P_{\lambda_0}\left(\min_{\tau \in [0, 1-\ell^*]} N(\tau,\tau+\ell^*] \leq x\right) &= P_{\lambda_0} \left(  \max_{\tau \in [0, 1-\ell^*]} -M_\tau^{\tau+\ell^*}  \geq  \lambda_{0} L\ell^* -x  \right) \\
 &\leq P_{\lambda_0} \left( \max_{s,t \in [0,1]} \vert M_s^t \vert \geq  \lambda_{0} L\ell^* -x    \right) \\
 &\leq 2 \exp \left(  -\lambda_{0} L g\left(  \frac{ \lambda_{0} L\ell^* -x}{2\lambda_{0} L}  \right)   \right) \\
 &\leq \alpha(1-\varepsilon)  \\
 &< \alpha \enspace.
 \end{align*}
 Thus the $\alpha$-quantile $p_{\lambda_0,\ell^*}^-(\alpha)$ of $\min_{\tau\in [0,1-\ell^*]}{N(\tau,\tau + \ell^*]}$ under $\hzero$ satisfies $p_{\lambda_0,\ell^*}^-(\alpha) >x$ for every $x$ such that
 \eqref{bproofDeltaLength7N1} holds. In particular
 \[ p_{\lambda_0,\ell^*}^-(\alpha) > \lambda_{0} L\ell^* - 2  \lambda_{0} L g^{-1} \left( \frac{\log (2/ (\alpha(1-\varepsilon)))}{\lambda_{0} L}   \right)\]
 for every $ \varepsilon$ in $(0,1).$ The result then follows by continuity of $g^{-1}.$
\end{proof}

\begin{lemma}[Quantile bound for $\sup_{\ell \in (0,1- \tau^{*})} S_{\delta^*,\tau^*,\tau^*+\ell}(N)$]\label{QuantilessupShifted}
Let $\gamma >0$ and $L \geq 1$. Let $(N_t^{\lambda_0})_{t \geq 0}$ be an homogeneous Poisson process with a known constant intensity $\lambda_0 L >0$ w.r.t. the Lebesgue measure. Then, the $u$-quantile of the supremum $\sup_{t \geq 0} (N_t ^{\lambda_0}- (\lambda_0 + \gamma)Lt)$ does not depend on $L$, and will therefore be denoted by $s_{\lambda_0,\gamma}^+(u)$. Now considering $s_{\lambda_0,\delta^*,\tau^*,L}^+(u)$,  the $u$-quantile under $(H_0)$
of the statistic $\sup_{\ell \in (0,1- \tau^{*})} S_{\delta^*,\tau^*,\tau^*+\ell}(N)$ with $S_{\delta^*,\tau^*,\tau^*+\ell}(N)$ defined by \eqref{stat_alt5}, we have for all $L\geq 1$
\begin{equation} \label{UBQuantilemaxShifted}
\left\{\begin{array}{ccl}
s_{\lambda_0,\delta^*,\tau^*,L}^+(u)&\leq& s_{\lambda_0,\delta^{*}/2}^+(u) \quad \textrm { when }\delta^*>0\enspace,\\
s_{\lambda_0,\delta^*,\tau^*,L}^+(u)&\leq& \frac{\log \pa{1/(1-u)}}{\log \pa{\lambda_{0}/\pa{\lambda_{0} - \vert \delta^{*} \vert/2}}} \quad \textrm { when }\delta^{*} \in (-\lambda_{0}^{*}, 0) \enspace.
\end{array}\right.
\end{equation} 
\end{lemma}

\begin{proof}
Equation (7) in \cite{pyke1959} directly enables to state the first part of the result.

Under $(H_0)$, since the processes $(N(\tau^{*}, \tau^{*} +\ell])_{\ell \in (0,1- \tau^{*}]}$ and $(N(0,\ell])_{\ell \in (0,1- \tau^{*}]}$ are left-continuous and have the same finite dimensional laws, we get
$$ \sup_{\ell \in (0,1- \tau^{*})} S_{\delta^*,\tau^*,\tau^*+\ell}(N) \overset{d}{=} \sup_{\ell \in (0,1- \tau^{*})} \pa{\mathrm{sgn}(\delta^{*})(N(0,\ell] - \lambda_{0} \ell L) - \frac{\vert \delta^{*}{} \vert}{2}\ell L}.$$
Assume first that $\delta^{*}{}>0$. We compute
\begin{align*}
P_{\lambda_0} \Big(\sup_{\ell \in (0,1- \tau^{*})} S_{\delta^*,\tau^*,\tau^*+\ell}(N) &> s_{\lambda_0,\delta^{*}/2}^+(u)\Big)\\
&\leq \mathbb{P} \pa{ \sup_{t \in [0, + \infty)} \pa{N^{\lambda_0}(0,t]  - \pa{\lambda_0 + \frac{\delta^{*}{}}{2}}Lt  }  > s_{\lambda_0,\delta^{*}/2}^+(u)}\\
& \leq 1-u\enspace,
\end{align*}
by definition of $s_{\lambda_0,\delta^{*}/2}^+(u)$. This allows to conclude that  the first part of \eqref{UBQuantilemaxShifted} holds. 

Assume then that $\delta^{*}$ belongs to $(-\lambda_{0},0)$. For all $x>0$,
\[P_{\lambda_0} \pa{ \sup_{\ell \in (0,1- \tau^{*})} S_{\delta^*,\tau^*,\tau^*+\ell}(N)  >x} 
 = P_{\lambda_0} \pa{\sup_{\ell \in (0,1- \tau^{*})} \pa{\pa{\lambda_{0}- \frac{\vert \delta^{*}\vert}{2}}\ell L - N(0,\ell] }  > x}\enspace.\]
Theorem  3 in \cite{pyke1959} then entails 
$$P_{\lambda_0} \pa{\sup_{\ell \in (0,1- \tau^{*})} S_{\delta^*,\tau^*,\tau^*+\ell}(N) >x} \leq \exp \pa{-\omega x }\enspace,$$ where $\omega$ is the largest real root of the equation $\lambda_{0}(1- e^{- \omega})= \omega \pa{\lambda_{0} - \vert \delta^{*}{} \vert /2}$. Notice that $\omega  > \log\pa{\lambda_{0}/(\lambda_{0} - \vert \delta^{*}{} \vert/2)}$. Then, correctly choosing $x$ in the above exponential inequality leads to
$$P_{\lambda_0} \pa{\sup_{\ell \in (0,1- \tau^{*})} S_{\delta^*,\tau^*,\tau^*+\ell}(N)  >\frac{\log \pa{1/(1-u)}}{\log \pa{\lambda_{0}/\pa{\lambda_{0} -{\vert \delta^{*} \vert}/{2}}}}} \leq 1-u\enspace,$$ which implies the second part of \eqref{UBQuantilemaxShifted}.
\end{proof}

\begin{lemma}[Quantile bound for $\sup_{\tau \in (0,1)} S_{\delta^*,\tau,1}(N)$]\label{QuantilessupShifted2}
Let $L\geq 1$. With the same notation as in  Lemma \ref{QuantilessupShifted} and $S_{\delta^*,\tau,1}(N)$ defined by \eqref{stat_alt5},  the $u$-quantile $s_{\lambda_0,\delta^*,L}^+(u)$
of the statistic $\sup_{\tau \in (0,1)} S_{\delta^*,\tau,1}(N)$ under $(H_0)$ satisfies
\begin{equation} \label{UBQuantilesupShifted2}
\left\{\begin{array}{ccl}
s_{\lambda_0,\delta^*,L}^+(u)&\leq& s_{\lambda_0,\delta^{*}/2}^+(u) \quad \textrm { when }\delta^*>0\enspace,\\
s_{\lambda_0,\delta^*,L}^+(u)&\leq& \frac{\log \pa{1/(1-u)}}{\log \pa{\lambda_{0}/\pa{\lambda_{0} - \vert \delta^{*} \vert /2}}} \quad \textrm { when }\delta^{*} \in (-\lambda_{0}^{*}, 0) \enspace.
\end{array}\right.
\end{equation} 
\end{lemma}

\begin{proof}
Under $(H_0)$, since the processes $(N(\tau,1])_{\tau \in (0,1)}$ and $(N(0,1-\tau])_{\tau \in (0,1)}$ are left-continuous and have the same finite dimensional laws, we get
$$ \sup_{\tau\in (0,1)} S_{\delta^*,\tau,1}(N) \overset{d}{=} \sup_{\tau \in (0,1)} \pa{\mathrm{sgn}(\delta^{*})(N(0,\tau] - \lambda_{0} \tau L) - \frac{\vert \delta^{*}{} \vert}{2}\tau L}.$$
The result then follows from the same arguments as in the proof of Lemma \ref{QuantilessupShifted}, by noticing that when $\delta^{*}{}>0$,
\begin{align*}
P_{\lambda_0} \Big( \sup_{\tau\in (0,1)} S_{\delta^*,\tau,1}(N)  &> s_{\lambda_0,\delta^{*}/2}^+(u)\Big)\\
&\leq \mathbb{P} \pa{ \sup_{t \in [0, + \infty)} \pa{N^{\lambda_0}(0,t]  - \pa{\lambda_0 + \frac{\delta^{*}{}}{2}}Lt  }  > s_{\lambda_0,\delta^{*}/2}^+(u)}\enspace,
\end{align*}
and when $\delta^{*}$ belongs to $(-\lambda_{0},0)$, 
\[P_{\lambda_0} \pa{ \sup_{\tau\in (0,1)} S_{\delta^*,\tau,1}(N)   >x} 
 = P_{\lambda_0} \pa{\sup_{t \in (0,1)} \pa{\pa{\lambda_{0}- \frac{\vert \delta^{*}\vert}{2}}t L - N(0,t] }  > x}\]
 for all $x>0.$
\end{proof}

\begin{lemma}[Quantile bound for $|S_{\tau_1, \tau_2}(N)|$] \label{QuantilesAbsShifted}
Let $\lambda_{0}>0$ and assume that $N$ is a homogeneous Poisson process of intensity $\lambda_{0}$ with respect to the measure $\Lambda$. For $0 \leq \tau_1 < \tau_2 \leq 1$ and  $u$ in $(0,1)$,  let $s_{\lambda_0,\tau_1, \tau_2}(1-u)$ be the $(1-u)$-quantile of $|S_{\tau_1,\tau_2}(N)|$, with $S_{\tau_1,\tau_2}(N)= N(\tau_1, \tau_2] - \lambda_{0} (\tau_2-\tau_1) L $ as defined by \eqref{stat_alt6_N1}. 
Then 
\[ s_{\lambda_0,\tau_1, \tau_2}\left( 1- u\right) \leq \lambda_{0} L (\tau_2 - \tau_1) g^{-1} \pa{\frac{\log \pa{{2}/{u}}}{\lambda_{0} (\tau_2 - \tau_1) L}}  \enspace, \]
where $g$ is defined by \eqref{defg}.
 \end{lemma}
 
 \begin{proof}
Since $S_{\tau_1,\tau_2}(N)$ can be written as
\[S_{\tau_1,\tau_2}(N)= \int_0^{1} \mathds{1}_{(\tau_1, \tau_2]}(t) (dN_t - \lambda_{0} L dt)\enspace,\]
we can apply  the exponential inequality of Theorem 1 in \cite{LeGuevel2021} and obtain for all $x>0$
\begin{equation} \label{lambda0quantileN1proof}
P_{\lambda_0} (\vert S_N(\tau_1, \tau_2)\vert >x) \leq 2 \exp \left(-\lambda_{0}L (\tau_2 - \tau_1) g \left(  \frac{x}{\lambda_{0} L (\tau_2 - \tau_1)  } \right)    \right)\enspace.
\end{equation}
We get then $P_{\lambda_0} (\vert S_N(\tau_1, \tau_2) \vert >x) \leq u $ if
\[ x \geq \lambda_{0} L (\tau_2 - \tau_1) g^{-1} \pa{\frac{\log \pa{{2}/{u}}}{\lambda_{0} L (\tau_2 - \tau_1) }} \enspace,\]
which leads to the expected result.

\end{proof}

 \begin{lemma}[Expectation and variance of $T'_{\tau_1, \tau_2}(N)$] \label{MomentsT'}
 
Let $\tau_1$ and $\tau_2$ in $(0,1]$ such that $0 < \tau_1 < \tau_2 \leq 1$ and let $N$ be a Poisson process such that $N( 0,\tau_1]$, $N( \tau_1,\tau_2]$, and $N( \tau_2,1]$ follow a Poisson distribution with respective parameters $Lx>0$, $Ly>0$ and $Lz>0.$ Considering the statistic $T'_{\tau_1, \tau_2}$ defined by \eqref{def_T'}, one has

\begin{equation} \label{MomentsT'_E}
 \mathbb{E}[T'_{\tau_1, \tau_2}(N)] = \left(\sqrt{\frac{\tau_2 - \tau_1}{1- \tau_2 + \tau_1}}(x+z) - \sqrt{\frac{1-\tau_2 + \tau_1}{\tau_2 - \tau_1}}y   \right)^2 \enspace,
 \end{equation}
 and
\begin{multline} \label{MomentsT'_V}
\mathrm{Var}(T'_{\tau_1, \tau_2}(N))= \frac{2}{L^2} \left( \frac{ \tau_2- \tau_1}{1- \tau_2 + \tau_1}( x+z) + \frac{1- \tau_2 + \tau_1}{\tau_2 - \tau_1}y   \right)^2\\+ \frac{4}{L} \left(  \frac{1- \tau_2 + \tau_1}{\tau_2 - \tau_1}  \right)^2
 \left( y - \frac{\tau_2 - \tau_1}{1- \tau_2 + \tau_1} (x+z) \right)^2 \left( \left( \frac{ \tau_2 - \tau_1}{1- \tau_2 + \tau_1} \right)^2 (x+z) +y       \right) \enspace.
\end{multline}
 \end{lemma}

 \begin{proof}
Recall that
\begin{multline*}
T'_{\tau_1, \tau_2}(N)= \frac{1}{L^2} \Bigg[ \frac{\tau_2 - \tau_1}{1- \tau_2 + \tau_1}  \left( \left( N(0,\tau_1] + N(\tau_2,1] \right)^2 - \left( N(0,\tau_1]  +N(\tau_2,1] \right)  \right) \\
+ \frac{1- \tau_2 + \tau_1}{\tau_2 - \tau_1}  \left( N(\tau_1,\tau_2]^2  - N(\tau_1,\tau_2]  \right) -2 N(\tau_1,\tau_2] \left( N(0, \tau_1]+ N(\tau_2,1] \right) \Bigg]\enspace.
\end{multline*}
Set, as in the proof of Lemma \ref{MomentsT}, $m_{i}$ and $\bar{m}_i$  the moments of order $i$ of $N( \tau_1, \tau_2]$ and $(N(0, \tau_1]+N(\tau_2, 1])$ respectively, and $c_i$ and $\bar{c}_i$ the corresponding centered moments of order $i$. Then, by independence of  $N( 0,\tau_1]$, $N( \tau_1,\tau_2]$ and $N( \tau_2,1]$, and since $m_2=m_1^2+m_1$, $\bar{m}_2=\bar{m}_1^2+\bar{m}_1$ with $m_1=Ly$ and $\bar{m}_1=L(x+z)$,
\begin{align*}
 \mathbb{E}[T'_{\tau_1, \tau_2}(N)]&= \frac{1}{L^2} \left[ \frac{\tau_2 - \tau_1}{1- \tau_2 + \tau_1} (\bar{m}_2-\bar{m}_1)  + \frac{1-\tau_2 + \tau_1}{\tau_2 - \tau_1} (m_{2} -m_1)- 2m_{1}\bar{m}_{1} \right] \\
 &= \frac{\tau_2 - \tau_1}{1- \tau_2 + \tau_1} (x+z)^2  + \frac{1-\tau_2 + \tau_1}{\tau_2 - \tau_1} y^2- 2y(x+z)\enspace,
 \end{align*}
 which gives \eqref{MomentsT'_E}.
Moreover,
\begin{multline*}
T'_{\tau_1, \tau_2}(N)= \frac{1}{L^2} \Bigg[ \frac{\tau_2 - \tau_1}{1- \tau_2 + \tau_1}  \left( \left( N(0,\tau_1] + N(\tau_2,1] \right)^2 - \left( N(0,\tau_1]  +N(\tau_2,1] \right)  \right) \\
+ \frac{1- \tau_2 + \tau_1}{\tau_2 - \tau_1}  \left( N(\tau_1,\tau_2]^2  - N(\tau_1,\tau_2]  \right) -2 N(\tau_1,\tau_2] \left( N(0, \tau_1]+ N(\tau_2,1] \right) \Bigg]\enspace,
\end{multline*}
and then
\begin{align*}
T'_{\tau_1, \tau_2}(N)-  \mathbb{E}[T'_{\tau_1, \tau_2}(N)]=& \frac{1}{L^2} \Bigg[ \frac{\tau_2 - \tau_1}{1- \tau_2 + \tau_1}  \Big( (N(0,\tau_1]+N(\tau_2, 1]-\bar{m}_{1})^2 \\
 &+(2 \bar{m}_{1} -1) \left(    N(0,\tau_1]+N(\tau_2,1] -\bar{m}_{1} \right)  -\bar{m}_{1}\Big)\\ 
 &+ \frac{1- \tau_2 + \tau_1}{\tau_2 - \tau_1}  \Big( (N(\tau_1,\tau_2]-m_{1})^2  +(2m_{1}-1) (N(\tau_1,\tau_2]-m_{1})  -m_{1}\Big)\\
  &-2 (N(\tau_1, \tau_2]- m_{1}) \Big( (N(0,\tau_1]+N(\tau_2,1]-\bar{m}_{1}) + \bar{m}_{1}\Big)\\
  & -2m_{1} \left( N(0, \tau_1] +N(\tau_2,1]  - \bar{m}_{1} \right) \Bigg] \enspace.
\end{align*}
This entails
\begin{align*}
\mathrm{Var}(T'_{\tau_1, \tau_2}(N))= \frac{1}{L^4} \Bigg[ &\left( \frac{\tau_2 - \tau_1}{1- \tau_2 + \tau_1} \right)^2  \left(\bar{c}_4+ (2\bar{m}_{1}-1)^2 \bar{c}_2+ \bar{m}_{1}^2 + 2(2\bar{m}_1 -1)\bar{c}_3 -2\bar{m}_{1} \bar{c}_{2}\right) \\
 & +\left(  \frac{1- \tau_2 + \tau_1}{\tau_2 - \tau_1} \right)^2   \left( c_{4}+ (2m_{1}-1)^{2}c_{2}  + m_{1}^2 + 2(2m_{1}-1)c_{3} -2m_{1} c_{2} \right)\\
 &+ 4 \left( c_{2}\bar{m}_2  + m_{1}^2\bar{c}_2\right)+ 2 ( c_{2}- m_{1}) (\bar{c}_2-\bar{m}_1)\\
 & -4m_1 \frac{ \tau_2 - \tau_1}{1- \tau_2 + \tau_1}  \left(\bar{c}_3+(2\bar{m}_1-1)\bar{c}_2 \right)\\
 &- 4\bar{m}_1 \frac{1- \tau_2 + \tau_1}{\tau_2 - \tau_1} \left(c_{3}+(2m_{1}-1)c_{2}  \right) \Bigg] \enspace.
\end{align*}
This leads using  Lemma \ref{momentPoisson} to
\begin{align*}
\mathrm{Var}(T'_{\tau_1, \tau_2}(N))= \frac{1}{L^4} \Bigg[ &\left( \frac{\tau_2 - \tau_1}{1- \tau_2 + \tau_1} \right)^2 \left( 4L^3(x+z)^3+2L^2(x+z)^2\right)\\
& +\left(  \frac{1- \tau_2 + \tau_1}{\tau_2 - \tau_1} \right)^2   \left( 4L^3y^3+2L^2y^2\right)\\
 &+ 4 \left(Ly (L^2(x+z)^2 +L(x+z))  + L^2y^2L(x+z)\right)\\
 & -4Ly  \frac{ \tau_2 - \tau_1}{1- \tau_2 + \tau_1}  \left(L(x+z)+(2L(x+z)-1)L(x+z) \right)\\
 & - 4 L(x+z) \frac{1- \tau_2 + \tau_1}{\tau_2 - \tau_1} \left(Ly+(2Ly-1)Ly  \right) \Bigg] \enspace.
\end{align*}
The second statement of Lemma \ref{MomentsT'} given in \eqref{MomentsT'_V} finally follows from direct computations.
\end{proof}

 \begin{lemma}[Conditional quantile bound for $T'_{\tau_1, \tau_2}(N)$] \label{quantile_T'}

Assume that $N$ is a homogeneous Poisson process with a constant intensity $\lambda_0$ with respect to the measure $\Lambda=Ldt$. For $\tau_1$ and $\tau_2$ in $(0,1)$ such that $0 < \tau_1 < \tau_2 \leq 1$, $u$ in $(0,1)$ and $n$ in $\mathbb{N}$, let $t'_{n,\tau_1, \tau_2}(1-u)$ the $(1-u)$-quantile of the conditional distribution of $T'_{\tau_1, \tau_2}(N)$ defined by \eqref{def_T'} given the event $\{N_1=n\}$. Then
$$t'_{n,\tau_1, \tau_2}(1-u) \leq \frac{C}{L^2} \!\! \left( 5n  \log \! \left(\!  \frac{2.77 }{u} \! \right)\!\! + 3\max \left( \frac{1- \tau_2 + \tau_1}{\tau_2 - \tau_1},\frac{\tau_2 - \tau_1}{1- \tau_2 + \tau_1}   \right)    \log^2\!\! \left(\!\frac{2.77 }{u} \! \right)\!\!  \right)     \!\!   \enspace,$$
with $C= \min_{\varepsilon>0} \pa{e(1+\varepsilon^{-1})^2 (2.5+32 \varepsilon^{-1}} + \pa{(2\sqrt{2}(2+\varepsilon+\varepsilon^{-1})) \vee ((1+\varepsilon)^2)/\sqrt{2}}\enspace.$\end{lemma}

\begin{proof}
For $n\geq 2$ and conditionally on the event $\{N_1 =n\},$ the points of the process $N$ obey the same law as a $n$-sample $(U_1,\ldots,U_n)$ of i.i.d. random variables uniformly distributed on $(0,1).$ $t'_{n,\tau_1, \tau_2}(1-u)$ is thus equal to the $(1-u)$-quantile of the following $U$-statistic of order $2$
\[T'_{n,L,\tau^1,\tau^2}=\frac{1}{L^2} \sum_{i \neq j=1}^{n} \psi_{\tau_1, \tau_2}(U_i) \psi_{\tau_1, \tau_2}(U_j)= \sum_{i=2}^{n} \sum_{j=1}^{i-1} H_{L,\tau_1, \tau_2}(U_i,U_j) \enspace,\]
where $H_{L,\tau_1, \tau_2}(x,y)=  2 \psi_{\tau_1, \tau_2}(x) \psi_{\tau_1, \tau_2}(y)/L^2$ for any $x$ and $y$ in $[0,1]$.

Since for all $0 \leq \tau_1 < \tau_2 \leq 1$, $\psi_{\tau_1, \tau_2}$ is orthogonal to $\psi_{0}$ (in $\bbL^2([0,1])$), the variables $\psi_{\tau_1, \tau_2}(U_i)$ are centered and we can apply Theorem 3.4 in \cite{HoudreRB}. We obtain that there exists some absolute constant $C>0$ such that for all $x>0$ and for all $n \geq 2$
$$ \mathbb{P} \left( T'_{n,L,\tau^1,\tau^2} \geq C( A_{1} \sqrt{x} + A_{2} x +A_{3} x^{3/2} + A_{4} x^{2}  )    \right) \leq 2.77 e^{-x} \enspace,$$
where
\begin{align*}
A_{1}^{2}&= n^2 \mathbb{E} [H_{\tau_1, \tau_2}(U_1,U_2)^2] \enspace, \\
A_{2} &= \sup\Bigg( \Bigg|    \mathbb{E} \Bigg[  \sum_{i=2}^{n} \sum_{j=1}^{i-1} H_{L,\tau_1, \tau_2}(U_i, U_j) f_{i}(U_i) g_{j}(U_j)   \Bigg] \Bigg|~,\\
&\quad\mathbb{E} \Bigg[    \sum_{i=2}^{n} f_{i}^{2}(U_i) \Bigg] \leq 1,~  \mathbb{E} \Bigg[    \sum_{j=1}^{n-1} g_{j}^{2}(U_j) \Bigg] \leq 1,~\textrm{$f_{i}  $ and $g_{i}$ Borel measurable functions}    \Bigg)\enspace, \\
A_{3}^{2}&=n \sup_{y \in [0,1]} \int_{0}^{1} H_{L,\tau_1, \tau_2}^2(x,y)dx \enspace, \\
A_{4}&= \sup_{x,y \in [0,1]} \vert H_{L,\tau_1, \tau_2}(x,y) \vert \enspace.
\end{align*}
From Theorem 3.4 in \cite{HoudreRB}, notice that the constant $C$ can be taken equal to $C= \min_{\varepsilon>0} \pa{e(1+\varepsilon^{-1})^2 (2.5+32\varepsilon^{-1}} + \pa{(2\sqrt{2}(2+\varepsilon+\varepsilon^{-1})) \vee ((1+\varepsilon)^2 /\sqrt{2})}.$ 

Let us now evaluate $A_{1}, A_{2}, A_{3}$ and $A_{4}$. ~\\
Since the function $ \psi_{\tau_1, \tau_2}$ has a $\bbL_2$-norm equal to $1$ and the $U_i$'s are independent, we get
\[A_{1}^{2} = \frac{4n^2}{L^{4}} \mathbb{E}\cro{ \psi_{\tau_1, \tau_2}^2(U_1)}^2= \frac{4n^{2}}{L^{4}}\enspace.\]
The independence of the $U_i$'s and the Cauchy-Schwarz inequality applied twice also yields
\begin{align*}
\Bigg|    \mathbb{E} \Bigg[ & \sum_{i=2}^{n} \sum_{j=1}^{i-1} H_{L,\tau_1, \tau_2}(U_i, U_j) f_{i}(U_i) g_{j}(U_j)   \Bigg] \Bigg|\\
&\leq \frac{2}{L^2} \sum_{i=2}^{n} \sum_{j=1}^{i-1} \left| \int_{0}^{1} \psi_{\tau_1, \tau_2}(x) f_{i}(x) dx    \right|     \left|  \int_{0}^{1} \psi_{\tau_1, \tau_2}(y)  g_{j}(y) dy    \right| \\
&\leq \frac{2}{L^2} \sum_{i=2}^{n} \sqrt{\int_{0}^{1} f_{i}^2(x) dx}  \sum_{j=1}^{n-1} \sqrt{\int_{0}^{1}   g_{j}^2(y) dy} \\
&\leq \frac{2(n-1)}{L^2} \sqrt{\sum_{i=2}^{n} \underbrace{\int_{0}^{1} f_{i}^2(x) dx}_{= \mathbb{E}[f_{i}^2(U_i)]}}  \sqrt{\sum_{j=1}^{n-1}\underbrace{ \int_{0}^{1} g_{j}(y)^2 dy}_{= \mathbb{E}[g_{j}^2(U_j)]}} \enspace,
\end{align*} 
hence
$ A_{2} \leq 2n/L^2.$
Moreover,
\begin{align*}
A_{3}^{2} &= \frac{4n}{L^4} \sup_{y \in [0,1]} \int_{0}^{1}  \psi_{\tau_1, \tau_2}^2 (x) \psi_{\tau_1, \tau_2}^2 (y) dx\\
&=  \frac{4n}{L^4} \sup_{y \in [0,1]} \psi_{\tau_1, \tau_2}^2 (y)\\
&= \frac{4n}{L^{4}} \max \left( \frac{1- \tau_2 + \tau_1}{\tau_2 - \tau_1}~,~ \frac{\tau_2 - \tau_1}{1- \tau_2 + \tau_1}   \right)\enspace.
\end{align*}
To conclude,
$$A_{4} = \frac{2}{L^2} \left( \sup_{x \in [0,1]} \vert \psi_{\tau_1, \tau_2}(x) \vert \right)^2 = \frac{2}{L^2} \max \left( \frac{1- \tau_2 + \tau_1}{\tau_2 - \tau_1}~,~ \frac{\tau_2 - \tau_1}{1- \tau_2 + \tau_1}   \right)\enspace.$$
We finally obtain for all $x>0$ and for all $n \geq 2$
\begin{multline*}
 \mathbb{P} \Bigg(T'_{n,L,\tau^1,\tau^2}\geq \frac{C}{L^2} \Bigg( 2n \sqrt{x} + 2n x +2\sqrt{n\max \left( \frac{1- \tau_2 + \tau_1}{\tau_2 - \tau_1}~,~ \frac{\tau_2 - \tau_1}{1- \tau_2 + \tau_1}   \right)} x^{3/2}\\
  + 2 \max \left( \frac{1- \tau_2 + \tau_1}{\tau_2 - \tau_1}~,~ \frac{\tau_2 - \tau_1}{1- \tau_2 + \tau_1}   \right) x^{2}  \Bigg)   \Bigg)\leq 2.77 e^{-x}\enspace.
\end{multline*}
Then notice that for all $x >0$
$$ 2\sqrt{n\max \left( \frac{1- \tau_2 + \tau_1}{\tau_2 - \tau_1}~,~ \frac{\tau_2 - \tau_1}{1- \tau_2 + \tau_1}   \right)}x^{3/2} \leq nx + \max \left( \frac{1- \tau_2 + \tau_1}{\tau_2 - \tau_1}~,~ \frac{\tau_2 - \tau_1}{1- \tau_2 + \tau_1}   \right)x^2,$$
whereby
\begin{multline} \label{quantileT'_eq1}
 \mathbb{P} \left(T'_{n,L,\tau^1,\tau^2}\geq \frac{C}{L^2} \left( 2n \sqrt{x} + 3n x + 3 \max \left( \frac{1- \tau_2 + \tau_1}{\tau_2 - \tau_1}~,~ \frac{\tau_2 - \tau_1}{1- \tau_2 + \tau_1}   \right) x^{2}  \right)    \right)\\
  \leq 2.77 e^{-x}\enspace.
 \end{multline}
By convention, \eqref{quantileT'_eq1} also holds for $n$ in $\lbrace 0,1 \rbrace$ such that $T'_{n,L,\tau^1,\tau^2}=0.$ We then obtain for all $n$ in $\mathbb{N}$ and  $x= \log \left(  2.77 /u  \right)$ (which satisfies $x \geq 1$ for all $u$ in $(0,1)$)
\[\mathbb{P} \left(T'_{n,L,\tau^1,\tau^2} \geq \frac{C}{L^2} \!\! \left( 5n  \log \! \left(\!  \frac{2.77 }{u} \! \right)\!\! + 3\max \left( \frac{1- \tau_2 + \tau_1}{\tau_2 - \tau_1},\frac{\tau_2 - \tau_1}{1- \tau_2 + \tau_1}   \right)    \log^2\!\! \left(\!\frac{2.77 }{u} \! \right)\!\!  \right)     \!\!  \right) \leq u \enspace.\]
This allows to end the proof.
\end{proof}

\begin{lemma}[Conditional quantile bound for $\max/\min_{\tau\in [0,1-\ell]}{N(\tau, \tau + \ell]}$] \label{bquantile_maxminNbis_u}
Let $L \geq 1$ and $0<\ell<1$. For any $n$ in $\N\setminus\{0\}$, the $(1-\alpha)$-quantile $b_{n,\ell}^+(1-\alpha)$ of the conditional distribution of $\max_{\tau\in [0,1-\ell]}{N(\tau, \tau + \ell]}$ given $N_1=n$ under $\hzero$ satisfies
$$ b_{n,\ell}^+(1- \alpha) \leq \ell n + \frac{n}{2} g^{-1} \pa{\frac{32}{n} \log \pa{\frac{320}{\alpha}}}\enspace,$$
where $g$ is defined by \eqref{defg}. Moreover, $b_{0,\ell}^+(1-\alpha)=0.$

For any $n$ in $\N$, the $\alpha$-quantile $b_{n,\ell}^-(\alpha)$ of the conditional distribution of  $\min_{\tau\in [0,1-\ell]}{N(\tau, \tau+ \ell]}$ given $N_1=n$ under $\hzero$ satisfies
$$ b_{n,\ell}^-(\alpha) \geq  \ell n - 4\sqrt{2n \log \pa{\frac{320}{\alpha}} } \enspace.$$
\end{lemma}

\begin{proof}
Let $n$ in $\mathbb{N}\setminus\set{0}$. If $U_1,\ldots,U_n$ denote independent uniform random variables on $(0,1)$, let us consider the empirical process $\mathbb{U}_n$ associated with these random variables, defined for $0 \leq t \leq 1$ by
\[ \mathbb{U}_n(t) = \sqrt{n} \pa{\frac{1}{n} \sum_{i=1}^{n} \mathds{1}_{X_i \leq t} -t}\enspace.\]
Let $\ell$ in $(0,1/2]$ and $x>0$ satisfying
\begin{equation} \label{bquantileDeltaLengthN1proofxg}
 x \geq n \ell + \frac{n}{2} g^{-1} \pa{\frac{32}{n} \log \pa{\frac{320}{\alpha}}}\enspace.
\end{equation}
Then, we may compute for all $\lambda_0$ in $(0,R)$
\begin{align*}
P_{\lambda_0} \Bigg(&\max_{t \in [0,1-\ell]} N(t,t+\ell] >x ~\Big\vert ~ N_1 =n\Bigg)\\
 &= P_{\lambda_0} \pa{\max_{t \in [0,1-\ell]} \pa{\sum_{i=1}^{n} \mathds{1}_{X_i \leq t+\ell} - \sum_{i=1}^{n} \mathds{1}_{X_i \leq t} } >x~\Big\vert ~ N_1=n} \\
&= P_{\lambda_0} \pa{\! \max_{t \in [0,1-\ell]} \sqrt{n}\pa{\! \frac{1}{n}\sum_{i=1}^{n} \mathds{1}_{X_i \leq t+\ell} -\frac{1}{n} \sum_{i=1}^{n} \mathds{1}_{X_i \leq t} -\ell \! } \! >\! \sqrt{n}\pa{\frac{x}{n}-\ell}\Big\vert  N_1=n \!} \\
&\leq P_{\lambda_0} \pa{\max_{t \in [0,1-\ell]} \vert \mathbb{U}_n(t+\ell) - \mathbb{U}_n(t) \vert > \sqrt{n}\pa{\frac{x}{n}-\ell}\Big\vert  N_1=n } \\
&\leq P_{\lambda_0} \pa{\max_{0 \leq z \leq \frac{1}{2}}\max_{t \in [0,1-z]} \vert \mathbb{U}_n(t+z) - \mathbb{U}_n(t) \vert > \sqrt{n}\pa{\frac{x}{n}-\ell}\Big\vert  N_1=n }\enspace.
\end{align*}
We shall use now an inequality for controlling the oscillations of the empirical process which is due to Mason, Shorack and Wellner, and which can be found in Chapter 14 of \cite{ShorackWellner} page 545 under the name of Inequality 1. We get
\begin{align*}
\sup_{\lambda_0 \in (0,R)} P_{\lambda_0} \pa{\max_{t \in [0,1-\ell]} N(t,t+\ell] >x ~\Big\vert ~ N_1 =n} &\leq 320 \exp \pa{-\frac{n}{32} g \pa{2\pa{\frac{x}{n}-\ell}   }} \\
 &\leq \alpha ~~\text{using (\ref{bquantileDeltaLengthN1proofxg})}\enspace.
 \end{align*}
 This yields $ b_{n,\ell}^+(1- \alpha)\leq n \ell + n g^{-1} \pa{32 \log \pa{320/\alpha}/n}/2$ for every $n$ in $\N\setminus\{0\}$, which is the first statement of Lemma \ref{bquantile_maxminNbis_u}. The fact that $b_{0,\ell}^+(1- \alpha)=0$ is obvious.
 
Now let again $n$ in $\mathbb{N}\setminus\set{0},$ $\varepsilon >0$ and $y$ in $\mathbb{R}$ satisfying 
\begin{equation} \label{bquantileDeltaLengthN1proofyg}
 y \leq n \ell - 4 \sqrt{ 2n\log \pa{\frac{320}{\alpha - \varepsilon}}} \enspace.
\end{equation}

We compute for all $\lambda_0$ in $(0,R)$
\begin{align*}
P_{\lambda_0} \Bigg(&\min_{t \in [0,1-\ell]} N(t,t+\ell] \leq  y ~\Big\vert ~ N_1 =n\Bigg) \\
&= 
 P_{\lambda_0} \pa{\min_{t \in [0,1-\ell]}  \pa{\mathbb{U}_n(t+\ell) - \mathbb{U}_n(t)}  \leq \sqrt{n}\pa{\frac{y}{n}-\ell}~\Big\vert~  N_1=n } \\
 &= 
 P_{\lambda_0} \pa{\max_{t \in [0,1-\ell]} \pa{ \mathbb{U}_n(t) - \mathbb{U}_n(t+\ell)}  \geq \sqrt{n}\pa{\ell-\frac{y}{n}}~\Big\vert~  N_1=n } \\
&\leq P_{\lambda_0} \pa{\max_{0 \leq z \leq \frac{1}{2}}\max_{t \in [0,1-z]} \pa{ \mathbb{U}_n(t) -\mathbb{U}_n(t+z)}  \geq \sqrt{n}\pa{\ell-\frac{y}{n}}~\Big\vert~  N_1=n }\enspace.
\end{align*}
Applying now the second part of  Inequality 1 page 545 in \cite{ShorackWellner}, we obtain
\begin{align*}
 \sup_{\lambda_0 \in (0,R)} P_{\lambda_0} \pa{\min_{t \in [0,1-\ell]} N(t,t+\ell]  \leq y ~\Big\vert ~ N_1 =n} &\leq 320 \exp \pa{-\frac{n}{32}\pa{\ell-\frac{y}{n}}^2   } \\
 &\leq \alpha - \varepsilon <\alpha   ~~\text{using (\ref{bquantileDeltaLengthN1proofyg})}\enspace.
 \end{align*}

This entails $  b_{n,\ell}^-(\alpha)> n \ell -4\sqrt{2n \log \pa{320/(\alpha- \varepsilon)}}$ for all $\varepsilon >0$. With $\varepsilon$ tending to $0$, we get $  b_{n,\ell}^-(\alpha) \geq n \ell -4\sqrt{2n \log \pa{320/\alpha}}$,  which leads to the second statement of Lemma \ref{bquantile_maxminNbis_u} with the obvious fact that $b_{0,\ell}^-(\alpha)=0.$
 \end{proof}

\begin{lemma}[Conditional quantile bound for $\sup_{\ell \in (0,1- \tau^{*})} S'_{\delta^*,\tau^*,\tau^*+\ell}(N)$] \label{QuantilessupShifted_u}
Let $L \geq 1$ and $n_0 \geq 1$. For all $0 \leq n \leq n_0 L$, the $(1-\alpha)$-quantile $s_{n,\delta^*,\tau^*,L}^{'+}(1-\alpha)$ of the conditional distribution of $\sup_{\ell \in (0,1- \tau^{*})}S'_{\delta^*,\tau^*,\tau^*+\ell}(N)$ given $N_1=n$ under $(H_0)$ with $S_{\delta^*,\tau^*,\tau^*+\ell}'(N)$ defined by \eqref{stat_alt5_u} satisfies
$$ s_{n,\delta^*,\tau^*,L}^{'+}(1-\alpha) \leq Q(n_0, \delta^{*}{}, \alpha)\enspace ,$$
where
$$ Q(n_0, \delta^{*}{}, \alpha) =  \log(2)  \frac{(6 n_0 + \vert \delta^{*}{} \vert) (6 n_0 + 2\vert \delta^{*}{} \vert /3)}{\delta^{*}{}^2} + \log \pa{\frac{\pi^2}{3 \alpha}} \frac{9n_0 + \vert \delta^{*}{} \vert}{3\vert \delta^{*}{} \vert} \enspace.$$

\end{lemma}

\begin{proof}

First recall that (see \eqref{stat_alt5_u})
\[S_{\delta^*,\tau^*,\tau^*+\ell}'(N)=  \mathrm{sgn}(\delta^{*}{}) \Big(N(\tau^*,\tau^*+\ell] - \ell N_1\Big) - \vert \delta^{*}{} \vert \ell (1-\ell)L/2 \enspace.\]

Since $N$ is an homogeneous Poisson process, the processes $(N_1, N(\tau^*,\tau^*+\ell])_{\ell \in (0,1-\tau^*]}$ and $(N_1, N(0,\ell])_{\ell \in (0,1-\tau^*]}$ have the same finite dimensional law and since $N$ is a right continuous process, we get that $\sup_{\ell \in (0,1- \tau^{*})} S'_{\delta^*,\tau^*,\tau^*+\ell}(N)$ is distributed as 
 $$ \sup_{\ell \in (0,1-\tau^*)} \left(  \mathrm{sgn}\pa{\delta^*}( N(0,\ell] - \ell N_1 ) - \frac{\vert \delta^{*}{} \vert}{2}\ell(1-\ell)L    \right) \enspace.$$ 

 As seen above, for $n\geq 1$ and conditionally on the event $\{N_1 =n\}$, the points of the process $N$ obey the same law as a $n$-sample $(U_1,\ldots,U_n)$ of i.i.d. random variables uniformly distributed on $(0,1)$. Therefore, considering the empirical distribution function $F_n$ associated with this sample defined  for $0 \leq t \leq 1$ by
\[ F_n(t) = \frac{1}{n} \sum_{i=1}^{n} \mathds{1}_{X_i \leq t}\enspace,\]
conditionally on the event $\{N_1 =n\}$, $\sup_{\ell \in (0,1- \tau^{*})}S_{\delta^*,\tau^*,\tau^*+\ell}'(N)$ is distributed as
$$\sup_{\ell \in (0,1- \tau^{*})} \pa{\mathrm{sgn}(\delta^{*}{}) \big(n F_n(\ell)- n\ell \big) - \frac{\vert \delta^{*}{} \vert}{2} \ell (1-\ell)L}\enspace.$$

Notice first that for all $n \geq 1$ and for all $x>0$
\begin{multline*}
 \mathbb{P} \pa{ \sup_{\ell \in (0,1- \tau^{*})} \pa{ \mathrm{sgn}\pa{\delta^*}( n F_n(\ell)- n\ell) - \frac{\vert\delta^{*}{}\vert}{2} \ell (1-\ell)L} > x } \\
  \leq \mathbb{P} \pa{ \sup_{\ell \in (0,1)} \pa{ \mathrm{sgn}\pa{\delta^*}( n F_n(\ell)- n\ell)  - \frac{\vert\delta^{*}{}\vert}{2} \ell (1-\ell)L} > x } \enspace.
\end{multline*}
Moreover, since the process $\big( - (nF_n(\ell) -n \ell)    \big)_{\ell \in (0,1)}$ has the same distribution as the process $\big( nF_n(1-\ell) -n (1-\ell)    \big)_{\ell \in (0,1)}$,  one has when $ \delta^* <0$
\begin{align*}
 \mathbb{P} \Bigg( \sup_{\ell \in (0,1)} &\pa{ \mathrm{sgn}\pa{\delta^*}( n F_n(\ell)- n\ell )  - \frac{\vert \delta^{*}{} \vert}{2} \ell (1-\ell)L} > x \Bigg)\\
 &= \mathbb{P} \pa{ \sup_{\ell \in (0,1)} \pa{  n F_n(1-\ell)- n(1-\ell)  - \frac{\vert \delta^{*}{} \vert}{2} \ell (1-\ell)L} > x } \\
 &= \mathbb{P} \pa{ \sup_{\ell \in (0,1)} \pa{  n F_n(\ell)- n\ell  - \frac{\vert \delta^{*}{} \vert}{2} \ell (1-\ell)L} > x }\enspace.
 \end{align*}

Therefore, whatever the sign of $\delta^*$,
\begin{multline*}
 \mathbb{P} \Bigg(\sup_{\ell \in (0,1- \tau^{*})} \pa{ \mathrm{sgn}\pa{\delta^*}( n F_n(\ell)- n\ell) - \frac{\vert\delta^{*}{}\vert}{2} \ell (1-\ell)L} > x \Bigg) \\
 \leq \mathbb{P} \pa{ \sup_{\ell \in (0,1)} \pa{ n F_n(\ell)- n\ell - \frac{\vert\delta^{*}{}\vert}{2} \ell (1-\ell)L} > x }\enspace,\\
\end{multline*}
and
\begin{multline}\label{QuantilessupShifted_uProof0}
 \mathbb{P} \Bigg(\sup_{\ell \in (0,1- \tau^{*})} \pa{ \mathrm{sgn}\pa{\delta^*}( n F_n(\ell)- n\ell) - \frac{\vert\delta^{*}{}\vert}{2} \ell (1-\ell)L} > x \Bigg) \\
 \leq  \mathbb{P} \pa{ \sup_{\ell \in (0,1/2]} \pa{  n F_n(\ell)- n\ell  - \frac{\vert \delta^{*}{} \vert}{2} \ell (1-\ell)L} > x } \\ 
 + \mathbb{P} \pa{ \sup_{\ell \in (1/2,1)} \pa{  n F_n(\ell)- n\ell  - \frac{\vert \delta^{*}{} \vert}{2} \ell (1-\ell)L} >x }\enspace.
\end{multline}
Let us now prove that on the one hand
\begin{equation} \label{QuantilessupShifted_uProof1}
 \mathbb{P} \pa{ \sup_{\ell \in (0,1/2]} \pa{  n F_n(\ell)- n\ell  - \frac{\vert \delta^{*}{} \vert}{2} \ell (1-\ell)L} > Q(n_0, \delta^{*}{}, \alpha) }\leq \frac{\alpha}{2}  \enspace,
 \end{equation}
and on the other hand
\begin{equation} \label{QuantilessupShifted_uProof2}
 \mathbb{P} \pa{ \sup_{\ell \in (1/2,1)} \pa{  n F_n(\ell)- n\ell  - \frac{\vert \delta^{*}{} \vert}{2} \ell (1-\ell)L} > Q(n_0, \delta^{*}{}, \alpha) }\leq \frac{\alpha}{2}  \enspace.
 \end{equation}

Let $1 \leq n \leq n_0 L$. Set $C(\delta^{*}{},n_0)=|\delta^{*}{}|/\left(6 n_0\right)$ and $C'(\delta^{*}{},n_0,n)=\left\lfloor \exp \pa{\frac{3n}{2}\frac{C(\delta^{*}{},n_0)^2}{3+2C(\delta^{*}{},n_0)}}\right\rfloor$.

For all $k$ in $\left\{0,\ldots, C'(\delta^{*}{},n_0,n)-1 \right\}$, we define
$$\ell_k=\frac{3+2C(\delta^{*}{},n_0)}{3C(\delta^{*}{},n_0)^2} \frac{\log(k+1)}{n} \enspace,$$
and 
$$\ell_{C'(\delta^{*}{},n_0,n)}=\pa{ \frac{3+2C(\delta^{*}{},n_0)}{3C(\delta^{*}{},n_0)^2} \frac{\log(C'(\delta^{*}{},n_0,n)+1)}{n}} \wedge 1 \enspace.$$
Notice that with such definitions,  $\ell_{k}$ belongs to $[0,1/2]$ for all $k$ in $\left\{0,\ldots, C'(\delta^{*}{},n_0,n)-1 \right\}$ and $\ell_{C'(\delta^{*}{},n_0,n)}$ belongs to $(1/2,1]$.

Applying Bernstein's inequality, as stated in equation (2.6) in \cite{BercuDelyonRio}[page 12], one obtains for every $x>0$ and every  $k$ in $\{1,\ldots,C'(\delta^{*}{},n_0,n)\}$
\[\P\left(n F_n(\ell_k) > n\ell_{k} +\sqrt{2n\ell_{k}(1-\ell_{k})(x+2\log k)}+ \frac{x+2\log k}{3} \right)\leq \frac{e^{-x}}{k^2}\enspace.\]

A union bound therefore gives for every $x>0$
\begin{multline*} \P\left(\forall k\in\{1,\ldots,C'(\delta^{*}{},n_0,n)\},\ n F_n(\ell_k) \leq n\ell_{k} +\sqrt{2n\ell_{k}(1-\ell_{k})(x+2\log k)}+ \frac{x+2\log k}{3}\right)\\
\geq 1-\frac{\pi^2}{6}e^{-x}\enspace.\end{multline*}
Using the inequality $\sqrt{2ab}\leq aC(\delta^{*}{},n_0) +{b}/(2C(\delta^{*}{},n_0))$ and the fact that $\ell \mapsto n F_n(\ell)$ is nondecreasing, we get
\begin{multline*}
\P\Bigg(\forall k\in\{1,\ldots,C'(\delta^{*}{},n_0,n)\},\ \forall \ell \in (\ell_{k-1},\ell_{k}],\ n F_n(\ell) \leq n\ell_{k}
+C(\delta^{*}{},n_0) n\ell_{k}(1-\ell_{k}) \\+\frac{x+2\log k}{2C(\delta^{*}{},n_0)}+ \frac{x+2\log k}{3}\Bigg)\geq 1-\frac{\pi^2}{6} e^{-x}\enspace.
\end{multline*}
Therefore, with probability larger than $1-\pi^2 e^{-x}/6$, for all $k$ in $\{1,\ldots,C'(\delta^{*}{},n_0,n)\}$ and all $\ell$ in $(\ell_{k-1},\ell_{k}]$,
\begin{align*}
n F_n(\ell)
\leq &n\ell_{k} +C(\delta^{*}{},n_0) n \ell_{k}(1-\ell_{k}) +x \frac{3+2C(\delta^{*}{},n_0)}{6C(\delta^{*}{},n_0)}+\log k\frac{3+2C(\delta^{*}{},n_0)}{3C(\delta^{*}{},n_0)}\\
 = &n\ell_{k} +C(\delta^{*}{},n_0) n \ell_{k}(1-\ell_{k}) +x \frac{3+2C(\delta^{*}{},n_0)}{6C(\delta^{*}{},n_0)}+ C(\delta^{*}{},n_0) n\ell_{k-1}\\ 
 < &n\ell_{k} +C(\delta^{*}{},n_0) n \ell_{k}(1-\ell_{k}) +x \frac{3+2C(\delta^{*}{},n_0)}{6C(\delta^{*}{},n_0)}+ C(\delta^{*}{},n_0) n \ell\\
 = &n \ell  +C(\delta^{*}{},n_0) n \ell(1-\ell) + C(\delta^{*}{},n_0) n\ell+ n(\ell_{k}-\ell) +C(\delta^{*}{},n_0) n \left( \ell_{k}(1-\ell_{k}) - \ell(1-\ell)\right) \\ 
 &+x \frac{3+2C(\delta^{*}{},n_0)}{6C(\delta^{*}{},n_0)}\\
  \leq &n \ell +C(\delta^{*}{},n_0) n \ell(1-\ell) + C(\delta^{*}{},n_0) n\ell + n(\ell_{k}-\ell) +C(\delta^{*}{},n_0) n (1-\ell)\left( \ell_{k}- \ell\right)\\
  &+  x \frac{3+2C(\delta^{*}{},n_0)}{6C(\delta^{*}{},n_0)}\\
< &n \ell +C(\delta^{*}{},n_0) n \ell(1-\ell) + C(\delta^{*}{},n_0) n\ell+ (1+C(\delta^{*}{},n_0)(1-\ell))  \frac{3+2C(\delta^{*}{},n_0)}{3C(\delta^{*}{},n_0)^2}\log 2 \\ 
& +  x \frac{3+2C(\delta^{*}{},n_0)}{6C(\delta^{*}{},n_0)}
\enspace.
\end{align*}
Finally, we have proved that with probability larger than $1-\pi^2 e^{-x}/6$, for all $\ell$ in $(0,1/2]$,
$$n F_n(\ell) \leq n \ell +3 C(\delta^{*}{},n_0) n_0 L \ell(1-\ell) + (1+C(\delta^{*}{},n_0))  \frac{3+2C(\delta^{*}{},n_0)}{3C(\delta^{*}{},n_0)^2}\log 2 +  x \frac{3+2C(\delta^{*}{},n_0)}{6C(\delta^{*}{},n_0)}\enspace,$$
hence
\begin{equation*}
 \mathbb{P} \pa{ \sup_{\ell \in (0,1/2]} \pa{  n F_n(\ell)- n\ell  - \frac{\vert \delta^{*}{} \vert}{2} \ell (1-\ell)L} \leq Q(n_0, \delta^{*}{},\alpha) } \geq 1-\frac{\alpha}{2} \enspace,
 \end{equation*}
 that is \eqref{QuantilessupShifted_uProof1}.
 
Define now $\ell'_k=1-\ell_k \enspace$ for all $k$ in $\left\{0,\ldots, C'(\delta^{*}{},n_0,n)\right\}$
and notice that $\ell'_{k}$ belongs to $[1/2,1]$ for all $k$ in $\left\{0,\ldots, C'(\delta^{*}{},n_0,n)-1 \right\}$ and  $\ell'_{C'(\delta^{*}{},n_0,n)}$ belongs to $[0,1/2).$

Applying Bernstein's inequality again, one obtains for every $x>0$
\begin{multline*}
\P\Bigg(\forall k\in\{1,\ldots,C'(\delta^{*}{},n_0,n)\},\  n F_n(\ell'_{k-1})  \leq n\ell'_{k-1} +\sqrt{2n \ell'_{k-1} (1-\ell'_{k-1} )(x+2\log k)}\\
+ \frac{x+2\log k}{3}\Bigg)\geq 1-\frac{\pi^2}{6}e^{-x}\enspace.
\end{multline*}

With the same computations as in the above case, we get 
\begin{multline*}
\P\Bigg(\forall k\in\{1,\ldots,C'(\delta^{*}{},n_0,n)\}, \forall \ell \in [\ell'_k, \ell'_{k-1}),\ n F_n(\ell ) \leq n \ell'_{k-1} +C(\delta^{*}{},n_0) n \ell'_{k-1}(1-\ell'_{k-1})\\
 +\frac{x+2\log k}{2C(\delta^{*}{},n_0)}+ \frac{x+2\log k}{3}\Bigg) \geq 1-\frac{\pi^2}{6}e^{-x}\enspace.
\end{multline*}
Therefore, with probability larger than $1-\pi^2 e^{-x}/6$, for all $k$ in $\{1,\ldots,C'(\delta^{*}{},n_0,n)\}$ and all $\ell$ in $[\ell'_{k},\ell'_{k-1} ),$
\begin{align*}
n F_n(\ell)
\leq &n \ell'_{k-1}  +C(\delta^{*}{},n_0) n \ell'_{k-1}  (1-\ell'_{k-1} ) +x \frac{3+2C(\delta^{*}{},n_0)}{6C(\delta^{*}{},n_0)}+\log k\frac{3+2C(\delta^{*}{},n_0)}{3C(\delta^{*}{},n_0)}\\
 = &n \ell'_{k-1}  +C(\delta^{*}{},n_0) n \ell'_{k-1}  (1-\ell'_{k-1} ) +x \frac{3+2C(\delta^{*}{},n_0)}{6C(\delta^{*}{},n_0)}+ C(\delta^{*}{},n_0) n(1-\ell'_{k-1})\\ 
< &n \ell'_{k-1}  +C(\delta^{*}{},n_0) n \ell'_{k-1} (1-\ell'_{k-1} ) +x \frac{3+2C(\delta^{*}{},n_0)}{6C(\delta^{*}{},n_0)}+ C(\delta^{*}{},n_0) n(1-\ell)\\
 < &n \ell +C(\delta^{*}{},n_0) n \ell(1-\ell) + C(\delta^{*}{},n_0) n(1-\ell)+ n(\ell'_{k-1} -\ell) +C(\delta^{*}{},n_0) n (1-\ell)\left( \ell'_{k-1} - \ell\right) \\
&+  x \frac{3+2C(\delta^{*}{},n_0)}{6C(\delta^{*}{},n_0)}\\
& \leq n \ell +C(\delta^{*}{},n_0) n \ell(1-\ell) + C(\delta^{*}{},n_0) n(1-\ell)+ (1+C(\delta^{*}{},n_0)(1-\ell))  \frac{3+2C(\delta^{*}{},n_0)}{3C(\delta^{*}{},n_0)^2}\log 2 \\ 
&+  x \frac{3+2C(\delta^{*}{},n_0)}{6C(\delta^{*}{},n_0)}\enspace.
\end{align*}
Finally, we have proved that with probability larger than $1-\pi^2 e^{-x}/6$, for all $\ell$ in $[1/2,1)$,
$$n F_n(\ell) \leq n \ell +3 C(\delta^{*}{},n_0) n_0 L \ell(1-\ell) + (1+C(\delta^{*}{},n_0))  \frac{3+2C(\delta^{*}{},n_0)}{3C(\delta^{*}{},n_0)^2}\log 2 +  x \frac{3+2C(\delta^{*}{},n_0)}{6C(\delta^{*}{},n_0)}\enspace,$$
hence
\begin{equation*}
 \mathbb{P} \pa{ \sup_{\ell \in [1/2,1)} \pa{  n F_n(\ell)- n\ell  - \frac{\vert \delta^{*}{} \vert}{2} \ell (1-\ell)L} \leq Q(n_0, \delta^{*}{},\alpha) } \geq 1-\frac{\alpha}{2} \enspace,
 \end{equation*}
 that is \eqref{QuantilessupShifted_uProof2}.
Combined with \eqref{QuantilessupShifted_uProof0}, the equations \eqref{QuantilessupShifted_uProof1} and \eqref{QuantilessupShifted_uProof2} conclude the proof using the obvious fact that $s_{0,\delta^*,\tau^*,L}^{'+}(1-\alpha) =0$.
 \end{proof}

\begin{lemma}[Conditional quantile bound for $\big| S'_{\tau_1,\tau_2}(N)\big|$] \label{QuantilesAbsS_u}
Let $L \geq 1$, $u$ in $(0,1),$ $\tau_1$ and $\tau_2$ such that $0 < \tau_1 < \tau_2 \leq 1$ and $n$ in $\mathbb{N}.$ The $(1-u)$-quantile $s'_{n,\tau_1,\tau_2}(1-u)$ of the conditional distribution of $\big| S'_{\tau_1,\tau_2}(N)\big| $ given $N_1=n$ under $(H_0)$ with $S'_{\tau_1,\tau_2}(N)$ defined by \eqref{stat_alt6_N1_u} satisfies
\[s'_{n,\tau_1,\tau_2}(1-u) \leq \frac{2}{3} \log(2 / u) + \sqrt{n(\tau_2 - \tau_1) (1- \tau_2 + \tau_1)} \sqrt{2 \log(2 / u)} \enspace.\]
\end{lemma}

\begin{proof}
Let $n \geq 1$. Under $\hzero$ and conditionally on $N_1=n$, $N(\tau_1, \tau_2] $ follows a binomial distribution with parameters $(n, \tau_2 - \tau_1)$ and then $S'_{\tau_1, \tau_2}(N)=N(\tau_1 ,\tau_2] - n (\tau_2 - \tau_1) $. Applying Bennett's inequality as stated in Theorem 2.28 in \cite{BercuDelyonRio}, we obtain for all $x>0$
\begin{multline*}
\sup_{\lambda_0 \in \mathcal{S}^u_0[R]}P_{\lambda_0} \pa{\big| S'_{\tau_1,\tau_2}(N)\big|>x ~ \vert N_1=n  }\\\leq 2\exp \left(  -n (\tau_2 - \tau_1) (1- \tau_2 + \tau_1)  g \left(  \frac{x}{n (\tau_2 - \tau_1) (1- \tau_2 + \tau_1)}  \right)  \right),
\end{multline*}
where $g(y)= (1+y) \log (1+y)-y$ for all $y \geq 0.$
It directly follows that
 \[s'_{n,\tau_1,\tau_2}(1-u) \leq n (\tau_2 - \tau_1) (1-\tau_2 + \tau_1) g^{-1}\left( \frac{\log(2 / u)}{n (\tau_2 - \tau_1) (1- \tau_2 + \tau_1)} \right) \enspace.\]
 The inequality $g^{-1}(y) \leq 2 y/3 + \sqrt{2y}$ for all $y \geq 0$ and the fact that $s'_{0,\tau_1,\tau_2}(1-u)=0$ allow to conclude.

\end{proof}

\begin{lemma}[Conditional quantile bound for $\sup_{\tau \in (0,1)} S'_{\delta^*,\tau,1}(N)$] \label{QuantilessupShifted_ubis}
Let $L \geq 1$ and $n_0 \geq 1$. For all $0 \leq n \leq n_0 L$, the $(1-\alpha)$-quantile  $s_{n,\delta^*,L}^{'+}(1-\alpha)$ of the conditional distribution of $\sup_{\tau \in (0,1)} S'_{\delta^*,\tau,1}(N)$ given $N_1=n$ under $(H_0)$ with $S'_{\delta^*,\tau,1}(N)$ defined by \eqref{stat_alt5_u} satisfies
\[ s_{n,\delta^*,L}^{'+}(1-\alpha) \leq Q(n_0, \delta^{*}{}, \alpha)\enspace ,\]
where $ Q(n_0, \delta^{*}{}, \alpha)$ is defined in Lemma \ref{QuantilessupShifted_u}.
\end{lemma}

\begin{proof}
Notice first that
$\sup_{\tau \in (0,1)} S'_{\delta^*,\tau,1}(N)$ is distributed as 
 $$ \sup_{\tau \in (0,1)} \left(  \mathrm{sgn}\pa{\delta^*}( N(0,1-\tau] - (1-\tau) N_1 ) - \frac{\vert \delta^{*}{} \vert}{2}\tau(1-\tau)L    \right) \enspace.$$ 
Now, recall that for $n\geq 1$ and conditionally on the event $\{N_1 =n\}$, the points of the process $N$ obey the same law as a $n$-sample $(U_1,\ldots,U_n)$ of i.i.d. random variables uniformly distributed on $(0,1)$ and consider the empirical distribution function $F_n$ associated with this sample defined  for $0 \leq t \leq 1$ as in the proof of Lemma  \ref{QuantilessupShifted_u}. Then, conditionally on the event $\{N_1 =n\}$, $\sup_{\tau \in (0,1)} S'_{\delta^*,\tau,1}(N)$ is distributed as
$$\sup_{\tau \in (0,1)} \pa{\mathrm{sgn}(\delta^{*}{}) \big(n F_n(1-\tau)- n(1-\tau) \big) - \frac{\vert \delta^{*}{} \vert}{2} \tau (1-\tau)L}\enspace.$$
Since the process $(-(n F_n(1-\tau)- n(1-\tau)))_{\tau \in (0,1)}$ has the same distribution as the process $(nF_n(\tau) - n \tau)_{\tau \in (0,1)}$, whatever the sign of $\delta^*$, $\sup_{\tau \in (0,1)} S'_{\delta^*,\tau,1}(N)$ is distributed as
$$\sup_{\tau \in (0,1)} \pa{ n F_n(\tau)- n\tau - \frac{ \vert \delta^{*}{} \vert }{2} \tau (1-\tau)L}\enspace.$$
The proof now follows the same line as the proof of Lemma \ref{QuantilessupShifted_u}, just replacing $\ell$ by $\tau$ and using the fact that $s_{0,\delta^*,L}^{'+}(1-\alpha)=0$.
\end{proof}

\section*{Acknowledgements}
The authors want to thank Nicolas Verzelen for fruitful discussions.

The PhD grant of Fabrice Grela is funded by the french Région Bretagne and Direction Générale des Armées.

\bibliographystyle{apalike}
\bibliography{biblio_changepointPP}

\end{document}